\documentclass[a4paper,10pt,reqno]{amsart}

\usepackage[utf8]{inputenc}
\usepackage[T1]{fontenc}
\usepackage{amsmath,amssymb}
\usepackage[english]{babel}
\usepackage{mathrsfs}
\usepackage{tikz-cd}
\usepackage{verbatim}
\usepackage{color}
\usepackage[all]{xy}

\usepackage[shortlabels]{enumitem}
\setlist[enumerate]{label=\rm{(\arabic*)}}
\setlist[enumerate,2]{label=\rm{(\roman*)}}
\setlist[itemize]{label=\raisebox{0.25ex}{\tiny$\bullet$}}
\usepackage[backref, colorlinks, linktocpage, citecolor = blue, linkcolor = blue]{hyperref}

\usepackage[left=2cm,right=2cm,top=3cm,bottom=3cm]{geometry}

\theoremstyle{plain}
\newtheorem{theorem}{Theorem}[section]
\newtheorem{corollary}[theorem]{Corollary}
\newtheorem{proposition}[theorem]{Proposition}
\newtheorem{lemma}[theorem]{Lemma}

\newtheorem{notation}[theorem]{Notations}

\newtheorem*{acknowledgments}{Acknowledgments}

\theoremstyle{definition}
\newtheorem{definition}[theorem]{Definition}
\newtheorem{remark}[theorem]{Remark}

\theoremstyle{plain}
\newtheorem{theoremA}{Theorem}
\newtheorem{propositionA}[theoremA]{Proposition}
\newtheorem{corollaryA}[theoremA]{Corollary}

\def\h{\mathrm{h}}
\def\H{\mathrm{H}}
\def\Mat{\mathrm{Mat}}
\def\GL{\mathrm{GL}}
\def\PGL{\mathrm{PGL}}
\def\Autz{\mathrm{Aut}^\circ}
\def\Aut{\mathrm{Aut}}
\def\Bir{\mathrm{Bir}}
\def\AA{\mathbb{A}}
\def\PP{\mathbb{P}}
\def\FF{\mathbb{F}}
\def\O{\mathcal{O}}
\def\Pic{\mathrm{Pic}}

\def\kk{\mathbf{k}}

\def\AAA{\mathcal{A}}

\def\D{\mathcal{D}}

\def\E{\mathcal{E}}
\def\L{\mathcal{L}}
\def\seg{\mathfrak{S}}
\begin{document}
	
	\title{Automorphism groups of $\PP^1$-bundles over \\geometrically ruled surfaces}
	\author{Pascal Fong}
	\address{Leibniz Universität Hannover, Institut für Algebraische Geometrie, Welfengarten 1, 30167 Hannover, Deutschland}
	\email{fong@math.uni-hannover.de}

	\begin{abstract}
		We classify the pairs $(X,\pi)$, where $\pi\colon X\to S$ is a $\PP^1$-bundle over a non-rational geometrically ruled surface $S$ and $\Autz(X)$ is \emph{relatively maximal}, i.e., maximal with respect to the inclusion in the group $\mathrm{Bir}(X/S)$. The results hold over any algebraically closed field of characteristic zero.
	\end{abstract}
	\maketitle	
	
	\tableofcontents
	
	\section{Introduction}
	
	The classification of algebraic subgroups of the Cremona groups was initiated by the Italian school of Algebraic Geometry. In \cite{Enriques}, Enriques classified the maximal connected algebraic subgroups of $\Bir(\PP^2)$ over $\mathbb{C}$, showing that they are conjugate to either $\Aut(\PP^2) = \PGL_3(\mathbb{C})$ or to $\Autz(\FF_b)$ for $b\neq 1$, where $\FF_b$ denotes the $b$-th Hirzebruch surface. Later with Fano, they stated the classification of maximal connected algebraic subgroups of $\Bir(\PP^3)$ over $\mathbb{C}$; see \cite{EnriquesFano}.
	A proof of their classification for $\Bir(\PP^3)$ was provided by Umemura in \cite{Umemura80,Umemura82a,Umemura82b,Umemura85}, and recently recovered and generalized to any algebraically closed field of characteristic zero by Blanc-Fanelli-Terpereau in \cite{BFT,BFT2}. 
	
	It follows from these classifications that $\Bir(\PP^2)$ and $\Bir(\PP^3)$ satisfy a remarkable property:
	every connected algebraic subgroup is contained in a maximal one. In higher dimensions, no complete classification is known. However, in the pioneering article \cite{Demazure}, which is considered as the starting point of toric geometry, Demazure classified the algebraic subgroups of $\Bir(\PP^n)$ containing a torus of maximal rank. 
	
	\bigskip
	
	If $C$ is a smooth projective curve of genus $g(C)\geq 1$, the maximal connected algebraic subgroups of $\Bir(C\times \PP^1)$ are classified in \cite{Fong}. An unexpected phenomenon arises: there exist \emph{unbounded connected algebraic subgroups} in $\Bir(C\times \PP^1)$, meaning connected algebraic subgroups that are not contained in any maximal one. This result was later generalized in higher dimensions; see \cite{fosok}, and \cite{fanelli_floris_zimmermann,kollar2024} for the rational case.
	
	\bigskip
	
	\emph{From now on, unless otherwise stated, $C$ denotes a smooth projective curve of genus $g \geq 1$.}
	
	\subsection{Automorphism groups of geometrically ruled surfaces}
	
	Up to conjugacy, the maximal connected algebraic subgroups of $\Bir(C\times \PP^1)$ are of the form $\Autz(S)$, where $S$ is one of the following \emph{geometrically ruled surfaces} (see Definition \ref{def:ruled}):
	\begin{enumerate}
		\item The product $C\times \PP^1$,
		\item The indecomposable geometrically ruled surface $\AAA_0$,
		\item The indecomposable geometrically ruled surface $\AAA_1$,
		\item The geometrically ruled surfaces $\PP(\O_C\oplus \L)$, where $\L$ is a non-trivial line bundle of degree zero.
	\end{enumerate}
	 The latter three cases occur only when $g=1$, i.e. $C$ is an elliptic curve. The \emph{Atiyah's geometrically ruled surfaces} $\AAA_0$ and $\AAA_1$ are the two indecomposable\footnote{This means that there do not exist two disjoint sections of the $\PP^1$-bundle $S\to C$.} $\PP^1$-bundles over the elliptic curve $C$ (see Section \ref{ruled} for definitions and properties). 
	
	\bigskip
	
	Using an explicit description of the automorphism groups of the above geometrically ruled surfaces, which is obtained by Maruyama in \cite{Maruyama}, it follows that the maximal connected algebraic subgroups of $\Bir(C\times \PP^1)$ have dimension at most $4$, while there exist unbounded connected algebraic subgroups of arbitrary large dimension. This contrasts with Enriques' classification, as there are only decomposable $\PP^1$-bundles over $\PP^1$, whose automorphism groups produce maximal connected algebraic subgroups of $\Bir(\PP^2)$ of arbitrarily large dimension. 
	
	\bigskip
	
	This article is motivated by the classification of maximal connected algebraic subgroups of $\Bir(C\times \PP^2)$. Inspired by the strategy used in \cite{BFT}, we focus on the automorphism groups of $\PP^1$-bundles over non-rational geometrically ruled surfaces. Their approach relies on the \emph{Minimal Model Program} (MMP) and the \emph{Sarkisov Program}, and the arguments developed in this article likewise rely on these theories, which hold in characteristic zero.
	
	\bigskip
	
	\emph{Accordingly, we assume throughout that $\kk$ is algebraically closed of characteristic zero.}
	
	\subsection{Reduction via the Minimal Model Program}
	
	Starting with a connected algebraic subgroup $G\subset \Bir(C\times \PP^2)$, the regularization theorem of Weil (see \cite{Weil}, and \cite{Sumihiro1,Sumihiro2, Brionnormal} for refinements) provides a smooth projective threefold that is $G$-birationally equivalent to $C\times \PP^2$, on which $G$ acts regularly and faithfully. Running an MMP conjugates $G$ to a connected algebraic subgroup of $\Autz(X)$, where $X$ is a conic bundle over a birationally ruled surface or a Mori del Pezzo fibration over $C$.  To prove or disprove the maximality of $\Autz(X)$ as a connected algebraic subgroup of $\Bir(X)$, we use the Sarkisov Program, which was developed by Corti in dimension three, extended to higher dimensions by Hacon and McKernan, and adapted to the equivariant setting by Floris (see \cite{Corti,HM13, Floris}).
	
	\bigskip
	
	The strategy outlined above can be adapted to study algebraic subgroups of $\Bir(X)$ in a more general setting, e.g. in \cite{Blanc,fong2} when $X$ is a surface of negative Kodaira dimension but $G$ is not necessarily connected, or in \cite{RobayoZimmermann, SchneiderZimmermann} when $X$ is a rational surface defined over a non-algebraically closed field. Finite subgroups of groups of birational transformations have also been of central interest. For finite subgroups of $\Bir(\PP^2)$, see e.g., \cite{Blancthesis, DI}, and \cite{Yasinsky,CMYZ} for non-closed fields. In higher dimension, research on finite subgroups of $\Bir(\PP^3)$ is still ongoing, see e.g., \cite{Prokhorov1,Prokhorov2,PS2,PS1,CheltsovShramov}, or also \cite{CSbook} and the references therein. When $X$ is a conic bundle over a non-uniruled variety, finite subgroups of $\Bir(X)$ have also been studied through the Jordan property by Bandman and Zarhin, see e.g., \cite{BZ,BZ2}.
	
	\subsection{The main results}
	
	In \cite{BFT}, the notion of \emph{relatively maximal automorphism group} was introduced to restrict the classification to the category of $\PP^1$-bundles over geometrically ruled surfaces, excluding the equivariant birational maps whose target is a rational Mori del Pezzo fibration or a rational Fano threefold. The automorphism groups of rational Mori del Pezzo fibrations were studied in a second phase in \cite{BFT2}, using different techniques from those developed in \cite{BFT}. 
	
	\bigskip
	
	We combine the techniques developed in \cite{BFT} and the classification of maximal connected algebraic subgroups of $\Bir(C\times \PP^1)$ obtained in \cite{Fong}, to classify the pairs $(X,\pi)$, where $\pi\colon X\to S$ is a $\PP^1$-bundle over a geometrically ruled surface $S$ and $\Autz(X)$ is \emph{relatively maximal} (see Definition \ref{def}). We obtain:

	\bigskip
	
	\begin{theoremA}\label{thmA}
		Let $\kk$ be an algebraically closed field of characteristic zero. Let $C$ be a smooth projective curve of genus $g\geq 1$, and let $\tau \colon S\to C$ and $\pi\colon X\to S$ be $\PP^1$-bundles. 
		
		\bigskip
		
		\begin{enumerate}
			\item[$\mathrm{(I)}$] If $g=1$, then the following hold:
			
			\medskip
			
			\begin{enumerate}
				\item\label{thmA.i} Let $X= S'\times \PP^1$, where $\tau'\colon S'\to C$ is one of the following geometrically ruled surfaces: \medskip
				\begin{enumerate}
					\item[(a)] $C\times \PP^1$,
					\item[(b)] $\AAA_0$,
					\item[(c)] $\AAA_1$,
					\item[(d)] $\PP(\O_C\oplus \O_C(D))$, where $D$ is a non-trivial divisor of degree zero.
				\end{enumerate}
				\medskip
				In each case, $X$ admits exactly two $\PP^1$-bundle structures for which $\Autz(X)$ is relatively maximal: 
				\begin{itemize}
					\item $S=S'$ and $\pi$ is the trivial $\PP^1$-bundle over $S$.
					\item $S=C\times \PP^1$ and $\pi=\tau' \times \mathrm{id}_{\PP^1}$; which is a decomposable $\PP^1$-bundle if $S'=C\times \PP^1$ or $\PP(\O_C\oplus \O_C(D))$ (cases (a) and (d)), and it is an indecomposable one if $S'=\AAA_0$ or $\AAA_1$ (cases (b) and (c)).
				\end{itemize}
				\vspace{3pt}
				
				\item\label{thmA.iv} Let $X=\AAA_0 \times_C \AAA_1$ and $\pi$ be the projection onto the first factor, so that $S = \AAA_0$. Then $\pi$ is an indecomposable $\PP^1$-bundle and $\Autz(X)$ is relatively maximal. \vspace{3pt}
				
				\item\label{thmA.v} Let $X=\AAA_1 \times_C \AAA_1$ and $\pi$ be the projection onto either factor, so that $S = \AAA_1$. Then $\pi$ is an indecomposable $\PP^1$-bundle and $\Autz(X)$ is relatively maximal. \vspace{3pt}
				
				\item\label{thmA.vi} Let $X=\PP(\O_C\oplus \O_C(D))\times_C \AAA_1$, where $D$ is a non-trivial divisor of degree zero. In this case, $X$ admits two $\PP^1$-bundle structures: \medskip

				\begin{itemize}
					\item $S=\PP(\O_C\oplus \O_C(D))$ and $\pi$ is the projection onto the first factor, which is an indecomposable $\PP^1$-bundle, and $\Autz(X)$ is relatively maximal.
					\item $S= \AAA_1$ and $\pi$ is the projection onto the second factor, which identifies $X$ with the decomposable $\PP^1$-bundle 
					\[
					\PP(\O_{\AAA_1} \oplus \O_{\AAA_1}(\tau^*(D))),
					\]
					and $\Autz(X)$ is relatively maximal provided that $D$ is not $2$-torsion. 
				\end{itemize}
				 \vspace{3pt}
				 
				\item\label{thmA.vii} Let $X=\PP(\O_C\oplus \O_C(D))\times_C \AAA_0$, where $D$ is a non-trivial divisor of degree zero. In this case, $X$ admits two $\PP^1$-bundle structures: \medskip
				\begin{itemize}
					\item $S=\PP(\O_C\oplus \O_C(D))$ and $\pi$ is the projection onto the first factor, which is an indecomposable $\PP^1$-bundle, and $\Autz(X)$ is relatively maximal provided that $D$ has infinite order.
					\item $S= \AAA_0$ and $\pi$ is the projection onto the second factor, which identifies $X$ with the decomposable
					$\PP^1$-bundle 
					\[
					\PP(\O_{\AAA_0} \oplus \O_{\AAA_0}(\tau^*(D))),
					\]
					and $\Autz(X)$ is relatively maximal.
				\end{itemize}
				\vspace{3pt}
				\item\label{thmA.viii} Let $X=\PP(\O_C\oplus \O_C(D))\times_C \PP(\O_C\oplus \O_C(E))$, where $D,E$ are non-trivial divisors of degree zero. In this case, $X$ admits two $\PP^1$-bundles structures:
				\medskip
				\begin{itemize}
					\item $S = \PP(\O_C\oplus \O_C(D))$ and $\pi$ is the projection onto the first factor, which identifies $X$ with the decomposable
					$\PP^1$-bundle 
					\[
					\PP(\O_{\PP(\O_C\oplus \O_C(D))} \oplus \O_{\PP(\O_C\oplus \O_C(D))}(\tau^*(E))),
					\]
					and $\Autz(X)$ is relatively maximal provided that $E$ is not a multiple of $D$.
					\item $S = \PP(\O_C\oplus \O_C(E))$ and $\pi$ is the projection onto the second factor, which identifies $X$ with the decomposable
					$\PP^1$-bundle 
					\[
					\PP(\O_{\PP(\O_C\oplus \O_C(E))} \oplus \O_{\PP(\O_C\oplus \O_C(E))}(\tau^*(D))),
					\]
					and $\Autz(X)$ is relatively maximal provided that $D$ is not a multiple of $E$.
				\end{itemize}				
				 \vspace{3pt}
				
				\item\label{thmA.ix} Let $X=\PP(\O_{C\times \PP^1} \oplus \O_{C\times \PP^1}(b\sigma + \tau^*(D)))$, where $b>0$, $D\in \Pic^0(C)$ and $\sigma$ is a constant section in $C\times \PP^1$. Then $\pi$ is a decomposable $\PP^1$-bundle over $S=C\times \PP^1$ and $\Autz(X)$ is relatively maximal. Moreover, if $D$ is trivial, then $X\cong C\times \FF_b$. \vspace{3pt}
				
				\item\label{thmA.x} Let $X=\PP(\O_{\AAA_1}\oplus \O_{\AAA_1}(2\sigma+\tau^*(D)))$, where $D$ is a non-trivial $2$-divisor on $C$ (see Definition $\ref{def:(m_2^*,b)})$, and $\sigma$ is a minimal section of $\AAA_1$ such that $\O_{\AAA_1}(\sigma)\simeq \O_{\AAA_1}(1)$. Then $\pi$ is a decomposable $\PP^1$-bundle over $S=\AAA_1$ and $\Autz(X)$ is relatively maximal. \vspace{3pt}
			\end{enumerate}\vspace{3pt}
			Moreover, let $T$ be a geometrically ruled surface over $C$, and let $\pi'\colon Y\to T$ be a $\PP^1$-bundle. If $\Autz(Y)$ is relatively maximal, then $T$ is isomorphic to one of the geometrically ruled surfaces $S$ listed above, and there exists an $\Autz(Y)$-equivariant birational map $Y\dashrightarrow X$ over $S$ to one of the $\PP^1$-bundles $\pi\colon X\to S$ listed above, which conjugates $\Autz(Y)$ and $\Autz(X)$. \vspace{3pt}
			
			\item[$\mathrm{(II)}$] If $g\geq 2$, then $\Autz(X)$ is relatively maximal if and only if  $X= C\times \PP^1 \times \PP^1$ and $\pi$ is the trivial $\PP^1$-bundle over $S=C\times \PP^1$.\vspace{3pt}
		\end{enumerate}
	\end{theoremA}
	
	In Theorem \ref{thmA}, except for cases \ref{thmA.ix} and \ref{thmA.x}, the threefolds listed are fiber products of two geometrically ruled surfaces over $C$; in particular, they are equipped with two $\PP^1$-bundle structures. In case \ref{thmA.iv}, $\Autz(X)$ is relatively minimal with respect to the projection onto $\AAA_0$, but not with respect to the second projection (see Proposition \ref{A_0A_1notmax}). In particular, $\Autz(\AAA_0 \times_C \AAA_1)$ is not conjugate to a maximal connected algebraic subgroup of $\Bir(C\times \PP^2)$.
	
	If $g\geq 2$, notice that $\Autz(C\times \PP^1 \times \PP^1) = (\PGL_2(\kk))^2$ is the unique relatively maximal automorphism group. This result generalizes the classification of maximal connected algebraic subgroups of $\Bir(C\times \PP^1)$, as $\Autz(C\times \PP^1)=\PGL_2(\kk)$ is the unique such subgroup.
	
	\bigskip
	
	In Proposition \ref{propB}, we describe the automorphism groups of the $\PP^1$-bundles listed in Theorem \ref{thmA}. Blanchard's lemma (see Proposition \ref{prop:Blanchard}) induces a morphism of algebraic groups
	\[
	\pi_*\colon \Autz(X) \to \Autz(S).
	\]
	For each case, we show that $\pi_*$ is surjective and compute its kernel. This determines the dimension of $\Autz(X)$. We also describe the orbits of $\Autz(X)$.
	
	\begin{propositionA}\label{propB}
		For each $\PP^1$-bundle $\pi\colon X\to S$ in Theorem \ref{thmA}, the induced morphism of algebraic groups 
		\[
		 \pi_* \colon \Autz(X) \to \Autz(S)
		\]
		is surjective. We keep the same assumptions, numbering, and notation as in Theorem \ref{thmA}. In cases \ref{thmA.iv} - \ref{thmA.viii}, we also denote by $p_1$ and $p_2$ the projections from $X$ onto the first and second factors, respectively. Then the following statements hold:
		\medskip
		\begin{enumerate}
			\item[(I)] 
			\begin{enumerate}
				\item For each $S'$, we have $\Autz(X)\cong \Autz(S') \times \PGL_2(\kk)$. More precisely: \medskip
				\begin{enumerate}
					\item[(a)] $\Autz(X) \cong C\times \PGL_2(\kk) \times \PGL_2(\kk)$ has dimension $7$ and acts transitively on $X$. \vspace{3pt}
					\item[(b)] $\Autz(X) \cong \Autz(\AAA_0) \times \PGL_2(\kk)$ has dimension $5$ and acts with two orbits: $\sigma \times \PP^1$, where $\sigma$ is the unique minimal section of $\AAA_0$, and its complement $X\setminus (\sigma \times \PP^1)$. \vspace{3pt}
					\item[(c)] $\Autz(X) \cong \Autz(\AAA_1) \times \PGL_2(\kk)$ has dimension $4$ and acts with orbits of the form $\omega \times \PP^1$, where $\omega \subset \AAA_1$ is an $\Autz(\AAA_1)$-orbit that is mapped $4$-to-$1$ onto $C$ via $\tau'$. \vspace{3pt}
					\item[(d)] $\Autz(X) \cong \Autz(\PP(\O_C\oplus \O_C(D))) \times \PGL_2(\kk)$ has dimension $5$ and acts with three orbits: $\sigma_1\times \PP^1$, $\sigma_2\times \PP^1$ and their complement $X\setminus ((\sigma_1 \sqcup \sigma_2) \times \PP^1)$, where $\sigma_1$ and $\sigma_2$ are the two disjoint minimal sections of $\PP(\O_C\oplus \O_C(D))$.
				\end{enumerate} \medskip
				If $\pi$ is the projection onto $S'$, then $\ker(\pi_*)\cong \PGL_2(\kk)$. If $\pi=\tau' \times \mathrm{id}_{\PP^1}\colon X\to C\times \PP^1$, then
				 \[
				 \ker(\pi_*) \cong
				 \left\{
				 \begin{aligned}
				 &	\PGL_2(\kk) & \text{ in (a),} \\
				 &	\mathbb{G}_a& \text{ in (b),}  \\
				&	(\mathbb{Z}/2 \mathbb{Z})^2 &\text{ in (c),}\\ 
				 &	\mathbb{G}_m & \text{ in (d).}
				  \end{aligned}
				 \right.
				 \]
				
				\item $\Autz(X) \cong \Autz(\AAA_0) \times_{\Autz(C)} \Autz(\AAA_1)$ has dimension $2$. We denote by $\sigma \subset \AAA_0$ the unique minimal section. Then $\Autz(X)$ acts with:\medskip
				\begin{itemize}
					\item orbits of dimension one of the form $p_1^{-1}(\sigma)\cap p_2^{-1}(\omega)$, which are mapped $4$-to-$1$ onto $\sigma$,\vspace{3pt}
					\item orbits of dimension two of the form $p_1^{-1}(\AAA_0 \setminus \sigma)\cap p_2^{-1}(\omega)$, which are mapped $4$-to-$1$ onto $\AAA_0\setminus \sigma$,
				\end{itemize}
				and in both cases above, $\omega\subset \AAA_1$ is an $\Autz(\AAA_1)$-orbit.
				Moreover, for the $\PP^1$-bundle structure $\pi = p_1$, we have
				\[
				\ker(\pi_*) \cong (\mathbb{Z}/2\mathbb{Z})^2.
				\]
				
				\item $\Autz(X) \cong \Autz(\AAA_1) \times_{\Autz(C)} \Autz(\AAA_1)$ has dimension $1$. Every orbit is one-dimensional and is of the form 
				$
				p_1^{-1}(\omega_1)\cap p_2^{-1}(\omega_2),
				$
				where $\omega_1,\omega_2\subset \AAA_1$ are $\Autz(\AAA_1)$-orbits. Each orbit is mapped $4$-to-$1$ onto $\omega_1$ via $p_1$, and also $4$-to-$1$ onto $\omega_2$ via $p_2$. Moreover, for $\pi=p_1$ and $\pi=p_2$, we have
				\[
				\ker(\pi_*) \cong (\mathbb{Z}/2\mathbb{Z})^2.
				\]
				
				\item $\Autz(X) \cong \Autz(\PP(\O_C\oplus \O_C(D))) \times_{\Autz(C)} \Autz(\AAA_1)$ has dimension $2$. We denote by $\sigma_1$ and $\sigma_2$ the two disjoint minimal sections of $\PP(\O_C\oplus \O_C(D))$. Then $\Autz(X)$ acts with:
				 \medskip
				\begin{itemize}
					\item orbits of dimension one of the form $p_1^{-1}(\sigma_1)\cap p_2^{-1}(\omega)$ or $p_1^{-1}(\sigma_2)\cap p_2^{-1}(\omega)$, that are mapped $4$-to-$1$ onto their images via $p_1$ and $1$-to-$1$ via $p_2$.\vspace{3pt}
					\item orbits of dimension two of the form
					$p_1^{-1}(\sigma_1 \sqcup \sigma_2)\cap p_2^{-1}(\omega)$,\vspace{3pt}
				\end{itemize}
				and in both cases above, $\omega\subset \AAA_1$ is an $\Autz(\AAA_1)$-orbit. Moreover,
				\[
				\ker({p_1}_*) \cong (\mathbb{Z}/2\mathbb{Z})^2,~ \ker({p_2}_*) \cong \mathbb{G}_m.
				\]
				
				\item $\Autz(X) \cong \Autz(\PP(\O_C\oplus \O_C(D))) \times_{\Autz(C)} \Autz(\AAA_0)$ has dimension $3$. We denote by $s_1,s_2$ the disjoint minimal sections of $\PP(\O_C\oplus \O_C(D))$, and $\sigma$ the unique minimal section of $\AAA_0$.
				Then $\Autz(X)$ acts with:
				\medskip
				
				\begin{itemize}
					\item two one-dimensional orbits, which are $l_1 = p_1^{-1}(\sigma_1) \cap \pi^{-1}(\sigma)$ and $l_2 = p_1^{-1}(\sigma_2) \cap \pi^{-1}(\sigma)$. \vspace{3pt}
					\item three two-dimensional orbits, which are $p_1^{-1}(\sigma_1) \setminus l_1$, $p_1^{-1}(\sigma_2) \setminus l_2$, and $\pi^{-1}(\sigma)\setminus (l_1 \sqcup l_2)$.\vspace{3pt}
					\item one three-dimensional orbit, which is the complement in $X$ of the union of the orbits of dimension $\leq 2$.
				\end{itemize}
				Moreover,
				\[
				\ker({p_1}_*) \cong \mathbb{G}_a, ~ \ker({p_2}_*) \cong \mathbb{G}_m.
				\]
				 \medskip
				
				\item $\Autz(X) \cong \Autz(\PP(\O_C\oplus \O_C(D))) \times_{\Autz(C)} \Autz(\PP(\O_C\oplus \O_C(E)))$ has dimension $3$. We denote by $\sigma_1,\sigma_2$ the two disjoint minimal sections of $\PP(\O_C\oplus \O_C(D))$ and by $\sigma'_1,\sigma'_2$ the two disjoint minimal sections of $\PP(\O_C\oplus \O_C(E))$. Then $\Autz(X)$ acts with:
				\medskip
				\begin{itemize}
					\item four one-dimensional orbits, which are $l_{11} = p_1^{-1}(\sigma_1) \cap p_2^{-1}(\sigma'_1)$, 
					$l_{12} = p_1^{-1}(\sigma_1) \cap p_2^{-1}(\sigma'_2)$, 
					$l_{21} = p_1^{-1}(\sigma_2) \cap p_2^{-1}(\sigma'_1)$, 
					and $l_{22} = p_1^{-1}(\sigma_2) \cap p_2^{-1}(\sigma'_2)$. \vspace{3pt}
					\item four two-dimensional orbits, which are $p_1^{-1}(\sigma_1) \setminus (l_{11} \sqcup l_{12})$, 
					$p_1^{-1}(\sigma_2) \setminus (l_{21} \sqcup l_{22})$,
					$p_2^{-1}(\sigma'_1) \setminus (l_{11} \sqcup l_{21})$,
					$p_2^{-1}(\sigma'_2) \setminus (l_{12} \sqcup l_{22})$, \vspace{3pt}
					\item one three-dimensional orbit, which is the complement in $X$ of the union of the orbits of dimension $\leq 2$.
				\end{itemize} 
				
				Moreover, for $\pi=p_1$ and $\pi=p_2$, we have
				\[
				\ker(\pi_*) \cong \mathbb{G}_m.
				\]
				
				\medskip
				
				\item 
				If $D$ is trivial, then $\Autz(X)\cong C\times \Aut(\FF_b)$ has dimension $b+6$ and acts on $X$ with two orbits: $C\times \sigma_b$ where $\sigma_b$ denotes the minimal section of $\FF_b$, and its complement. In this case,
				\[
				\ker(\pi_*) \cong \mathbb{G}_m \rtimes \Gamma(\O_{\PP^1}(b)).
				\]
				
				If $D$ is not trivial, then $\Autz(X)$ has dimension $5$ and acts with two orbits of dimension two, which are $\Autz(X)$-invariant disjoint sections of $\pi$, and the complement of their union is a three-dimensional orbit. In this case,
				\[
				\ker(\pi_*) \cong \mathbb{G}_m.
				\]

				\item $\Autz(X)$ has dimension $2$ and acts with two orbits of dimension two, which are $\Autz(X)$-invariant disjoint sections of $\pi$, and the complement of their union is a three-dimensional orbit. Moreover,
				\[
				\ker(\pi_*) \cong \mathbb{G}_m.
				\]
			\end{enumerate}
			
		\item [(II)]  $\Autz(X) \simeq \PGL_2(\kk) \times \PGL_2(\kk)$ has dimension $6$ and acts with orbits of dimension two which are fibers of the projection $\tau\pi\colon X\to C$. Moreover,
		\[
		\ker(\pi_*) \cong  \PGL_2(\kk).
		\]
		\end{enumerate}
	\end{propositionA}
	
	\bigskip
	
	Theorem \ref{thmA} provides a list of relatively maximal automorphism groups, up to conjugacy via a square birational map. For each such conjugacy class, it gives a birational model $X$ whose $\Autz(X)$ is a representative. However, Theorem \ref{thmA} does not describe all representatives within a given class. In particular, it does not specify when two automorphism groups $\Autz(X)$ and $\Autz(Y)$ lie in the same class.
	
	\bigskip
	
	To obtain a complete classification of pairs $(X,\pi)$ for which $\Autz(X)$ is relatively maximal, we study all equivariant square birational maps starting from the $\PP^1$-bundles in Theorem \ref{thmA}. For each relatively maximal automorphism group $\Autz(X)$ in Theorem \ref{thmA}, we determine the list of $\PP^1$-bundles $\pi'\colon Y\to T$ such that $\Autz(X)$ and $\Autz(Y)$ are conjugate via a square birational map. This leads to Theorem \ref{thmB}.
	
	We also determine the pairs $(X,\pi)$ that are \emph{superstiff}. This notion was introduced by Blanc, Fanelli, and Terpereau in \cite{BFT}. It is an analogue of superrigidity in the setting of $\PP^1$-bundles over geometrically ruled surfaces: a pair $(X,\pi)$ is superstiff if $\Autz(X)$ is relatively maximal and is the only representative in its conjugacy class (see Definition \ref{def}).
	
	\begin{theoremA}\label{thmB}
		Assumptions, numberings and notations as Theorem $\ref{thmA}$, and let $\pi'\colon Y\to T$ be a $\PP^1$-bundle. If there exists an $\Autz(X)$-equivariant square birational map $X\dashrightarrow Y$, then $S \cong T$ are isomorphic as geometrically ruled surfaces over $C$. Moreover,
		the following hold:\medskip
		
		\begin{enumerate}
			\item[$\mathrm{(I)}$] If $g=1$:\medskip
			
			\begin{enumerate}
				\item[$\mathrm{(i)}$-$\mathrm{(iii)}$] The pairs $(X,\pi)$ are superstiff.\medskip
				
				\item[$\mathrm{(iv)}$] 
				
				 {\tiny$\bullet$} If $\pi$ is the projection onto $\PP(\O_C \oplus \O_C(D))$, then the pair $(X,\pi)$ is superstiff.\medskip
				
				 {\tiny$\bullet$} If $\pi$ is the projection onto $\AAA_1$, then the pair $(X,\pi)$ is not stiff: $\Autz(X)$ is conjugate to $\Autz(Y)$ via a square birational map if and only if there exists $n\in \mathbb{Z}$ such that
					\[
					Y\cong \PP(\O_{\AAA_1}\oplus \O_{\AAA_1}(4n\sigma + \tau^*(D-nD_\sigma))),
					\]
					where $\sigma$ is a minimal section of $\AAA_1$ such that $\O_{\AAA_1}(\sigma)\simeq \O_{\AAA_1}(1)$, and where $D_\sigma$ is a divisor of degree two such that every $\Autz(\AAA_1)$-orbit $\omega\subset \AAA_1$ is linearly equivalent to $4\sigma - \tau^*(D_\sigma)$. \medskip

				\item[$\mathrm{(v)}$] With respect to both projections, the pair $(X,\pi)$ is not stiff. \vspace{3pt}
					
					{\tiny$\bullet$} If $\pi$ is the projection onto $S=\PP(\O_C\oplus \O_C(D))$, then $\Autz(X)$ is conjugate to $\Autz(Y)$ via a square birational map if and only if 
					\[
					Y\cong \AAA_{(S,b,nD)}
					\]
					for some $b\geq 0$ and 
					$n\in \{0,\ldots,b\}$ (see Lemma $\ref{lem:induction}$). \vspace{3pt}
					
					{\tiny$\bullet$} If $\pi$ is the projection onto $\AAA_0$, then $\Autz(X)$ is conjugate to $\Autz(Y)$ via a square birational map if and only if there exists $b\in \mathbb{Z}$ such that $$Y\cong \PP(\O_{\AAA_0}\oplus \O_{\AAA_0}(b\sigma +\tau^*(D))),$$
					where $\sigma$ is the minimal section of $\AAA_0$. \medskip
				
				\item[$\mathrm{(vi)}$] With respect to both projections, the pair $(X,\pi)$ is not stiff. \vspace{3pt}
				
				{\tiny$\bullet$} If $\pi$ is the projection onto $\PP( \O_C\oplus \O_C( D))$, then $\Autz(X)$ is conjugate to $\Autz(Y)$ via a square birational map if and only if there exist $b,n\in \mathbb{Z}$ such that 
				\[Y\cong \PP(\O_{S}\oplus \O_{S}(b\sigma +{\tau}^{*}(nD+E)))),\]
				where $\sigma$ is the minimal section of $\tau$ corresponding to the line bundle $\O_C( D) \subset \O_C\oplus \O_C( D)$.\vspace{3pt}
				
				 {\tiny$\bullet$} If $\pi$ is the projection onto $\PP( \O_C\oplus \O_C( E))$, then $\Autz(X)$ is conjugate to $\Autz(Y)$ via a square birational map if and only if there exist $b,n\in \mathbb{Z}$ such that 
				 \[Y\cong \PP(\O_{S}\oplus \O_{S}(b\sigma +{\tau}^{*}(D+nE)))),\]
				where $\sigma$ is the minimal section of $\tau$ corresponding to the line bundle $\O_C( E) \subset \O_C\oplus \O_C( E)$.\medskip
				
				\item[$\mathrm{(vii)}$] The pair $(X,\pi)$ is superstiff.\medskip
				
				\item[$\mathrm{(viii)}$] The pair $(X,\pi)$ is not stiff: $\Autz(X)$ is conjugate to $\Autz(Y)$ via a square birational map if and only if there exists $n\in \mathbb{Z}$ such that 
				\[
				Y\cong \PP(\O_{\AAA_1}\oplus \O_{\AAA_1}((4n+2)\sigma + \tau^*(D-nD_\sigma))),
				\]
				where $\sigma$ is a minimal section of $\AAA_1$ such that $\O_{\AAA_1}(\sigma)\simeq \O_{\AAA_1}(1)$, and where $D_\sigma$ is a divisor of degree two such that every $\Autz(\AAA_1)$-orbit $\omega\subset \AAA_1$ is linearly equivalent to $4\sigma - \tau^*(D_\sigma)$.\vspace{3pt}
			\end{enumerate}\medskip
			
			\item[$\mathrm{(II)}$] If $g\geq 2$, the pair $(X,\pi)$ is superstiff.\medskip
		\end{enumerate}
	\end{theoremA}
	
	\subsection{Plan of the article}
	
	We fix an algebraically closed field $\kk$ of characteristic zero. Through this article, we work with a smooth projective curve $C$ of positive genus, and $\PP^1$-bundles $\tau \colon S\to C$ and $\pi \colon X\to S$. 
	
	\bigskip
	
	Section \ref{SectionPreliminaries} recalls general properties of $\PP^1$-bundles and introduces several reduction steps. Assume that $\Autz(X)$ is relatively maximal. First we show that the morphism $\tau \pi\colon X\to C$ is an $\FF_b$-bundle over $C$, and $\Autz(S)$ is a maximal connected algebraic subgroup of $\Bir(S)$ (see Propositions \ref{removal jumping} and \ref{basesurfacemaximal}). 
	
	If $b=0$, then $X$ is isomorphic to a fiber product $S\times_C S'$, where $S'$ is a geometrically ruled surface over $C$, and its automorphism group is isomorphic to the fiber product $\Autz(S)\times_{\Autz(C)} \Autz(S')$ (see Lemmas \ref{fiberproduct} and \ref{autofiberproduct}). Therefore, $\Autz(S)$ and $\Autz(S')$ are conjugate to maximal connected algebraic subgroups of $\Bir(C\times \PP^1)$, and we obtain a list of $\FF_0$-bundles that we study individually throughout the article (see Proposition \ref{candidateFF0max}). 
	
	If $b>0$, we show that the connected algebraic group $\Autz(X)$ acts transitively on $C$ via the structure morphism (see Proposition \ref{equivarianttodecomposable}). We also show that there exist a divisor $D$ on $C$, whose linear class is invariant for the $S$-isomorphism class of $X$, and a rank-$2$ vector bundle $\E$ such that $\PP(\E)\simeq X$, fitting into a short exact sequence
	\[
	0 \to \O_S(b\sigma +\tau^*(D)) \to \E \to \O_S \to 0,
	\]
	where $\sigma$ is a minimal section of $\tau$ such that $\O_S(\sigma)\simeq \O_S(1)$ is the twisting sheaf. In Brosius' terminology, the short exact sequence above is the so-called \emph{canonical extension} of $\E$ (see \cite[Theorem 1]{Brosius}, and also \cite[Proposition 3.3.1]{BFT} for the analogue in the rational case). To the $\PP^1$-bundle $\pi\colon X\to S$, we associate the triplet of invariants $(S,b,D)$ (see Definition \ref{definvariants}). Using these invariants, we examine the possible $\Autz(X)$-equivariant Sarkisov links to study the maximality of $\Autz(X)$.
	
	\bigskip
	
	Then, we proceed with our analysis according to the different possibilities for $S$, as outlined above, in Sections \ref{SectionCxPP1}, \ref{SectionAAA1}, \ref{SectionAAA0}, and \ref{SectionSL}. In Section \ref{section:proofs}, we collect the results obtained and prove the main results.
	
	\bigskip
	
	To classify the maximal connected algebraic subgroups of $\Bir(C\times \PP^2)$, it remains to study the automorphism groups of Mori del Pezzo fibrations over $C$ and the equivariant birational maps from the $\FF_b$-bundles to them. If $\Autz(X)$ is a relatively maximal automorphism group as in Theorem \ref{thmA} and there exists no $\Autz(X)$-equivariant birational map to a Mori del Pezzo fibration, then $\Autz(X)$ is a maximal connected algebraic subgroup of $\Bir(X)$. Thus, we may already state:
	
	\begin{corollaryA}\label{coroC}
		Assumptions and numberings as in Theorem \ref{thmA}. The following hold:
		\begin{enumerate}
			\item[$\mathrm{(I)}$] Assume that $g=1$. Then the followings hold:
			\begin{enumerate}
				\item[(a)] The connected algebraic subgroup $\Autz(X)$ is maximal in $\Bir(X)$ in cases \ref{thmA.i} and \ref{thmA.v}.
				\item[(b)] The connected algebraic subgroup $\Autz(X)$ is not maximal in $\Bir(X)$ in case \ref{thmA.iv}.
				\item[(c)] There is no $\Autz(X)$-equivariant birational map starting from $X$ to a $\PP^2$-bundle in cases \ref{thmA.vi}, \ref{thmA.ix} provided that $b\neq 1$, and \ref{thmA.x}.
			\end{enumerate}
			
			\item[$\mathrm{(II)}$] If $g\geq 2$, then $\Autz(X)$ is a maximal connected algebraic subgroup of $\Bir(X)$.
		\end{enumerate}
	\end{corollaryA}

	In a forthcoming article, we plan to give a complete answer to the question of whether $\Autz(X)$ is a maximal connected algebraic subgroup in cases \ref{thmA.vi}, \ref{thmA.vii}, \ref{thmA.viii}, \ref{thmA.ix}, and \ref{thmA.x}, by studying the equivariant birational maps to the Mori del Pezzo fibrations.
	
	\acknowledgments{The author is thankful to Jérémy Blanc, Michel Brion, Adrien Dubouloz, Enrica Floris, Ronan Terpereau, Immanuel Van Santen, Egor Yasinsky, Sokratis Zikas, Susanna Zimmermann for their helpful discussions; and to the anonymous referees for their useful comments and careful reading. The author also acknowledges support by the Swiss National Science Foundation Grant “Geometrically Ruled Surfaces” 200020–192217, the CNRS through the grant PEPS JC/JC, and the ERC StG Saphidir.}

	\section{Preliminaries}\label{SectionPreliminaries}
	
	\subsection{Relatively maximal automorphism groups, (super)stiffness, and Blanchard's Lemma}
	
	\begin{definition}\label{def}
		Let $S$ be a projective surface and $\pi\colon X\to S$ be a $\PP^1$-bundle. For every $\PP^1$-bundle $\pi'\colon X'\to S'$ over a surface $S'$, we say that a birational map $\phi\colon X\dasharrow X'$ is
		\begin{enumerate}			
			\item \emph{$\Autz(X)$-equivariant} if $\phi\Autz(X)\phi^{-1} \subseteq \Autz(X')$;
			\item \label{def.2} a \emph{square birational map} if there exists a birational map $\psi\colon S\dasharrow S'$ such that $\pi' \phi = \psi\pi $, i.e.\ the following diagram commutes:
			\[
			\begin{tikzcd}[column sep=4em,row sep = 3em]
				X \arrow[r,dashed, "\phi"] \arrow[d, "\pi"]& X'\arrow[d,"\pi'"] \\
				S\arrow[r,dashed,"\psi"] & S',
			\end{tikzcd}
			\]
			if moreover $\phi$ is an isomorphism (resp.\ an automorphism), then $\phi$ is called a \emph{square isomorphism} (resp.\ a \emph{square automorphism}).
		\end{enumerate}
		We say that:
		\begin{enumerate}	
			\item[(3)] $\Autz(X)$ is \emph{relatively maximal} (with respect to $\pi$) if $\phi\Autz(X)\phi^{-1} = \Autz(X')$ for every $\Autz(X)$-equivariant square birational map $\phi\colon X\dasharrow X'$;
			\item[(4)] the pair $(X,\pi)$ is \emph{stiff} if for each $\Autz(X)$-equivariant square birational map $\phi\colon X\dasharrow X'$, the $\PP^1$-bundles $\pi\colon X\to S$ and $\pi'\colon X'\to S'$ are square isomorphic;
			\item[(5)] the pair $(X,\pi)$ is \emph{superstiff} if for each $\Autz(X)$-equivariant square birational map $\phi\colon X\dasharrow X'$, the birational maps $\phi$ and $\psi$ from the commutative diagram in \ref{def.2} are isomorphisms.
		\end{enumerate}
	\end{definition}
	
	\begin{remark}
	The definitions above were introduced by Blanc, Fanelli, and Terpereau; see \cite[Definitions 1.2.1 and 1.2.3]{BFT}. In loc. cit., the notion "relatively maximal" is just referred to as "maximal", which could be misleading for non-experts, as $\Autz(X)$ being relatively maximal does not necessarily imply that it is maximal as a connected algebraic subgroup of $\Bir(X)$.
	\end{remark}
	
	Working with the connected algebraic group $\Autz(X)$, we will make heavy use of the Blanchard's lemma and its corollary:
	
	\begin{proposition}\cite[Proposition 4.2.1]{BSU}\label{prop:Blanchard}
		Let $f\colon X\to Y$ be a proper morphism of schemes such that $f_*(\mathcal{O}_X) = \mathcal{O}_Y$, and let $G$ be a connected group scheme acting on $X$. Then there exists a unique action of $G$ on $Y$ such that $f$ is $G$-equivariant.
	\end{proposition}
	
	\begin{corollary}\cite[Corollary 4.2.6]{BSU}\label{blanchard}
		Let $f\colon X\to Y$ be a proper morphism of schemes such that $f_*(\mathcal{O}_X) = \mathcal{O}_Y$. Then $f$ induces a homomorphism of group schemes $f_*\colon \Autz(X)\to \Autz(Y)$ such that for every $\alpha\in \Autz(X)$, the following diagram commutes:
		\[
		\begin{tikzcd}[column sep=4em,row sep = 3em]
			X \arrow[r, "\alpha"] \arrow[d, "f"]& X\arrow[d,"f"] \\
			Y\arrow[r,"f_*(\alpha)"] & Y.
		\end{tikzcd}
		\]
	\end{corollary}

	\subsection{Geometrically ruled surfaces and maximal connected algebraic subgroups of $\Bir(C\times \PP^1)$}\label{ruled}
	
		\begin{definition}\label{def:ruled} Let $C$ be a smooth projective curve and let $S$ be a smooth projective surface. 
		\begin{enumerate}
			\item We say that $S$ is a \emph{geometrically ruled surface} over $C$ if there exists a surjective morphism $\tau\colon S\to C$ whose fibers are isomorphic to $\PP^1$.
			\item Let $\tau\colon S\to C$ be a geometrically ruled surface. The \emph{Segre invariant} of $S$ is the integer $$\seg(S) = \min\{\sigma^2 \vert~ \sigma\text{ section of } \tau\}.$$
			A section $\sigma$ of $\tau$ such that $\sigma^2 = \seg(S)$ is called a \emph{minimal section} of $\tau$.
			\item Let $T$ and $B$ be smooth projective varieties, and let $p\colon T\to B$ be a $\PP^1$-bundle. Let $\E$ be a rank-$2$ vector bundle over $B$ such that $\PP(\E)$ is $B$-isomorphic to $T$. We say that $T$ (or $p$) is \emph{decomposable} if $\E$ is \emph{decomposable}, i.e.\ $\E$ is $B$-isomorphic to the direct sum of two line bundles on $B$. Equivalently, $p$ is \emph{decomposable} if and only if it admits two disjoint sections. 
		\end{enumerate}
	\end{definition}
	
	Equivalently, a geometrically ruled surface $S$ over $C$ is a $\PP^1$-bundle, obtained as the projectivization of a rank-$2$ vector bundle (see \cite[Proposition 2.8]{Hartshorne}).
	
	\begin{definition}\label{defHirzebruch}
		Let $b\in \mathbb{Z}$. The \emph{$b$-th Hirzebruch surface} $\FF_b$ is the quotient of $(\AA^2\setminus \{0\})^2$ by the action of $(\mathbb{G}_m)^2$ defined as:
		\begin{align*}
			(\mathbb{G}_m)^2 \times (\AA^2\setminus \{0\})^2 & \longrightarrow (\AA^2\setminus \{0\})^2 \\
			((\mu,\rho),(y_0,y_1,z_0,z_1)) & \longmapsto (\rho^{-b}\mu y_0,\mu y_1 ,\rho z_0,\rho z_1).
		\end{align*}
		The class of $(y_0,y_1,z_0,z_1)\in (\AA^2\setminus \{0\})^2$ is denoted by $[y_0:y_1\ ;z_0:z_1]$.
	\end{definition}
	
	\begin{remark}		
		\begin{enumerate}
			\item The morphism $\tau_b \colon \FF_b  \to\PP^1$, $[y_0:y_1\ ;z_0:z_1]  \mapsto [z_0:z_1]$ identifies $\FF_b$ with $\PP(\O_{\PP^1}\oplus \O_{\PP^1}(b))$ as $\PP^1$-bundle over $\PP^1$. The disjoint sections $s_{-b}$ and $s_b$ of $\tau_b$, given by $\{y_0=0\}$ and $\{y_1=0\}$, have respectively self-intersection $-b$ and $b$. 
			\item The surface $\FF_b$ is isomorphic to ${\FF}_{-b}$, via $[y_0:y_1 \ ;z_0:z_1]\mapsto [y_1:y_0\ ;z_0:z_1]$.
		\end{enumerate}
	\end{remark}
	
	Recall that every $\PP^1$-bundle over $\PP^1$ is decomposable, hence isomorphic to $\FF_b$ for some $b\geq 0$. This property fails over non-rational curves. In fact, by Atiyah's classification of vector bundles over an elliptic curve $C$ (see also \cite[V. Theorem 2.15]{Hartshorne}), there exist exactly two indecomposable $\PP^1$-bundles over $C$. 
		
	\begin{theorem}\cite[Theorems 5, 6, 11]{Atiyah}
		Let $C$ be an elliptic curve and $\mathcal{L}$ be a line bundle of degree $1$ over $C$. There exist two rank-$2$ indecomposable vector bundles $\E_0$ and $\E_1$ over $C$, unique up to $C$-isomorphism, such that $\mathrm{H}^0(C,\E_0)\neq 0$ and $\mathrm{H}^0(C,\E_1)\neq 0$, and which fit into the following exact sequences:
		\[
		\begin{array}{cccccccccc}
			&0 & \to &\O_C  &\to& \E_0 &\to  &\O_C  &\to 0, & \\
			& 0& \to &\O_C & \to &\E_1 &\to & \mathcal{L}  & \to 0. &
		\end{array}
		\]
		Up to $C$-isomorphism, $ \PP(\E_0)$ and $\PP(\E_1)$ are exactly the two classes of indecomposable $\PP^1$-bundles over $C$ and the $C$-isomorphism class of $\PP(\E_1)$ does not depend on the choice of $\mathcal{L}$. 
	\end{theorem}
	
	\begin{notation}
		Let $C$ be an elliptic curve. From now on, we denote by $\AAA_0 = \PP(\E_0)$ and $\AAA_1 = \PP(\E_1)$ the two indecomposable $\PP^1$-bundles over $C$. 
	\end{notation}
	
	The two geometrically ruled surfaces $\AAA_0$ and $\AAA_1$ are birational to $C\times \PP^1$. In particular, these birational maps can be decomposed into a sequence of finitely many elementary transformations, namely, the blow-up of a point followed by the contraction of the strict transform of the fiber passing through that point. This explicit construction enables us to compute their transition maps as $\PP^1$-bundles.
	
	\begin{lemma}\label{transitionAtiyah}		
		Let $C$ be an elliptic curve and $p\in C$. For every $\alpha\in \kk(C)^*$ having a zero of order one at $p$, there exists an open subset $U_j\subset C$ such that $\mathrm{div}_{|U_j}(\alpha) = p$ and for every trivializing open subset $U_i\subset C$ that does not contain $p$:
		\begin{enumerate}
			\item\label{transitionAtiyah.2} We can choose the transition maps of the indecomposable geometrically ruled surface $\tau_0\colon \AAA_0\to C$ as
			\[
			\begin{array}{cccc}
				& U_j\times \PP^1 & \dashrightarrow & U_i\times \PP^1 \\
				& (x,[u:v]) & \longmapsto & (x,[u:\alpha(x)^{-1} u+ v]).
			\end{array}
			\]
			\item\label{transitionAtiyah.3} We can choose the transition maps of the indecomposable geometrically ruled surface $\tau_1\colon \AAA_1\to C$ as
			\[
			\begin{array}{cccc}
				& U_j\times \PP^1 & \dashrightarrow & U_i\times \PP^1 \\
				& (x,[u:v]) & \longmapsto & (x,[u:\alpha(x)^{-1} u + \alpha(x) v]).
			\end{array}
			\]
		\end{enumerate}
	\end{lemma}
	
	\begin{proof}
		Let $q=(p,[0:1])\in C\times \PP^1$ and $(U_k)_k$ be a trivializing open cover of $\AAA_0$ and $\AAA_1$. Choose $j$ such that $p\in U_j$ and we can assume that for each $i\neq j$, the point $p$ does not belong to $U_i$. By shrinking the open subset $U_j$ if necessary, we assume that $\mathrm{div}_{|U_j}(\alpha) = p$. 
		Up to isomorphisms, an elementary transformation $\epsilon_q \colon C\times \PP^1 \dasharrow S$ centered at $q$ is locally given by
		\[
		\begin{array}{cccc}
		&U_j\times \PP^1 & \dasharrow & U_j\times \PP^1 \\
		&(x,[u:v]) & \longmapsto & (x,[u:\alpha(x)v])
		\end{array}
		\]	
		and identity on $U_i\times \PP^1$. Therefore, we can choose the transition maps $\tau_{ij}$ of the $\PP^1$-bundle $\tau\colon S \to C$ as 
		\[
		\begin{array}{cccc}
					\epsilon_q^{-1}\colon & U_j\times \PP^1 & \dasharrow & U_i\times \PP^1 \\
					& (x,[u:v]) & \longmapsto & (x,[\alpha(x)u:v]).
		\end{array}
		\]
		An elementary transformation of $S$ centered at a point $q'\in \tau^{-1}(p)$ which does not lie on the unique $(-1)$-section of $\tau\colon S\to C$ and distinct from the base point of $\epsilon_q^{-1}$ yields a birational map $\epsilon_{q'} \colon S\dashrightarrow \AAA_0$ (see e.g.\ \cite[Proposition 2.21 (1)]{Fong}). Locally, we can assume that the birational map $\epsilon_{q'}$ is given by 
		\[
		\begin{array}{cccc}
			& U_j\times \PP^1 & \dasharrow & U_j\times \PP^1 \\
			& (x,[u:v]) & \longmapsto & (x,[\alpha(x)u:v-u]),
		\end{array}
		\]
		and identity on $U_i\times \PP^1$. Therefore, we can choose the transition maps $\tau_{0,ij}$ of $\AAA_0$ as 
		\[
		\begin{array}{cccc}
			\tau_{ij}\epsilon_{q'}^{-1} \colon &U_j\times \PP^1 & \dasharrow &U_i\times \PP^1 \\
			&(x,[u:v]) & \longmapsto &(x,[u:\alpha(x)^{-1}u+v]),
		\end{array}
		\]
		and this proves \ref{transitionAtiyah.2}. Under this choice of trivializations, the unique section of $\AAA_0$ of self-intersection zero is the zero section $\{u=0\}$. An elementary transformation of $\AAA_0$ centered on any point $q''$ which is not lying on the minimal section yields a birational map $\epsilon_{q''}\colon \AAA_0\dashrightarrow \AAA_1$ (see e.g.\ \cite[Proposition 2.21 (2)]{Fong}). Thus, we can choose $q''$ such that $\epsilon_{q''}$ is locally given by
		\[
		\begin{array}{cccc}
			& U_j\times \PP^1 & \dasharrow & U_j\times \PP^1 \\
			& (x,[u:v]) & \longmapsto & (x,[\alpha(x)u:v])	
		\end{array}
		\]
		and identity on $U_i\times \PP^1$. So we can choose the transition maps of $\AAA_1$ as 
		\[
		\begin{array}{cccc}
			\tau_{0,ij}\epsilon_{q''}^{-1}  \colon & U_j\times \PP^1 & \dasharrow & U_i\times \PP^1 \\
			& (x,[u:v]) & \longmapsto & (x,[u:\alpha(x)^{-1}u+\alpha(x)v]),
		\end{array}
		\]
		and this proves \ref{transitionAtiyah.3}.
	\end{proof}
	
	Next, we recall some properties of the geometry of non-rational geometrically ruled surfaces and their automorphism groups. The following result is due to Maruyama; see \cite[Theorem 3]{Maruyama}.
	
	\begin{lemma}\label{geometryofruledsurface}
		Let $C$ be a smooth projective curve of genus $g\geq 1$ and $\tau \colon S\to C$ be a $\PP^1$-bundle. Then the following hold:
		\begin{enumerate}
			\item If $S=C\times \PP^1$ and $\tau$ is the projection onto $C$, then $\seg(S)=0$ and $\tau$ admits infinitely many minimal sections. Moreover, $\Autz(S) \simeq \Autz(C) \times \PGL_2(\kk)$.
			\item \label{geometryofruledsurface.3} If $g=1$ and $S$ is $C$-isomorphic to $\AAA_0$, then $\seg(S)=0$ and $\tau$ admits a unique minimal section $\sigma_0$. Moreover, $S$ contains exactly two $\Autz(S)$-orbits: $\sigma_0$ and its complement; and $\Autz(S)$ fits into the short exact sequence
			$$
			0 \to \mathbb{G}_a \to \Autz(S) \to \Autz(C) \to 0.
			$$
			\item \label{geometryofruledsurface.4} If $g=1$ and $S$ is $C$-isomorphic to $\AAA_1$, then $\seg(S)=1$ and $\tau$ admits infinitely many minimal sections that are pairwise not linearly equivalent. Moreover, $\Autz(S)$ fits into the short exact sequence
			$$
			0 \to \left(\mathbb{Z}/2\mathbb{Z}\right)^2 \to \Autz(S) \to \Autz(C) \to 0.
			$$
			\item \label{geometryofruledsurface.2} If $g=1$ and $S$ is $C$-isomorphic to $\PP(\O_C\oplus\O_C(D))$ for some non-trivial divisor class $D\in \Pic^0(C)$, then $\seg(S)=0$ and $\tau$ admits exactly two disjoint minimal sections, $\sigma$ and $\sigma_D$, corresponding to the line subbundles $\O_C$ and $\O_C(D)$ of $\O_C\oplus  \O_C(D)$. These sections are not linearly equivalent and $\sigma \sim \sigma_D + \tau^*(D)$.
			
			Moreover, $S$ contains three $\Autz(S)$-orbits: $\sigma$, $\sigma_D$ and $S\setminus (\sigma\cup \sigma_D)$;
			and $\Autz(S)$ fits into the short exact sequence
			$$
			0 \to \mathbb{G}_m \to \Autz(S) \to \Autz(C) \to 0.
			$$
		\end{enumerate}
	\end{lemma}
	
	\begin{proof} 
			The computations of the Segre invariants $\seg(S)$ and the short exact sequences can be found in \cite{Maruyamabook, Maruyama} (see also \cite[Section \S 2]{Fong}). The only statement not covered in the references above is that $\sigma \sim \sigma_D + \tau^*(D)$ in \ref{geometryofruledsurface.2}.
			
			Write $D \sim p-p_0$, where $p_0$ is the neutral element of the elliptic curve $C$ and $p\neq p_0$. We denote by $\tau_1\colon C\times \PP^1\to C$ the projection onto the first factor and let $\sigma_0=\{z_0=0\}$ and $\sigma_1=\{z_1=0\}$ be two disjoint constant sections of $\tau_1$. By \cite[Proposition 2.20 (2)]{Fong}, there exists a birational map $\phi\colon C\times \PP^1\dashrightarrow S$, consisting of the blowup of $\tau_1^{-1}(p) \cap \sigma_0$ and $\tau_1^{-1}(p_0) \cap \sigma_1$
			followed by the contraction of the strict transforms of their corresponding fibers. The strict transforms by $\phi$ of the constant sections $\sigma_0$ and $\sigma_1$ are respectively $\sigma_D$ and $\sigma$. Since $\sigma_0$ and $\sigma_1$ are linearly equivalent, there exists $h\in \kk(C\times \PP^1)^*$ such that $\mathrm{div}(h)=\sigma_0 - \sigma_1$. Then $\mathrm{div}(h \phi^{-1}) = \sigma_D - \sigma + \tau^{*}(p) - \tau^{*}(p_0)$, i.e., $\sigma \sim \sigma_D +\tau^*(D)$.
	\end{proof}
	
	In the lemma above, the short exact sequences in the cases \ref{geometryofruledsurface.3} and \ref{geometryofruledsurface.2} do not split. The corresponding middle terms are equipped with a structure of $\mathbb{G}_a$ and $\mathbb{G}_m$-torsors over $C$ (see \cite[Chapter VII, \S 3]{Serre}, and also \cite[Theorem 3]{Maruyama}). From the perspective of algebraic groups, they provide examples of anti-affine groups, i.e., their global sections are constant (see e.g., \cite[Example 4.2.4]{BSU}). Using these results, we obtain the following lemma:
	
	\begin{lemma}\label{trickBrion}	
		Let $C$ be an elliptic curve, $G=\Autz(S)$ where $S=\AAA_0$ or $\PP(\O_C\oplus \L)$, with $\L\in \Pic^0(C)\setminus \{\O_C\}$ of infinite order. Then $G$ does not contain any complete curve.
	\end{lemma}
	
	\begin{proof}
		First, let us recall that $G$ is anti-affine, i.e. $\O(G)\simeq \kk$, see e.g., \cite[Example 4.2.4]{BSU}. 
		
		We denote by $\sigma$ the unique minimal section of $\AAA_0$, by $\sigma_1$ and $\sigma_2$ the two minimal sections of $\PP(\O_C\oplus \L)$. Then by \cite[Theorem 3]{Maruyama}, the connected algebraic group $G$ is isomorphic to $\AAA_0\setminus \sigma$ if $S = \AAA_0$, or isomorphic to $\PP(\O_C\oplus \L) \setminus (\sigma_1\cup \sigma_2)$ if $S=\PP(\O_C\oplus \L)$. Therefore, the restriction of the structure morphism yields a surjective morphism $p\colon G\to C$. 
		
		Assume by contradiction that $G$ contains a complete curve $D$. Then $D$ surjects onto $C$ via $p$. Without loss of generality, we can assume that $0_G \in D$. By \cite[Chapter I, \S2, 2.2 Proposition]{Borel}, the group $A$ generated by $D$ and $-D$ is a connected subgroup of $G$. Since $D$ is complete, it follows that $A$ is also complete; thus, $A$ is an abelian subvariety of $G$ that surjects onto $C$ via the restriction of $p$. By Poincaré's complete reducibility theorem (see e.g., \cite[Corollary 4.2.6]{Brion_structure}), there exists an elliptic curve $C'\subset A$ that surjects onto $C$. Recall that Blanchard's lemma induces a morphism of algebraic groups $\pi_*\colon G\to \Autz(C)$ (see Corollary \ref{blanchard}). We denote by $H$ its kernel, which is isomorphic to $\mathbb{G}_a$ or $\mathbb{G}_m$ by Lemma \ref{geometryofruledsurface}. The action of $H$ on $G$ yields a short exact sequence
		\[
		0 \to F \to C'\times H \to G \to 0,
		\]
		where $F$ is a finite group. Therefore, $\O(G)= \O(C'\times H)^F = \O(H)^F$ has infinite dimension, and this contradicts that $\O(G) \simeq \kk$.
	\end{proof}
	
	\begin{remark}\label{rem:A_0_elliptic}
		In this remark, we assume that $\kk$ is an algebraically closed field of arbitrary characteristic. In the previous lemma, we did use the assumption that $\mathrm{char}(\kk)=0$. Indeed, by \cite[Example 4.2.4]{BSU}, the group $\Autz(\AAA_0)$ is anti-affine if and only if $\mathrm{char}(\kk)=0$. 
		
		In fact, the previous lemma does not extend to positive characteristic. By \cite[Theorem 4]{Maruyama} (see also \cite[Propositions 2.10 and 2.15]{Togashi_Uehara}), the geometrically ruled surface $\AAA_0$ is elliptic if and only if $\mathrm{char}(\kk)=p>0$. In this case, there exists a finite subgroup scheme $F\subset C$ of the elliptic curve $C$ acting diagonally on the product $C\times \PP^1$, such that $$\AAA_0 \cong (C\times \PP^1)/F.$$ Under this isomorphism, the structure morphism of $\PP^1$-bundle is identified with the projection onto the quotient $C/F$.
		An explicit construction of $\AAA_0$ as a quotient of $C\times \PP^1$ is given in \cite[Examples 4.7 and 4.8]{Katsura_Ueno} (see also \cite[Remark 4.4]{Togashi_Uehara}), with $F = \mathbb{Z}/p\mathbb{Z}$ or $\alpha_p$, depending on whether $C$ is ordinary or supersingular. 
		
		Since the minimal section $\sigma$ is the unique section of self-intersection zero, it is contracted by the projection $\AAA_0 \to \PP^1/F \simeq \PP^1$. In particular, every fiber of this projection is contained in $\Autz(\AAA_0) \simeq \AAA_0 \setminus \sigma$.
	\end{remark}
	
	For the sake of completeness, let us also recall the classification of maximal connected algebraic subgroups of $\Bir(C\times \PP^1)$:
	
	\begin{theorem}\cite[Theorem B]{Fong}\label{dim2max}
		Let $C$ be a smooth projective curve of genus $g\geq 1$ and $G$ be a connected algebraic subgroup of $\Bir(C\times \PP^1)$. Then the following hold:
		\begin{enumerate}
			\item \label{dim2max.1} If $g=1$, then $G$ is maximal if and if $G=\Autz(S)$, where $S$ is a geometrically ruled surface isomorphic to $C\times \PP^1$, $\AAA_0$, $\AAA_1$, or $\PP(\O_C\oplus \mathcal{L})$ for some $\mathcal{L}\in \Pic^0(C)\setminus \{0\}$. 
			\item \label{dim2max.2} If $g\geq 2$, then $G$ is maximal if and only if $G = \Autz(C\times \PP^1)\simeq \PGL_2(\kk)$. 
		\end{enumerate}
	\end{theorem}		
	
	\section{First reduction step: towards $\FF_b$-bundles}
	
	\subsection{Removal of jumping fibers}
	
	\begin{definition}
		Let $X$ and $B$ be smooth projective varieties. A surjective morphism $p\colon X\to B$ is \emph{an $\FF_b$-bundle} if there exists a Zariski open cover $(U_i)_i$ of $B$ such that, for each open subset $U_i$, there exists an isomorphism $\phi_i \colon p^{-1}(U_i) \to U_i \times \FF_b$ making the following diagram commute:
		\[
		\begin{tikzcd}[column sep=2em,row sep = 3em]
			p^{-1}(U_i)\arrow[rd,"p" swap] \arrow[rr,"\sim", "\phi_i" swap] &&  U_i \times \FF_b \arrow[ld,"p_{U_i}"]\\
		& U_i &.
		\end{tikzcd}
		\] 
		The isomorphisms $\phi_i$ are called the \emph{trivializations} of the $\FF_b$-bundle. The birational maps 
		\[
		\phi_i \phi_j^{-1}\colon U_j \times \FF_b \dashrightarrow U_i \times \FF_b,
		\]
		which are isomorphisms over $U_j\cap U_i$, are called the \emph{transition maps} of the $\FF_b$-bundle.
	\end{definition}
	
	Let $C$ be a smooth projective curve, and let $\tau\colon S\to C$ and $\pi\colon X \to S$ be $\PP^1$-bundles. In this section, we follow the strategy adopted in \cite[Section 3.2]{BFT} to show that $\Autz(X)$ is conjugate to a subgroup of the automorphism group of an $\FF_b$-bundle over $C$. 
	
	\begin{lemma}\label{lemma removal jumping}
		Let $C$ be a smooth projective curve, and let $\tau\colon S\to C$ and $\pi\colon X \to S$ be $\PP^1$-bundles. Then there exists an integer $b\geq 0$ such that the generic fiber of the morphism $\tau\pi$ is isomorphic to the Hirzebruch surface $\FF_b$ over the function field $\kk(C)$.
		\medskip
		
		Moreover, for each $p\in C$, there exists an affine open subset $U\subset C$ containing $p$ and a local trivialization $\phi_U\colon \tau^{-1}(U)\to U\times \PP^1$ such that the following hold:
		\begin{enumerate}
			\item \label{removal.2} Let $V=U\setminus \{p\}$. Then there exists a local trivialization 
			$
			\psi_V\colon (\tau \pi)^{-1}(V)\xrightarrow{\sim} V\times \FF_b
			$
			such that the following diagram is commutative:
			\[
			\begin{tikzcd}[column sep=4em,row sep = 3em]
				(\tau \pi)^{-1}(V)\arrow[r,"\sim","\psi_V" swap] \arrow[d,"\pi",swap] &  V\times \FF_b \arrow[d,"id\times \tau_b"]\\
				\tau^{-1}(V) \arrow[r,"{\phi_U}_{|V}" below ,"\sim"]& V\times \PP^1,
			\end{tikzcd}
			\] 
			where $\tau_b\colon \FF_b \to \PP^1$ denotes the structure morphism of the $b$-th Hirzebruch surface.
			\item \label{removal.3} Assume that the commutative diagram in \ref{removal.2} does not extend across $p$, that is, $\psi_V$ does not extend to an isomorphism $\psi_U\colon (\tau \pi)^{-1}(U)\xrightarrow{\sim} U\times \FF_b$.
			Then there exists a positive integer $\epsilon$ such that 
			\[
			(\tau \pi)^{-1}(p)\simeq \FF_{b+2\epsilon}.
			\]
			In this case, we say that $\tau^{-1}(p)$ is a jumping fiber.
			\item \label{removal.4} If $\tau^{-1}(p)$ is a jumping fiber, one can blowup the minimal section in $(\tau \pi)^{-1}(p)$ and then contract the strict transform of $(\tau \pi)^{-1}(p)$. This replaces $X$ with another $\PP^1$-bundle $\pi'\colon X'\to S$, which is also trivial over $V$. After finitely many such steps, $\tau^{-1}(p)$ is no longer a jumping fiber, and $\tau\pi$ becomes the trivial $\FF_b$-bundle over $U$.
		\end{enumerate} 
	\end{lemma}
	
	\begin{proof} 
		We adapt the proof of \cite[Lemma 3.2.1]{BFT}. Let $p\in C$ and $U\subset C$ be an affine trivializing open subset of $\tau$ containing $p$, with a trivialization $\phi_U\colon \tau^{-1}(U)\to U\times \PP^1$. Let $V_0=\phi_U^{-1}(U\times \{[y:1],y\in \kk\})$ and $V_{\infty}=\phi_U^{-1}(U\times \{[1:y],y\in \kk\})$, which are affine and isomorphic to $U\times \AA^1$. Let $\mathcal{E}$ be a vector bundle of rank two over $S$ such that $\PP(\mathcal{E}) \simeq  X$. By \cite[Theorem 4']{Quillen}, $\mathcal{E}_{|V_0}$ and $\mathcal{E}_{|V_{\infty}}$ are pullbacks of rank-$2$ vector bundles over $U$. Therefore, by choosing $U$ smaller if necessary, we can assume that $V_0$ and $V_\infty$ trivialize $\mathcal{E}$. This implies that $V_0$ and $V_{\infty}$ also trivialize the $\PP^1$-bundle $\pi$ and we can write the transition maps as follows:
		\[
		\begin{array}{cccc}
			\psi_\infty\psi_0^{-1}\colon & V_{0} \times \PP^1 & \dashrightarrow &  V_\infty \times \PP^1 \\
			& ((x,y),[u:v]) & \longmapsto  & \left((x,y^{-1}),A(x,y)\cdot [u:v] \right),
		\end{array}
		\]
		where $\psi_0$ and $\psi_\infty$ respectively denote the trivializations of $V_0$ and $V_\infty$, and $A\in \mathrm{GL}_2(\mathcal{O}_C(U)[y,y^{-1}])$. The $\PP^1$-bundle $\pi$ is determined by the equivalence class of $A$ modulo $A\sim \lambda M A M'$, where $\lambda \in {\mathcal{O}_C(U)}^* y^\mathbb{Z}$, $M\in \mathrm{GL}_2(\mathcal{O}_C(U)[y^{-1}])$ and $M'\in \mathrm{GL}_2(\mathcal{O}_C(U)[y])$. Since $A$ is invertible, it follows that $\det(A)\in{\mathcal{O}_C(U)}^* y^\mathbb{Z}$. After multiplying by a matrix 
		$\begin{pmatrix}
			\mu & 0 \\ 0 & 1
		\end{pmatrix}\in \mathrm{GL}_2(\O_C(U)^*)$,
		we can assume that $\mathrm{det}(A)\in y^\mathbb{Z}$.
		
		Working over $\kk(C)$, we get a $\PP^1$-bundle over $\PP^1_{\kk(C)}$, which is isomorphic to a Hirzebruch surface $(\FF_b)_{\kk(C)}$ (see e.g.\ \cite{HM}). Hence there exist $M\in \mathrm{GL}_2(\kk(C)[y^{-1}])$ and $N\in \mathrm{GL}_2(\kk(C)[y])$ such that 
		\[
		M^{-1} A N = D = 
		\begin{pmatrix}
			y^m & 0 \\
			0 & y^n
		\end{pmatrix},
		\]
		with $m,n\in \mathbb{Z}$ satisfying $m-n=b\geq 0$.
		
		\ref{removal.2} and \ref{removal.3}: Notice that $\mathrm{det}(M), \mathrm{det}(N)\in \kk(C)^*$. Since $\mathrm{det}(A),\mathrm{det}(D) \in y^\mathbb{Z}$, it follows from the equality $AN=MD$ that $\mathrm{det}(A)= \mathrm{det}(D)$ and $\mathrm{det}(M)= \mathrm{det}(N)$. Multiplying both $M$ and $N$ with the same element of $\O_C(U)^*$, we can assume that $M\in \GL_2(\O_C(U)[y^{-1}])$, $N\in \GL_2(\O_C(U)[y])$. Choosing $U$ smaller if necessary and multiplying both $M$ and $N$ by the same element of $\kk(C)^*$, we can furthermore assume that $(M(x),N(x))\neq (0,0) $ for each $x\in U$, where $M(x)\in \Mat_2(\kk[y^{-1}])$ and $N(x)\in \Mat_2(\kk[y])$ are obtained by evaluating respectively the coefficients of $M$ and $N$ at $x$.
		
		Let $Z\subset C$ be the zero set of $\mathrm{det}(M)$, which is finite. 
		If $Z\cap U=\emptyset$, one can take $V= U$ as $M$ and $N$ correspond respectively to automorphisms of $V_\infty\times \PP^1$ and $V_0\times \PP^1$, and this ends the proof.
		
		Now we assume that $\{p\}= Z\cap U$ and we prove the result by induction on $\nu_p(\det(M))$. We write $\widetilde{b}\geq 0$ such that $(\tau \pi)^{-1} (p) \simeq \FF_{\widetilde{b}}$. There exist $\widetilde{M}\in \mathrm{GL}_2(\kk[y^{-1}])$ and $\widetilde{N}\in \mathrm{GL}_2(\kk[y])$ such that 
		\[
		\widetilde{M}^{-1} A(p) \widetilde{N} = 
		\begin{pmatrix}
			y^{\widetilde{m}} & 0 \\
			0 & y^{\widetilde{n} }
		\end{pmatrix}
		\]
		with $\widetilde{m} - \widetilde{n} = \widetilde{b}\geq 0$. Since $\mathrm{det}(A)\in y^\mathbb{Z}$ is independent of $x$, it follows that $m+n = \widetilde{m} + \widetilde{n}$. Writing $\epsilon = \widetilde{m} - m = n - \widetilde{n}\in \mathbb{Z}$, one gets $\widetilde{b} - b = (\widetilde{m} - m ) - (\widetilde{n} - n ) = 2\epsilon$. Replacing $A,M,N$ with respectively $\widetilde{M}^{-1} A \widetilde{N},\widetilde{M}^{-1}M,\widetilde{N}^{-1}N$; we keep the equation $AN=MD$, do not change the set $Z$ and the multiplicity of $p$ as zero of $\mathrm{det}(M)$, but we can assume that $A(p)=
		\begin{pmatrix}
			y^{\widetilde{m}} & 0 \\
			0 & y^{\widetilde{n} }		 
		\end{pmatrix}.
		$
		Writing $M(p) = 
		\begin{pmatrix}
			\beta_{11} & \beta_{12} \\
			\beta_{21} & \beta_{22}
		\end{pmatrix}\in \Mat_2(\kk[y^{-1}])$
		yields 
		\begin{equation}\label{eq3}
			N(p)= A(p)^{-1}M(p) D = 
			\begin{pmatrix}
				y^{-\epsilon}\beta_{11} & y^{-b-\epsilon} \beta_{12} \\
				y^{b+\epsilon} \beta_{21} & y^{\epsilon} \beta_{22}
			\end{pmatrix}.\tag{$\star$}
		\end{equation}
		\begin{enumerate}
			\item[(i)] If the first column of $M(p)$ is zero, then so is the first column of $N(p)$. Let $f\in \kk(C)^{*}$ having a zero of order $1$ at $p$, by restricting $U$ one can assume that $U$ does not contain other zeroes or poles of $f$, and let 
			$\Delta=
			\begin{pmatrix}
				f & 0 \\
				0 & 1
			\end{pmatrix}.
			$
			Since $D$ commutes with $\Delta$, replacing $M$ by $M'=M\Delta^{-1}$ and $N$ by $N'=N\Delta^{-1}$ does not change the equation $AN=MD$ , and decreases the multiplicity of $p$ as zero of $\mathrm{det}(M)$. A similar argument works if the second column of $M(p)$ is zero.
			\item[(ii)] If $\epsilon <0$ then $-b-\epsilon = \epsilon - \widetilde{b}<0$. Since $M(p)\in \Mat_2(\kk[y^{-1}])$ and $N(p)\in \Mat_2(\kk[y])$, it follows that the second column of $N(p)$ is zero. Hence the second column of $M(p)$ is also zero and we can apply (i).
			\item[(iii)] If $\epsilon = b = 0$ then $M(p) = N(p)$. In particular they do not depend on $y$ or $y^{-1}$, i.e.\ they belong to $\Mat_2(\kk)$. Since $M(p)$ is not invertible, there exists $R\in \mathrm{GL}_2(\kk)$ such that the first column of $M(p)R$ is zero. We replace $M,N$ by $MR,NR$, so that the equation $AN=MD$ still holds because $D=y^mI_2=y^nI_2$ commutes with $R$. Then we apply (i).
			\item[(iv)] If $\epsilon = 0$ and $b>0$, then $\beta_{12}=0$ and $\beta_{22}\in \kk$. If $\beta_{22}=0$, we apply (i). Else $\beta_{11}=0$ because $\mathrm{det}(M(p))=0$, and the first column of  $M(p)R$ is zero if  
			$R=\begin{pmatrix}
				\beta_{22} & 0 \\ -\beta_{21} & 1
			\end{pmatrix}\in \mathrm{GL}_{2}(\kk[y^{-1}])$.
			Replacing also $N$ by $NR'$ with $R' = D^{-1}RD=
			\begin{pmatrix}
				\beta_{22} & 0 \\ -\beta_{21}y^b & 1
			\end{pmatrix}\in \mathrm{GL}_2(\kk[y])$, the equation $AN=MD$ still holds and we can apply (i).
			\item[(v)] \label{v} The last case is when $\epsilon>0$, which implies that $\widetilde{b}=b+2\epsilon\geq 2$ and that $\beta_{11}=\beta_{12}=0$. 
		\end{enumerate}
		
		Taking $U$ small enough, we get $Z\cap U =\{p\}$ and $(\tau\pi)^{-1}(p)\simeq \FF_{b+2\epsilon}$. Writing $V=U\setminus Z$, this yields \ref{removal.2} and \ref{removal.3}.
		
		\ref{removal.4}:  By \ref{removal.3}, the fiber $(\tau\pi)^{-1}(p)$ is isomorphic to an Hirzebruch surface $\FF_{\widetilde{b}}$ with $\widetilde{b}\geq 2$. In particular, it contains a unique minimal section that we may assume to be defined by $u=0$ in the charts $V_0 \times \mathbb{P}^1$ and $V_\infty\times \PP^1$, both isomorphic to $(U\times \AA^1)\times \PP^1$. By restricting $U$ if necessary, there exists $f\in \kk(C)^{*}$ such that $\mathrm{div}(f)_{|U}=p$. The blowup of the minimal section followed by the contraction of the strict transform of the surface $\FF_{\widetilde{b}}$ corresponds locally to the birational map:
		\[
		\begin{array}{ccc}
			(U\times \mathbb{A}^1)  \times \mathbb{P}^1 & \dashrightarrow & (U\times \mathbb{A}^1) \times \mathbb{P}^1\\
			((x,y),[u:v] )& \longmapsto & ((x,y),[u:f(x)v]),
		\end{array}
		\]
		This replaces the transition matrix $A$ with $A'= \Delta^{-1}A\Delta$, where: 
		\begin{center}
			\begin{minipage}{0.3\textwidth}
				$\Delta = 
				\begin{pmatrix}
					f & 0 \\
					0 & 1
				\end{pmatrix}$,
			\end{minipage}
			\begin{minipage}{0.3\textwidth}
				$
				A=
				\begin{pmatrix}
					\alpha_{11} & \alpha_{12} \\
					\alpha_{21} & \alpha_{22}
				\end{pmatrix}$,
			\end{minipage}
			\begin{minipage}{0.3\textwidth}
				$A' =
				\begin{pmatrix}
					\alpha_{11} & \alpha_{12}f^{-1} \\
					\alpha_{21}f & \alpha_{22}
				\end{pmatrix}		
				$.
			\end{minipage}
		\end{center}
		
		Since $A(p)$ is diagonal, it follows that $\alpha_{12}$ and $\alpha_{21}$ vanish at $q$, hence $A'\in \mathrm{GL}_2(\mathcal{O}_C(U)[y,y^{-1}])$. Moreover, $\epsilon >0$ implies that the first line of $N(p)$ and $M(p)$ is zero by (\ref{eq3}), hence the matrices $M'=\Delta^{-1}M$ and $N'=\Delta^{-1}N$ belong to $\mathrm{Mat}_{2}(\mathcal{O}_C(U)[y,y^{-1}])$. Replacing $A,M,N$ with $A',M',N'$ does not change the equation $AN=MD$, but decreases the multiplicity of $p$ as zero of $\mathrm{det}(M)$. After finitely many steps, one gets the trivial $\FF_b$-bundle over $U$. 
	\end{proof}
	
	Applying Lemma \ref{lemma removal jumping} inductively to remove all jumping fibers, we obtain:
	
	\begin{proposition}\label{removal jumping}
	Let $C$ be a smooth projective curve, and let $\tau\colon S\to C$ and $\pi\colon X \to S$ be $\PP^1$-bundles. Then there exist $b\geq 0$, a $\PP^1$-bundle $\widetilde{\pi}\colon \widetilde{X}\to S$ such that $\tau\widetilde{\pi}$ is an $\FF_b$-bundle, and an $\Autz(X)$-equivariant birational map $\psi\colon X\dashrightarrow \widetilde{X}$.
	\end{proposition}
	
	\begin{proof}
		We adapt the proof of \cite[Proposition 3.2.2]{BFT}.
		By Lemma \ref{lemma removal jumping} \ref{removal.2}, we can cover $C$ by finitely many trivializing affine open subsets $(U_i)_i$ of $\tau\colon S\to C$, such that each $\tau^{-1}(U_i)$ has at most one jumping fiber. Hence there exist $b\geq 0$ and a dense open subset $U\subset C$ such that for all $p\in U$, $(\tau \pi)^{-1}(p)\simeq\FF_b$; and for each $q_i\in U_i\cap (C\setminus U)$, there exists an integer $c_i$ such that $(\tau \pi)^{-1}(q_i) \simeq \FF_{c_i}$, where $c_i-b$ is a positive even integer by Lemma \ref{lemma removal jumping} \ref{removal.3}. Applying Lemma \ref{lemma removal jumping} \ref{removal.4} to each jumping fiber, we get a $\PP^1$-bundle $\widetilde{\pi}\colon \widetilde{X}\to S$ such that $\tau \widetilde{\pi}$ is a  $\FF_b$-bundle and a birational map $\psi\colon X\dashrightarrow \widetilde{X}$. The connected algebraic group $\Autz(X)$ acts trivially on the set of the jumping fibers, since this set is finite. Therefore, the birational map $\psi\colon X\dashrightarrow \widetilde{X}$ is $\Autz(X)$-equivariant.
	\end{proof}
	
	Let us recall that an automorphism $f\in \Aut(X)$ preserves a fibration $\pi\colon X\to S$ if there exists a unique automorphism $g\in \Aut(S)$ such that $\pi f = g \pi$. The subgroup $G\subseteq \Aut(X)$ consisting of automorphisms preserving $\pi$ is closed and contains the connected component $\Autz(X)$ (see e.g., \cite[Proposition 2.22]{Brionnotes}). 
	In the following proposition, we show that $G=\Aut(X)$ when $X$ is an $\FF_b$-bundle with $b>0$:
	
	\begin{proposition}\label{Autalgebraicgroup}
		Let $C$ be a smooth projective curve of positive genus, $\tau\colon S\to C$  and $\pi\colon X\to S$ be $\PP^1$-bundles such that $\tau\pi$ is an $\FF_b$-bundle with $b>0$. Then the following hold:
		\begin{enumerate}
			\item\label{Autalgebraicgroup.2} For every $f \in \Aut(X)$, there exist $g \in \Aut(S)$ and $h \in \Aut(C)$ such that the following diagram commutes:
			\[
			\begin{tikzcd}[column sep=4em,row sep = 3em]
				X \arrow[r, "f"] \arrow[d, "\pi"]& X\arrow[d,"\pi"] \\
				S\arrow[r,"g"] \arrow[d, "\tau"]& S\arrow[d, "\tau"] \\
				C \arrow[r, "h"] & C.
			\end{tikzcd}
			\]
			
			\item \label{Autalgebraicgroup.1} The groups $\Aut(S)$ and $\Aut(X)$ are algebraic subgroups of $\Bir(S)$ and $\Bir(X)$, respectively.
			
			\item  \label{Autalgebraicgroup.3} There exists a morphism of algebraic groups $\Pi\colon \Aut(X)\to \Aut(S)$, mapping $f$ to $g$, which restricts to the morphism $\pi_*\colon \Autz(X)\to \Autz(S)$ induced by Blanchard's lemma $(\ref{blanchard})$. Moreover, $\Pi(\Aut(X))$ is an algebraic subgroup of $\Aut(S)$ and $ \Pi(\Aut(X))^\circ = \pi_*(\Autz(X))$.
		\end{enumerate}
	\end{proposition}
	
	\begin{proof}
		(1): Let $F$ be a fiber of $\pi\tau$, that is isomorphic to $\FF_b$ by assumption. For any $f \in \Aut(X)$, the restriction of $\tau\pi f \colon X\to C$ to $F$ is constant, as $C$ is not rational. Therefore, there exists a unique bijective map $h \colon C\to C$ such that $h \tau \pi=  \tau \pi f$. Since $\tau\pi$ admits a section, it follows that $h$ is an automorphism and this shows that $\Aut(X)$ preserves the fibration $\tau \pi \colon X\to C$. This shows that there exists a group homomorphism $\Aut(X)\to \Aut(C)$.
		
		We proceed similarly with the fibration $\pi\colon X\to S$. Let $\widetilde{F}$ be a fiber of $\pi$, that is isomorphic to $\PP^1$. There is a unique fiber $F$ of $\tau\pi$ such that $\widetilde{F}\subset F$, and $\widetilde{F}$ is a fiber of the $b$-th Hirzebruch surface $\pi_{|F}\colon F \to \mathbb{P}^1$. Then $f(\widetilde{F}) \subset f(F)$. Since $f$ preserves the fibration $\tau\pi$, the surface $f(F)$ is again a fiber of $\tau\pi$. As $b>0$ and $f$ respects the intersection form, it follows that $f(\widetilde{F})$ is a fiber of the $b$-th Hirzebruch surface $\pi_{|f(F)}\colon f(F)\to \PP^1$. This implies that $f$ induces a unique bijective map $g\colon S\to S$ such that $g \pi = \pi f$. Moreover, $\pi$ admits a section, so $g$ is an automorphism. This shows that there exists a group homomorphism $\Pi\colon \Aut(X) \to \Aut(S)$, that maps $f$ to $g$.

		(2): By \cite[Proposition 3.6]{Fong}, $\Aut(S)$ is an algebraic subgroup of $\Bir(S)$. We recall the proof, which goes as follows.
		Let $D= -K_S + m\tau^*(A)$, where $A$ is an ample divisor on $C$, and notice that $D$ is ample for $m>0$ large enough. Notice also that the numerical class of $D$ is fixed by $\Aut(S)$; thus, $\Aut(S)$ has finitely many connected components and it is an algebraic subgroup of $\Bir(S)$ by \cite[Theorem 2.10]{Brionnotes}.
		
		Now, to show that $\Aut(X)$ is an algebraic subgroup of $\Bir(X)$, we consider the divisor $D'=-K_X+ n \pi^*(D) $ with $n>0$. Since $\pi$ is a $\PP^1$-bundle, we have $-K_X\cdot F=2$ for any fiber $F$ of $\pi$. Then for $n$ large enough, the divisor $D'$ is ample. Now, it remains to see that the numerical class of $D'$ is fixed by $\Aut(X)$.
		Let $f\in \Aut(X)$. Then $f^*(-K_X) \sim -K_X$ and $f^*(\pi^*D) = \pi^*(g^*D)$, where $g = \Pi(f)$ is defined in \ref{Autalgebraicgroup.2}. The numerical class of $D$ is fixed by $\Aut(S)$, so the numerical class of $\pi^*D$ is fixed by $\Aut(X)$. Hence, the numerical class of $D'$ is also fixed by $\Aut(X)$. Using again \cite[Theorem 2.10]{Brionnotes}, it follows that $\Aut(X)$ has only finitely many connected components, and it is an algebraic subgroup of $\Bir(X)$.
		
		(3): First, we show that the group homomorphism $\Pi\colon \Aut(X) \to \Aut(S)$ is in fact a morphism of algebraic groups. By uniqueness of Blanchard's lemma, the restriction of $\Pi$ on $\Autz(X)$ equals $\pi_*\colon \Autz(X) \to \Autz(C)$ (see Proposition \ref{prop:Blanchard} and Corollary \ref{blanchard}).
		
		By \ref{Autalgebraicgroup.1}, we can write $\Aut(X) = G_0 \sqcup \cdots \sqcup G_n$, where each $G_i$ is connected and $G_0 =\Autz(X)$.
		For every $t \in G_i$, the following diagram commutes:
		\[
		\begin{tikzcd}[column sep=4em,row sep = 3em]
			G_0 \arrow[r, "t\cdot "] \arrow[d, "\pi_*"]& G_i\arrow[d,"\Pi_{\vert G_i}"] \\
			\Pi(G_0) \arrow[r,"\Pi(t) \cdot "] & \Pi(G_i),
		\end{tikzcd}
		\]
		where the horizontal arrows are the translations by $t$ and $\Pi(t)$. This shows that $\Pi$ is a morphism of algebraic groups. Therefore, $ \Pi(\Aut(X))$ is an algebraic subgroup of $\Aut(S)$, and we can write
		$ \Pi(\Aut(X))= \Pi(G_0)\cup \cdots \cup \Pi(G_n)$. Since $\Pi(G_0)$ is connected, it follows that $\Pi (G_0)\subseteq  \Pi(\Aut(X))^\circ$. In fact, this is an equality as they have the same dimension, thus $\pi_*(\Autz(X)) =  \Pi(\Aut(X))^\circ$.
	\end{proof}
	
	\begin{remark}
		\begin{itemize}
			\item Proposition \ref{Autalgebraicgroup} \ref{Autalgebraicgroup.2} is false if $b=0$. For example, let $X = C\times \PP^1 \times \PP^1$, and $\pi\colon X\to C\times \PP^1$ be the projection onto the first two factors. The permutation of two $\PP^1$-factors defines an automorphism $f\in \Aut(X)$, which does not lie in $\Autz(X)$ and does not preserve the fibration $\pi$. 
			\item Let $g\in \Autz(S)$. If $g^*X$ is $S$-isomorphic to $X$, then we can lift $g$ to an automorphism of $X$, i.e., $g$ lies in the image of $\Pi$. A priori, this lifting may not be in the connected component $\Autz(X)$. Proposition \ref{Autalgebraicgroup} \ref{Autalgebraicgroup.3} can be reformulated as follows: if $g \in \Pi(\Aut(X))^\circ \subset \Autz(S)$, then we can choose its lifting in $\Autz(X)$.
		\end{itemize}
	\end{remark}

	\subsection{Fiber products of geometrically ruled surfaces as $\FF_0$-bundles}

	~\bigskip

	The following lemmas show that the case of $\FF_0$-bundles over $C$ is particularly explicit, as they correspond to fiber products of geometrically ruled surfaces.
	
	\begin{lemma}\label{fiberproduct}
		Let $C$ be a smooth projective curve, and let $\tau_1\colon S_1\to C$ be a geometrically ruled surface. Then the following statements hold:
		\begin{enumerate}
			\item\label{fiberproduct.1} Let $\tau_2\colon S_2\to C$ be a geometrically ruled surface, and let $\pi_1\colon S_1\times_C S_2\to S_1$ and $\pi_2\colon S_1\times_C S_2 \to S_2$ be the natural projections. Then the morphism $\tau_1 \pi_1= \tau_2 \pi_2$ endows the fiber product $S_1\times_C S_2$ with the structure of an $\FF_0$-bundle over $C$. 
			\item\label{fiberproduct.2} Let $\pi_1\colon X\to S_1$ be a $\PP^1$-bundle such that $\tau_1\pi_1$ is an $\FF_0$-bundle. Then there exists a geometrically ruled surface $\tau_2\colon S_2\to C$ such that $X$ is $S_1$-isomorphic to the fiber product $S_1\times_C S_2$.  Moreover, the $\PP^1$-bundle $\pi_1$ is decomposable if and only if $\tau_2$ is decomposable.
		\end{enumerate}
	\end{lemma}
	
		\begin{proof}
		\ref{fiberproduct.1}: We fix an open cover $(U_i)_i$ of $C$ that simultaneously trivializes $\tau_1$ and $\tau_2$, and we denote the transition matrices of $\tau_1$ and $\tau_2$ by $s_{1,ij},s_{2,ij}\in \GL_2(\mathcal{O}_C(U_{ij}))$, respectively. For each $U_i$, define the isomorphisms:
		\[
		\begin{array}{cccc}
		 \alpha_i\colon  & (U_i\times \PP^1)\times_{U_i} (U_i\times \PP^1) & \longrightarrow& U_i\times \FF_0 \\
			 & (x,[y_0:y_1]),(x,[z_0:z_1]) & \longmapsto & (x,[y_0:y_1\ ;z_0:z_1]).
		\end{array} 
		\]
		Together, the isomorphisms $(\alpha_i)_i$ induces a $C$-isomorphism $\alpha\colon S_1\times_C S_2\to X$, where $X$ is an $\FF_0$-bundle over $C$.
		
		\ref{fiberproduct.2}: There exist trivializations of the $\FF_0$-bundle $\tau_1\pi_1$ such that its  transition maps are:
		\[
		\begin{array}{ccc}\label{transitionFF0}
			U_j\times \FF_0  & \dashrightarrow & U_i\times \FF_0 \notag\\
			(x,[y_0:y_1\ ;z_0:z_1])& \longmapsto & \left(x,\left[s_{2,ij}(x)\cdot \begin{pmatrix}
				y_0 \\ y_1
			\end{pmatrix}; s_{1,ij}(x)\cdot \begin{pmatrix}
			z_0\\z_1
		\end{pmatrix}\right] \right), \tag{$\star$}
		\end{array}
		\]
		where $s_{1,ij},s_{2,ij} \in \GL_2(\O_C(U_{ij}))^*$. Under this choice of trivializations, the $\PP^1$-bundle structure morphism $\pi_1$ corresponds to the projection onto $(x,[z_0:z_1])$. Hence, $s_{1,ij}$ are the transition matrices of the geometrically ruled surface $\tau_1\colon S_1\to C$.
		The projection onto $(x,[y_0:y_1])$ gives a geometrically ruled surface $\tau_2\colon S_2\to C$, with transition matrices $s_{2,ij}$, such that $X$ is $C$-isomorphic to $S_1\times_C S_2$.
		
		By the transition maps (\ref{transitionFF0}) above, the $\PP^1$-bundle $\pi_1$ is decomposable if and only if there exist a choice of trivializations such that $s_{2,ij}$ are diagonal matrices, or equivalently, if and only if $\tau_2$ is decomposable.
	\end{proof}
	
	\begin{remark}
		Since every $\PP^1$-bundle over $\PP^1$ is decomposable, it follows that $\FF_0$-bundles over $\PP^1$ are fiber products of the form $\FF_a\times_{\PP^1} \FF_{b}$. Then the projection onto either of the two factors yields a decomposable $\PP^1$-bundle. 
		Let $C$ be a smooth projective curve, that is not isomorphic to $\PP^1$. By Lemma \ref{fiberproduct}, there exists indecomposable $\PP^1$-bundles $\pi\colon X\to S$, where $\tau\colon S\to C$ is a geometrically ruled surface, such that $\tau\pi$ is an $\FF_0$-bundle.
	\end{remark}
	
	As a direct consequence of Blanchard's lemma, we obtain an explicit description of the automorphism groups of $\FF_0$-bundles:
	
	\begin{lemma}\label{autofiberproduct}
		Let $C$ be a smooth projective curve, and let $\tau_1\colon S_1\to C$ and $\tau_2\colon S_2\to C$ be geometrically ruled surfaces. Let $X$ be the fiber product $S_1\times_C S_2$, equipped with the two projections $\pi_1\colon X\to S_1$ and $\pi_2\colon X\to S_2$. Then the following hold:
		\begin{enumerate}
			\item\label{autofiberproduct.1} $\Autz(X)\simeq \Autz(S_1)\times_{\Autz(C)}  \Autz(S_2)$.
			\item $\ker({\pi_1}_*) = \ker({\tau_2}_*)$ and $\ker({\pi_2}_*) = \ker({\tau_1}_*)$.
			\item \label{FF0relativelymax} Assume that $g(C)>0$. If $\Autz(X)$ is relatively maximal with respect to one of the two $\PP^1$-bundle structures, then $\Autz(S_1)$ and $\Autz(S_2)$ are maximal connected algebraic subgroups of $\Bir(S_1)$ and $\Bir(S_2)$, respectively.
		\end{enumerate}
	\end{lemma}
	
	\begin{proof}
		(1): By Blanchard's lemma (\ref{blanchard}) and the universal property of fiber products, there exists a unique morphism $\Phi \colon \Autz(X)\to \Aut^\circ(S_1) \times_{\Autz(C)} \Autz(S_2)$, sending $f\in \Autz(X)$ to $({\pi_1}_*(f),{\pi_2}_*(f))$, and such that ${\tau_1}_*{\pi_1}_*(f)={\tau_2}_*{\pi_2}_*(f)$. The inverse of $\Phi$ sends $(f_1,f_2)\in \Aut^\circ(S_1) \times_{\Autz(C)} \Autz(S_2)$ to the element $(x_1,x_2) \mapsto (f_1(x_1),f_2(x_2))\in \Autz(X)$. Hence $\Phi$ is an isomorphism.
		
		(2): This follows from \ref{autofiberproduct.1}. The kernel of ${\pi_1}_*$ equals $\{\mathrm{id}_{S_1}\} \times_{\Autz(C)} \Autz(S_2)$. Since $\mathrm{id}_{S_1}$ is mapped to $\mathrm{id}_{C}$ by ${\tau_1}_*$, the second factor is also mapped to $\mathrm{id}_{C}$ by ${\tau_2}_*$. Thus, $\ker({\pi_1}_*) = \ker({\tau_2}_*)$. Permutting the roles of $S_1$ and $S_2$, we get the second statement.
		
		(3): Without loss of generality, we can assume that $\Autz(X)$ is relatively maximal with respect to the structure of $\PP^1$-bundle $\pi_1\colon X\to S_1$. 
		
		First, assume by contradiction that $\Autz(S_1)$ is not a maximal connected algebraic subgroup of $\Bir(S_1)$. Since $C$ is not rational, there exists a geometrically ruled surface $S'$ and an $\Autz(S_1)$-equivariant $C$-birational map $\phi\colon S_1 \dashrightarrow S'$ such that $\phi \Autz(S_1) \phi^{-1} \subsetneq  \Autz(S')$. The following diagram commutes
		\[
		\begin{tikzcd}[column sep=4em,row sep = 3em]
			X \arrow[r,dashed, "\phi\times \mathrm{id}_{S_2}"] \arrow[d, "\pi_1"]& S'\times_C S_2\arrow[d,"\pi'"] \\
			S_1\arrow[r,dashed,"\phi"] & S',
		\end{tikzcd}
		\]
		where $\pi'$ denotes the projection on $S'$. By \ref{autofiberproduct.1}, the square birational map $\phi\times \mathrm{id}_{S_2} $ is $\Autz(X)$-equivariant and conjugates $ \Autz(X)$ to a strict subgroup of $ \Autz(S'\times_C S_2)$, which contradicts that $\Autz(X)$ is relatively maximal with respect to $\pi_1$.
		
		If instead, we assume by contradiction that $\Autz(S_2)$ is not a maximal connected algebraic subgroup of $\Bir(S_2)$, the proof proceeds similarly by considering the $\Autz(X)$-equivariant square birational map $ \mathrm{id}_{S_1} \times \psi$, where $\psi \colon S_2 \dashrightarrow S'$ is $\Autz(S_2)$-equivariant and $\psi \Autz(S_2) \psi^{-1} \subsetneq  \Autz(S')$.
	\end{proof}
	
	Taking $S_1$ and $S_2$ such that $\Autz(S_1)$ and $\Autz(S_2)$ are maximal connected algebraic subgroups of $\Bir(S_1)$ and $\Bir(S_2)$, respectively, we obtain the following proposition:
	
	\begin{proposition}\label{candidateFF0max}
		Let $C$ be a smooth projective curve of genus $g\geq 1$ and $X=S_1\times_C S_2$ as in Lemma $\ref{autofiberproduct}$. If $\Autz(X)$ is relatively maximal with respect to one of the two $\PP^1$-bundle structures, then the following hold:
		\begin{enumerate}
			\item If $g=1$, then $X$ is isomorphic to one of the following:
			\begin{enumerate}
				\item $C\times \PP^1 \times \PP^1$,
				\item $\PP(\O_C\oplus \L)\times \PP^1$ for some non-trivial $\L\in \Pic^0(C)$,
				\item $\AAA_0\times \PP^1$,
				\item $\AAA_1\times \PP^1$,
				\item\label{candidateFF0max.v} $\PP(\O_C\oplus \L) \times_C \PP(\O_C\oplus \L')$ for some non-trivial $\L,\L'\in \Pic^0(C)$, 
				\item\label{candidateFF0max.vi} $\PP(\O_C\oplus \L)\times_C \AAA_0$ for some non-trivial $\L\in \Pic^0(C)$,
				\item $ \PP(\O_C\oplus \L)\times_C \AAA_1$ for some non-trivial $\L\in \Pic^0(C)$, 
				\item\label{candidateFF0max.viii} $\AAA_0\times_C \AAA_0$,
				\item $\AAA_0\times_C \AAA_1$,
				\item $\AAA_1\times_C \AAA_1$.
			\end{enumerate}
			\item If $g\geq 2$, then $X$ is isomorphic to $C\times \PP^1 \times \PP^1$.
		\end{enumerate}
		Moreover, in every case above, the morphism of connected algebraic groups $(\tau\pi)_*\colon \Autz(X) \to \Autz(C)$ induced by Blanchard's lemma is surjective.
	\end{proposition}
	
	\begin{proof} 
		By Lemma \ref{autofiberproduct}, $X$ is a fiber product of two geometrically ruled surfaces whose automorphism groups are maximal connected algebraic subgroups. Considering all the possibilities given in Theorem \ref{dim2max}, we obtain the list stated in this proposition. The last assertion follows from Lemmas \ref{geometryofruledsurface} and \ref{autofiberproduct} \ref{autofiberproduct.1}.
	\end{proof}
	
	\begin{remark}
		If $\Autz(S_1\times_C S_2)$ is relatively maximal with respect to one of the two $\PP^1$-bundle structures $\pi_1$ or $\pi_2$, then $S_1\times_C S_2$ is among the cases listed in Proposition \ref{candidateFF0max}. In Sections \ref{SectionCxPP1}, \ref{SectionAAA1}, \ref{SectionAAA0}, and \ref{SectionSL}, we determine precisely, case by case, which of these $\PP^1$-bundles have relatively maximal automorphism groups.
	\end{remark}
	
	\subsection{Transition maps of $\FF_b$-bundles for $b>0$}
	
	~\bigskip
	
	 For any ring $R$, we denote by $R[z_0,z_1]_b$ the set of homogeneous polynomials of degree $b$ with coefficients in $R$. In view of working with $\FF_b$-bundles via local trivializations, we describe their transition maps in the following lemmas:
	
	\begin{lemma}\label{autotrivializations}
		Let $b>0$, $U$ be an affine variety and let $f$ be an automorphism of $U\times \FF_b$ that fixes $U$ pointwise. Then there exist $\lambda \in \mathcal{O}(U)^*$, $p\in \mathcal{O}(U)[z_0,z_1]_b$ and $s=\begin{pmatrix}
			\alpha & \beta \\
			\gamma & \delta
		\end{pmatrix} \in \GL_2(\mathcal{O}(U))$ such that $f$ equals
		$$
		(x,[y_0:y_1;z_0:z_1]) \mapsto (x,\left[y_0:\lambda(x) y_1 + p(x,z_0,z_1)y_0\ ; \alpha(x) z_0 + \beta(x) z_1 : \gamma(x) z_0 + \delta(x) z_1 \right]).
		$$
	\end{lemma}
	
	\begin{proof}
		Since $b>0$, the structure morphism $\tau_b\colon \FF_b\to \PP^1$, $[y_0:y_1\ ;z_0:z_1]\mapsto [z_0:z_1]$ is $\Aut(\FF_b)$-equivariant. Then $\pi=\mathrm{id}_U\times\tau_b\colon U\times \FF_b \to U\times \PP^1$ is $\Aut_U(U\times \FF_b)$-equivariant, so $f$ induces an automorphism $\pi_*(f)$ of $U\times \PP^1$, i.e., there exists $s\in \GL_2(\mathcal{O}(U))$ such that $$\pi_*(f)=\left((x,[z_0:z_1])\mapsto \left(x,s(x)\cdot \begin{bmatrix} z_0 \\ z_1 \end{bmatrix}\right)\right).$$ 
		Define the $U$-automorphism $g\colon U\times \FF_b \to U\times \FF_b$, $(x,[y_0:y_1\ ;z_0:z_1]) \mapsto \left( x,\left[y_0: y_1\ ; s^{-1}(x)\cdot 
		\begin{pmatrix}z_0\\z_1\end{pmatrix}
		\right] \right)$ and $h=gf$. By construction, $h$ is a $U$-automorphism of $U\times \FF_b$ such that $\pi_*(h) = id_{U\times \PP^1}$. 
		
		Let $\AA^1_0=\PP^1\setminus \{[1:0]\}$ and $\AA^1_\infty=\PP^1\setminus \{[0:1]\}$ be trivializing open subsets, and we choose local trivializations of $\tau_b$ such that $\tau_b^{-1}(\AA^1_0) \simeq \AA^1_0 \times \PP^1$ and $\tau_b^{-1}(\AA^1_\infty) \simeq \AA^1_\infty \times \PP^1$.
		Since $h$ fixes pointwise the union of $(-b)$-sections $\{y_0=0\}\subset U\times \FF_b$, it implies that $h$ induces automorphisms of $U\times \AA^1\times \PP^1$ that fix $U\times \AA^1$ pointwise and acts on $\PP^1$ by lower triangular matrices in $\GL_2(\O(U)[z])$. The latters are defined up to multiplication by an element of $\O(U)^*$, so we may assume the upper left coefficient to be $1$. Therefore, there exist $p_0,p_\infty\in \O(U)[z]$ and $\lambda_0,\lambda_\infty \in \O(U)^*$ such that $h$ induces the following automorphisms:
		\[
		\begin{array}{cccc}
			h_0\colon &U\times \AA^1_0 \times \PP^1  & \longrightarrow &U\times \AA^1_0\times \PP^1 \\
			&(x, z,[y_0:y_1]) &\longmapsto& \left(x, z,
			\begin{bmatrix}
				1 & 0  \\  p_0(x,z) & \lambda_0(x) 
			\end{bmatrix} \cdot 
			\begin{bmatrix}
				y_0 \\ y_1
			\end{bmatrix}\right),
		\end{array}
		\]
		\[
		\begin{array}{cccc}
			h_\infty\colon &U \times \AA^1_\infty\times \PP^1 & \longrightarrow &U \times \AA^1_\infty\times \PP^1\\
			&(x, z,[y_0:y_1]) & \longmapsto &\left( x, z,
			\begin{bmatrix}
				1 & 0  \\  p_\infty(x,z) & \lambda_\infty(x) 
			\end{bmatrix} \cdot 
			\begin{bmatrix}
				y_0 \\ y_1
			\end{bmatrix}\right).
		\end{array}
		\]
		Moreover, $U\times \FF_b$ is the gluing of $U\times \AA_0^1\times \PP^1$ and $U\times\AA_\infty^1\times \PP^1$ through the transition maps
		\begin{align*}
			t_{\infty 0}\colon  U\times  \AA_0^1\times  \PP^1  & \dashrightarrow  U\times \AA_\infty^1 \times \PP^1  \\
			(x,z,[y_0:y_1] ) & \longmapsto \left( x,1/z,\begin{bmatrix} 1 & 0 \\ 0 & z^{-b}\end{bmatrix} \cdot \begin{bmatrix} y_0 \\ y_1 \end{bmatrix}\right) .
		\end{align*}
		The condition $ t_{\infty 0} h_0= h_\infty  t_{\infty 0}$ implies that $\lambda_0 = \lambda_\infty$ and $ p_0(x,z)=z^bp_\infty(x,1/z)$. In particular, $\deg_z(p_0)\leq b$ and $\deg_z(p_\infty)\leq b$. Therefore $h$ equals
		$$
		(x,[y_0:y_1\ ;z_0:z_1]) \mapsto( x,[y_0:\lambda(x) y_1+ p(x,z_0,z_1)y_0\ ;z_0:z_1]),
		$$
		for $\lambda=\lambda_0 = \lambda_\infty\in \O(U)^*$ and $p(x,z_0,z_1)=z_1^bp_0(x,z_0/z_1)=z_0^bp_\infty(x,z_1/z_0)\in \O(U)[z_0,z_1]_b$.
		Finally, compose $h$ by $g^{-1}$ on the left to conclude.
	\end{proof}
		
	Following \cite[Proposition 3.3.1, Lemma 3.3.6]{BFT} and their proofs, we obtain their analogues in the following lemmas:
	
	\begin{lemma}\label{FFb-bundle}
		Let $C$ be a smooth projective curve, $\tau\colon S\to C$ and $\pi\colon X\to S$ be $\PP^1$-bundles such that $\tau\pi$ is an $\FF_b$-bundle with $b>0$. For any section $\sigma$ of $\tau$, the following hold:
		\begin{enumerate}
			\item \label{FFb-bundle1} There exists a trivializing open cover $(U_i)_i$ with local trivializations $\phi_i\colon (\tau\pi)^{-1}(U_i) \overset{\sim}{\longrightarrow} U_i\times \FF_b$, such that the transition maps $\phi_{ij}$ of $\tau\pi$ equal:
			\[
			\begin{array}{cccc}\label{transitionmaps}
				\phi_{ij} = \phi_i\phi_j^{-1}\colon & U_j\times \FF_b & \dashrightarrow &U_i\times \FF_b \notag \\
				&(x,[y_0:y_1 \ ;z_0:z_1]) & \longmapsto&\left( x,\left[y_0:\lambda_{ij}(x)y_1 + p_{ij}(x,z_0,z_1)y_0\ ; s_{ij}(x)\cdot \begin{pmatrix}z_0\\z_1\end{pmatrix}\right]\right), 
			\end{array} 
			\]
			where $\lambda_{ij} \in \mathcal{O}_C(U_{ij})^*$,  $p_{ij}\in \mathcal{O}_C(U_{ij})[z_0,z_1]_b$ and
			$$s_{ij} = \begin{pmatrix}
				1 & 0 \\ b_{ij} &a_{ij}
			\end{pmatrix},$$
			with $a_{ij}\in \O_C(U_{ij})^*$ and $b_{ij}\in \O_C(U_{ij})$. Moreover, $s_{ij}$ are the transition matrices of the $\PP^1$-bundle $\tau$ and $\sigma$ is the zero section $\{z_0=0\}$.
			
			\item\label{canonicalextension} 
			There exists a unique rank-$2$ vector bundle $\mathcal{E}$ such that $\PP(\mathcal{E})=X$ and which fits into a short exact sequence:
			\begin{equation*}
				0 \to  \mathcal{O}_S(b\sigma+\tau^*(D)) \to \mathcal{E}\to\mathcal{O}_S \to 0,
			\end{equation*}
			where $D\in \Pic(C)$ is a divisor class, and such that the line subbundle $\mathcal{O}_S(b\sigma+\tau^*(D)) \subset \mathcal{E}$ corresponds to the section $S_b=\{y_0=0\}\subset X$ spanned by the $(-b)$-sections along the fibers of $\tau\pi$. Moreover, the coefficients $\lambda_{ij}$ in $\ref{FFb-bundle1}$ are the cocycles of the line bundle $\mathcal{O}_S(\tau^*(D))$.
			
			\item \label{FFb-decomposable} We can choose $p_{ij}=0$ in $\ref{FFb-bundle1}$ if and only if $\pi$ is a decomposable $\PP^1$-bundle; and in this case, $X$ is $S$-isomorphic to $\PP(  \mathcal{O}_S\oplus \mathcal{O}_S(b\sigma+\tau^*(D)))$. 
		\end{enumerate}
	\end{lemma}
	
	\begin{proof}
			(1): As in the proof of Lemma \ref{lemma removal jumping}, choosing an affine open covering $(U_i)_i$ of $C$ such that each $U_i$ is small enough, we can assume that $U_i$ trivializes the $\FF_b$-bundle $(\tau\pi)$. This gives the local trivializations $\phi_i\colon (\tau \pi)^{-1}(U_i)\to U_i\times \FF_b$. Now using Lemma \ref{autotrivializations}, there exist $\lambda_{ij} \in \mathcal{O}_C(U_{ij})^*$,  $p_{ij}\in \mathcal{O}_C(U_{ij})[z_0,z_1]_b$ and $s_{ij}\in \mathrm{GL}_2(\mathcal{O}_C(U_{ij}))$ such that the transition maps $\phi_{ij}$ equal
			\[
			\begin{array}{ccc}
				U_j\times \FF_b & \dashrightarrow & U_i \times \FF_b \notag \\
				(x,[y_0:y_1 \ ;z_0:z_1]) & \longmapsto & \left( x,\left[y_0:\lambda_{ij}(x)y_1 + p_{ij}(x,z_0,z_1)y_0\ ; s_{ij}(x)\cdot \begin{pmatrix}z_0\\z_1\end{pmatrix}\right]\right). 
			\end{array} 
			\]
			Denote by $\tau_b\colon \FF_b\to \PP^1$ the morphism $[y_0:y_1;z_0:z_1]\mapsto [z_0:z_1]$, then the following diagram commutes: 
			\[
			\begin{tikzcd}[column sep=4em,row sep = 3em]
				U_j\times \FF_b \arrow[r,dashrightarrow,"\phi_{ij}"] \arrow[d,rightarrow,"\mathrm{id}\times\tau_b" left] &  U_i\times \FF_b \arrow[d,rightarrow,"\mathrm{id}\times\tau_b"]\\
				U_j\times \PP^1\arrow[r,dashrightarrow,"\mathrm{id}\times s_{ij}\cdot" below] & U_i\times \PP^1.
			\end{tikzcd}
			\] 
			Therefore, $s_{ij}$ are the transition matrices of the $\PP^1$-bundle $\tau$. Next we show that we can choose $s_{ij}$ as stated. Choose some local trivializations of $\tau$ such that $\sigma$ is the zero section, i.e., there exist matrices $M_i\in \GL_2(\O_C(U_{i}))$ such that the automorphisms 
			\[
			\begin{array}{cccc}
				\alpha_i \colon  &U_i\times \FF_b & \longrightarrow & U_i\times \FF_b \\
				& (x,[y_0:y_1\ ;z_0:z_1]) & \longmapsto & \left( x,\left[y_0:y_1 \ ;M_i(x)\cdot \begin{pmatrix}z_0\\z_1 \end{pmatrix}\right]\right)
			\end{array}
			\]
			send $\pi^{-1}(\sigma)$ to the zero section $\{z_0=0\}$. Replacing $\phi_{ij}$ by $\alpha_i \phi_{ij} \alpha_j^{-1}$, the matrices $s_{ij}$ are replaced by lower triangular matrices $M_i \cdot s_{ij} \cdot M_j^{-1}$. Multiplying by an element in $\O_C(U_{ij})^*$, we may also assume that the upper left coefficients of the new transition matrices are $1$. We also obtain new $\lambda_{ij}\in \mathcal{O}_C(U_{ij})^*$ and $p_{ij}\in \mathcal{O}_C(U_{ij})[z_0,z_1]_b$; but the new transition maps $\phi_{ij}$ are as we want.
			
			(2): We write $s_{ij} = \begin{pmatrix}
				1 & 0 \\ b_{ij} &a_{ij}
			\end{pmatrix}$, where $a_{ij}\in \O_C(U_{ij})^*$ and $b_{ij}\in \O_C(U_{ij})$. For each $i$, let 
			\[
			\begin{array}{ccc}
				V_{i,0} & =&  U_i\times (\PP^1 \setminus [1:0]), \\
				V_{i,\infty} & = & U_i\times (\PP^1 \setminus [0:1]).
			\end{array}
			\]
			These open subsets of $U_i \times \PP^1$ trivialize the $\PP^1$-bundle $\pi$. Using \ref{FFb-bundle1}, for all $i,j$, we can write the transition maps of $\pi$ as follows:
			\[
			 \begin{array}{cccc}
			\phi_{ij,0}\colon	&V_{j,0} \times \PP^1 & \dashrightarrow &V_{i,0} \times \PP^1, \\
			&(x,[z:1],[y_0:y_1]) & \longmapsto &\left( x, \left[\frac{z}{b_{ij}(x)z+a_{ij}(x)}:1\right], \left[y_0:\left(b_{ij}(x)z+a_{ij}(x))^{-b}(p_{ij}(x,z,1)y_0 + \lambda_{ij}(x)y_1 \right)\right] \right);\\[10pt]
			
			\phi_{ij,\infty}\colon & V_{j,\infty} \times \PP^1 & \dashrightarrow & V_{i,\infty} \times \PP^1 \\
			& (x,[1:z],[y_0:y_1]) & \longmapsto & \left( x, \left[1:b_{ij}(x)+a_{ij}(x)z\right], \left[y_0: p_{ij}(x,1,z)y_0+  \lambda_{ij}(x)y_1\right]\right); \\[10pt]
			
			\phi_{j,\infty0}\colon & V_{j,0} \times \PP^1 & \dashrightarrow & V_{j,\infty} \times \PP^1 \\
			& (x,[z:1],[y_0:y_1]) & \longmapsto & \left( x, \left[1:z\right], \left[y_0: {z}^{-b}y_1\right]\right).	
			\end{array}
			\]
			In $\PGL_2(\kk(C))$, the following equality holds: $\phi_{i,\infty 0}\phi_{ij,0}=\phi_{ij,\infty}\phi_{j,\infty 0}$; and since the coefficients of $y_0$ are always $1$, we can lift this equality in $\GL_2(\kk(C))$. 
			
			Then, these birational maps are the transition maps of a rank-$2$ vector bundle $\E$, trivial over $V_{i,0}$ and $V_{i,\infty}$ for each $i$, and such that $\PP(\E)$ is $S$-isomorphic to $X$. The coefficients $\lambda_{ij}$ depend only on $x$, so they are the cocycles of a line bundle $\O_S(\tau^*(D))$ for some $D$ divisor on $C$. Moreover, $\sigma$ is the zero section $\{z_0=0\}$, so the transition maps of the line bundle $\O_S(b\sigma)$ are given by 
			\[
			\begin{array}{cccc}
				\psi_{ij,0}\colon	&V_{j,0} \times \AA^1 & \dashrightarrow &V_{i,0} \times \AA^1 \\
				&(x,[z:1],y) & \longmapsto &\left( x, \left[\frac{z}{b_{ij}(x)z+a_{ij}(x)}:1\right], (b_{ij}(x)z+a_{ij}(x))^{-b}y \right);\\[10pt]
				
				\psi_{ij,\infty}\colon & V_{j,\infty} \times \AA^1 & \dashrightarrow & V_{i,\infty} \times \AA^1 \\
				& (x,[1:z],y) & \longmapsto & \left( x, \left[1:b_{ij}(x)+a_{ij}(x)z\right], y\right); \\[10pt]
				
				\psi_{j,\infty0}\colon & V_{j,0} \times \AA^1 & \dashrightarrow & V_{j,\infty} \times \AA^1 \\
				& (x,[z:1],y) & \longmapsto & \left( x, \left[1:z\right], {z}^{-b}y\right).	
			\end{array}
			\]
			This implies that the section $S_b=\{y_0=0\}\subset X$ corresponds to the line subbundle $\O_S(b\sigma+ \tau^*(D))\subset \E$. The coefficients in front of $y_0$ are always equal to $1$, so the quotient line bundle $\E/\O_S(b\sigma+ \tau^*(D))$ is trivial. This gives the short exact sequence stated in \ref{canonicalextension}. Since we impose the condition that $\O_S(b\sigma+ \tau^*(D))$ corresponds to the section $S_b$, the rank-$2$ vector bundle $\E$ is uniquely determined.
			
			(3): If $p_{ij}=0$, it follows from \ref{FFb-bundle1} that the sets $\{y_0=0\}$ and $\{y_1=0\}$ are invariant by the transition maps $\phi_{ij}$; hence, they define two disjoint sections of $\pi$, which is the decomposable $\PP^1$-bundle $\PP(  \mathcal{O}_S\oplus \mathcal{O}_S(b\sigma+\tau^*(D)))$. Conversely, if $\pi$ is decomposable, we can choose its transition matrices being diagonal; thus, $p_{ij}=0$. 
	\end{proof}
	
	Next, we provide a criterion in terms of these explicit transition maps, for two $\PP^1$-bundles to be isomorphic.
	
	\begin{lemma}\label{Sisomorph}
		Let $C$ be a smooth projective curve, $\tau\colon S\to C$ and $\pi\colon X\to S$ be $\PP^1$-bundles such that $\tau\pi$ is an $\FF_b$-bundle with $b>0$. Let $\pi'\colon X'\to S$ be a $\PP^1$-bundle such that $\tau\pi'$ is a $\FF_{b'}$-bundle for some $b'>0$. Then there exists an open cover $(U_i)_i$, such that each $U_i$ trivializes $\tau\pi$ and $\tau\pi'$ simultaneously, and their transition maps are given by
		\[
		\begin{array}{cccc}
			\phi_{ij}\colon & U_j \times \FF_b & \dashrightarrow & U_i\times \FF_b \\
			&(x,[y_0:y_1\ ;z_0:z_1]) & \longmapsto& \left(x,\left[y_0:\lambda_{ij}(x)y_1 + p_{ij}(x,z_0,z_1)y_0\ ; s_{ij}(x)\cdot \begin{pmatrix}z_0\\z_1\end{pmatrix}\right]\right); \\[10pt]
			\phi'_{ij}\colon  & U_j \times \FF_{b'} & \dashrightarrow &U_i\times \FF_{b'} \\
			& (x,[y_0:y_1\ ;z_0:z_1]) & \longmapsto& \left(x,\left[y_0:\lambda'_{ij}(x) y_1 + p'_{ij}(x,z_0,z_1)y_0\ ; s_{ij}(x)\cdot \begin{pmatrix}z_0\\z_1\end{pmatrix}\right]\right);	
		\end{array}
		\]
		where $\lambda_{ij},\lambda'_{ij}\in \mathcal{O}_C(U_{ij})^*$, $p_{ij}\in \mathcal{O}_C(U_{ij})[z_0,z_1]_b$, $p'_{ij}\in \mathcal{O}_C(U_{ij})[z_0,z_1]_{b'}$ and $s_{ij}\in \GL_2(\O_C(U_{ij}))$ are the transition matrices of the $\PP^1$-bundle $\tau$. Moreover, the following are equivalent:
		\begin{enumerate}
			\item \label{binvariant} $\pi$ and $\pi'$ are $S$-isomorphic,
			\item \label{conditioniso} $b=b'$, and there exist $\mu_i\in \mathcal{O}_C(U_{i})^*$, $\mu_j\in \mathcal{O}_C(U_{j})^*$, $q_{j}\in \mathcal{O}_C(U_{j})[z_0,z_1]_b$, $q_i\in \mathcal{O}_C(U_{i})[z_0,z_1]_b$ such that $\lambda_{ij}' = \mu_{i} \mu_{j}^{-1} \lambda_{ij}$,  and 
			$$p'_{ij}(x,z_0,z_1) = \mu_i(x) p_{ij}(x,z_0,z_1)-\mu_{i}(x) \mu_{j}^{-1} (x)\lambda_{ij}(x) q_j(x,z_0,z_1)+q_i\left(x,s_{ij}(x)\cdot \begin{pmatrix} z_0\\ z_1\end{pmatrix}\right).$$
		\end{enumerate}
	\end{lemma}
	
	\begin{proof}
		By choosing the trivializing opens $U_i$ smaller if necessary, we can assume that the opens $U_i$ trivialize $\tau\pi$ and $\tau\pi'$ simultaneously. Then the first part of the statement follows from Lemma \ref{FFb-bundle} \ref{FFb-bundle1}.
		
		 Assume that $\pi$ and $\pi'$ are isomorphic as $\PP^1$-bundles. In particular, the fibers of $\tau\pi$ and $\tau\pi'$ are isomorphic, hence $b=b'$. By Lemma \ref{autotrivializations}, there exist automorphisms $\alpha_j$ and $\alpha_i$ defined as
		\[
		\begin{array}{cccc}
			\alpha_j  \colon & U_j\times \FF_b & \longrightarrow & U_j\times \FF_{b} \\
			& (x,[y_0:y_1\ ;z_0:z_1]) & \longmapsto & \left(x,\left[y_0:\mu_{j}(x)y_1 + q_{j}(x,z_0,z_1)y_0\ ;  z_0:z_1\right]\right), \\[10pt]
			\alpha_i \colon & U_i\times \FF_b & \longrightarrow & U_i\times \FF_{b} \\
			&(x,[y_0:y_1\ ;z_0:z_1]) & \longmapsto & \left(x,\left[y_0:\mu_{i}(x)y_1 + q_{i}(x,z_0,z_1)y_0\ ;  z_0:z_1\right]\right), \\
		\end{array}
		\]
		where $\mu_j\in  \mathcal{O}_C(U_{j})^*$, $\mu_i\in  \mathcal{O}_C(U_{i})^*$, $q_{j}\in \mathcal{O}_C(U_{j})[z_0,z_1]_b$, $q_i\in \mathcal{O}_C(U_{i})[z_0,z_1]_b$ and such that $\phi'_{ij} = \alpha_i \phi_{ij} \alpha_j^{-1}$. Using the latter equality and identifying the coefficients, we get that $\lambda_{ij}' = \mu_{i} \mu_{j}^{-1} \lambda_{ij}$ and 
		$$p'_{ij} (x)= \mu_i(x) p_{ij}(x,z_0,z_1)-\mu_{i}(x) \mu_{j}^{-1}(x) \lambda_{ij}(x)q_j(x,z_0,z_1)+q_i\left(x,s_{ij}(x)\cdot \begin{pmatrix} z_0\\z_1\end{pmatrix} \right).$$
		
		Conversely, if the equalities above hold, then we define the automorphisms $\alpha_i\colon U_i\times \FF_b \to U_i\times \FF_b$ as before and we get that $\phi'_{ij} = \alpha_i \phi_{ij} \alpha_j^{-1}$, i.e., $\pi$ and $\pi'$ are $S$-isomorphic.
	\end{proof}
	
	\begin{remark}\label{rem.tauto}
		By Lemma \ref{Sisomorph}, once we fix the section $\sigma$, the linear class of $D$ in the short exact sequence of Lemma \ref{FFb-bundle} \ref{canonicalextension} is an invariant for the $S$-isomorphism class of $X$.
		
		In the next sections, we choose $\sigma$ to be a minimal section of $\tau$ such that $\O_S(\sigma)\simeq \O_{\PP(\E_S)}(1)$ is the twisting sheaf and $\E_S$ is a normalized rank-$2$ vector bundle over $C$ such that $\PP(\E_S) \simeq S$ (see e.g.\ \cite[V. 2. Notation 2.8.1 and Proposition 2.9]{Hartshorne}). This minimal section depends on the choice of $\E_S$.
	\end{remark}
	
	To each $\PP^1$-bundle $\pi\colon X\to S$ we associate a triplet of invariants $(S,b,D)$, which is a generalization of the numerical invariants $(a,b,c)$ introduced in \cite[Definition 1.4.1]{BFT}. 
	
	\begin{definition}\label{definvariants}
		Let $b>0$, let $C$ be a smooth projective curve, and $\tau\colon S\to C$ be a $\PP^1$-bundle. We fix a section $\sigma$ such that $\O_S(\sigma)\simeq \O_S(1)$ is the twisting sheaf. For every $\PP^1$-bundle $\pi\colon X \to S$ such that $\tau\pi$ is an $\FF_{b}$-bundle, we say that \emph{$\pi$ has invariants $(S,b,D)$}, where $D$ is the linear class of the divisor in Lemma \ref{FFb-bundle} \ref{canonicalextension}. 
	\end{definition}

\subsection{Sarkisov Program for $\FF_b$-bundles}

~\bigskip

The Sarkisov program is a non-deterministic algorithm that decomposes every birational map between two Mori fiber spaces into elementary ones, the so-called Sarkisov links. This result is due to Corti in dimension three, Hacon and McKernan in higher dimensions; see \cite{Corti,HM13}. Its equivariant version, for connected algebraic groups, is due to Floris:  

\begin{theorem}[{\cite[Theorem 1.2]{Floris}}] \label{Floris}
	Let $G$ be a connected algebraic group acting on two Mori fiber spaces $\pi\colon X\to B$ and $\pi'\colon X'\to B'$. If $\phi\colon X\dashrightarrow X'$ is a $G$-equivariant birational map, then there exists a decomposition of $\phi$ as a product of $G$-equivariant Sarkisov links.
\end{theorem}

Let $\pi\colon X\to S$ be a $\PP^1$-bundle. Recall from Definition \ref{def} that $\Autz(X)$ is  relatively maximal if, for every $\Autz(X)$-equivariant square birational map
\[
\phi\colon X\dashrightarrow X'
\]
where $\pi'\colon X'\to S'$ is a $\PP^1$-bundle, we have $\phi \Autz(X) \phi^{-1} = \Autz(X')$. By Theorem \ref{Floris}, $\phi$ factorizes as a composition of $\Autz(X)$-equivariant Sarkisov links. In the forthcoming Sections \ref{SectionCxPP1}, \ref{SectionAAA1}, \ref{SectionAAA0}, and \ref{SectionSL}, we will use the $\Autz(X)$-equivariant Sarkisov Program for every $\PP^1$-bundle structure $\pi\colon X\to S$ to determine whether $\Autz(X)$ is relatively maximal.

\begin{remark}\label{rem:no_contraction_bgeq2}
	Let us recall that Mori fiber spaces are $\mathbb{Q}$-factorial with terminal singularities. In particular, there is no Sarkisov link whose target Mori fiber space is a variety with non-terminal singularities. 
	
	Let $b>0$, let $C$ be a smooth projective curve and $X\to C$ be an $\FF_b$-bundle. Let $\kappa \colon X\to X'$ be the contraction of the $(-b)$-sections along the fibers. Then $X'$ is non-terminal if and only if $b \geq 2$.
	To see this, it suffices to compute the discrepancies, as follows.
	
	Let $\eta \colon \FF_b \to T$ be the contraction of the $(-b)$-section $\sigma_b \subset \FF_b$. In particular, $T$ is a normal projective surface and
	\[
	K_{\FF_b} \sim_{\mathbb{Q}} \eta^*(K_{T}) + a  \sigma_b,
	\]
	for some $a\in \mathbb{Q}$. In dimension two, being smooth and terminal are equivalent (see e.g. \cite[Corollary 5.18]{Kollar-Mori}); hence, $T$ is non-terminal, i.e., $a \leq 0$ if and only if $b\geq 2$.
	
	Since it is a local computation and $X$ is locally isomorphic to the trivial $\FF_b$-bundle, we can assume that $X = C\times \FF_b$ and $X' = C\times T$ and $\kappa = \mathrm{id}_C \times \eta$. Moreover, the following diagram commutes
	\[
	\begin{tikzcd}[column sep=2em,row sep = 3em]
		\FF_b \arrow[rr, "\eta"] && T \\
		C \times \FF_b \arrow[rr,"\mathrm{id}_C \times \eta"] \arrow[u, "p_{\FF_b}"]\arrow[rd, "p_C" swap]&& C\times T\arrow[u, "q_T" swap]\arrow[ld, "q_C"] \\
		& C &,
	\end{tikzcd}
	\]
	where $p_{\FF_b}$, $p_C$, $q_T$, $q_C$ are projections. Using that $K_{C\times \FF_b} \sim p_C^*K_C + p_{\FF_b}^*K_{\FF_b} $ and $K_{C\times T} \sim q_C^*K_C + q_T^*K_T$, we obtain that
	\[
	K_{C\times \FF_b} \sim_{\mathbb{Q}} (\mathrm{id}_C \times \eta)^*(K_{C\times T}) + a(C\times \sigma_b).
	\]
	Thus, $X'$ is non-terminal if and only if $b\geq 2$.
\end{remark}

The following lemma provides all possible Sarkisov diagrams of type III and IV, starting from a $\PP^1$-bundle $\pi\colon X\to S$ over a geometrically ruled surface.

\begin{lemma}\label{SarkisovIIIandIV}
	Let $C$ be a smooth projective curve of positive genus, $\tau\colon S\to C$ and $\pi\colon X\to S$ be $\PP^1$-bundles such that $\tau\pi$ is an $\FF_b$-bundle with $b\geq 0$. The following hold:
	\begin{enumerate}
		\item\label{SarkisovIIIandIV.1} If $b\geq 2$, there is no Sarkisov diagram of type III and IV starting from $\pi$.
		\item\label{SarkisovIIIandIV.2} If $b=1$, there exists a Sarkisov diagram of type III:
		\[
		\begin{tikzcd} [column sep=4em,row sep = 3em]
			X \arrow[d,"\pi",swap]\arrow[dr,"\kappa"] &  \\
			S\arrow[dr,"\tau",swap] & X'\arrow[d,"\pi'"]  \\
			&C,
		\end{tikzcd}
		\] 
		where $\kappa$ is the contraction of the surface spanned by the $(-1)$-sections along the fibers of $\tau\pi$, and $\pi'\colon X'\to C$ is a $\PP^2$-bundle.
		\item\label{SarkisovIIIandIV.3} If $b=0$, there exists a $\PP^1$-bundle $\tau'\colon S'\to C$ such that $X\simeq S\times_C S' $. Via this isomorphism, the structure morphism $\pi$ is identified with the projection on the first factor. Denoting by $\pi'$ the projection on the second factor, the pullback diagram yields a Sarkisov diagram of type IV:
		\[
		\begin{tikzcd} [column sep=2em,row sep = 3em]
			S\times_C S'\arrow[d,"\pi",swap]\arrow[rr,equal] && S\times_C S'\arrow[d,"\pi'"] \\
			S\arrow[dr,"\tau",swap] && S'\arrow[dl,"\tau'"]  \\
			&C&.
		\end{tikzcd}
		\] 
	\end{enumerate}
	In the cases \ref{SarkisovIIIandIV.2} and \ref{SarkisovIIIandIV.3}, these diagrams are the only Sarkisov diagrams of type III and IV starting from $\pi\colon X\to S$, and the induced Sarkisov links are $\Autz(X)$-equivariant.
\end{lemma}

\begin{proof}
	We consider Sarkisov diagrams of type III and IV starting from $\pi\colon X\to S$. Since $S$ is a geometrically ruled surface, it has Picard rank $2$, and one extremal ray of the cone of curves $\overline{\mathrm{NE}}(S)$ determines the structure morphism $\tau\colon S\to C$. The other extremal ray contracts the numerical class of the minimal section, so in particular the fibers of this morphism are not rationally connected, and we may exclude this case in the Sarkisov diagrams; see \cite{HMShokurov}. This implies that in our diagrams of type III and IV, the lower left arrow is the structure morphism $\tau\colon S\to C$ and we can determine the rest of the Sarkisov diagram by the $2$-ray game.
	
	If $b>0$, we denote by $S_b\subset X$ the surface spanned by the $(-b)$-sections along the fibers of $\tau\pi$. Over the generic fiber of $\tau\pi\colon X\to C$, which is isomorphic to $\FF_b$ over $\kk(C)$, we contract the $(-b)$-section. Therefore, if there exists a Sarkisov diagram of type III or IV, the induced Sarkisov link necessarily contracts $S_b$. 
	If $b\geq 2$, this yields a threefold with non-terminal singularities, which is prohibited by the Sarkisov program (see Remark \ref{rem:no_contraction_bgeq2}). If $b=1$, this yields a birational morphism $\kappa\colon X\to X'$, where $\pi'\colon X'\to C$ is a $\PP^2$-bundle over $C$. Since $S_1$ is $\Autz(X)$-invariant, it follows that $\kappa$ is $\Autz(X)$-equivariant. This proves \ref{SarkisovIIIandIV.1} and \ref{SarkisovIIIandIV.2}.
	
	If $b=0$, there exists a $\PP^1$-bundle $\tau'\colon S'\to C$ such that $X \simeq S\times_C S'$ by Lemma \ref{fiberproduct}. This yields an $\Autz(X)$-equivariant Sarkisov diagram of type IV, where the RHS Mori fiber space is given by the projection on $S'$. The induced Sarkisov link is an isomorphism; hence, it is $\Autz(X)$-equivariant. This proves \ref{SarkisovIIIandIV.3}.
\end{proof}

To complete Lemma $\ref{SarkisovIIIandIV}$, we would need to determine the $\Autz(X)$-orbits and the $\Autz(X)$-Sarkisov diagrams of type I and II. If every $\Autz(X)$-orbit has dimension at least $2$, we can already state:

\begin{lemma}\label{LinkfromFFb}
	Let $C$ be a smooth projective curve of positive genus, $\tau\colon S\to C$ and $\pi\colon X\to S$ be $\PP^1$-bundles such that $\tau\pi$ is an $\FF_b$-bundle with $b\geq 0$. If the $\Autz(X)$-orbits have dimension at least $2$, then $\Autz(X)$ is relatively maximal and the pair $(X,\pi)$ is superstiff. 
\end{lemma}

\begin{proof}
	Since every $\Autz(X)$-orbit has dimension at least two, its blowup yields an isomorphism. Therefore, there is no $\Autz(X)$-equivariant Sarkisov diagram of type I and II. By Lemma \ref{SarkisovIIIandIV}, there is also no $\Autz(X)$-equivariant Sarkisov diagrams of type III and IV if $b\geq 2$; and there is no $\Autz(X)$-equivariant square birational map if $b\in \{0,1\}$.
	This implies that $\Autz(X)$ is relatively maximal and the pair $(X,\pi)$ is superstiff.
\end{proof}

\section{Second reduction step: towards $\FF_b$-bundles without invariant fibers}

\subsection{$\FF_b$-bundles with invariant fibers}

\begin{remark}
	Let $C$ be a smooth projective curve of positive genus $g\geq 1$ and $\tau\pi\colon X\to C$ be an $\FF_b$-bundle with $b>0$. Recall that, by Blanchard's lemma (see Corollary \ref{blanchard}), the structure morphism $\tau \pi$ induces a morphism of algebraic groups:
	\[
	(\tau\pi)_*\colon \Autz(X) \to \Autz(C).
	\]
	If $g\geq 2$, then $\Autz(C)$ is trivial, so $(\tau\pi)_*$ is trivial. If $g=1$, then $\Autz(C)\simeq C$ and $(\tau\pi)_*$ may be trivial or surjective.
\end{remark}

Following the ideas of \cite[\S 3]{BFT}, we show in this section that if $b>0$ and the $\FF_b$-bundle $(\tau\pi)_*$ is trivial, then $\Autz(X)$ is not relatively maximal. This allows us to restrict our study to $\PP^1$-bundles over geometrically ruled surfaces whose automorphism groups are maximal connected algebraic subgroups of $\Bir(C\times \PP^1)$.

\begin{remark}
Unlike the case where $S$ is isomorphic to a Hirzebruch surface, not every element $f\in \Aut_C(S)$ is induced by an element of $\Aut_C(\E)$, where $\E$ is a rank-$2$ vector bundle such that $\PP(\E) \cong S$. In fact, by \cite[\S 5]{Grothendieck}, $f$ is induced by an isomorphism $\E\cong \E\otimes \mathcal{N}$, where $\mathcal{N}$ is a line bundle, and there exists a short exact sequence
\[
0 \to \Aut_C(\E) /\mathbb{G}_m \to \Aut_C(S) \to \Delta \to 0,
\]
where $\Delta$ is a subgroup of the two-torsion subgroup of the Jacobian of $C$. More precisely,
\[
\Delta \simeq \{\mathcal{N} \in \Pic^0(C), \E \cong \E\otimes \mathcal{N}\}.
\]
\end{remark}

Adapting the proof of \cite[Lemma 3.7.1]{BFT}, we obtain:

\begin{lemma}\label{lem:equivarianttodecomposable}
	Let $C$ be a smooth projective curve of genus $g(C)\geq 1$, $\tau\colon S \to C$ be a geometrically ruled surface with a section $\sigma$ such that $\O_S(\sigma)\simeq \O_S(1)$, and $\pi\colon X\to S$ be a $\PP^1$-bundle such that $\tau\pi$ is an $\FF_b$-bundle with $b>0$. If the morphism $(\tau \pi)_*\colon \Autz(X)\to \Autz(C)$ is trivial and one of the following condition holds:
	\begin{enumerate}
		\item $b$ is odd,
		\item $\Delta$ is trivial,
		\item $\seg(S)>0$.
	\end{enumerate}
	Then there exists an $\Autz(X)$-equivariant square birational map $X \dashrightarrow \PP( \mathcal{O}_S\oplus \mathcal{O}_S(b\sigma+\tau^*(D)) )$, where $D$ is a divisor on $C$ that can be chosen with arbitrarily large degree.
\end{lemma}

\begin{proof}
	We fix a trivializing open cover $(U_i)_i$ of the $\FF_b$-bundle $\tau\pi$, such that $\sigma$ is the zero section $\{z_0=0\}$. We also choose a rank-$2$ vector bundle $\E$, such that $\PP(\E)=S$, and with transition matrices 
	\[s_{ij} =
	\begin{pmatrix}
		1 & 0 \\ b_{ij} & a_{ij}
	\end{pmatrix} \in \GL_2(O_C(U_{ij})).
	\]		
	Let $\rho \in \Autz(X)$. Then $\rho$ induces an automorphism $\pi_*(\rho)\in \Autz(S)$ by Blanchard's lemma (\ref{blanchard}). For each $i$, the automorphism $\pi_*(\rho)$ induces an automorphism of $U_i \times \PP^1$ via the local trivializations of $\tau$, given by
	\[
	(x,[z_0:z_1]) \mapsto (x,r_i(x)\cdot [z_0:z_1]),
	\]
	where $r_i \in \GL_2(\O_C(U_i))$. By \cite[\S 5]{Grothendieck}, $\pi_*(\rho)$ is induced by an isomorphism $\E\cong \mathcal{N}\otimes \E$, where $\mathcal{N}$ is a two-torsion element of the Jacobian of $C$; hence, for all $i,j$, we have the equality $r_i s_{ij} = \xi_{ij} \otimes s_{ij} r_j$, where $\xi_{ij} \in \O_C(U_{ij})^*$ denote the cocycles of $\mathcal{N}$.

	Via the local trivializations of $\tau\pi$, the automorphism $\rho$ induces an automorphism of $U_i\times \FF_b$ 
	$$
	\rho_i\colon (x,[y_0:y_1\ ;z_0:z_1])\longmapsto \left(x,\left[y_0:\mu_i(x)y_1 +y_0\sum_{k=0}^{b}q_{i,k}(x)z_0^kz_1^{b-k}\ ;r_i(x)\cdot \begin{pmatrix}z_0\\z_1\end{pmatrix}\right]\right)
	$$
	where $\mu_i\in {\mathcal{O}_C}(U_i)^*$ and $q_{i,k}\in \mathcal{O}_C(U_i)$ (see Lemma \ref{autotrivializations}). 
	
	Fix an index $j$. First we prove that for all $i$, $\mu_j=\mu_i \in \kk^*$.  Developing the equality $\phi_{ij} \rho_j = \rho_i \phi_{ij}$, where $\phi_{ij}$ denote the transition maps of $(\tau\pi)_*$ as in Lemma \ref{FFb-bundle}, one gets for all $x\in U_{ij}$: 
	\begin{multline*}
		\left[y_0: \mu_j(x) \lambda_{ij}(x) y_1 +  \left(\lambda_{ij}(x) \sum_{k=0}^{b}q_{j,k}(x)z_0^kz_1^{b-k}+p_{ij}\left(x,r_j(x)\cdot \begin{pmatrix}z_0\\z_1\end{pmatrix}\right) \right) y_0\ ;s_{ij}(x)\cdot r_j(x)\cdot \begin{pmatrix}z_0\\z_1\end{pmatrix}\right]= \\ 	
		\left[y_0:\lambda_{ij}(x)\mu_i (x)y_1 +  \left(\mu_i (x)p_{ij}(x,z_0,z_1)+  \sum_{k=0}^{b}q_{i,k}(x)z_0^k\left(b_{ij}(x)z_0+a_{ij}(x)z_1\right)^{b-k}\right)y_0\ ;r_i(x)\cdot s_{ij}(x)\cdot \begin{pmatrix}z_0\\z_1\end{pmatrix}\right].
	\end{multline*}
	Now, using that $\xi_{ij}\otimes s_{ij}r_j = r_is_{ij}$, we get $\xi_{ij}^b \mu_j= \mu_i$ on $U_{ij}$. This yields an element $\mu\in \Gamma(C,\mathcal{N}^b)$. 
	
	Now we prove that under one of the three given additional assumptions, $\pi_*(\rho)$ is induced by an element of $\Aut_C(\E)$. If $b$ is odd, then $\mathcal{N}^b \cong \mathcal{N}$. Since $\mu \neq 0$, it follows that $\mathcal{N}$ is trivial. The same conclusion holds as well if $\Delta$ is trivial. If $\seg(S)>0$ and $g(C)\geq 2$, then $\Autz(S)$ is trivial by \cite[Theorem 2]{Maruyama}; hence, $\pi_*(\rho)$ is the identity on $S$. If $\seg(S)>0$ and $g(C)=1$, then $S\cong \AAA_1$ and $\Autz(S)$ is an elliptic curve which surjects onto $\Autz(C)$ via $\tau_*$; hence, the assumption that $(\tau\pi)_*$ is trivial already implies that $\pi_*$ is trivial and $\pi_*(\rho)$ is the identity. In all those cases, we obtain that $r_i s_{ij} =  s_{ij} r_j$ and $\mu\in \kk^*$. 
	
	Let $\{p_1,\cdots ,p_n\}$ be the complement of $U_j$. Since $\Autz(X)$ has a structure of an algebraic group, the rational map
	\[
	\begin{array}{clcc}
		\Phi\colon &\Autz(X) \times (U_j \times \FF_b) & \dasharrow &U_j\times \FF_b \\
		&\big(\rho,(x,[y_0:y_1\ ;z_0:z_1])\big) & \longmapsto &\rho_j (x,[y_0:y_1\ ;z_0:z_1])
	\end{array}
	\]
	has bounded poles. Therefore, for each $i$, the integer $-\nu_{p_i}(q_{j,k})$ is bounded above by an integer which is independent of $\rho$. It follows that for each $i$, the quantity
	$$-\nu_{(p_i,z_0,z_1)}\left( (x,z_0,z_1)\mapsto \sum_{k=0}^{b} q_{j,k}(x)(z_0-b_{ij}(x)a_{ij}^{-1}(x)z_1)^ka_{ij}^{-(b-k)}(x)z_1^{b-k} \right)$$ is also bounded by an integer $c_i\geq 0$ which is independent of $\rho$.
	
	Set $D=\sum_{i=1}^{n} d_ip_i$, with $d_i \geq c_i$. To conclude, we show that $\rho_j$ extends to an automorphism of $\PP(  \mathcal{O}_S\oplus \mathcal{O}_S(b\sigma+\tau^*(D)))$. Denote the transition maps of $\PP(  \mathcal{O}_S\oplus \mathcal{O}_S(b\sigma+\tau^*(D)))$ by
	\[
	\begin{array}{cccc}
		\theta_{ij} \colon & U_j\times \FF_b & \dashrightarrow & U_i\times \FF_b \\
		& (x,[y_0:y_1\ ;z_0:z_1])& \longmapsto & \left(x,\left[y_0: \widetilde{\lambda}_{ij}(x) y_1\ ;s_{ij}(x)\cdot 
		\begin{pmatrix}
			z_0 \\ z_1
		\end{pmatrix}\right]\right),
	\end{array}
	\]
	where $\widetilde{\lambda}_{ij}$ are the cocyles of the line bundle $\O_C(D)$ (see Lemma \ref{FFb-bundle} \ref{FFb-bundle1} and \ref{canonicalextension}). Then for all $\left(x,[y_0:y_1\ ;z_0:z_1]\right)\in U_{ij}\times \FF_b$, we get that  $\theta_{ij}\rho_j\theta_{ij}^{-1}\left(x,[y_0:y_1\ ;z_0:z_1]\right)$ equals
	\begin{multline*}
		\Bigg(x,\Bigg[y_0:\mu y_1+ \left(\widetilde{\lambda}_{ij}(x)\sum_{k=0}^{b} q_{j,k}(x)(z_0-b_{ij}(x)a_{ij}^{-1}(x)z_1)^ka_{ij}^{-(b-k)}(x)z_1^{b-k}\right) y_0 \ ;\\ 
		s_{ij}(x)\cdot r_j(x)\cdot {s_{ij}}^{-1}(x)\cdot \begin{pmatrix}z_0\\z_1\end{pmatrix} \Bigg] \Bigg).
	\end{multline*}
	
	By shrinking $U_i$ if necessary, it follows from Lemma \ref{transitionmaplinebundle}, that we can choose $\widetilde{\lambda}_{ij}\in \kk(C)^*$ such that $\mathrm{div}_{\vert U_i}(\widetilde{\lambda}_{ij})= D_{\vert U_i}$. Therefore, in the formula just above, the zeroes of $\widetilde{\lambda}_{ij}$ cancel the poles inside the sum. Hence, for every $\rho \in \Autz(X)$, the composition $\theta_{ij}\rho_j\theta_{ij}^{-1}$ is an automorphism of $U_i\times \FF_b$, i.e., $\rho_j$ extends to an element of $\Autz(\PP(  \mathcal{O}_S\oplus \mathcal{O}_S(b\sigma+\tau^*(D))))$. This yields an $\Autz(X)$-equivariant square birational map $X\dashrightarrow \PP(  \mathcal{O}_S\oplus \mathcal{O}_S(b\sigma+\tau^*(D)) ))$ above $S$.
	
	Finally, notice that $\{d_i\geq c_i\}_i$ are the only conditions imposed in the choice of $D$, so the latter can be chosen with arbitrarily large degree.
\end{proof}

We are now ready to state the main result of this section, which is analogue of \cite[Lemma 3.7.1]{BFT}. Before, we make the following observation, which is elementary:

\begin{lemma}\label{transitionmaplinebundle}
	Let $D$ a divisor on a smooth projective curve $C$. Then there exists a divisor $E$ such that the following hold:
	\begin{enumerate}
		\item $\mathrm{Supp}(E)\cap \mathrm{Supp}(D)= \emptyset$,
		\item The line bundle $\O_C(D)$ is trivial over $U=C\setminus \mathrm{Supp}(E)$ and $V=C\setminus \mathrm{Supp}(D)$,
		\item The transition maps of $\O_C(D)$ can be chosen as follows:
		\[
		\begin{array}{cccc}
			\phi_{UV}\colon &V \times \AA^1 & \dashrightarrow &U\times \AA^1 \\
			&(x,z) & \longmapsto& (x,f(x)z),
		\end{array}
		\]
		where $f \in \kk(C)^*$ satisfies $\mathrm{div}(f)=D-E$.
	\end{enumerate}
\end{lemma}

\begin{proof}
	Let $A$ be a very ample divisor on $C$, such that $B=D+A$ is also very ample. Then there exist $A' \sim A$ and $B'\sim B$ such that $\mathrm{Supp}(A')\cap \mathrm{Supp}(D)= \emptyset$ and $\mathrm{Supp}(B') \cap \mathrm{Supp}(D) = \emptyset$. 
	Let $E= B'-A'$, which is linearly equivalent to $D$. Then $\mathrm{Supp}(E) \cap \mathrm{Supp}(D) = \emptyset$, and there exists $f\in \kk(C)^*$ such that $\mathrm{div}(f) = D-E$.
	
	Since $U\cap \mathrm{Supp}(E)=\emptyset$ and $V\cap \mathrm{Supp}(D)=\emptyset$, the line bundle $\O_C(D)$ is trivial over $U=C\setminus \mathrm{Supp}(E)$ and $V=C\setminus \mathrm{Supp}(D)$. The transition map of $\O_C(D)$ are given by
	\[
	\begin{array}{ccc}
		V\times \AA^1 & \dashrightarrow & U\times \AA^1 \\
		(x,z) & \longmapsto & \left(x,\frac{\phi_U}{\phi_V}(x)z\right),
	\end{array}
	\]
	where $\phi_U$, $\phi_V\in \kk(C)^*$ satisfy $\mathrm{div}(\phi_U) = D_{|U} =D$ and $\mathrm{div}(\phi_V) = D_{|V}=0$. In particular, we can choose $\phi_U=f$ and $\phi_V=1$.
\end{proof}

\begin{proposition}\label{equivarianttodecomposable}
	Let $C$ be a smooth projective curve of genus $g(C)\geq 1$, let $\tau\colon S \to C$ be a geometrically ruled surface, and let $\pi\colon X\to S$ be a $\PP^1$-bundle such that $\tau\pi$ is an $\FF_b$-bundle with $b>0$. If the morphism $(\tau \pi)_*\colon \Autz(X)\to \Autz(C)$ is trivial, then $\Autz(X)$ is not relatively maximal.
\end{proposition}

\begin{proof}
		First, notice that the trivial ruled surface $C\times \PP^1$ satisfies the assumptions of Lemma \ref{lem:equivarianttodecomposable} (in this case, $\Delta$ is trivial). If $S$ is not $C$-isomorphic to $C\times \PP^1$ and $\seg(S)\leq 0 $, then $\sigma$ is an $\Autz(S)$-invariant section. Let $S_b$ be the $\Autz(X)$-invariant surface spanned by the $(-b)$-sections along the fibers of $\tau\pi$. Then $\pi^{-1}(\sigma) \cap S_b$ is an $\Autz(X)$-invariant curve, whose blowup followed by the contraction of the strict transform of $\pi^{-1}(\sigma)$ yields an $\Autz(X)$-equivariant birational map $X\dashrightarrow X'$, where $\tau\pi'\colon X' \to C$ is an $\FF_{b+1}$-bundle. If $(\tau\pi')_*\colon \Autz(X')\to \Autz(C)$ is not trivial, then $\Autz(X)$ is not relatively maximal. Else, we can replace $X$ by $X'$ and assume that $b$ is an odd integer. In all cases, we can assume that $\pi\colon X\to S$ satisfies the hypothesis of Lemma \ref{lem:equivarianttodecomposable}.
		
		Applying Lemma \ref{lem:equivarianttodecomposable}, we can also replace $X$ by $\PP(  \mathcal{O}_S\oplus \mathcal{O}_S(b\sigma+\tau^*(D)))$, where $\deg(D)>0$ may be chosen arbitrarily large in the following proof. Let $p\in C$. The curve $l_p=\{y_0=0\}\cap (\tau\pi)^{-1}(p)\subset X$ is $\Autz(X) $-invariant, because $(\tau\pi)_*$ is trivial and $b>0$. Then the blowup of $l_p$ followed by the contraction of the strict transform of $(\tau\pi)^{-1}(p)$ yields an $\Autz(X)$-equivariant square birational map $\phi\colon X\dashrightarrow X'$ above $S$.
		
		Let $j$ such that $p\in U_j$. We can assume that $\phi$ is given locally by 
		\[
		\begin{array}{ccc}
			U_j\times \FF_b&\dashrightarrow &U_j\times \FF_b\\
			(x,[y_0:y_1\ ;z_0:z_1])&\longmapsto & (x,[y_0:\delta(x) y_1\ ;z_0:z_1]),
		\end{array}
		\]
		where $\delta\in \kk(C)^*$ satisfies $\mathrm{div}(\delta)_{|U_j}=p$. We can assume that $p\notin U_i$ if $i\neq j$, and that $\phi$ induces the identity on $U_i\times \FF_b$. Writing the transition maps $\theta_{ij}$ of the $\FF_b$-bundle $\tau\pi \colon X\to C$ as in Lemma \ref{lem:equivarianttodecomposable}, we get that the transition maps of the $\FF_b$-bundle $\tau\pi'\colon X'\to C$ equal
		\[
		\begin{array}{cccc}
			\theta'_{ij} \colon &U_j\times \FF_b&\dashrightarrow &U_i\times \FF_b\\
			& (x,[y_0:y_1\ ;z_0:z_1])&\longmapsto & \left(x,\left[y_0:\delta^{-1}(x) \widetilde{\lambda}_{ij}(x)y_1\ ;s_{ij}(x)\cdot \begin{pmatrix}z_0\\z_1\end{pmatrix}\right] \right).
		\end{array}
		\]
		By Lemma \ref{FFb-bundle}, $X'$ is $S$-isomorphic to $\PP(\O_S\oplus \O_S(b\sigma+{\tau}^*(D+p)))$. To conclude, it suffices to show that the curve $l_p' = \{y_1=0\}\cap (\tau\pi')^{-1}(p) \subset X'$ is not $\Autz(X')$-invariant.
		
		Choosing $\deg(D)$ large enough, we can assume that $D'=D+p$ is very ample. In particular, there exists a section $\gamma\in \Gamma(C,\O_C(D'))$ such that $\gamma(p)\neq 0$. For each $i$ and $\xi\in \kk$, we define the automorphism 
		\[
		\begin{array}{cccc}
			\alpha_{\xi,i}\colon &U_i\times \FF_b & \longrightarrow &U_i\times \FF_b \\
			&(x,[y_0:y_1;z_0:z_1])& \longmapsto &(x,[y_0:y_1+\xi\gamma_{|U_i}(x)z_0^by_0\ ;z_0:z_1]).
		\end{array}
		\]
		We get that $\alpha_{\xi,i} \theta'_{ij}=\theta'_{ij}\alpha_{\xi,j} $, i.e., the $(\alpha_{\xi,i})_i$ glue into an element $\alpha_\xi\in \Autz(X')$ which does not leave $l'_p$ invariant.
\end{proof}

Combining the previous lemma with the classification of maximal connected algebraic subgroups of $\Bir(C\times \PP^1)$, we obtain the following proposition:

\begin{proposition}\label{basesurfacemaximal}
	Let $C$ be a curve of genus $g\geq 1$, let $\tau\colon S\to C$ and $\pi \colon X \to S$ be $\PP^1$-bundles such that $\tau\pi$ is an $\FF_b$-bundle with $b> 0$. If $\Autz(X)$ is relatively maximal, then $\Autz(S)$ is a maximal connected algebraic subgroup of $\Bir(S)$.

\end{proposition}

\begin{proof}
	We assume that $\Autz(S)$ is not a maximal connected algebraic subgroup of $\Bir(S)$ and prove the contraposition. The algebraic group $\Autz(C)$ is trivial if $g\geq 2$, or isomorphic to $C$ if $g=1$, so the morphism $\tau_*\colon \Autz(S)\to \Autz(C)$ is trivial or surjective. If $g=1$ and $\tau_*$ is surjective, then the $\Autz(S)$-orbits have dimension at least $1$, so $\Autz(S)$ is a maximal connected algebraic subgroup of $\Bir(S)$, which is a contradiction. Therefore, $\tau_*$ is trivial, so $(\tau \pi)_*$ as well. Then by Proposition \ref{equivarianttodecomposable}, the automorphism group $\Autz(X)$ is not relatively maximal.
\end{proof}

\subsection{Lifting automorphisms of $S$}

~\bigskip

Let $\pi\colon X\to S$ be a $\PP^1$-bundle. Then $f\in \Autz(S)$ induces a cartesian square:
\[
\begin{tikzcd}
	f^*X \arrow[r,"\widetilde{f}"] \arrow[d,"f^*\pi" swap] & X \arrow[d,"\pi"] \\ S \arrow[r,"f"] & S,
\end{tikzcd}
\]
where $\widetilde{f}$ is an isomorphism and $f^*\pi\colon f^*X \to S$ is the pullback $\PP^1$-bundle.
We start with the following observation:

\begin{lemma}\label{pullbackiso}
	Let $\pi \colon X\to S$ be a morphism of projective varieties such that $\pi_*(\O_X)=\O_S$. If $f\in \Autz(S)$ lies in the image of $\pi_*\colon \Autz(X) \to \Autz(S)$, then $f^*X$ and $X$ are isomorphic over $S$.
\end{lemma}

\begin{proof}
	By assumption, there exists $\widetilde{f}\in \Autz(X)$ such that $\pi \widetilde{f} =f \pi$. By the universal property of pullback, there exists a morphism $X\to f^*X=S\times_S X$ over $S$, given by $x\mapsto (\pi(x),\widetilde{f}(x))$. The inverse of this morphism is $(y,x) \mapsto {\widetilde{f}}^{-1}(x)$. Thus, $X$ and  $f^*X$ are isomorphic over $S$.
\end{proof}

Assume that the $\PP^1$-bundle $\pi\colon X\to S$ has invariants $(S,b,D)$, introduced in Definition $\ref{definvariants}$. If $f^*\pi\colon f^*X \to S$ has invariants different than $(S,b,D)$, then $f\notin \pi_*(\Autz(X))$ by Lemma $\ref{pullbackiso}$. When $C$ is an elliptic curve, we obtain:

\begin{proposition}\label{degnonnuldec}
	Let $C$ be an elliptic curve, let $D\in \Pic(C)$ such that $\deg(D)\neq0$, and let $\tau \colon S\to C$ be a geometrically ruled surface admitting a section $\sigma$ whose linear class is $\Autz(S)$-invariant and such that $\O_S(\sigma)\simeq \O_S(1)$. Then for every $\PP^1$-bundle $\pi\colon X\to S$ with invariants $(S,b,D)$, where $b>0$, the morphism $(\tau\pi)_*\colon \Autz(X)\to \Autz(C)$ is trivial and the automorphism group $\Autz(X)$ is not relatively maximal.
\end{proposition}

\begin{proof}
	If $\tau_*$ is trivial, then the morphism $(\tau\pi)_* \colon \Autz(X)\to \Autz(C)$ is also trivial. Now, assume that $\tau_*$ is not trivial. Since the image of $\tau_*$ is a connected algebraic subgroup of $\Autz(C)$, it follows that $\tau_*$ is surjective. By Lemma \ref{FFb-bundle} \ref{canonicalextension}, there exists a rank-$2$ vector bundle $\E$ over $S$ such that $\PP(\E)\simeq X$, and which fits into the short exact sequence 
	$$ 0 \to \mathcal{O}_S(b\sigma+\tau^*(D)) \to \mathcal{E}\to \mathcal{O}_S \to 0.$$
	Since $\deg(D)\neq 0$, there exists $p\in C$ such that $D$ is linearly equivalent to $\deg(D)p$. Then there exists $g\in \Autz(C)$ such that $D$ is not linearly equivalent to $g^*D$. The morphism $\tau_*$ being surjective, there exists $f\in \Autz(S)$ such that $\tau_*(f)=g$. Moreover, the pullback functor is exact, so $f^*\E$ fits into the short exact sequence:
	$$ 0 \to \mathcal{O}_S(bf^*\sigma+f^*(\tau^*(D))) \to f^*\mathcal{E}\to \mathcal{O}_S \to 0.$$
	By assumption, the linear class of $\sigma$ is $\Autz(S)$-invariant, so we get $\O_S(bf^*\sigma) \simeq \O_S(b\sigma)$. Since $\tau_*(f)=g$, it follows that $f^*(\tau^*(D)) = \tau ^* (g^*(D))$. Hence, $f^*\pi \colon f^*X \to S$ has invariants $(S,b,g^*(D))$, but we have chosen $g$ such that $D$ and $g^*D$ are not linearly equivalent. So by Lemmas \ref{Sisomorph} and \ref{pullbackiso}, $X$ and $f^*X$ are not $S$-isomorphic, and $f$ does not lift to an automorphism of $X$. Thus, the image of $(\tau \pi)_*$ does not contain $g$ and $(\tau \pi)_*$ is trivial. We conclude that $\Autz(X)$ is not relatively maximal by Proposition \ref{equivarianttodecomposable}.
\end{proof}

In \cite{Brosius,BrosiusII}, Brosius studies the moduli space of rank-$2$ vector bundles over a geometrically ruled surface. He shows that such a rank-$2$ vector bundle $\mathcal{E}$ sits into a short exact sequence, called the \emph{canonical extension} of $\mathcal{E}$. We recall his result below:

\begin{definition}\cite[\S 1]{Brosius}
	Let $\tau\colon S\to C$ be a geometrically ruled surface and let $\E$ be a rank-$2$ vector bundle over $S$. We say that the pair $(n,m)$ is the $\emph{fiber type}$ of $\E$ if the restriction of $\E$ on a general fiber $f$ of $\tau$ is isomorphic to $\O_f(n)\oplus \O_f(m)$.
\end{definition}

\begin{theorem}\cite[\S 2. Proposition 7]{Brosius}\label{Brosius}
	Let $\tau\colon S\to C$ be a geometrically ruled surface over a smooth projective curve $C$, and let $\E$ be an indecomposable rank-$2$ vector bundle over $S$ of fiber type $(n,m)$, where $n>m$ are integers. Then $\E$ fits into the following short exact sequence, called the \emph{canonical extension} of $\E$:
	$$ 0\to \E' \to \E \to \mathscr{I}_Z\otimes \mathscr{M} \to 0,$$
	where $\E' = \tau^*(\tau_*(\E\otimes \O_S(-n)))\otimes \O_S(n)$ is a line bundle of fiber degree $n$, $\mathscr{M}$ is a line bundle of fiber degree $m$, and $\mathscr{I}_Z$ is the ideal sheaf on a l.c.i.\ zero cycle $Z$. Conversely, the data of $\E'$, $Z$, $\mathscr{M}$, together with an invertible orbit in $\PP(\mathrm{Ext}^1_{\O_S}(\mathscr{I}_Z\otimes \mathscr{M} ,\E'))$, uniquely determine the isomorphism class of $\E$.
\end{theorem}

In the theorem above, by \emph{"an invertible orbit"}, we mean an element of $\PP(\mathrm{Ext}^1_{\O_S}(\mathscr{I}_Z\otimes \mathscr{M} ,\E'))$ that arises from a rank-$2$ vector bundle; see \cite[\S 2, Proposition 6]{Brosius}. Our setting is simplier, as we work with $\PP^1$-bundles without jumping fibers. In fact, the short exact sequence obtained in Lemma $\ref{FFb-bundle}$ is a particular case of Theorem $\ref{Brosius}$:

\begin{lemma}\label{Brosiusrem}
		Let $C$ be a smooth projective curve, let $\tau\colon S\to C$ and $\pi\colon X\to S$ be $\PP^1$-bundles such that $\tau\pi$ is an $\FF_b$-bundle with $b>0$.
		Let $\E$ be the rank-$2$ vector bundle over $S$ such that $\PP(\E) \simeq X$ and which fits into the short exact sequence
		\begin{equation}\label{canonicalextensionseq}
			0 \to \mathcal{O}_S(b\sigma+\tau^*(D)) \to \mathcal{E}\to  \mathcal{O}_S \to 0 \tag{$\dagger$}
		\end{equation}
	    obtained in Lemma $\ref{FFb-bundle}$ $\ref{canonicalextension}$. If $\E$ is indecomposable, then this short exact sequence is the canonical extension of $\E$ in the sense of Theorem $\ref{Brosius}$, with $\E' \simeq \O_S(b\sigma+ \tau^*(D))$, $\mathscr{M}\simeq \O_S$ and $Z=\emptyset$.
\end{lemma}

\begin{proof}
	Restricting (\ref{canonicalextensionseq}) to a fiber $f$ of $\tau$, we get $\E_f\simeq \O_{f}\oplus \O_{f}(b)$, i.e., $\E$ is of fiber type $(b,0)$. Tensoring (\ref{canonicalextensionseq}) by $\O_S(-b\sigma)$, we get the short exact sequence:
	$$
	0 \to \mathcal{O}_S(\tau^*(D)) \to \mathcal{E}\otimes \O_S(-b\sigma)\to  \mathcal{O}_S(-b\sigma) \to 0.
	$$
	Applying $\tau_*$ and and using that $\tau_*(\O_S(-b\sigma)) = 0$ (see e.g., \cite[II. Proposition 7.11. (a)]{Hartshorne}), we get that $\tau_*(\mathcal{E}\otimes \O_S(-b\sigma)) \simeq \O_C(D)$. Applying now $\tau^*$ and then tensoring by $\O_S(b\sigma)$, we get that $$\E' =  \tau^*(\tau_*(\E\otimes \O_S(-b\sigma))) \otimes \O_S(b\sigma)  \simeq \O_S(b\sigma+\tau^*(D)).$$
	This also implies that $\mathscr{M}\simeq \O_S$ and $Z=\emptyset$.
\end{proof}

\begin{remark}
	In our setting, $\mathscr{M}\simeq \mathcal{O}_S$, $Z=\emptyset$, and Theorem \ref{Brosius} implies that the isomorphism class of $\E$ is uniquely determined by the data of the invariants $(S,b,D)$  (once we have fixed the section $\sigma$ such that $\mathcal{O}_S(\sigma) \simeq \mathcal{O}_S(1)$) and an invertible orbit of $\PP(\mathrm{Ext}^1_{\O_S}(\mathcal{O}_S,\mathcal{O}_S(b\sigma+\tau^*(D))))$.
\end{remark}

\begin{proposition}\label{imagepi*}
	Let $C$ be an elliptic curve, $\tau \colon S\to C$ and $\pi\colon X\to S$ be $\PP^1$-bundles, such that $\pi$ has invariants $(S,b,D)$ for some $b>0$ and $D\in \Pic^0(C)$. Then the following hold:
	\begin{enumerate}
		\item\label{imagepi*.1} If $S=C\times \PP^1$, then there exists a subgroup $G$ of $\PGL_2(\kk)$ such that $\pi_*(\Autz(X)) = \Autz(C)\times G$.
		\item\label{imagepi*.2} If $S$ is $C$-isomorphic to $\AAA_0$ or $\PP(\O_C\oplus \mathcal{L})$, for some $\mathcal{L}\in \Pic^0(C)\setminus \{0\}$ of infinite order, then $\pi_*$ is surjective.
	\end{enumerate}
\end{proposition}

\begin{proof}
	  Let $\sigma$ be a section of $\tau$ such that $\O_S(\sigma)\simeq \O_S(1)$ and $\E$ be a rank-$2$ vector bundle over $S$, such that $\PP(\E)\simeq X$, and whose canonical extension is
	 $$
	 0 \to \O_S(b\sigma + \tau^*(D)) \to \E \to \O_S \to 0.
	 $$
	 Notice that for every $f\in \Autz(S)$, we have $f^*(\tau^*(D))=\tau^*(\tau_*(f)^*(D))$, and this divisor is linearly equivalent to $\tau^*(D)$ because the linear class of $D$ is $\Autz(C)$-invariant. Therefore, the pullback bundle $f^*X$ has also invariants $(S,b,D)$. 
	 	 
 	(1): If $S=C\times \PP^1$, then $\Autz(S) = \Autz(C) \times \PGL_2(\kk)$. The section $\sigma$ is a constant section (see e.g., \cite[V. Example 2.11.1]{Hartshorne}) and its linear class is $\Autz(S)$-invariant. The group $\Autz(C)\times \{\mathrm{id}_{\PP^1}\}$ acts on $\PP(\mathrm{Ext}^1_{\O_S}(\O_S,\mathcal{O}_S(b\sigma+\tau^*(D))))$ via pullback, which is a trivial action because $\Autz(C)$ is projective. Therefore, by Theorem \ref{Brosius}, $f^*X$ is $(C\times \PP^1)$-isomorphic to $X$ for every $f\in \Autz(C)\times \{\mathrm{id}_{\PP^1}\}$.
 	So $\Autz(C)\times \{\mathrm{id}_{\PP^1}\}$ is contained in $\Pi(\Aut(X))^\circ$, which equals $\pi_*(\Autz(X))$, by Proposition \ref{Autalgebraicgroup}. This implies that $\pi_*(\Autz(X)) = \Autz(C)\times G$ for some subgroup $G\subset \PGL_2(\kk)$.
 	
 	(2): The section $\sigma$ is $\Autz(S)$-invariant (see e.g., \cite[Chapter V. Example 2.11.2]{Hartshorne} and Lemma \ref{geometryofruledsurface}). Here as well, the group $\Autz(S)$ acts on $\PP(\mathrm{Ext}^1_{\O_S}(\O_S,\mathcal{O}_S(b\sigma+\tau^*(D))))$ via pullback and this action is necessarily trivial, as the automorphism group $\Autz(S)$ is anti-affine, i.e., $\O(\Autz(S))\simeq \kk$ (see e.g., \cite[Example 1.2.3]{BSU} and Lemma \ref{geometryofruledsurface}). By Theorem \ref{Brosius}, the rank-$2$ vector bundles $f^*{\E}$ and $\E$ are $S$-isomorphic for every $f\in \Autz(S)$, and using once more Proposition \ref{Autalgebraicgroup}, we conclude that the morphism $\pi_*$ is surjective.
\end{proof}

\section{Automorphisms of $\PP^1$-bundles over $C\times \PP^1$}\label{SectionCxPP1}

In this section, $\tau\colon C\times \PP^1\to C$ denotes the projection onto the first factor. Then $\tau$ admits infinitely many minimal sections, which are the constant sections; they are all of self-intersection zero and pairwise linearly equivalent. We also fix $\sigma = C\times \{[0:1]\} \subset C\times \PP^1$ as the constant section such that
$\O_{C\times \PP^1}(\sigma)\simeq \O_{C\times \PP^1}(1)$ (see Remark $\ref{rem.tauto}$ and \cite[V. 2.11.1]{Hartshorne}). The main result of this section is the following:

\begin{proposition}\label{main_prop:CtimesP^1}
	Let $C$ be a smooth projective curve of genus $g\geq 1$ and let $\pi\colon X\to C\times \PP^1$ be a $\PP^1$-bundle such that $\tau\pi$ is an $\FF_b$-bundle, with $b\in \mathbb{Z}$.
	Then $\Autz(X)$ is relatively maximal if and only if one of the following cases occurs:
	\begin{enumerate}
		\item Case $g=1$ and $b=0$. Then there exists a geometrically ruled surface $\tau'\colon S\to C$ such that
		\[
		X \cong S\times \PP^1,
		\]
		and $S$ is isomorphic to one of the following geometrically ruled surfaces: $C\times \PP^1$, $\AAA_0$, $\AAA_1$, or $\PP(\O_C\oplus \O_C(D))$ where $D\in \Pic^0(C)$ is non-trivial. Under this isomorphism, $\pi$ is identified with the morphism $\tau' \times \mathrm{id}_{\PP^1}$.
		\item Case $g=1$ and $b\neq0$. Then 
		\[
		X \cong \PP(\O_{C\times \PP^1}\oplus \O_{C\times \PP^1}(b\sigma+\tau^*(D))),
		\]
		where $\sigma$ is a constant section of $C\times \PP^1$ and $D\in \Pic^0(C)$. In this case, there is a short exact sequence
		\[
		1\to \mathbb{G}_m \to \Autz(X) \to \Autz(C) \to 1.
		\]
		\item Case $g\geq2$. Then $X\cong C\times \PP^1 \times \PP^1$ is the trivial $\PP^1$-bundle over $C\times \PP^1$ and $\Autz(X) \cong \PGL_2(\kk)^2$.
	\end{enumerate}
	Moreover, in all cases above, the pair $(X,\pi)$ is superstiff.
\end{proposition}

\subsection{$\FF_0$-bundles with relatively maximal automorphism groups}

~\bigskip

If $C$ is a smooth projective curve of genus $g\geq 2$, we have all the ingredients to classify the relatively maximal automorphism groups of $\PP^1$-bundles over $C\times \PP^1$:

\begin{proposition}\label{prop:ggeq2}
	Let $C$ be a smooth projective curve of genus $g\geq 2$ and $\pi \colon X\to C\times \PP^1$ be a $\PP^1$-bundle. Then $\Autz(X)$ is relatively maximal if and only if $X=C\times \PP^1 \times \PP^1$ and $\pi$ is the trivial $\PP^1$-bundle over $C\times \PP^1$. In this case, the pair $(X,\pi)$ is superstiff.
\end{proposition}

\begin{proof}
	Assume that $\Autz(X)$ is relatively maximal. By Propositions \ref{removal jumping} and \ref{equivarianttodecomposable}, $\tau\pi$ is an $\FF_0$-bundle. Using Proposition \ref{candidateFF0max}, $X$ is $(C\times \PP^1)$-isomorphic to $C\times \PP^1\times \PP^1$. Conversely, $\Autz(X) \simeq (\PGL_2(\kk))^2$ acts on $X$ with orbits of dimension $2$, so $\Autz(X)$ is relatively maximal and the pair $(X,\pi)$ is superstiff by Lemma \ref{LinkfromFFb}.
\end{proof}

From now on, let $C$ be an elliptic curve. Unlike the case where $C$ is of general type, $\Autz(C)$ acts on $C$ by translations, and the $\FF_b$-bundle $\tau\pi\colon \Autz(X)\to \Autz(C)$ may have no $\Autz(X)$-invariant fibers. Hence, we cannot apply directly Proposition $\ref{equivarianttodecomposable}$ as in the proof above, when $b>0$. If $\tau\pi$ is an $\FF_0$-bundle, $X$ is a fiber product and we obtain:

\begin{proposition}\label{FF0viaCxPP1}
	Let $C$ be an elliptic curve and $X=S\times \PP^1$, where $\tau'\colon S \to C$ is a geometrically ruled surface. Denote by $p_1$ the projection onto $S$ and $p_2=\tau' \times \mathrm{id}_{\PP^1}$. Then $\Autz(X)$ is relatively maximal with respect to $p_1$ and $p_2$, if and only if $X$ is one of the following:
	\begin{enumerate}
		\item $C\times \PP^1 \times \PP^1$,
		\item $\PP(\O_C\oplus \L)\times \PP^1$ for some non-trivial $\L\in \Pic^0(C)$,
		\item $\AAA_0\times \PP^1$,
		\item $\AAA_1\times \PP^1$;
	\end{enumerate}
	and the pairs $(X,p_1)$ and $(X,p_2)$ are superstiff. Moreover, in each case above, $\Autz(X)$ is a maximal connected algebraic subgroup of $\Bir(X)$.
\end{proposition}

\begin{proof}
	The direct implication follows from Proposition \ref{candidateFF0max}. Conversely, let $X$
	be $(C\times \PP^1)$-isomorphic to one of the threefolds in the given list. Then every $\Autz(X)$-orbit has dimension at least two by Lemmas \ref{geometryofruledsurface} and \ref{autofiberproduct}. Hence, $\Autz(X)$ is relatively maximal by Lemma \ref{SarkisovIIIandIV}. The only $\Autz(X)$-Sarkisov diagram is of type IV, that exchanges the two fibrations of the fiber product, and the induced Sarkisov link is an isomorphism. Therefore, $\Autz(X)$ is a maximal connected algebraic subgroup of $\Bir(X)$. Moreover, there is no $\Autz(X)$-equivariant square birational map starting from any of the two $\PP^1$-bundle structures $p_1$ and $p_2$. Thus $(X,p_1)$ and $(X,p_2)$ are superstiff.
\end{proof}

 To treat the cases where $\tau\pi$ has not $\Autz(X)$-invariant fibers, we explicitly compute the moduli space of $\PP^1$-bundles over $C\times \PP^1$ with some fixed invariants $(C\times \PP^1,b,D)$, where $b>0$ and $\deg(D) = 0$ (see Proposition $\ref{degnonnuldec}$).

\subsection{Moduli space of $\PP^1$-bundles over $C\times \PP^1$}

\begin{notation}
	Let $f\in \kk(C)$. Then $\mathrm{div}(f)= \mathrm{Z}(f)- \mathrm{P}(f)$, where $\mathrm{Z}(f)$ and $\mathrm{P}(f)$ denote the divisors of zeroes and poles of $f$, respectively.
\end{notation} 

The following two lemmas are elementary observations, but useful in the explicit computation of the moduli spaces in Lemma $\ref{representativedeg0}$.

\begin{lemma}\label{rationalfunction}
	Let $C$ be an elliptic curve, and $p,q\in C$. Then the following hold:
	\begin{enumerate}
		\item \label{rationalfunction.1} There exists $f\in \kk(C)^*$ such that $\mathrm{P}(f)=p+q$.
		\item \label{rationalfunction.2} For each $n\geq 2$, there exists $f\in \kk(C)^*$ such that $\mathrm{P}(f)=np$.
	\end{enumerate}
\end{lemma}

\begin{proof}
	By Riemann-Roch formula, $\h^0(C,p+q) =2$ and there is no rational function with only one pole, this gives \ref{rationalfunction.1}. For $n\geq 2$, it follows also by Riemann-Roch that $\H^0(C,np)\setminus \H^0(C,(n-1)p)\neq \emptyset$, and this gives \ref{rationalfunction.2}.
\end{proof}

\begin{lemma}\label{transitiondeg0}
	Let $C$ be an elliptic curve with a neutral point $p_0\in C$, and let $p\in C$ such that $p\neq p_0$. For every $t\in \Autz(C)$ such that $p,p_0,t(p),t(p_0)$ are pairwise distinct, the line bundle $\O_C(p-p_0)$ is trivial over $U = C\setminus \{p,p_0\}$ and $V = C\setminus \{t(p),t(p_0)\}$, and we can choose the transition map as
	\[
	\begin{array}{cccc}
		\phi \colon &V\times \AA^1 & \dashrightarrow &U\times \AA^1 \\
		&(x,z)& \longmapsto &(x,f(x)z),
	\end{array}
	\]
	where $f\in \kk(C)^*$ satisfies $\mathrm{div}(f)=(t(p)-t(p_0))-(p-p_0)$.
\end{lemma}

\begin{proof}
	Using the isomorphism $C\to \Pic^0(C)$, $p \mapsto p-p_0$, we have that $D=t(p)-t(p_0)$ is linearly equivalent to $E=p-p_0$, so there exists $f\in \kk(C)^*$ with $\mathrm{div}(f) = D-E$. This implies that the line bundle $\O_C(p-p_0)$ is trivial over the opens $U$ and $V$ defined in the statement. The transition map $\phi$ is given by the multiplication of an element $\mu\in \O_C(U\cap V)$ on the second factor and by construction, $\mu$ can be chosen as $\phi_U/\phi_V$, where $\phi_U = 1$ and $\phi_V = 1/f$.
\end{proof}

Following the ideas of \cite[\S 3.3]{BFT}, we compute explicit trivializations and transition maps for the $\PP^1$-bundles with invariants $(C\times \PP^1,b,D)$; where $C$ is an elliptic curve, $b>0$ and $D$ has degree zero. 

\begin{lemma}\label{representativedeg0}
	Let $C$ be an elliptic curve with a neutral element $p_0\in C$. For any $\PP^1$-bundle $\pi\colon X\to C\times \PP^1$ with invariants $(C\times\PP^1,b,D)$, where $b>0$ and $\deg(D)=0$, the following hold:
	\begin{enumerate}
		\item \label{representativedeg0.1} If $D$ is trivial, then for any $p\neq p_0$, the $\FF_b$-bundle $\tau\pi$ is trivial over $U=C\setminus \{p\} $ and $V=C\setminus \{p_0\}$, and we can choose the transition map as follows:
		\[
		\begin{array}{cccc}
			\phi \colon &V\times \FF_b & \dashrightarrow &U\times \FF_b \\
			&(x,[y_0:y_1\ ;z_0:z_1]) & \longmapsto& (x,[y_0:y_1 + \xi(x) g(z_0,z_1) y_0\ ;z_0:z_1]),
		\end{array}
		\]
		where $\xi\in \kk(C)$ satisfies $\mathrm{P}(\xi) = p+p_0$ and $g\in \kk[z_0,z_1]_b$. Moreover, the $(C\times\PP^1)$-isomorphism class of $X$ is uniquely determined by $g$, up to multiplication by a non-zero scalar.
		\item \label{representativedeg0.2} If $D$ is not trivial, then $X$ is $(C\times \PP^1)$-isomorphic to the decomposable $\PP^1$-bundle 
		\[
		\PP(\O_{C\times \PP^1}\oplus \O_{C\times \PP^1}(b\sigma+\tau^*(D))).
		\]
		Let $p\in C$ such that $D\sim p-p_0$ and $t\in \Autz(C)$ such that $p,p_0,t(p),t(p_0)$ are pairwise distinct. Then the $\FF_b$-bundle $\tau\pi$ is trivial over $U=C\setminus \{p,p_0\}$ and $V=C\setminus \{t(p),t(p_0)\}$, and we can choose the transition map as follows:
		\[
		\begin{array}{cccc}
			\phi\colon & V\times \FF_b & \dashrightarrow& U\times \FF_b \\
			&(x,[y_0:y_1\ ;z_0:z_1]) & \longmapsto &(x,[y_0:f(x) y_1 \ ;z_0:z_1]),
		\end{array}
		\]
		where $f\in \kk(C)^*$ satisfies $\mathrm{div}(f)=(t(p)-t(p_0))-(p-p_0)$.
	\end{enumerate}
\end{lemma}

\begin{proof}
	 Let $(U_i)_i$ be an open cover of $C$ such that each $U_i$ trivializes the $\FF_b$-bundle $\tau \pi\colon X\to C$. By shrinking $U_i$ if necessary, we can assume that $p_0\in U_0$ and for every $i\neq 0$, $U_i \setminus U_0 = \{p_i\}$ for some $p_i\in C$, and $p_0\notin U_i$. We choose the trivializations of $\tau\pi$ such that the induced transition matrices of $\tau$ are identity matrices:
	\[
	\begin{array}{ccc}
		U_0 \times \FF_b & \dashrightarrow &U_i\times \FF_b \\
		\left( x,[y_0:y_1\ ;z_0:z_1] \right) & \longmapsto &\left(x,\left[y_0:\lambda_{i0}(x)y_1 + p_{i0}(x,z_0,z_1)y_0\ ; z_0:z_1\right]\right),
	\end{array}
	\]
	where $\lambda_{i0}\in \O_C(U_{i0})^*$ and $p_{i0}\in \O_C(U_{i0})[z_0,z_1]_b$ by Lemma \ref{FFb-bundle} \ref{FFb-bundle1}. Then by Lemma \ref{Sisomorph} \ref{conditioniso}, we can replace $p_{i0}$ by $p_{i0}'$ without changing the $(C\times \PP^1)$-isomorphism class of $X$, if there exist $\mu_0\in \O_C(U_0)^*$, $q_0\in \O_C(U_0)[z_0,z_1]_b$, $\mu_i\in \O_C(U_i)^*$ and $q_i\in \O_C(U_i)[z_0,z_1]_b$ such that	
	\begin{equation}
		p'_{i0}(x,z_0,z_1) =  \mu_i (x)p_{i0}(x,z_0,z_1)- \mu_i(x) \mu_0^{-1}(x)\lambda_{i0}(x) q_0(x,z_0,z_1)+q_i\left(x, z_0 ,z_1 \right). \label{eqdivtrivial}\tag{$\dagger$}
	\end{equation} 
	In the rest of the proof, we distinguish between the cases where $D$ is trivial and $D$ is non-trivial.

	(1): Fix $p\in C$ such that $p\neq p_0$. We can assume that $p\notin U_0$. By assumption, $D$ is trivial so by Lemma \ref{Sisomorph}, we can choose $\lambda_{i0}=1$ for every $i$. Choosing $\mu_i = \mu_0=1$ in the equality (\ref{eqdivtrivial}), we get:
		\begin{equation}
			p'_{i0}(x,z_0,z_1) =  p_{i0}(x,z_0,z_1)- q_0(x,z_0,z_1)+q_i\left(x, z_0 ,z_1 \right).  \label{eqdivtrivial.1}\tag{$\dagger\dagger$}
		\end{equation}
		Write $p_{i0}(x,z_0,z_1) = \sum_{k=0}^b a_{i0,k}(x) z_0^k z_1^{b-k}$, with $a_{i0,k} \in \O_C(U_{i0}) $ for every $k$. 
		
		Assume that $a_{i0,k}$ admits a pole at $p_i\notin \{p,p_0\}$ of order $n\geq 1$. By Lemma \ref{rationalfunction} \ref{rationalfunction.1}, there exists $f_1 \in \O_C(U_0)$ such that $\mathrm{P}(f_1) =  p + p_i$. Using the equality (\ref{eqdivtrivial.1}) and replacing $p_{i0}$ by $p_{i0}-\alpha f_1^n z_0^kz_1^{b-k}$, for a suitable $\alpha\in \kk^*$, we can decrease the multiplicity of $a_{i0,k}$ at the pole $p_i$. Repeating this finitely many times, we can assume that the rational function $a_{i0,k}$ is regular at $p_i$. Doing this for every $k$, we obtain new transition maps for $\tau\pi$, which are regular on $C\setminus \{p,p_0\}$. This implies that $\tau\pi$ is trivial over $V=C\setminus \{p\}$ and $U=C\setminus \{p_0\}$, and we can write the transition map as
		\[
		\begin{array}{cccc}
			\phi\colon &V\times \FF_b & \dashrightarrow &U\times \FF_b \\ 
			&(x,[y_0:y_1\ ;z_0:z_1]) & \longmapsto &(x,[y_0:y_1 + \widetilde{g}(x,z_0,z_1) y_0\ ;z_0:z_1]),
		\end{array}
		\]
		where $ \widetilde{g}(x,z_0,z_1)=\sum_{k=0}^b a_{k}(x) z_0^k z_1^{b-k}$ and $a_{k}\in \O_C(U\cap V)$ for every $k$.
		
		By Lemma \ref{rationalfunction}, there exists $\xi\in \kk(C)^*$ such that $\mathrm{P}(\xi)=p+p_0$. Rewriting (\ref{eqdivtrivial.1}) with the opens $U$ and $V$, we get
		\begin{equation*}
			g'(x,z_0,z_1) =  \widetilde{g}(x,z_0,z_1)- q_v(x,z_0,z_1)+q_u\left(x, z_0 ,z_1 \right).
		\end{equation*}
		Using the equality above and Lemma \ref{rationalfunction}, with $q_u$ and $q_v$ cancelling the poles of $a_{k}$ of order $n\geq 2$, we can now assume that $a_{k}$ has poles of order at most $1$ at $p$ and at $p_0$. Since $h^0(C,p+p_0)=2$, this implies that for each $k$, there exists $\alpha_k,\beta_k\in \kk$ such that $a_{k} = \alpha_k \xi +\beta_k $. Substracting once more by $\sum_{k=0}^{b} \beta_kz_0^kz_1^{b-k}$, we can now assume that $ \widetilde{g}(x,z_0,z_1)=\xi(x) g(z_0,z_1)$, where $g(z_0,z_1)=\sum_{k=0}^b \alpha_k z_0^k z_1^{b-k}\in \kk[z_0,z_1]_b$. This shows that we can choose the trivializations of $\tau\pi$ and the transition map as stated.
		
		Next, we show that the $(C\times\PP^1)$-isomorphism class of $X$ is uniquely determined by $g\in \kk[z_0,z_1]_b$, up to multiplication by a non-zero scalar. Assume that we can choose two distinct trivializations of $\tau\pi$, which induce transition maps $\phi$ and $\phi'$. By Lemma \ref{Sisomorph}, there exist automorphisms 
		\[
		\begin{array}{cccc}
			\alpha_u \colon &U\times \FF_b & \longrightarrow &U\times \FF_{b} \\
			& (x,[y_0:y_1\ ;z_0:z_1]) & \longmapsto &(x,\left[y_0:\mu_{u}(x)y_1 + h_{u}(x,z_0,z_1)y_0\ ;  z_0:z_1\right]), \\[10pt]
			\alpha_v \colon & V\times \FF_b & \longrightarrow& V\times \FF_{b} \\
			& (x,[y_0:y_1\ ;z_0:z_1]) & \longmapsto& (x,\left[y_0:\mu_{v}(x)y_1 + h_{v}(x,z_0,z_1)y_0\ ;  z_0:z_1\right]), 
		\end{array}
		\]
		with $\mu_u\in \O_C(U)^*$, $\mu_v\in \O_C(V)^*$, $h_{u}\in \O_C(U)[z_0,z_1]_b$, $h_{v}\in \O_C(V)[z_0,z_1]_b$; and such that $ \phi' =\alpha_u \phi\alpha_v^{-1}$. This implies that $\mu_u = \mu_v \in\kk^*$, and we denote by $\mu \in \kk^*$ the corresponding constant, and the equality
		$$\xi(x)(g'(z_0,z_1)-\mu g(z_0,z_1))  =h_u(x,z_0,z_1)- h_v(x,z_0,z_1).$$
		
		 If $g'(z_0,z_1)\neq \mu g(z_0,z_1)$, the LHS is a polynomial in $\kk(C)[z_0,z_1]_b$, having a non-zero coefficient with poles of order $1$ at $p$ and at $p_0$. Since $h_{u}\in \O_C(U)[z_0,z_1]_b$ and $h_{v}\in \O_C(V)[z_0,z_1]_b$, this is not possible, as there is no rational function on $C$ with only one pole. Thus $g'(z_0,z_1)=\mu g(z_0,z_1) $ and this proves \ref{representativedeg0.1}.
		
		(2): Now, we assume that $D$ is not trivial and has degree zero. So $D \sim p-p_0$ for some $p\in C\setminus \{p_0\}$. Let $t\in \Autz(C)$ such that $p,p_0,t(p),t(p_0)$ are pairwise distinct, and $f\in \kk(C)^*$ such that $\mathrm{div}(f)= (t(p)-t(p_0))-(p-p_0)$ as in Lemma \ref{transitiondeg0}.
		
		First, we prove that we can choose $U_0 = C\setminus \{p,t(p),t(p_0)\}$ as a trivializing open subset of $\tau\pi$. Without loss of generality, we can assume that $p,t(p),t(p_0)\notin U_0$. Recall also that from the beginning of the proof, we have $U_i \setminus U_0 = \{p_i\}$ for each $i$. Let $p_i \neq p,t(p),t(p_0)$, then we can choose $\lambda_{i0}=f$, as the coefficients $\{\lambda_{ij}\}_{i,j}$ are the cocycles of the line bundle $\O_C(p-p_0)$ by Lemma \ref{FFb-bundle}. Choosing moreover $\mu_i=\mu_0=1$, the equality (\ref{eqdivtrivial}) becomes
		\begin{equation}
			p'_{i0}(x,z_0,z_1) =  p_{i0}(x,z_0,z_1)-f(x) q_0(x,z_0,z_1)+q_i\left( x,z_0 ,z_1 \right).  \label{eqdivnontrivial}\tag{$\star$}
		\end{equation}
		Write $p_{i0}(x,z_0,z_1) = \sum_{k=0}^b a_{i0,k}(x) z_0^k z_1^{b-k}$, with $a_{i0,k} \in \O_C(U_{i0}) $ for every $k$, and assume that $a_{i0,k}$ has a pole of order $n\geq 1$ at $p_i$.
		
		 By Lemma \ref{rationalfunction} \ref{rationalfunction.1}, there exists $f_i\in \O_C(U_0)$ such that $\mathrm{P}(f_i)=p_i+t(p_0)$. Using (\ref{eqdivnontrivial}) and substracting $a_{i0,k}$ by $fq_0$, where $q_0=\alpha f_i^nz_0^k z_1^{b-k}$ for a suitable $\alpha\in\kk^*$, we replace $a_{i0,k}$ and decrease the multiplicity of the pole at $p_i$. After finitely many steps, we can assume that $a_{i0,k}$ is regular at $p_i$ for each $k$. Moreover, $f(p_i)\in \kk^*$, so the new transition maps are regular at $p_i$ and we can extend the trivialization over $U_0$ to $p_i$. This proves that we can choose $U_0 = C\setminus \{p,t(p),t(p_0)\}$.
		
		Next, we prove that we can also extend the trivialization over $U_0$ to $p$. Let $l$ be the index such that $U_l \setminus U_0 = \{p\}$, and recall that the opens $U_l$ and $U_0$ do not contain $t(p)$ and $t(p_0)$. Again by Lemma \ref{FFb-bundle}, the coefficients $\{\lambda_{ij}\}_{i,j}$ are the cocycles of the line bundle $\O_C(p-p_0)$; hence, we can choose $\lambda_{l0}=1$. Taking again $\mu_l = \mu_0=1$ in (\ref{eqdivtrivial}), we get now the equality (\ref{eqdivtrivial.1}). As before, write $p_{l0}(x,z_0,z_1) = \sum_{k=0}^b a_{l0,k}(x) z_0^k z_1^{b-k}$, with $a_{l0,k} \in \O_C(U_{l0}) $ for every $k$, and assume that $a_{l0,k}$ admits a pole at $p$ of order $n\geq 1$. By Lemma \ref{rationalfunction} \ref{rationalfunction.1}, there exists $f_p \in \O_C(U_0)$ such that $\mathrm{P}(f_p) =  p + t(p_0)$. Using the equality (\ref{eqdivtrivial.1}), we replace $p_{l0}$ by $p_{l0}- \alpha {f_p}^n z_0^kz_1^{b-k}$ for a suitable $\alpha\in \kk^*$; this decreases the multiplicity of $a_{l0,k}$ at the pole $p$. Repeating this process and after finitely many steps, we can assume that $a_{l0,k}$ is regular at $p$ for each $k$. This shows that $\tau\pi$ is trivial over $V=C\setminus \{t(p),t(p_0)\}$.
		
		By the same argument as above, $\tau\pi$ is also trivial over $U=C\setminus \{p,p_0\}$ and we can write the transition map as
		\[
		\begin{array}{cccc}
			\phi\colon & V\times \FF_b & \dashrightarrow& U\times \FF_b \\
			&(x,[y_0:y_1\ ;z_0:z_1]) & \longmapsto &(x,[y_0:f(x) y_1 + g(x,z_0,z_1) y_0\ ;z_0:z_1]),
		\end{array}
		\]
		where $ g(x,z_0,z_1)=\sum_{k=0}^b a_{k}(x) z_0^k z_1^{b-k} \in \O_C(U\cap V)[z_0,z_1]_b$. Rewriting the equality (\ref{eqdivnontrivial}), we get:
		\begin{equation}
			g'(x,z_0,z_1) =  g(x,z_0,z_1)-f(x) q_v(x,z_0,z_1)+q_u\left(x, z_0 ,z_1 \right).  \label{eqdivnontrivial2}\tag{$\star\star$}
		\end{equation}
			
		The rational functions $a_{k}\in \O_C(U\cap V)$ have poles at most at the points $p,p_0,t(p),t(p_0)$. Using (\ref{eqdivnontrivial2}), we will replace $a_{k}$ by zero. For each $k$, we proceed as follows:
			
		\begin{enumerate}
		\item[(i)] If $a_k$ has a pole of order $n\geq 1$ at $t(p_0)$, we replace $g$ by $g-fq_v$, where $q_v=\alpha f_1^{n-1} z_0^k z_1^{b-k}\in \O_C(V)[z_0,z_1]_b$, for a suitable $\alpha\in\kk^*$ and $f_1\in \O_C(V)$ satisfies $\mathrm{P}(f_1) = t(p) + t(p_0)$ (Lemma \ref{rationalfunction}). Since $f$ has a pole of order $1$ at $t(p_0)$, this decreases the multiplicity of the pole of $a_{k}$ at $t(p_0)$. Repeating this step finitely many times, we can assume that $a_k$ is regular at $t(p_0)$.
		
		\item[(ii)] If $a_k$ has a pole of order $m\geq 1$ at $t(p)$, we replace $g$ by $g-fq_v$, where $q_v= \beta f_{m+1} z_0^k z_1^{b-k}$, for a suitable $\beta\in \kk^*$ and $f_{m+1}\in \O_C(V)$ satisfies $\mathrm{P}(f_{m+1})=(m+1)t(p)$ (Lemma \ref{rationalfunction}).
		Since $f$ has a zero of order $1$ at $t(p)$, it follows that $ff_{m+1}$ has a pole of order $m\geq 1$ at $t(p)$. This decreases the multiplicity of $a_k$ at the pole $t(p)$. Repeating this step finitely many times, we can assume that $a_{k}$ is regular at $t(p)$ and $t(p_0)$.
	\end{enumerate}	
	We obtain a new polynomial $g\in \O_C(U)[z_0,z_1]_b$. Using a last time (\ref{eqdivnontrivial2}) with $q_u = -g$, we replace $g$ by the zero polynomial. By Lemma \ref{FFb-bundle} \ref{FFb-decomposable}, it follows that $X$ is $(C\times\PP^1)$-isomorphic to $\PP(\O_{C\times \PP^1}\oplus \O_{C\times \PP^1}(b\sigma+\tau^*(D)))$.
\end{proof}

We obtain the analogue of \cite[Corollary 3.3.8]{BFT}:

\begin{proposition}\label{modulioverCtimesPP^1}
	Assumptions and notations as in Lemma \ref{representativedeg0}. Then the following hold:
	\begin{enumerate}
		\item\label{modulioverCtimesPP^1.1} If $D$ is trivial, then the $(C\times \PP^1)$-isomorphism class of indecomposable $\PP^1$-bundles with invariants $(C\times \PP^1,b,D)$ are parametrised by the projective space 
		$$ \PP(\kk[z_0,z_1]_b).$$
		If $\pi$ is indecomposable, then the image of the morphism $\pi_*\colon \Autz(X)\to \Autz(C\times \PP^1)$ equals $\Autz(C)\times G$, where $G$ is a proper subgroup of $\PGL_2(\kk)$. Else, if $\pi$ is decomposable, then $X$ is the trivial $\FF_b$-bundle $C\times \FF_b$ and $\pi_*$ is surjective and $\ker(\pi_*) \cong \mathbb{G}_m \rtimes \Gamma(\O_{\PP^1}(b))$.
		\item\label{modulioverCtimesPP^1.2} If $D$ is not trivial, then $\pi_*\colon \Autz(X)\to \Autz(C\times \PP^1)$ is surjective and $\mathrm{ker}(\pi_*)=\mathbb{G}_m$. Moreover, the two disjoint sections of $\pi$ corresponding to the line subbundles $\O_{C\times \PP^1}$ and $\O_{C\times \PP^1}(b\sigma+\tau^*(D))$ are $\Autz(X)$-invariant.
	\end{enumerate}
\end{proposition}

\begin{proof}
			(1): By Lemma \ref{representativedeg0} \ref{representativedeg0.1}, the $(C\times \PP^1)$-isomorphism class of $X$ is uniquely determined by a polynomial $g\in \kk[z_0,z_1]_b$, up to multiplication by a non-zero scalar. Hence, the indecomposable $\PP^1$-bundles with invariants $(C\times \PP^1,b,0)$ are parametrised by the projective space 
			$ \PP(\kk[z_0,z_1]_b)$. 
			
			By Proposition \ref{imagepi*} \ref{imagepi*.1}, there exists a subgroup $G\subset \PGL_2(\kk)$ such that $\pi_*(\Autz(X))= \Autz(C)\times G$. Next, we show that $G\subsetneq \PGL_2(\kk)$ is a proper subgroup if and only if $\pi$ is indecomposable. The subgroup $\{\mathrm{id}_C\}\times \PGL_2(\kk)\subset \Autz(C\times \PP^1)$ acts on $\PP(\kk[z_0,z_1]_b)$ as follows: an automorphism 
			$$f = \mathrm{id}_C\times \begin{bmatrix} a&b\\c&d\end{bmatrix} \in \{\mathrm{id}_C\}\times \PGL_2(\kk)$$ sends the $\mathbb{G}_m$-orbit of $g(z_0,z_1)$ to the $\mathbb{G}_m$-orbit of $g\left(a z_0+ b z_1, cz_0+dz_1\right)$, which defines the transition maps of the $\FF_b$-bundle $\tau f^*\pi \colon f^*(X)\to C$. 
			The only fixed point for this action is the zero polynomial, so by Lemma \ref{pullbackiso} and Proposition \ref{Autalgebraicgroup}, $G=\PGL_2(\kk)$ if and only if $\pi$ is decomposable; in which case $X$ is $(C\times \PP^1)$-isomorphic to $C\times \FF_b$. In this case, $\pi_* = \mathrm{id}_C \times (\tau_b)_*$, where $\tau_b\colon \FF_b \to \PP^1$ is the structure morphism. It follows that $\ker(\pi_*) = \ker({\tau_b}_*) = \mathbb{G}_m \rtimes \Gamma(\O_{\PP^1}(b))$.
		
		(2): By Lemma \ref{representativedeg0} \ref{representativedeg0.2}, the unique $\PP^1$-bundle over $C\times \PP^1$ with invariants $(C\times \PP^1,b,D)$, where $D$ is non-trivial and has degree zero, is the decomposable $\PP^1$-bundle $\PP(\O_{C\times \PP^1}\oplus \O_{C\times \PP^1}(b\sigma+\tau^*(D)))$. Since $\deg(D)=0$, the linear class of $b\sigma+\tau^*(D)$ is $\Autz(C\times \PP^1)$-invariant. Again by Proposition \ref{Autalgebraicgroup} \ref{Autalgebraicgroup.3}, we conclude that $\pi_*$ is surjective. 
		
		By Lemma \ref{FFb-bundle}, we can write the transition maps of $\tau\pi$ as
		\[
		\begin{array}{cccc}
			\phi_{ij}\colon & U_j\times \FF_b & \dashrightarrow & U_i\times \FF_b \\
			& (x,[y_0:y_1\ ;z_0:z_1]) & \longmapsto & (x,[y_0:\lambda_{ij}(x)y_1 \ ; z_0:z_1]) 
		\end{array}
		\]
		where $\lambda_{ij}\in \O_C(U_{ij})^*$ are the cocyles of the line bundle $\O_{C}(D)$. An element $f\in \mathrm{ker}(\pi_*)$ induces automorphisms 
		\[
		\begin{array}{cccc}
			f_i\colon & U_i\times \FF_b & \longrightarrow & U_i\times \FF_b \\
			& (x,[y_0:y_1\ ;z_0:z_1]) & \longmapsto & (x,[y_0:\mu_i(x)y_1+q_i(x,z_0,z_1)y_1 \ ; z_0:z_1]) 
		\end{array}
		\]
		where $\mu_i\in \O_C(U_{ij})^*$ and $q_i\in \O_C(U_i)[z_0,z_1]_b$. For every $i,j$, we have $f_i \phi_{ij} = \phi_{ij} f_j$; hence, $\mu_i = \mu_j \in \kk^*$ is a non-zero constant and $q_i = \lambda_{ij} q_j$. Therefore, the coefficients of the polynomials $q_i$ and $q_j$ give global sections of the line bundle $\O_{C}(D)$. Since $D$ is non-trivial of degree zero, we have that $\mathrm{H}^0(C,\O_C(D))=0$; thus, $q_i=q_j=0$ and $\mathrm{ker}(\pi_*)=\mathbb{G}_m$. This also implies that the two sections $\{y_0=0\}$ and $\{y_1=0\}$ corresponding to the line subbundles $\O_{C\times \PP^1}$ and $\O_{C\times \PP^1}(b\sigma+\tau^*(D))$ are $\Autz(X)$-invariant.
\end{proof}

\subsection{$\FF_b$-bundles with relatively maximal automorphism groups}

~\bigskip

In this section, using the moduli space of $\PP^1$-bundles over $C\times \PP^1$, we show that those with relatively maximal automorphism groups are precisely the decomposable ones.

\begin{lemma}\label{CxP^1,Dtrivial}
	Let $C$ be an elliptic curve and $\pi\colon X\to C\times \PP^1$ be a $\PP^1$-bundle with invariants $(C\times \PP^1,b,0)$, where $b>0$. We denote by $S_{b}$ the surface spanned by the $(-b)$-sections along the fibers of $\tau\pi$. Then the following hold:
	\begin{enumerate}
		\item\label{CxP^1,Dtrivial.1} On every fiber $F$ of $\pi$, the kernel of the morphism $\pi_*\colon \Autz(X)\to \Autz(C\times \PP^1)$ acts transitively on $F \setminus (F\cap S_{b})$.
		\item If $\pi$ is indecomposable, then the automorphism group $\Autz(X)$ is not relatively maximal.
	\end{enumerate}
\end{lemma}

\begin{proof}
		(1): Let $(U_i)_i$ be a trivializing open cover of $\tau\pi$. By Lemma \ref{FFb-bundle}, we can write the transition maps of $\tau\pi$ as
		\[
		\begin{array}{cccc}
			\phi_{ij}\colon & U_j\times \FF_b & \dasharrow & U_i\times \FF_b \\
			&(x,[y_0:y_1\ ;z_0:z_1]) & \longmapsto& (x,[y_0:y_1+p_{ij}(x,z_0,z_1)y_0\ ;z_0:z_1]),
		\end{array}
		\]
		where $p_{ij}\in \O_C(U_{ij})[z_0,z_1]_b$. Let $q \in \kk[z_0,z_1]_b$. For each $i$, we define the automorphism 
		\[
		\begin{array}{cccc}
		\psi_{i}\colon &  U_i\times \FF_b & \dasharrow  &U_i\times \FF_b \\
		& (x,[y_0:y_1\ ;z_0:z_1]) & \longmapsto & (x,[y_0:y_1+ q(z_0,z_1) y_0\ ;z_0:z_1]).
		\end{array}		
		\]
		Then the equality $\psi_{i} \phi_{ij} = \phi_{ij} \psi_{j}$ holds, so $\psi= \{\psi_{i}\}_i$ defines a $(C\times \PP^1)$-automorphism of $X$. We obtain a subgroup of $\ker(\pi_*)$, that is isomorphic to $\kk[z_0,z_1]_b$ and acts transitively on $F \setminus S_{b}$.
	   
	   (2): By Proposition \ref{modulioverCtimesPP^1} \ref{modulioverCtimesPP^1.1}, $\pi_*(\Autz(X))=\Autz(C)\times G$, where $G\subsetneq \PGL_2(\kk)$ is a proper subgroup. Therefore, there exists a $\pi_*(\Autz(X))$-invariant constant section $\sigma_0$ in $C\times \PP^1$. Write $\sigma_0=\{az_0 - b z_1=0\}$, where $(a,b)\neq (0,0)$. The blowup of $\pi^{-1}(\sigma_0) \cap S_{b} = \{y_0=az_0 - b z_1=0\}$ followed by the contraction of the strict transform of the surface $\pi^{-1}(\sigma_0)=\{az_0 - b z_1=0\}$ yields an $\Autz(X)$-equivariant square birational map $\theta\colon X\dasharrow X'$ above $C\times \PP^1$, which is locally given by
	   \[
	   \begin{array}{cccc}
	   	\theta_i\colon & U_i \times \FF_b &\dashrightarrow & U_i\times \FF_{b+1}\\
	   & 	(x,[y_0:y_1\ ;z_0:z_1])& \longmapsto &	(x,[y_0:(az_0 - b z_1) y_1\ ;z_0:z_1]).
	   \end{array}
	   \]
	   We obtain a $\PP^1$-bundle $X'\to C\times \PP^1$ with invariants $(C\times \PP^1,b+1,0)$ and the base locus of $\theta^{-1}$ is not $\Autz(X')$-invariant by \ref{CxP^1,Dtrivial.1}. Thus, $\theta \Autz(X)\theta^{-1} \subsetneq \Autz(X')$ is a proper subgroup, so $\Autz(X)$ is not relatively maximal.
\end{proof}

\begin{proposition}\label{CxP^1,Dnontrivial}
	Let $C$ be an elliptic curve, $\pi\colon X\to C\times \PP^1$ be a $\PP^1$-bundle such that $\tau\pi$ is an $\FF_b$-bundle over $C$ with $b>0$. Then $\Autz(X)$ is relatively maximal if and only if $X$ is $(C\times \PP^1)$-isomorphic to the decomposable $\PP^1$-bundle
	$$\PP(\O_{C\times \PP^1}\oplus \O_{C\times \PP^1}(b\sigma+\tau^*(D))),$$
	where $D\in \Pic^0(C)$. In this case, $\Autz(X)$ is relatively maximal and the pair $(X,\pi)$ is superstiff. 
\end{proposition}

\begin{proof}
	Assume that $\Autz(X)$ is relatively maximal and let $(C\times \PP^1,b,D)$ be the invariants of $\pi$. By Proposition \ref{degnonnuldec}, $\deg(D) = 0$. If $D$ is trivial, then $X$ is $(C\times \PP^1)$-isomorphic to $\PP(\O_{C\times \PP^1}\oplus \O_{C\times \PP^1}(b\sigma))$ (i.e., isomorphic to the product $C\times \FF_b$) by Proposition \ref{modulioverCtimesPP^1} and Lemma \ref{CxP^1,Dtrivial}. If $D$ is not trivial, then $X$ is $(C\times \PP^1)$-isomorphic to $\PP(\O_{C\times \PP^1}\oplus \O_{C\times \PP^1}(b\sigma+\tau^*(D)))$ by Lemma \ref{representativedeg0}.
	
	Conversely, if $\pi$ is decomposable as above, then $\pi_*$ is surjective by Proposition \ref{modulioverCtimesPP^1}. Therefore, every $\Autz(X)$-orbit has dimension at least $2$, so by Lemma \ref{LinkfromFFb}, the pair $(X,\pi)$ is superstiff and in particular, $\Autz(X)$ is relatively maximal. 
\end{proof}

We can now conclude this section with

\begin{proof}[Proof of Proposition \ref{main_prop:CtimesP^1}]
	This follows from Propositions \ref{prop:ggeq2}, \ref{FF0viaCxPP1}, \ref{modulioverCtimesPP^1} and \ref{CxP^1,Dnontrivial}.
\end{proof}

\section{Automorphisms of $\PP^1$-bundles over $\AAA_1$}\label{SectionAAA1}

In this section, $C$ is an elliptic curve, $\tau\colon \AAA_1 \to C$ is the indecomposable geometrically ruled surface, where $\AAA_1 = \PP(\E_1)$ and $\E_1$ is an indecomposable and normalized rank-$2$ vector bundle which fits into a short exact sequence:
\[
0\to \O_C \to \E_1 \to \L \to 0,
\]
where $\L$ is a line bundle of degree one. Conversely, any line bundle $\L$ of degree one determines a rank-$2$ vector bundle $\E_1$ as above, and the $C$-isomorphism of $\AAA_1$ is independent of the choice of $\L$.
 
\medskip 
 
Recall that the geometrically ruled surface $\AAA_1$ admits infinitely many minimal sections of self-intersection one. We fix a line bundle $\L$ of degree one, that also determines a minimal section $\sigma$ such that $\O_{\AAA_1}(\sigma) \simeq \O_{\PP(\E_1)}(1)$ (see \cite[Example V.2.11.6]{Hartshorne}). As we will see in Lemma \ref{A_1quotient}, all $\Autz(\AAA_1)$-orbits are linearly equivalent: these are elliptic curves $\omega$ mapped $4$-to-$1$ onto $C$, and there exists $D_\sigma \in \Pic(C)$ of degree two, such that
\[
\omega \sim 4\sigma - \tau^*(D_\sigma).
\]
The main result of this section is the following:

\begin{proposition}\label{main_prop:A_1} 
	Let $\pi\colon X\to \AAA_1$ be a $\PP^1$-bundle such that $\tau\pi$ is an $\FF_b$-bundle, with $b\in \mathbb{Z}$. Then $\Autz(X)$ is relatively maximal if and only if one of the following case occurs:
	\begin{enumerate}
		\item Case $b=0$. Then $X$ is isomorphic to one of the following $\PP^1$-bundles:
		\[
		 \AAA_1 \times \PP^1,~ \AAA_1 \times_C \AAA_1,~  \AAA_1 \times_C \PP(\O_C\oplus \O_C(D)),
		 \] where $D\in \Pic^0(C)$ is not two-torsion. Moreover, the pair $(X,\pi)$ is superstiff in the first two cases, and not stiff in the last one. 
		\item\label{main_prop:A_1.2} Case $b = 4n\neq 0$ with $n\in \mathbb{Z}$. Then there exists $D\in \Pic^0(C)$ which is not two-torsion such that 
		\[
		X\cong \PP(\O_{\AAA_1}\oplus \O_{\AAA_1}(4n\sigma + \tau^*(D-nD_\sigma))),
		\]
		and there exists a square birational map $\AAA_1 \times_C \PP(\O_C\oplus \O_C(D)) \dashrightarrow X$ that conjugates \newline $\Autz(\AAA_1\times_C \PP(\O_C\oplus \O_C(D)))$ and $\Autz(X)$.
		
		\item Case $b=4n+2\neq 0$ with $n\in \mathbb{Z}$. Then there exists a non-trivial $2$-divisor $D$ (see Definition \ref{def:(m_2^*,b)}) such that
		\[
		X\cong \PP(\O_{\AAA_1}\oplus \O_{\AAA_1}((4n+2)\sigma + \tau^*(D-nD_\sigma))),
		\]
		and there exists a square birational map 
		$  \PP(\O_{\AAA_1}\oplus \O_{\AAA_1}(2\sigma + \tau^*(D))) \dashrightarrow X$ that conjugates $\Autz(\PP(\O_{\AAA_1}\oplus \O_{\AAA_1}(2\sigma + \tau^*(D))))$ and $\Autz(X)$.
	\end{enumerate}
	Let $\pi' \colon X'\to T$ be a $\PP^1$-bundle over a geometrically ruled surface $T$. Moreover, for $D\in \Pic^0(C)$ non two-torsion and with respect to the projection onto $\AAA_1$, there exists an $\Autz(\AAA_1 \times_C \PP(\O_C\oplus \O_C(D)))$-equivariant square birational map 
	\[
	\AAA_1 \times_C \PP(\O_C\oplus \O_C(D)) \dashrightarrow X',
	\]
	if and only if $T\cong \AAA_1$ and $X'\cong \AAA_1 \times_C \PP(\O_C\oplus \O_C(D))$, or $X'\cong \PP(\O_{\AAA_1}\oplus \O_{\AAA_1}(4n\sigma + \tau^*(D-nD_\sigma)))$ as in case \ref{main_prop:A_1.2}.
\end{proposition}

\subsection{$\AAA_1$ as an elliptic ruled surface}

~\medskip

The geometrically ruled surface $\AAA_1$ can be obtained as a quotient surface $(C\times \PP^1)/(\mathbb{Z}/2\mathbb{Z})^2$, where $(\mathbb{Z}/2\mathbb{Z})^2$ acts diagonally on $C\times \PP^1$. More precisely, we have the following construction:

\begin{lemma}\label{A_1quotient}
	Let $m_2\colon C\to C$ the multiplication by two. The kernel $C[2]$, that is isomorphic to $(\mathbb{Z}/2\mathbb{Z})^2$, acts on $C$ by translations and on $\PP^1$ via $x\mapsto \pm x^{\pm 1}$. We denote by $q\colon C\times \PP^1 \to (C\times \PP^1)/C[2]$ the quotient induced by the diagonal action. Then the following hold:
	\begin{enumerate}
		\item\label{A_1quotient.1} The quotient $(C\times \PP^1)/C[2]$ is $C$-isomorphic to the geometrically ruled surface $\AAA_1$, where the structure morphism is the projection $(C\times \PP^1)/C[2]\to C/C[2]$ composed with the isomorphism $C/C[2]\simeq C$ induced by $m_2$. This isomorphism yields the following cartesian square:
		\[
		\begin{tikzcd}[column sep=4em,row sep = 3em]
			C\times \PP^1 \arrow[r, "q"] \arrow[d, "\tau_1"]& \AAA_1 \arrow[d,"\tau"] \\
			C\arrow[r,"m_2"] & C,
		\end{tikzcd}
		\]
		where $\tau_1$ is the projection onto $C$.
		\item\label{A_1quotient.2} The morphism $q$ is $(\Autz(C)\times \{\mathrm{id}_{\PP^1}\})$-equivariant and induces an isomorphism $$q_*\colon \Autz(C)\times \{\mathrm{id}_{\PP^1}\} \longrightarrow \Autz(\AAA_1).$$ 
		\item\label{A_1quotient.3} Let $\sigma_0$ be a constant section of $C\times \PP^1$. There exists a divisor $D_0$ on $C$, which has degree two, and such that $q^*\sigma \sim \sigma_0+\tau_1^*(D_0)$. 
		
		\item\label{A_1quotient.4} Every $\Autz(\AAA_1)$-orbit $\omega$ is an elliptic curve mapped $4$-to-$1$ onto $C$ and there exists a divisor $D_\sigma$ on $C$, which has degree two, such that $\omega\sim 4\sigma -  \tau^{*}(D_\sigma)$ and $m_2^*(D_\sigma)\sim 4D_0$. Moreover, two $\Autz(\AAA_1)$-orbits are linearly equivalent, i.e., the linear class of $D_\sigma$ does not depend on $\omega$.
	\end{enumerate}
\end{lemma}

\begin{proof}
	(1) and (2): The fact that $\AAA_1$ is $C$-isomorphic to $(C\times \PP^1)/C[2]$ is well-known; see e.g. \cite[Section 3.3.3]{Fong} for the proof. Moreover, $ \Autz(C)$ acts faithfully on $(C\times \PP^1)/C[2]$, by $t \cdot (x,[u:v]) = (t(x),[u:v])$. Therefore, the morphism $q$ is $(\Autz(C)\times \{\mathrm{id}_{\PP^1}\})$-equivariant and induces a morphism $q_*\colon (\Autz(C)\times \{\mathrm{id}_{\PP^1}\})\to \Autz(\AAA_1)$. By Lemma \ref{geometryofruledsurface}, $\Autz(\AAA_1)$ has dimension one, so $q_*$ is an isomorphism. 
	
	(3): The Picard group of $C\times \PP^1$ is generated by the class of $\sigma_0$ and of the fibers; hence, there exist an integer $a\in \mathbb{Z}$ and a divisor $D_0$ such that $q^*\sigma\sim a \sigma_0 + \tau_1^*(D_0)$. Since $q^*\sigma$ is a section, we get $a=1$. By the projection formula, see e.g. \cite[Proposition 1.10]{Debarre}), and using that $\sigma^2 =1 $, it follows that $(q^*\sigma)^2 = 4$, which implies that $D_0$ has degree two.
	
	(4): Similarly, the Picard group of $\AAA_1$ is generated by the class of $\sigma$ and the fibers of $\tau$, so there exist $b\in \mathbb{Z}$ and $D_\sigma\in \Pic(C)$ such that $\omega \sim b\sigma - \tau^*(D_\sigma )$. By \ref{A_1quotient.1}, the orbit $\omega$ intersects every fiber at $4$ points; hence, $b=4$. Taking the pullback by $q$ both sides, we obtain that $q^*\omega \sim 4\sigma_0+{\tau_1}^*(4D_0 - {m_2}^*(D_\sigma))$. Since $q^*\omega \sim 4\sigma_0$ and $m_2$ is $4$-to-$1$, we get that ${m_2}^*(D_\sigma) \sim 4D_0$ and $D_\sigma$ has degree two. 

	Using the $C$-isomorphism $\AAA_1 \simeq (C\times \PP^1)/C[2]$, the projection onto $\PP^1/C[2]$, which is isomorphic to $\PP^1$, equips $\AAA_1$ with a structure of elliptic fibration, where the fibers are the $\Autz(\AAA_1)$-orbits. Thus, two $\Autz(\AAA_1)$-orbits are linearly equivalent.
\end{proof}	

\begin{remark}
	Let $\pi \colon X \to \AAA_1$ be a $\PP^1$-bundle with invariants $(\AAA_1,b,D)$. This triplet of invariants depends on the choice of a minimal section $\sigma\subset \AAA_1$, which is not $\Autz(\AAA_1)$-invariant. In general, pulling back $\pi$ by automorphisms of $S$ preserves neither $\sigma$ nor $D$. To compute the image of the morphism $\pi_*\colon \Autz(X) \to \Autz(\AAA_1)$, we do not directly compute the invariants of the pullback of $X$ by the automorphisms of $S$. Instead, our strategy is to pullback $X$ via the quotient morphism $q\colon C\times \PP^1 \to \AAA_1$, which yields a $\PP^1$-bundle over $C\times \PP^1$. This approach allows us to use the results from Section $\ref{SectionCxPP1}$.
\end{remark}

\subsection{Pullback bundles of $\PP^1$-bundles over $\AAA_1$}

\begin{lemma}\label{pullbackA1}
	 Let $\pi\colon X\to \AAA_1$ be a $\PP^1$-bundle with invariants $(\AAA_1,b,D)$, where $b>0$ is an integer and $D\in \Pic(C)$. Let $(\AAA_1)_b\subset X$ be the $\Autz(X)$-invariant surface spanned by the $(-b)$-sections along the fibers of $\tau\pi$. Then the following hold:
	\begin{enumerate}
		\item\label{pullbackA1.1} The pullback bundle $q^*\pi$ is a $\PP^1$-bundle with invariants $(C\times \PP^1,b,m_2^*(D)+bD_0)$.
		\item\label{pullbackA1.2} If $\deg(m_2^*(D)+bD_0)\neq 0$, then $\Autz(X)$ is not relatively maximal.
		\item\label{pullbackA1.3} Let $\omega\subset \AAA_1$ be an $\Autz(\AAA_1)$-orbit and $\widetilde{\omega}=(\AAA_1)_b\cap \pi^{-1}(\omega)$. Then the blowup of $\widetilde{\omega}$ followed by the contraction of the strict transform of $\pi^{-1}(\omega)$ yields an $\Autz(X)$-equivariant birational map $X\dasharrow X'$, where $\pi'\colon X'\to \AAA_1$ is a $\PP^1$-bundle with invariants $(\AAA_1,b+4,D-D_\sigma)$.
	\end{enumerate}
\end{lemma}

\begin{proof}
	(1): By Lemma \ref{FFb-bundle} \ref{canonicalextension}, there exists a rank-$2$ vector bundle $\E$ over $\AAA_1$, such that $\PP(\E) \simeq X$, and which fits into the short exact sequence
	\[
	0 \to \O_{\AAA_1}(b\sigma + \tau^*(D)) \to \E \to \O_{\AAA_1} \to 0.
	\]
	By Lemma \ref{A_1quotient}, we have $q^*\tau^*(D) = \tau_1^*m_2^*(D)$ and the section $q^*\sigma$ is linearly equivalent to $\sigma_0 +\tau_1^*(D_0)$. Pulling back the short exact sequence by $q\colon C\times \PP^1 \to \AAA_1$, we get:
	\[
	0 \to \O_{C\times \PP^1}(b\sigma_0 + \tau_1^*(m_2^*(D)+bD_0)) \to q^*\E \to \O_{C\times \PP^1} \to 0.
	\]
	This shows that $q^*\pi$ has invariants $(C\times \PP^1,b,m_2^*(D)+bD_0)$.
	
	(2): If $\deg(m_2^*(D)+bD_0)\neq 0$, it follows from Proposition \ref{degnonnuldec} that $(\tau_1q^*\pi)_*\colon \Autz(q^*X)\to \Autz(C)$ is trivial. Assume by contradiction that there exists $f\in \Autz(X)$ such that $\pi_*(f)\in \Autz(\AAA_1)$ is not the identity. Then by Lemma \ref{A_1quotient}, there exists $g\in (\Autz(C)\times \{\mathrm{id}_{\PP^1}\})$ such that $q_*(g) = \pi_*(f)$. Let 
	\[
	G = (\Autz(C)\times \{\mathrm{id}_{\PP^1}\}) \times_{\Autz(\AAA_1)} \Autz(X),
	\]
	that is a subset of $\Autz(q^*X)$ isomorphic to $\Autz(X)$, since $q_*$ is an isomorphism. Then $(g\times f) \in G$ is an automorphism of $q^*X = (C\times \PP^1)\times_{\AAA_1} X$ and its image by the morphism $(\tau_1 q^*\pi)_* \colon \Autz(q^*X) \to \Autz(C)$ is not the identity, which is a contradiction. Thus $\pi_*$ is trivial and by Proposition \ref{equivarianttodecomposable}, the automorphism group $\Autz(X)$ is not relatively maximal.
	
	(3): By Lemma \ref{FFb-bundle} \ref{FFb-bundle1}, we can choose trivializations of $\tau\pi$ such that the transition maps are
	\[
	\begin{array}{cccc}
		\phi_{ij}\colon & U_j\times \FF_b & \dasharrow & U_i\times \FF_b \\
		& (x,[y_0:y_1\ ;z_0:z_1]) & \longmapsto & \left(x,\left[y_0:\lambda_{ij}(x)y_1 + p_{ij}(x,z_0,z_1)y_0\ ;s_{ij}(x)\cdot 
		\begin{pmatrix}
			z_0 \\ z_1
		\end{pmatrix}\right] \right)
	\end{array}
	\]
	with $\lambda_{ij}\in \O_C(U_{ij})^*$, $p_{ij} \in \O_C(U_{ij})[z_0,z_1]_b$ and where $s_{ij}\in \GL_2(O_C(U_{ij}))$ are the transition matrices of $\tau\colon \AAA_1\to C$. Under this choice of trivializations, the $\Autz(\AAA_1)$-invariant surface $(\AAA_1)_b$ is $\{y_0=0\}$, and by Blanchard's lemma, $\widetilde{\omega}$ is also $\Autz(X)$-invariant. In each fiber of $\tau\pi$, that is isomorphic to $\FF_b$, the $(-b)$-section intersects $\widetilde{\omega}$ at four points counting with multiplicities. The blowup of $\widetilde{\omega}$ followed by the contraction of the strict transform of $\pi^{-1}(\omega)$ yields an $\Autz(X)$-equivariant birational map $\psi\colon X\dasharrow X'$, where $\pi'\colon X'\to \AAA_1$ is a $\PP^1$-bundle. Locally, $\psi$ is given by 
	\[
	\begin{array}{cccc}
		\psi_{i}\colon & U_i \times \FF_b & \dashrightarrow & U_i \times \FF_{b+4} \\
		& (x,[y_0:y_1\ ;z_0:z_1]) & \longmapsto & (x,[y_0:\delta_i(x,z_0,z_1)y_1\ ;z_0:z_1]),
	\end{array}
	\]
	where $\delta_i\in \O_C(U_{i})[z_0,z_1]_4$ satisfies $(\mathrm{div}(\delta_i))_{|U_i \times \PP^1} = \omega_{|U_i\times \PP^1}$. Computing the birational map $\psi_i \phi_{ij} \psi_j \colon U_j\times \FF_{b+4} \dasharrow  U_i\times \FF_{b+4}$, we get that $(x,[y_0:y_1\ ;z_0:z_1]) \in U_{ij}\times \FF_{b+4}$ is mapped to
	\[
	\bigg(x,\bigg[y_0:
	\delta_i \left(x,s_{ij}(x)\cdot 
	\begin{pmatrix}
		z_0 \\ z_1
	\end{pmatrix}\right)	\delta_j^{-1}(x,z_0,z_1)\left(\lambda_{ij}(x)y_1 + p_{ij}(x,z_0,z_1)y_0\right) 
	; s_{ij}(x)\cdot \begin{pmatrix} z_0\\z_1 \end{pmatrix}\bigg]\bigg).
	\]		
	The quantities $\delta_i \left(x,s_{ij}(x)\cdot 
	\begin{pmatrix}
		z_0 \\ z_1
	\end{pmatrix}\right)	\delta_j^{-1}(x,z_0,z_1)$ are the cocyles of the line bundle $\O_{\AAA_1}(\omega)$. Since $\omega\sim 4\sigma -  \tau^{*}(D_\sigma)$ by Lemma \ref{A_1quotient}, it follows that the $\PP^1$-bundle $X'\to \AAA_1$ has invariants $(\AAA_1,b+4,D-D_\sigma)$. 
\end{proof}

If $\Autz(X)$ is relatively maximal and $\pi \colon X\to \AAA_1$ has invariants $(\AAA_1,b,D)$, then $\deg(m_2^*(D) + bD_0)=0$ by Lemma $\ref{pullbackA1}$. Since $m_2$ is a $4$-to-$1$ morphism and $D_0$ has degree two, this implies that $b=-2\deg(D)$. When $b$ is even, we show in the following lemma that we can lift the action of $C[2]$ to $q^*X$, and the corresponding quotient $q^*X/C[2]$ is $X$.

\begin{lemma}\label{C[2]actiononFb}
 		Let $b>0$ be an even integer, $\pi\colon X\to \AAA_1$ be a $\PP^1$-bundle such that $\tau\pi$ is an $\FF_b$-bundle over $C$, and $p_X\colon q^*X \to X$ the projection onto $X$. We denote by $\iota\colon C[2] \to \PGL_2(\kk)$ the injective group homomorphism induced by the action of $C[2]$ on $\PP^1$, given by $x\mapsto \pm x^{\pm 1}$. Then the following hold:
 		\begin{enumerate}
			\item\label{C[2]actiononFb.1} There exists a lifting $\hat{\iota}\colon C[2] \to \GL_2(\kk)/\mu_b$ of $\iota$, such that for any $C[2]$-invariant open subset $U\subset C$, the finite group $C[2]$ acts diagonally on $U\times \FF_b$ as follows:
			\[
			\begin{array}{ccc}
				C[2] \times (U\times \FF_b) & \longrightarrow & U\times \FF_b \\
				(t,(x,[y_0:y_1\ ;z_0:z_1])) & \longmapsto & \left(t(x),\left[y_0:y_1\ ;\hat{\iota}(t)\cdot \begin{pmatrix} z_0\\z_1\end{pmatrix}\right]\right).
			\end{array}	
			\]
			\item\label{C[2]actiononFb.2} There exist trivializations of the $\FF_b$-bundle $\tau_1q^*\pi\colon q^*X \to C$ such that every trivializing open subset $U\subset C$ is $C[2]$-invariant, and via these trivializations, $p_X$ is locally the quotient of $U\times \FF_b$ by the action of $C[2]$ given in \ref{C[2]actiononFb.1}.
			\item The diagonal action of $C[2]$ on $C\times \PP^1$ lifts to an action on $q^*X$, such that $q^*X/C[2]\cong X$ as $\PP^1$-bundles over $\AAA_1$.
 		\end{enumerate}
\end{lemma}

\begin{proof}
	(1): Since $b$ is even, $-1$ belongs to the group of $b$-th root of unity $\mu_b$. This implies that 
	$$
	H=\left\{ \begin{pmatrix} 1 & 0 \\ 0 & 1 \end{pmatrix}, \begin{pmatrix} 1 & 0 \\ 0 & -1 \end{pmatrix}, \begin{pmatrix} 0 &1  \\1  & 0\end{pmatrix}, \begin{pmatrix} 0& -1\\1  &0 \end{pmatrix} \right\}
	$$
	is a subgroup of $\GL_2(\kk)/\mu_b$. We get an injective group homomorphism $\hat{\iota}\colon C[2] \to \GL_2(\kk)/\mu_b$, with image $H$, such that the composition with the quotient morphism $\GL_2(\kk)/\mu_b \to \PGL_2(\kk)$ coincides with $\iota$. Since $U$ is $C[2]$-invariant, we can define the diagonal action on $U\times \FF_b$ as stated.
	
	(2) and (3): Let $V$ be a trivializing open subset of $\tau\pi\colon X\to C$. Then $U=m_2^{-1}(V)$ is $C[2]$-invariant and the pullback of $(\tau\pi)^{-1}(V)\simeq V\times \FF_b$ by $q$ yields the following cartesian squares:
	\[
	\begin{tikzcd}[column sep=2em,row sep = 3em]
		(U\times \PP^1) \times_{V\times \PP^1} (V\times\FF_b) \arrow[r, "p_2",swap] \arrow[d, "p_1",swap]& V\times \FF_b \arrow[d,"\mathrm{id}_{V}\times \tau_b"] \\
		U\times \PP^1\arrow[r,"q_{|U\times \PP^1}", swap] \arrow[d, "{\tau_1}_{|U\times \PP^1}",swap] & V\times \PP^1\arrow[d, "\tau_{|V\times \PP^1}"]\\
		U \arrow[r, "{m_2}_{|U}",swap]& V,
	\end{tikzcd}
	\]
	where $p_2$ is the restriction of the morphism $p_X$ to the open subset $(U\times \PP^1) \times_{V\times \PP^1} (V\times\FF_b)$.
	
	There exist homogeneous polynomials $f_0,f_1\in \O_C(U)[z_0,z_1]_b$ of same degree $d$ such that
	\[q_{|U\times \PP^1}(x,[z_0:z_1]) = (m_2(x),[f_0(x,z_0,z_1):f_1(x,z_0,z_1)]).\]
	Since $m_2$ and $q$ are both $4$-to-$1$ morphisms, this implies that $d=1$.
	
	Moreover, the variety $U\times \FF_b$ is equipped with the morphisms $\mathrm{id}_{U}\times \tau_b\colon U\times \FF_b \to U\times \PP^1$, and 
	$$
	\begin{array}{cccc}
		\widetilde{q}\colon & U\times \FF_b & \longrightarrow & V\times \FF_b \\
		& (x,[y_0:y_1\ ;z_0:z_1])  & \longmapsto & (m_2(x),[y_0:y_1\ ;f_0(x,z_0,z_1):f_1(x,z_0,z_1)]),
	\end{array}
	$$
	which is the quotient of $U\times \FF_b$ by the diagonal action of $C[2]$ given in \ref{C[2]actiononFb.1}. By the universal property, there exists a unique morphism
	 $$
	 \begin{array}{cccc}
	 	f\colon & U\times \FF_b & \longmapsto & (U\times \PP^1) \times_{V\times \PP^1} (V\times\FF_b) \\
	 	& (x,[y_0:y_1\ ;z_0:z_1])  & \longmapsto & (x,[z_0:z_1]),(m_2(x),[y_0:y_1\ ;f_0(x,z_0,z_1):f_1(x,z_0,z_1)]),
	 \end{array}
	 $$
	 such that $p_2 f = \tilde{q} $ and $p_1 f = \mathrm{id}_{U}\times \tau_b$. Since $p_2$ and $\widetilde{q}$ are both $4$-to-$1$ onto $V\times \FF_b$, it follows that $f$ is bijective. Moreover, $U\times \FF_b$ is connected and $(U\times \PP^1) \times_{V\times \PP^1} (V\times\FF_b)$ is smooth; thus, $f$ is an isomorphism by Zariski's Main Theorem. We obtain new trivializations of $\tau_1 q^*\pi \colon q^*X \to C$ which identify $p_X$ with the quotient by the action of $C[2]$ defined in \ref{C[2]actiononFb.1}.
\end{proof}

We have seen that if $\Autz(X)$ is relatively maximal, then $\deg(m_2^*(D) + bD_0)=0$. Equivalently, $b=-2\deg(D)$. Next, we distinguish the cases where $m_2^*(D) + bD_0$ is trivial and non-trivial of degree zero.

\begin{lemma}\label{pullbackA1.trivialinv}
	Let $\pi\colon X\to \AAA_1$ be a $\PP^1$-bundle with invariants $(\AAA_1,b,D)$, where $b>0$ is an even integer and $D$ a divisor such that $m_2^*(D)+bD_0\sim 0$. Then $\Autz(X)$ is not relatively maximal.
\end{lemma}

\begin{proof}
	Let $\omega$ be an $\Autz(\AAA_1)$-orbit. By Lemma \ref{pullbackA1}, the blowup of $\pi^{-1}(\omega)\cap (\AAA_1)_b$ followed by the contraction of $\pi^{-1}(\omega)$ yields an $\Autz(X)$-equivariant square birational map $\psi \colon X\dashrightarrow X'$, where $\pi' \colon X' \to \AAA_1$ is a $\PP^1$-bundle with invariants $(\AAA_1,b+4,D-D_\sigma)$. The pullback bundle $q^*X'\to C\times \PP^1$ has invariants $(C\times \PP^1,b+4,m_2^*(D-D_\sigma)+(b+4)D_0)$. Since $m_2^*(D_\sigma) \sim4D_0$ by Lemma \ref{A_1quotient} \ref{A_1quotient.4}, it follows that the divisor $m_2^*(D-D_\sigma)+(b+4)D_0\sim m_2^*(D)+bD_0$, which is trivial by assumption. 
	
	By Lemma \ref{FFb-bundle}, we can choose the trivializations of $\tau_1q^*\pi'\colon q^*X'\to C$ such that the transition maps are
		\[
		\begin{array}{cccc}
			\phi_{ij}\colon & U_j \times \FF_{b+4} & \dasharrow & U_i \times\FF_{b+4} \\
			& (x,[y_0:y_1\ ; z_0:z_1]) & \longmapsto & (x,[y_0:y_1+p_{ij}(x,z_0,z_1)y_0\ ; z_0:z_1]),
		\end{array}
		\]
		for some $p_{ij}\in \O_C(U_{ij})[z_0,z_1]_{b+4}$. Shrinking the $U_i$ if necessary, we can also assume that each $U_i$ is $C[2]$-invariant, and $C[2]$ acts on $U_i\times \FF_b$ via the action defined in Lemma \ref{C[2]actiononFb}. For every $\mu\in \kk$ and index $i$, we define the automorphism 
		\[
		\begin{array}{cccc}
			h_i \colon& U_i \times \FF_{b+4} & \longrightarrow & U_i \times\FF_{b+4} \\
			& (x,[y_0:y_1\ ; z_0:z_1]) & \longmapsto & (x,[y_0:y_1+\mu (z_0^{b+4}+z_1^{b+4})y_0\ ; z_0:z_1]).
		\end{array}
		\]
		Then $h_i \phi_{ij} = \phi_{ij} h_j$, so this yields an element $h\in \Autz(q^*X')$. 
		
		The integer $b$ is even, so $h$ maps $C[2]$-orbits to $C[2]$-orbits, and it follows that the quotient morphism $q^*X'\to q^*X'/C[2]$ is $h$-equivariant. By Lemma \ref{C[2]actiononFb}, $q^*X'/C[2]$ is $\AAA_1$-isomorphic to $X'$; hence, $h$ induces an element of $\Autz(X')$ which does not leave invariant the base locus of $\psi^{-1}$. We conclude that $\psi \Autz(X) \psi^{-1} \subsetneq \Autz(X')$.
\end{proof}

\begin{definition}\label{def:(m_2^*,b)}
	We say that a divisor $D$ on $C$ is a \emph{non-trivial $2$-divisor} if $m_2^*(D)-2\deg(D)D_0$ is not trivial and has degree zero.
\end{definition}

\begin{lemma}\label{SarkisovoverA1}
	Let $\pi\colon X\to \AAA_1$ be a $\PP^1$-bundle with invariants $(\AAA_1,b,D)$, where $b>0$ is an even integer and $D$ a non-trivial $2$-divisor of degree $-b/2$, and let $\omega\subset \AAA_1$ be an $\Autz(\AAA_1)$-orbit. Then the following hold:
	\begin{enumerate}
		\item\label{SarkisovoverA1.0} $X$ is isomorphic to the decomposable $\PP^1$-bundle
		\[
		\PP(\O_{\AAA_1} \oplus \O_{\AAA_1}(b\sigma + \tau^*(D)))
		\]
		and the two disjoint sections $S_0$ and $S_1$  corresponding respectively to the line subbundles $\O_{\AAA_1}(b\sigma + \tau^*(D)) $ and $ \O_{\AAA_1}$ are $\Autz(X)$-invariant. Moreover, $\Autz(X)$ fits into a short exact sequence
		\[
		1\to \mathbb{G}_m \to \Autz(X)\to \Autz(\AAA_1) \to 1.
		\]
		\item\label{SarkisovoverA1.1} The blowup of $\pi^{-1}(\omega)\cap S_0$ followed by the contraction of the strict transform of $\pi^{-1}(\omega)$ yields an $\Autz(X)$-equivariant square birational map $\psi \colon X\dasharrow X'$ such that $\psi \Autz(X) \psi^{-1} = \Autz(X')$, where $X'=\PP(\O_{\AAA_1} \oplus \O_{\AAA_1}((b+4)\sigma + \tau^*(D-D_\sigma)))$. Moreover, $D-D_\sigma$ is a non-trivial $2$-divisor of degree $-(b+4)/2$.
		\item\label{SarkisovoverA1.2} The blowup of $\pi^{-1}(\omega)\cap S_1$ followed by the contraction of the strict transform of $\pi^{-1}(\omega)$ yields an $\Autz(X)$-equivariant square birational map $\phi\colon X\dasharrow X''$ such that  $\phi \Autz(X) \phi^{-1} = \Autz(X'')$, where $X''=\PP(\O_{\AAA_1} \oplus \O_{\AAA_1}((b-4)\sigma + \tau^*(D+D_\sigma)))$.
		Moreover, $D+D_\sigma$ is a non-trivial $2$-divisor of degree $-(b-4)/2$.
		\item\label{SarkisovoverA1.4} Every $\Autz(X)$-equivariant Sarkisov link starting from the $\PP^1$-bundle $\pi\colon X\to \AAA_1$ is given in \ref{SarkisovoverA1.1} or \ref{SarkisovoverA1.2}.
	\end{enumerate}
\end{lemma}

\begin{proof}
	\ref{SarkisovoverA1.0} By assumption, the divisor $\widetilde{D}=m_2^*(D)+bD_0$ is non-trivial of degree zero. By Lemmas \ref{modulioverCtimesPP^1} and \ref{pullbackA1}, the $\PP^1$-bundle $q^*\pi\colon q^*X\to C\times \PP^1$ is decomposable and admits two $\Autz(q^*X)$-invariant disjoint sections. In particular, these two sections are also $C[2]$-invariant and their images $S_0$ and $S_1$ are also disjoint sections of $\pi$. Hence, the $\PP^1$-bundle $\pi \colon X\to \AAA_1$ is decomposable and $\AAA_1$-isomorphic to $\PP(\O_{\AAA_1} \oplus \O_{\AAA_1}(b\sigma + \tau^*(D)))$.  
	
	Next, we show that $\pi_*$ is surjective by constructing explicitly an element $f\in \Autz(X)$ such that $\pi_*(f)$ is not the identity. Let $(U_i)_i$ be a trivializing open cover of $q^*\pi \tau_1\colon q^*X \to C$ such that each $U_i$ is $C[2]$-invariant and the transition maps are
	\[
	\begin{array}{cccc}
		\phi_{ij}\colon & U_j\times \FF_b & \dasharrow & U_i\times \FF_b \\
		&(x,[y_0:y_1\ ;z_0:z_1]) &\longmapsto & \left(x,\left[y_0:\widetilde{\lambda}_{ij}(x)y_1\ ;
			z_0 : z_1
	\right] \right),
	\end{array}
	\]
	where $\widetilde{\lambda}_{ij}\in \O_C(U_{ij})^*$ are the cocycles of $\mathcal{O}(\widetilde{D})$. 
	
	Since $\deg(\widetilde{D})=0$, every translation $t\in \Autz(C)$ lifts to an automorphism of the line bundle $\O_C(\widetilde{D})$, which is locally of the form
	\[
	\begin{array}{cccc}
		\widetilde{t}_i\colon & V_i\times \AA^1 & \longrightarrow & t(V_i) \times \AA^1 \\
		& (x,y) & \longmapsto & (t(x),\mu_i(x)y),
	\end{array}
	\]
	where $V_i=m_2(U_i)\subset C$ and $ \mu_i\in \O_C(V_i)^*$. Writing the transition maps of the line bundle $\O_C(\widetilde{D})$ as 
	\[
	\begin{minipage}{8cm}
		\[
		\begin{array}{ccccccc}
			V_j\times \AA^1 & \dashrightarrow & V_i \times \AA^1   \\
			(x,z) & \longmapsto & (x,\widetilde{\lambda}_{ij}(x)z),
		\end{array}
		\]	
	\end{minipage}
	\begin{minipage}{8cm}
		\[
		\begin{array}{ccccccc}
			 t(V_j)\times \AA^1 & \dashrightarrow & t(V_i) \times \AA^1 \\
			 (x,z) & \longmapsto & (x,\widetilde{\lambda}_{t(ij)}(x)z),
		\end{array}
		\]	
	\end{minipage}
	\]
	where $\widetilde{\lambda}_{ij}\in \O_C(V_{ij})^*$ and $\widetilde{\lambda}_{t(ij)}\in \O_C(t(V_{ij}))^*$. Since the $(\widetilde{t}_i)_i$ commute with the transition maps above, we obtain the gluing condition:
	\begin{equation}\label{eqn:gluing}
		\mu_i(x) = \widetilde{\lambda}_{t(ij)}(t(x)) \widetilde{\lambda}_{ij}(x)^{-1} \mu_j(x). \tag{$\dagger$}
	\end{equation}
	Hence, the $(\mu_i)_i$ give a trivializing section $\mu$ of the line bundle $\O_C(t^*\widetilde{D}) \otimes \O_C(\widetilde{D})^\vee$. 
	
	Moreover, $m_2^*\colon \Pic^0(C)\to \Pic^0(C)$ is the multiplication by two, so there exists a non-trivial divisor $D'$ of degree zero such that $m_2^*(D') = \widetilde{D}$. In particular, we can choose $\mu$ as a pullback section of $\O_C(D')$, i.e., we can assume that $\mu$ is $C[2]$-invariant. Then we define the following isomorphisms:
	\[
	\begin{array}{cccc}
		\widetilde{f}_i \colon & U_i\times \FF_b & \longrightarrow & t(U_i) \times \FF_b \\
		& (x,[y_0:y_1 \ ;z_0:z_1]) & \longmapsto & (t(x),[y_0:\mu_i(x)y_1 \ ;z_0:z_1]),
	\end{array}
	\]
	 and using (\ref{eqn:gluing}), we obtain that $\widetilde{f}_i\phi_{ij} = \phi_{t(ij)} \widetilde{f}_j$, where $\phi_{ij}\colon U_j\times \FF_b \dashrightarrow U_i\times \FF_b$ and $\phi_{t(ij)}\colon t(U_j) \times \FF_b \dashrightarrow t(U_i) \times \FF_b$ denote the transition maps of the $\FF_b$-bundle $q^*\pi \tau_1 \colon q^*X\to C$. Therefore, the $(f_i)_i$ yield an element $\widetilde{f}\in \Autz(q^*X)$, that maps $C[2]$-orbits to $C[2]$-orbits for the action given in Lemma \ref{C[2]actiononFb}. The quotient morphism $p_X \colon q^*X \to X$ is $\widetilde{f}$-equivariant and yields an element $f\in \Autz(X)$ that acts non-trivially on $\AAA_1$, whose automorphism group is one-dimensional by Lemma \ref{Autalgebraicgroup}. Thus, $\pi_*$ is surjective.
	
	Let $f\in \Autz(X)$. By Lemma \ref{A_1quotient}, we have the isomorphism $q_*\colon \Autz(C)\times \{\mathrm{id}_{\PP^1}\}\simeq \Autz(\AAA_1)$. This induces a closed embedding 
	\[
	\begin{array}{cccc}
		i\colon & \Autz(X) & \longrightarrow &   \Autz(q^*X) \\
		& f & \longmapsto & (q_*^{-1}(\pi_*(f)) \times \mathrm{id}_{\PP^1}) \times f 
	\end{array}
	\]
	with image $(\Autz(C) \times \{\mathrm{id}_{\PP^1}\}) \times_{\Autz(\AAA_1)} \Autz(X)$. By Lemma \ref{modulioverCtimesPP^1}, the two disjoint sections of $q^*X$ are $\Autz(q^*X)$-invariant; hence, $S_0$ and $S_1$ are $\Autz(X)$-invariant. Moreover, since $\pi$ is decomposable, it follows that $\mathbb{G}_m\subset \ker(\pi_*)$. Since $i(\ker(\pi_*))\subset \ker(q^*\pi) =\mathbb{G}_m$  and $q_*\colon \Autz(C) \times \{\mathrm{id}_{\PP^1}\} \to \Autz(\AAA_1)$ is an isomorphism, this implies that $\ker(\pi_*)= \mathbb{G}_m$.

	\ref{SarkisovoverA1.1} In Lemma \ref{pullbackA1} \ref{pullbackA1.3}, we prove that $\psi\colon X\dashrightarrow X'$ is $\Autz(X)$-equivariant and $\pi'\colon X'\to \AAA_1$ has invariants $(\AAA_1,b+4,D-D_\sigma)$. Using that $\deg(D_\sigma)=2$ and $m_2^*(D_\sigma) = D_0$, we get that $D-D_\sigma$ is a non-trivial $2$-divisor of degree $-(b+4)/2$. Therefore, $X'$ is $\AAA_1$-isomorphic to $\PP(\O_{\AAA_1} \oplus \O_{\AAA_1}((b+4)\sigma + \tau^*(D-D_\sigma)))$ and the base locus of $\psi$ is $\Autz(X)$-invariant. Therefore $\psi \Autz(X) \psi^{-1} = \Autz(X')$. 
	
	\ref{SarkisovoverA1.2} is proven similarly as \ref{SarkisovoverA1.1}, except if $b=4$, then we obtain that $\tau \pi''\colon X''\to C$ is the $\FF_0$-bundle $\PP(\O_C\oplus \O_C(D+D_\sigma)) \times_C \AAA_1$. Notice that by assumption, $\deg(D) = -2$. If $D+D_\sigma\sim 0$, we get that $m_2^*(D) + 4D_0 \sim 0$, which contradicts that $D$ is a non-trivial $2$-divisor. Therefore, $D+D_\sigma$ is non-trivial of degree $0$ and by Lemmas \ref{geometryofruledsurface} and \ref{autofiberproduct}, the base locus of $\phi^{-1}$ is an $\Autz(X'')$-invariant section. Thus, $\phi\Autz(X) \phi^{-1} = \Autz(X'')$. 
	
	Now, we prove \ref{SarkisovoverA1.4}. Since $\pi$ is decomposable, the connected group $\mathbb{G}_m$ acts fiberwise on the complement of the two disjoint sections $S_0$ and $S_1$. Moreover, the morphism induced by Blanchard's lemma $\pi_*\colon \Autz(X)\to \Autz(\AAA_1)$ is surjective.
	Therefore, the only $\Autz(X)$-orbits having dimension one are contained in $S_0$ or $S_1$; and they equal $\pi^{-1}(\omega)\cap S_0$ or $\pi^{-1}(\omega)\cap S_1$ for some $\Autz(\AAA_1)$-orbit $\omega$. By the $2$-ray game, the blowup of such $\Autz(X)$-orbit yields a Sarkisov diagram of type II, and depending whether this $\Autz(X)$-orbit lies in $S_0$ or $S_1$, we obtain the cases \ref{SarkisovoverA1.1} or \ref{SarkisovoverA1.2} stated in this lemma. Finally, by Lemma \ref{SarkisovIIIandIV} \ref{SarkisovIIIandIV.1}, there is no $\Autz(X)$-equivariant Sarkisov diagram of type III and IV starting from $\pi\colon X\to \AAA_1$.
\end{proof}

	Under the assumptions of Lemma $\ref{SarkisovoverA1}$, every $\Autz(X)$-equivariant square birational map has target a
	$\PP^1$-bundle with invariants $(\AAA_1,b-4,D+D_\sigma)$ or $(\AAA_1,b+4,D-D_\sigma)$. By induction, we get that $\Autz(X)$ is conjugate to $\Autz(\widetilde{X})$, where $\widetilde{\pi}\colon \widetilde{X}\to \AAA_1$ is a $\PP^1$-bundle such that $\tau\widetilde{\pi}\colon \widetilde{X}\to C$ is an $\FF_0$-bundle over $C$ or an $\FF_2$-bundle over $C$. The following proposition treats the case of $\FF_2$-bundles.

\begin{proposition}\label{4n+2}
	Let $\pi\colon X\to \AAA_1$ be the decomposable $\PP^1$-bundle $\PP(\O_{\AAA_1}\oplus \O_{\AAA_1}(2\sigma+\tau^*(D)))$, where $D$ is a non-trivial $2$-divisor. Then the following hold:
	\begin{enumerate}
		\item\label{4n+2.1} Let $X'$ be a $\PP^1$-bundle over $\AAA_1$. There exists an $\Autz(X)$-equivariant square birational map $X\dashrightarrow X'$ if and only if there exists $n\in \mathbb{Z}$ such that 
		\[
		X'\cong \PP(\O_{\AAA_1}\oplus \O_{\AAA_1}((4n+2)\sigma + \tau^*(D-nD_\sigma))).
		\]
		Moreover, this square birational map conjugates $\Autz(X)$ and $\Autz(X')$.
		\item\label{4n+2.4} The group $\Autz(X)$ is relatively maximal but the pair $(X,\pi)$ is not stiff. 
	\end{enumerate}
\end{proposition}

\begin{proof}
	 By Theorem \ref{Floris}, every $\Autz(X)$-equivariant birational map $X\dashrightarrow X'$, where $X'$ is a Mori fiber space, can be decomposed as product of $\Autz(X)$-equivariant Sarkisov links. Using Lemma \ref{SarkisovoverA1} inductively, we obtain \ref{4n+2.1}. Then \ref{4n+2.4} follows from \ref{4n+2.1}.
\end{proof}

\subsection{The $\FF_0$-bundles arising from a $\PP^1$-bundle over $\AAA_1$}

~\medskip

Next, we focus on the case of fiber products over $C$ of $\AAA_1$ with the following geometrically ruled surfaces:
\begin{enumerate}
	\item $C\times \PP^1$, 
	\item $\AAA_0$,
	\item $\AAA_1$, 
	\item $\PP(\O_C\oplus \L)$, where $\L$ is a non-trivial line bundle of degree zero.
\end{enumerate}
 In the first case, we have $\AAA_1\times_C (C\times \PP^1)  \cong \AAA_1 \times \PP^1$. By Proposition $\ref{FF0viaCxPP1}$, we already know that the pair $(\AAA_1\times \PP^1,\pi\colon \AAA_1\times \PP^1\to \AAA_1)$ is superstiff. In this section, we treat the other three fiber products one by one. Before, let us state the following lemma which will be useful:

\begin{lemma}\label{lem:A1_no_Sarkisov}
	Let $S$ be a geometrically ruled surface over $C$ and $\pi\colon X\to S$ a $\PP^1$-bundle. Assume that all $\Autz(X)$-orbits are of dimension at least one, and that those of dimension one are smooth curves mapped $4$-to-$1$ onto $S$. Then there is no $\Autz(X)$-equivariant Sarkisov diagram of type I and II starting from $\pi$.
\end{lemma}

\begin{proof}
	An $\Autz(X)$-Sarkisov diagram of type I or II is entirely determined by the $2$-ray game, which in turn is determined by the data of the blowup of a one-dimensional $\Autz(X)$-orbit $\omega$. 
	Let $\eta\colon Y \to X$ be the blow-up of $\omega$. The relative cone of effective curves $\mathrm{NE}(Y/S)$ is generated by the class of a fiber $f$ of $\eta$, that is isomorphic to $\PP^1$, and the class of the preimage of a fiber of $\pi$, that we denote by $g$. The class of $f$ generates an extremal ray of $\mathrm{NE}(Y/S)$, whose contraction corresponds to the birational morphism $\eta$. Therefore, the other extremal ray is generated by the class of $ag-bf$, where $a,b\geq 0$ and $b/a$ is maximal. 
	
	Let $\gamma$ be a fiber of $\pi$ passing through $\omega$ at $4$ distinct points, and $\widetilde{\gamma}$ the strict transform of $\gamma$ by $\eta$. Then $\eta^*\gamma \sim 4f + \widetilde{\gamma}$. It follows that $\widetilde{\gamma}$ is numerically equivalent to $g-4f$, and this implies that $b/a\geq 4$. Moreover, we have
	$$\mathrm{K}_Y\cdot (g-4f)=\mathrm{K}_X\cdot \eta_*(g) - 4E\cdot f = -2 +4 >0,$$
	where $E$ denotes the exceptional divisor of $\eta$, and the equality $E\cdot f = -1$ follows from \cite[Lemma 2.2.14]{Iskovskikh}. The curves numerically equivalent to $g-4f$ correspond to the strict transforms of fibers of $\pi$ passing through $\omega$ at $4$ points. Hence, these curves span the surface $\pi^{-1}(\pi(\omega))$ and are positive against $\mathrm{K}_Y$. Thus, the contraction of this surface yields a variety with non-terminal singularities, which is prohibited by the Sarkisov program.	
\end{proof}

\begin{proposition}\label{A_0A_1notmax}
	Let $X=\AAA_0\times_C \AAA_1$, let $p_1$ and $p_2$ be the projections onto $\AAA_0$ and $\AAA_1$, respectively. Then the following hold:
	\begin{enumerate}
		\item $\Autz(X)$ is relatively maximal with respect to $p_1$ and the pair $(X,p_1)$ is superstiff.
		\item $\Autz(X)$ is not relatively maximal with respect to $p_2$.
	\end{enumerate} 
\end{proposition}

\begin{proof}
	 We denote by $\sigma_0\subset \AAA_0$ the unique $\Autz(\AAA_0)$-invariant section. By Lemma \ref{autofiberproduct} and Proposition \ref{candidateFF0max}, we have the isomorphism $\Autz(X)\simeq \Autz(\AAA_0)\times_{\Autz(C)} \Autz(\AAA_1)$ and the morphism $(\tau p_1)_*\colon \Autz(X) \to \Autz(C)$, which also equals $(\tau p_2)_*$, is surjective. By Lemma \ref{geometryofruledsurface}, the $\Autz(X)$-orbits of dimension one are contained in the surface $S=(p_1)^{-1}(\sigma_0) \subset X$, which is $\Autz(X)$-invariant, and are precisely the curves of the form $\widetilde{\omega} = p_2^{-1}(\omega) \cap S$, where $\omega\subset \AAA_1$ is an $\Autz(\AAA_1)$-orbit. 
	
	(1): The curves $\widetilde{\omega}$ are mapped $4$-to-$1$ onto $\AAA_0$ via $p_1$. By Lemmas \ref{SarkisovIIIandIV} and \ref{lem:A1_no_Sarkisov}, the only $\Autz(X)$-equivariant diagram is of type IV and exchanges the two fibrations $p_1$ and $p_2$. Therefore, $\Autz(X)$ is relatively maximal with respect to $p_1$ and $(X,p_1)$ is superstiff. 
	
	(2): Fix an $\Autz(X)$-orbit $\widetilde{\omega}$ of dimension one. Using Lemmas \ref{transitionAtiyah} \ref{transitionAtiyah.2} and \ref{fiberproduct}, we can choose the trivializations of $\tau p_2$ such that the transition maps equal
	\[
	\begin{array}{cccc}
		\phi_{ij} \colon & U_j\times \FF_0 & \dasharrow & U_i\times \FF_0\\
		& (x,[y_0:y_1\ ; z_0:z_1 ]) & \longmapsto & \left(x,\left[y_0: y_1 + \alpha_{ij}(x)y_0 \ ; s_{ij}(x)\cdot 
		\begin{pmatrix}
			z_0\\z_1
		\end{pmatrix} \right] \right),
	\end{array}
	\]
	where $s_{ij} \in \GL_2(\O_C(U_{ij}))$ are the transition matrices of $\AAA_1$ and $\alpha_{ij}\in \O_C(U_{ij})^*$. Under this choice of trivializations, $S =\{y_0=0\}$.
	 The restriction of $S$ on each fiber of the $\FF_0$-bundle $\tau p_2\colon X\to C$ is a constant section of $\FF_0$ which intersects $\widetilde{\omega}$ at four points counting with multiplicities. 
	 Then the blowup of $\widetilde{\omega}$ followed by the contraction of $p_2^{-1}(\omega)$ yields an $\Autz(X)$-equivariant square birational map $\psi\colon X\dasharrow X'$, where $\pi'\colon X'\to \AAA_1$ is a $\PP^1$-bundle, and $\psi$ is locally given by 
	 \[
	 \begin{array}{cccc}
	 	\psi_{i}\colon & U_i \times \FF_0 & \dashrightarrow & U_i \times \FF_{4} \\
	 	& (x,[y_0:y_1\ ;z_0:z_1]) & \longmapsto & (x,[y_0:\delta_i(x,z_0,z_1)y_1\ ;z_0:z_1]),
	 \end{array}
	 \]
	 where $\delta_i\in \O_C(U_{i})[z_0,z_1]_4$ is such that $(\mathrm{div}(\delta_i))_{|U_i \times \PP^1} = \omega_{|U_i\times \PP^1}$. The quantities 
	$$\delta_i\left(x,s_{ij}(x)\cdot 
	\begin{pmatrix}
		z_0\\z_1
	\end{pmatrix}\right) \delta_j(x,z_0,z_1)^{-1}$$ are the cocycles of the line bundle $\O_{\AAA_1}(\omega)$, which is isomorphic to $\O_{\AAA_1}(4\sigma - \tau^*(D_\sigma))$ by Lemma \ref{A_1quotient}. Computing $\psi_i \phi_{ij} \psi_j^{-1}$, we get that the $\PP^1$-bundle $\pi'\colon X'\to \AAA_1$ has invariants $(\AAA_1,4,-D_\sigma)$. By Lemma \ref{pullbackA1}, $m_2^*(-D_\sigma)+4D_0 \sim 0$, and by Lemma \ref{pullbackA1.trivialinv}, $\Autz(X')$ is not relatively maximal. 
\end{proof}

\begin{proposition}\label{A_1A_1max}
	Let $ X = \AAA_1\times _C\AAA_1$, let $p_1$ and $p_2$ be the projections onto the first and second factor, respectively. Then the following hold:
	\begin{enumerate}
		\item\label{A_1max.1} There exists a unique $\Autz(X)$-equivariant Sarkisov diagram starting from $p_1$, that is the following one of type IV:
		\[
		\begin{tikzcd}[column sep=2em,row sep = 3em]
			X \arrow[d,"p_1",swap] \arrow[rr,equal] & & X  \arrow[d,"p_2"]  \\
			\AAA_1 \arrow[rd,"\tau",swap] & & \AAA_1 \arrow[ld,"\tau"] \\
			& C &,
		\end{tikzcd}
		\] 
		and where $p_2$ denotes the projection on the first factor.
		\item With respect to either projection, the automorphism group $\Autz(X)$ is relatively maximal and the pairs $(X,p_1)$ and $(X,p_2)$
		are superstiff. Moreover, $\Autz(X)$ is a maximal connected algebraic subgroup of $\Bir(X)$.
	\end{enumerate}
\end{proposition}

\begin{proof}
		(1): By Lemma \ref{SarkisovIIIandIV}, there is no Sarkisov diagram of type III starting from the $\PP^1$-bundle $p_1$, and there exists a unique one of type IV, which is the one given in the statement. Assume that there exists also an $\Autz(X)$-Sarkisov diagram of type I or II. This latter is entirely determined by the $2$-ray game, which in turn is determined by the datum of the blowup of an $\Autz(X)$-orbit $\omega$. By Lemmas \ref{geometryofruledsurface} and \ref{autofiberproduct}, we have an isomorphism $\Autz(X)\simeq \Autz(\AAA_1)\times_{\Autz(C)} \Autz(\AAA_1)$, and the curve $\omega$ is mapped $4$-to-$1$ onto the $\Autz(\AAA_1)$-orbit $p_1(\omega)$. Then we conclude with Lemma \ref{lem:A1_no_Sarkisov}.
		
		(2): Let $\phi\colon X\dasharrow X'$ be an $\Autz(X)$-equivariant birational map. By Theorem \ref{Floris}, $\phi$ decomposes as product of $\Autz(X)$-Sarkisov links. As seen in \ref{A_1max.1}, the only $\Autz(X)$-Sarkisov diagram starting from $p_1$ is of type IV, and where the induced Sarkisov link is an automorphism. Therefore, $\Autz(X)$ is a maximal connected algebraic subgroup of $\Bir(X)$, it is also relatively maximal with respect to $p_1$ and the pair $(X,p_1)$ is superstiff. By symmetry, the same holds for the pair $(X,p_2)$. 
\end{proof}

The last fiber product to consider is $X=\PP(\O_C\oplus \O_C(D))\times_C \AAA_1$. As a $\PP^1$-bundle over $\AAA_1$, we prove in Lemma \ref{lem:SLA1max} and Proposition \ref{SLA1max} that $\Autz(X)$ is relatively maximal if and only if $D$ is not two-torsion.

\begin{lemma}\label{lem:SLA1max}
	Let $ S=\PP(\O_C\oplus \O_C(D))$ with $D\in \Pic^0(C)$ non-trivial. Let $X=S\times_C \AAA_1$, and let $p_1$ and $p_2$ be the projections onto $S$ and $\AAA_1$, respectively. Then the following hold:
	
	\begin{enumerate} 
		\item As a $\PP^1$-bundle over $\AAA_1$, we have $X\cong \PP(\O_{\AAA_1}\oplus \O_{\AAA_1}(\tau^*(D)))$. We denote by $S_0$ and $S_1$ the disjoint sections corresponding respectively to the line subbundles $\O_{\AAA_1}(\tau^*(D)) \subset \O_{\AAA_1}\oplus \O_{\AAA_1}(\tau^*(D))$ and $\O_{\AAA_1}\subset \O_{\AAA_1}\oplus \O_{\AAA_1}(\tau^*(D))$.
		\item\label{SLA1max.1} There exists a unique $\Autz(X)$-equivariant Sarkisov diagram starting from $p_1$, that is of type IV and is the following one:
		\[
		\begin{tikzcd}[column sep=2em,row sep = 3em]
			X \arrow[d,"p_1",swap] \arrow[rr,equal] & & X  \arrow[d,"p_2"]  \\
			S \arrow[rd,"\tau'",swap] & & \AAA_1 \arrow[ld,"\tau"] \\
			& C &,
		\end{tikzcd}
		\] 
		where $\tau'\colon S\to C$ is the structure morphism.
		Moreover, $\Autz(X)$ is relatively maximal with respect to $p_1\colon X\to S$, and the pair $(X,p_1)$ is superstiff. 
		\item\label{SLA1max.2} The only $\Autz(X)$-equivariant Sarkisov diagrams starting from $p_2$ are either the inverse of the diagram given in \ref{SLA1max.1}, or the following one of type $II$:
		\[
		\begin{tikzcd}[column sep=4em,row sep = 3em]
			Y\arrow[d,"\eta",swap] \arrow[r,equal] & Y \arrow[d,"\kappa"] \\
			X \arrow[d,"p_2",swap]  & X_4  \arrow[d,"\pi'"]  \\
			\AAA_1 \arrow[r,equal]  & \AAA_1,
		\end{tikzcd}
		\] 
		where $\eta$ is the blowup of an $\Autz(X)$-orbit $\widetilde{\omega}$ of dimension one lying in $S_0$ or in $S_1$, and $\kappa$ is the contraction of the strict transform of $p_2^{-1}(p_2(\widetilde{\omega}))$. Two cases arise:
		\bigskip
		\begin{enumerate}
			\item If $\widetilde{\omega}\subset S_0$, then $X_4$ is $\AAA_1$-isomorphic to $\PP(\O_{\AAA_1}\oplus \O_{\AAA_1}(4\sigma + \tau^*(D-D_\sigma)))$. 
			\item If $\widetilde{\omega}\subset S_1$, then $X_4$ is $\AAA_1$-isomorphic to $\PP(\O_{\AAA_1}\oplus \O_{\AAA_1}(-4\sigma + \tau^*(D+D_\sigma)))$. 
		\end{enumerate}
		\bigskip
		In both cases, if $D$ is a two-torsion divisor, then $\Autz(X)$ is not relatively maximal. If $D$ is not two-torsion, then we have $\phi \Autz(X)\phi^{-1} = \Autz(X_4)$, where $\phi = \kappa\eta^{-1}$ is the induced $\Autz(X)$-equivariant Sarkisov link.
	\end{enumerate}
\end{lemma}

\begin{proof}
	(1): As in Lemma \ref{fiberproduct}, we can choose the trivializations of $\tau p_2$ such that the transition maps equal:
	\[
	\begin{array}{cccc}\label{transition_FF0}\tag{$\clubsuit$}
		\phi_{ij}\colon & U_j \times \FF_0 & \dashrightarrow & U_i \times \FF_0 \\
		& (x,[y_0:y_1 \ ; z_0:z_1]) & \longmapsto & \left(x,\left[y_0:\lambda_{ij}(x)y_1 \ ; s_{ij}(x)\cdot \begin{pmatrix} z_0 \\ z_1 \end{pmatrix}\right]\right),
	\end{array}
	\]
	where $\lambda_{ij}\in \O_C(U_{ij})^*$ are the cocyles of the line bundle $\O_C(D)$ and $s_{ij}\in \GL_2(O_C(U_{ij}))$ are the transition matrices of $\AAA_1$. In particular, $X$ is $\AAA_1$-isomorphic to $\PP(\O_{\AAA_1}\oplus \O_{\AAA_1}(\tau^*(D)))$. 
	
	(2): The proof is similar to the one of Lemma \ref{A_1A_1max} \ref{A_1max.1}. By Lemmas \ref{geometryofruledsurface} and \ref{autofiberproduct}, we have an isomorphism $\Autz(X)\simeq \Autz(S) \times_{\Autz(C)} \Autz(\AAA_1)$ and the morphisms $\tau_*\colon \Autz(\AAA_1)\to \Autz(C)$ and $\tau'_*\colon \Autz(S)\to \Autz(C)$ are both surjective. This implies that $ (p_1)_*\colon \Autz(X)\to \Autz(S)$ is also surjective. Moreover, by Lemma \ref{geometryofruledsurface}, $S$ admits exactly two minimal sections $s_0$ and $s_1$, which are $\Autz(S)$-invariant and such that $(p_1)^{-1}(s_0)=S_0$ and $(p_1)^{-1}(s_1)=S_1$. The complement of $s_0 \cup s_1$ is an $\Autz(S)$-orbit of dimension two. Therefore, the $\Autz(X)$-orbits of minimal dimension are one-dimensional: they are precisely of the form $S_0\cap p_2^{-1}(\omega)$ and $S_1\cap p_2^{-1}(\omega)$, where $\omega$ is an $\Autz(\AAA_1)$-orbit of dimension one. 
 	Every $\Autz(X)$-Sarkisov diagram of type I or II is entirely determined by the blowup of such an $\Autz(X)$-orbit, which is a smooth curve mapped $4$-to-$1$ onto $s_0$ or $s_1$. By Lemma \ref{lem:A1_no_Sarkisov}, there is no $\Autz(X)$-equivariant Sarkisov diagram of type I and II starting from $p_1\colon X \to S$, so $\Autz(X)$ is relatively maximal with respect to $p_1$ and the pair $(X,p_1)$ is superstiff. The $\Autz(X)$-Sarkisov diagram of type IV given in the statement follows from Lemma \ref{SarkisovIIIandIV} \ref{SarkisovIIIandIV.3}.
	
	(3): Now we show that every $\Autz(X)$-equivariant Sarkisov diagram of type I and II is actually of type II and is the one given in statement. 
	As in \ref{SLA1max.1}, we can only blowup a one-dimensional $\Autz(X)$-orbit $\widetilde{\omega}$ of the form $S_0\cap p_2^{-1}(\omega)$ or $S_1\cap p_2^{-1}(\omega)$, where $\omega$ is an $\Autz(\AAA_1)$-orbit.  
	
	We work with the trivializations $\phi_{ij}$ above (cf. \ref{transition_FF0}). Under this choice of trivializations, we have that $S_0 = \{y_0=0\}$ and $S_1 = \{y_1=0\}$.
	Let $\eta\colon Y\to X$ be the blowup of an $\Autz(X)$-orbit $\widetilde{\omega}\subset S_0$. The contraction of the strict transform of $p_2^{-1}(\omega)$ gives a birational morphism $\kappa \colon Y\to X_4$. This induces an $\Autz(X)$-equivariant birational map $\psi\colon X\dashrightarrow X_4$ locally given by
	\[
	\begin{array}{cccc}
		\psi_i \colon & U_i \times \FF_0 & \dashrightarrow & U_i \times \FF_4 \\
		& (x,[y_0:y_1\ ;z_0:z_1]) & \longmapsto & (x,[y_0:\delta_i(x,z_0,z_1)y_1\ ;z_0:z_1]),
	\end{array}
	\]
	where $\delta_i\in \O_C(U_{i})[z_0,z_1]_4$ satisfies $(\mathrm{div}(\delta_i))_{|U_i \times \PP^1} = \omega_{|U_i\times \PP^1}$. The quantities 
	$$\delta_i\left(x,s_{ij}(x)\cdot 
	\begin{pmatrix}
		z_0\\z_1
	\end{pmatrix}\right) \delta_j(x,z_0,z_1)^{-1}$$ are the cocyles of the line bundle $\O_{\AAA_1}(\omega)$, which is isomorphic to $\O_{\AAA_1}(4\sigma - \tau^*(D_\sigma))$ by Lemma \ref{A_1quotient}. Computing $\psi_i\phi_{ij}\psi_j^{-1}$, that are the transition maps of $\tau\pi'\colon X_4 \to C$, we get that $X_4$ is $\AAA_1$-isomorphic to $\PP(\O_{\AAA_1}\oplus \O_{\AAA_1}(4\sigma +\tau^*(D-D_\sigma)))$. 
	
	Since $m_2^*(D_\sigma)\sim 4D_0$ by Lemma \ref{A_1quotient} \ref{A_1quotient.4}, we get that $m_2^*(D- D_\sigma) +4D_0= m_2^*(D) $. The morphism $m_2^*\colon \Pic^0(C)\to \Pic^0(C)$ is the multiplication by two on the elliptic curve $\Pic^0(C)$. So, if $D\in \Pic^0(C)$ is a two-torsion divisor, it follows by Lemma \ref{pullbackA1.trivialinv} that $\Autz(X)$ is not relatively maximal. 
	Else, $D\in \Pic^0(C)$ is not two-torsion, i.e., $D-D_\sigma$ is a non-trivial $2$-divisor of degree $-2$. Then by Lemma \ref{SarkisovoverA1}, $\pi'\colon X_4\to \AAA_1$ is decomposable. The base locus of $\psi^{-1}$ is contained in the section $\{y_1=0\}$, which is $\Autz(X_4)$-invariant, and is an $\Autz(X_4)$-orbit. Thus $\psi \Autz(X) \psi^{-1} = \Autz(X_4)$.

	Similarly, if we blowup an $\Autz(X)$-orbit $\widetilde{\omega}\subset S_1$ and contract the strict transform of $p_2^{-1}(\omega)$, we get an $\Autz(X)$-equivariant birational map $\psi\colon X\dashrightarrow X_{-4}$ locally given by
	\[
	\begin{array}{cccc}
		\psi_i \colon & U_i \times \FF_0 & \dashrightarrow & U_i \times \FF_{-4} \\
		& (x,[y_0:y_1\ ;z_0:z_1]) & \longmapsto & (x,[\delta_i(x,z_0,z_1)y_0:y_1\ ;z_0:z_1]).
	\end{array}
	\]
	Therefore, we get that $X_{-4}$ is $\AAA_1$-isomorphic to $\PP(\O_{\AAA_1}\oplus \O_{\AAA_1}(-4\sigma + \tau^*(D+D_\sigma)))$, which is $\AAA_1$-isomorphic to $\PP(\O_{\AAA_1}\oplus \O_{\AAA_1}(4\sigma - \tau^*(D+D_\sigma)))$. Mutatis mutandis, the rest of the proof works the same. We get that $m_2^*(-D-D_\sigma) + 4D_0 = -m_2^*(D)$, $\Autz(X)$ is not relatively maximal if $D$ is two-torsion, else $\psi \Autz(X) \psi^{-1} = \Autz(X_{-4})$. 
\end{proof}

\begin{proposition}\label{SLA1max}
	Assumptions as in Lemma $\ref{lem:SLA1max}$ and assume moreover than $D$ is not two-torsion. Then the following hold:
	\begin{enumerate}
		\item There exists an $\Autz(X)$-equivariant square birational map $X\dashrightarrow X'$, where $X' $ is a $\PP^1$-bundle over $\AAA_1$,
		if and only if there exists $n\in \mathbb{Z}$ such that
		\[
		X'\cong \PP(\O_{\AAA_1}\oplus \O_{\AAA_1}(4n\sigma + \tau^*(D-nD_\sigma))).
		\]
		Moreover, this square birational map conjugates $\Autz(X)$ and $\Autz(X')$.
		\item $\Autz(X)$ is relatively maximal with respect to $p_2\colon X\to \AAA_1$, but the pair $(X,p_2)$ is not stiff. 
	\end{enumerate}
\end{proposition}

\begin{proof}
	 Consider the $\PP^1$-bundle $p_2\colon X \to \AAA_1$. By Lemma \ref{lem:SLA1max}, the only $\Autz(X)$-equivariant square birational map starting from $X$ are of the form $\phi\colon X\dashrightarrow X_4$, where 
	 \[
	 X_4=\PP(\O_{\AAA_1}\oplus \O_{\AAA_1}(\pm 4\sigma + \tau^*(D \mp D_\sigma))),
	 \]
	 and $\phi \Autz(X)\phi^{-1} = \Autz(X_4)$. 
	
	Assume first that the base locus of $\phi$ is $\widetilde{\omega} \subset S_0$, then $X_4=\PP(\O_{\AAA_1}\oplus \O_{\AAA_1}(4\sigma + \tau^*(D-D_\sigma)))$. Since $D$ is not two-torsion, $D-D_\sigma$ is a non-trivial $2$-divisor of degree $-2$. By Lemma \ref{SarkisovoverA1}, we have exactly two choices of $\Autz(X_4)$-equivariant square birational maps: the second is the inverse of $\phi$, and the first one is an $\Autz(X_4)$-equivariant Sarkisov link $\phi_4\colon X_4\dashrightarrow X_8$, where $X_8=\PP(\O_{\AAA_1}\oplus \O_{\AAA_1}(8\sigma + \tau^*(D-2D_\sigma)))$, and $\phi_4 \Autz(X_4)\phi_4^{-1} = \Autz(X_8)$. We obtain again that $D-2D_\sigma$ is a non-trivial $2$-divisor of degree $-4$, so we can iterate the process. By induction, for each $n>0$, we obtain a new $\Autz(X_{4n})$-equivariant Sarkisov link $\phi_{4n}\colon X_{4n} \dashrightarrow X_{4(n+1)}$, where $X_{4n}=\PP(\O_{\AAA_1}\oplus \O_{\AAA_1}(4n\sigma + \tau^*(D -n D_\sigma)))$, and such that $\phi_{4n} \Autz(X_{4n})\phi_{4n}^{-1} = \Autz(X_{4(n+1)})$. 
	The case $n<0$ is identical, instead we choose $\widetilde{\omega}\subset S_1$. 
	
	At each step, those are all the possible $\Autz(X)$-equivariant square birational maps, so this proves (1). This also implies that $\Autz(X)$ is relatively maximal with respect to $p_2$ but that the pair $(X,p_2)$ is not stiff. 
\end{proof}

We now gather the results of this section to prove Proposition \ref{main_prop:A_1}.

\begin{proof}[Proof of Proposition \ref{main_prop:A_1}]
	Assume that $\Autz(X)$ is relatively maximal and $\tau \pi$ is an $\FF_b$-bundle with $b>0$. By Lemmas \ref{pullbackA1} and \ref{pullbackA1.trivialinv}, we have that $m_2^*(D)+bD_0$ is a non-trivial divisor of degree zero. Equivalently, $D$ is non-trivial 2-divisor (see Definition \ref{def:(m_2^*,b)}) and $b>0$ is an even integer. Then by Lemma \ref{SarkisovoverA1},
	we have
	\[
	X\cong \PP(\O_{\AAA_1} \oplus \O_{\AAA_1}(b\sigma + \tau^*(D))).
	\]
	Applying Lemma \ref{SarkisovoverA1} inductively, we obtain an $\Autz(X)$-equivariant square birational map $X\dashrightarrow X'$, where $\pi'\colon X\to \AAA_1$ is an $\FF_0$-bundle if $b$ is a multiple of $4$, and an $\FF_2$-bundle otherwise. Two cases arise: 
	\begin{enumerate}
		\item The case where $\tau\pi$ is an $\FF_0$-bundle is studied in Propositions \ref{A_0A_1notmax} and \ref{A_1A_1max}, Lemma \ref{lem:SLA1max} and Proposition \ref{SLA1max}. This gives the first two cases of Proposition \ref{main_prop:A_1}.
		\item If $\tau\pi$ is an $\FF_2$-bundle, the result follows from Proposition \ref{4n+2}. This gives the last case of Proposition \ref{main_prop:A_1}.
	\end{enumerate}
\end{proof}

\section{Automorphisms of $\PP^1$-bundles over $\AAA_0$}\label{SectionAAA0}

In this section, $C$ is an elliptic curve, $\tau\colon \AAA_0\to C$ is the unique indecomposable geometrically ruled surface admitting a unique section $\sigma$ of self-intersection zero, which is of minimal self-intersection, and $\O_{\AAA_0}(\sigma) \simeq \O_{\AAA_0}(1)$ (see e.g.\ \cite[V. Proposition 2.9]{Hartshorne}). We will show that if $X$ is a $\PP^1$-bundle over $\AAA_0$ such that $\Autz(X)$ is relatively maximal, then $\Autz(X)$ is conjugate to the automorphism group of an $\FF_0$-bundle. Our main proposition is:

\medskip

\begin{proposition}\label{main_prop:A0}
	Let $\pi\colon X\to \AAA_0$ be a $\PP^1$-bundle such that $\tau\pi$ is an $\FF_b$-bundle for some $b\in \mathbb{Z}$. Then $\Autz(X)$ is relatively maximal if and only if one of the following case occurs:
	\begin{enumerate}
		\item Case $b=0$. Then $X$ is isomorphic to one of the following $\PP^1$-bundles:
		\[
		\AAA_0 \times \PP^1,~\AAA_0 \times_C \AAA_1,~ \AAA_0 \times_C \PP(\O_C\oplus \O_C(D)),
		\]
		where $D\in \Pic^0(C)$ is non-trivial. The pair $(X,\pi)$ is superstiff in the first two cases, and not stiff in the last one.
		\item\label{main_prop:A0.2} Case $b\neq 0$. Then there exists $D\in \Pic^0(C)$ non-trivial such that
		\[
		X \cong \PP(\O_{\AAA_0}\oplus \O_{\AAA_0}(b\sigma + \tau^*(D) )),
		\]
		and there exists an $\Autz(X)$-equivariant square birational map $X\dashrightarrow \PP(\O_{\AAA_0} \oplus \O_{\AAA_0}(\tau^*(D)))$. The latter is isomorphic to the fiber product $\AAA_0 \times_C \PP(\O_C \oplus \O_C(D))$.
	\end{enumerate}
	Let $\pi' \colon X'\to T$ be a $\PP^1$-bundle over a geometrically ruled surface $T$.  Moreover, there exists an $\Autz(\AAA_0 \times_C \PP(\O_C\oplus \O_C(D)))$-equivariant square birational map
	\[
	\AAA_0 \times_C \PP(\O_C\oplus \O_C(D)) \dashrightarrow X'
	\]
	if and only if $T \cong \AAA_0$, and either $X' \cong \AAA_0 \times_C \PP(\O_C\oplus \O_C(D))$ or $X'\cong \PP(\O_{\AAA_0}\oplus \O_{\AAA_0}(b\sigma + \tau^*(D) ))$ as in case \ref{main_prop:A0.2}.
\end{proposition}

\medskip

If $b>0$, by virtue of Proposition $\ref{degnonnuldec}$, it suffices to consider the $\PP^1$-bundles with invariants $(\AAA_0,b,D)$, where $\deg(D)=0$. In what follows, we distinguish between the case where $D$ is trivial and the case where it is non-trivial of degree zero.

\subsection{$\PP^1$-bundles with invariants $(\AAA_0,b,0)$ with $b>0$}\label{section:(A_0,b,0)}

~\bigskip

Using explicit transition maps, we first show there exists a unique indecomposable $\PP^1$-bundle with invariants $(\AAA_0,b,0)$ with $b>0$.

\begin{lemma}\label{A0_trivial_representative}
	Let $b>0$, $p,q\in C$ distinct, and $\pi\colon X\to \AAA_0$ be a $\PP^1$-bundle with invariants $(\AAA_0,b,0)$. 
	Then the $\FF_b$-bundle $\tau\pi\colon X\to C$ is trivial over $C\setminus \{p\}$ and $C\setminus \{q\}$ and we can choose the transition maps as
	\[
	\begin{array}{cccc}
		\phi \colon & (C\setminus \{q\}) \times \FF_b & \dashrightarrow & (C\setminus \{p\}) \times \FF_b \\
		& (x,[y_0:y_1 \ ; z_0:z_1]) & \longmapsto & \left(x,\left[y_0:y_1 + a_0 \alpha(x)^{-1}z_1^{b}y_0\ ; z_0: \alpha(x)^{-1}z_0+ z_1\right]\right),
	\end{array}
	\]
	where $\alpha\in \kk(C)^*$ has a zero exactly of order one at $p$ and $q$, and $a_0\in \{0,1\}$. If $a_0=0$, then $X$ is $ \AAA_0$-isomorphic to $\PP(\O_{\AAA_0}\oplus \O_{\AAA_0}(b\sigma))$, else $a_0=1$ and $\pi$ is the unique indecomposable $\PP^1$-bundle with invariants $(\AAA_0,b,0)$.
\end{lemma}

\begin{proof}
	Let $\alpha\in \kk(C)^*$ such that $\mathrm{P}(\alpha^{-1})=p+q$, where $\mathrm{P}$ denotes the pole divisor. By Lemma \ref{transitionAtiyah} \ref{transitionAtiyah.2}, we can choose trivializations of $\tau$ such that $p\in U_j$, the section $\sigma$ is the zero section $\{z_0=0\}$ and the transition maps of $\tau$ equal
	\[
	\begin{array}{ccc}
		U_j\times \PP^1 & \dashrightarrow & U_i \times \PP^1 \\
		(x,[z_0:z_1]) & \longmapsto & (x,[z_0:\alpha^{-1}(x)z_0+z_1]).
	\end{array}
	\]
	Taking the trivializing open subsets smaller if necessary, we can assume that $p\notin U_i$ for $i\neq j$ and $q\notin U_j$, and these opens trivialize also the $\FF_b$-bundle $\tau\pi \colon X\to C$. We write the transition maps of $\tau\pi$ as
	\[
	\begin{array}{cccc}
		& U_j\times \FF_b & \dashrightarrow & U_i\times \FF_b \\
		&(x,[y_0:y_1 \ ;z_0:z_1]) & \longmapsto &  \left(x,\left[y_0:y_1 +p_{ij}(x,z_0,z_1)y_0 \ ;z_0:\alpha^{-1}(x)z_0+ z_1\right]\right),
	\end{array}
	\]
	where $p_{ij}\in \O_C(U_{ij})[z_0,z_1]_b$. Using Lemma \ref{Sisomorph} with $\mu_i = \mu_j =1$ and
	\[
	\begin{array}{ccccc}
		& p_{ij}(x,z_0,z_1)& = & \sum_{k=0}^{b} c_{ij,k}(x) z_0^kz_1^{b-k}\in \O_C(U_{ij})[z_0,z_1]_b,& \\
		& q_j(x,z_0,z_1) & = & \sum_{k=0}^{b} c_{j,k}(x) z_0^kz_1^{b-k}\in \O_C(U_j)[z_0,z_1]_b, &\\
		& q_i(x,z_0,z_1) & = & \sum_{k=0}^{b} c_{i,k}(x) z_0^kz_1^{b-k}\in  \O_C(U_i)[z_0,z_1]_b,&
	\end{array}
	\]
	we can replace $p_{ij}(x,z_0,z_1)$ by 
	\begin{equation}\label{A0_trivial_=}
			p'_{ij}(x,z_0,z_1) = \sum_{k=0}^{b} c_{ij,k}(x) z_0^kz_1^{b-k} - \sum_{k=0}^{b} c_{j,k}(x) z_0^kz_1^{b-k} + \sum_{\substack{s,t\in \{0,\cdots, b\}, \\ s+t \leq b}} \binom{b-s}{t}c_{i,s}(x) \alpha^{-t}(x)z_0^{s+t}z_1^{b-(s+t)},\tag{$\star$}
	\end{equation}
	without changing the $\AAA_0$-isomorphism class of $X$. 
	
	For every integer $n\geq 2$ and $z\in C$, there exists $f_{n,z}\in \O_C(C\setminus \{z\})$ such that $\mathrm{P}(f_{n,z}) = nz$; if moreover $z\neq p$ (resp.\ $z\neq q$), there exists $f_{1,\{z,p\}} \in \O_C(C\setminus \{z,p\})$ (resp.\ $f_{1,\{z,q\}}\in \O_C(C\setminus \{z,q\})$) such that $\mathrm{P}(f_{1,\{z,p\}}) = z+p$ (resp.\ $\mathrm{P}(f_{1,\{z,q\}}) = z+q$) (see Lemma \ref{rationalfunction}).
	
	We proceed by iteration over $k\geq 0$. For each $k$, we proceed as follows:
	\begin{enumerate}
		\item[(i)] Assume that $c_{ij,k}$ has a pole of order $n\geq 1$ at a point $z\neq p,q$. If $z \in U_j\setminus U_i $, we can use the equality (\ref{A0_trivial_=}) with $c_{i,k} = d_k f_{1,\{z,p\}}^n$, where $d_k\in \kk^*$, we get new coefficients $c_{ij,l}$ for $l\geq k$ but we do not change the coefficients $c_{ij,l}$ for $l<k$. Moreover, choosing a suitable $d_k\in \kk^*$, we also decrease the order the pole of $c_{ij,k}$ at $z$. If $z \in U_i \setminus U_j$, we choose instead $c_{j,k} =  d_k f_{1,\{z,q\}}^n$ with a suitable $d_k\in \kk^*$ in (\ref{A0_trivial_=}) to decrease the order of the pole at $z$. Repeating this process and after finitely many steps, we can assume $c_{ij,k}$ is regular at $z$ and its poles are only at $p$ and $q$. 
		\item[(ii)] Now assume that $c_{ij,k}$ has a pole of order $n\geq 2$ at $p\in U_j$ (resp.\ $q\in U_i$ for some $i$), then we can also decrease its order by taking $c_{i,k} = d_k f_{n,p} \in \O_C(U_i)$ (resp.\ $c_{j,k} = d_k f_{n,q}\in \O_C(U_j)$) with a suitable $d_k\in \kk^*$. Here as well, we get new coefficients $c_{ij,l}$ for $l\geq k$ but we do not change the coefficients $c_{ij,l}$ for $l<k$. After finitely many steps, we can now assume that $c_{ij,k}$ has a pole of order at most one at $p$ and $q$ and is regular on $C\setminus \{p,q\}$. Since $h^0(C,p+q) = 2$, this implies that $c_{ij,k} = a_k \alpha^{-1} +b_k$ for some $a_k,b_k\in \kk$. Using once more (\ref{A0_trivial_=}) with $c_{j,k} = b_k$, we can assume now that $c_{ij,k} = a_k \alpha^{-1}$. 
		\item[(iii)] If $k\geq 1$ and $a_k\neq 0$, we use once more (\ref{A0_trivial_=}) with $c_{j,k-1} = c_{i,k-1} = -\frac{a_k}{b-(k-1)} \in \kk^*$; this replaces $c_{ij,k}$ by zero and we do not change $c_{ij,l}$ for $l<k$ and we get new coefficients $c_{ij,l}$ for $l>k$. 
	\end{enumerate}

	After the iteration, we can assume now that $c_{ij,0} = a_0 \alpha^{-1}$ for some $a_0\in \kk$ and $c_{ij,k} = 0$ for $k>0$. Finally, if $a_0\neq 0$, we use once more Lemma \ref{Sisomorph} with $\mu_i = \mu_j = 1/a_0$ and $q_i = q_j =0$, this replaces $a_0$ by $1$ without changing the $\AAA_0$-isomorphism class of $X$. Since $\alpha^{-1}$ has poles exactly at $p\in U_j$ and $q\in U_i$, this implies that $\tau\pi\colon X\to C$ is trivial over $U_j = C\setminus \{q\}$ and $U_i = C\setminus \{p\}$, and we can write the transition maps as stated. If $a_0 = 0$, then $\pi$ is decomposable and $X$ is $\AAA_0$-isomorphic to $\PP(\O_{\AAA_0}\oplus \O_{\AAA_0}(b\sigma))$ by Lemma \ref{FFb-bundle} \ref{FFb-decomposable}. In this case, the restriction of $\pi\colon X\to \AAA_0$ on $\sigma = \{z_0=0\}$ gives a $\PP^1$-bundle which is trivial. Else, $a_0=1$, the restriction of $\pi$ on $\sigma=\{z_0=0\}$ is isomorphic to $\AAA_0\to \sigma$ and a fortiori, the $\PP^1$-bundle $\pi\colon X\to \AAA_0$ is indecomposable.
\end{proof}

By the previous lemma, a $\PP^1$-bundle $X\to \AAA_0$ with invariants $(\AAA_0,b,0)$ is either $\PP(\O_{\AAA_0}\oplus \O_{\AAA_0}(b\sigma))$, or an indecomposable $\PP^1$-bundle which is uniquely determined by those invariants. In the following, we show that in both cases, $\Autz(X)$ is not relatively maximal.

\begin{lemma}\label{equivarianceoverA0trivial}
	Let $b>0$ and $X= \PP(\O_{\AAA_0}\oplus \O_{\AAA_0}(b\sigma))$. 
	Then the following hold:
	\begin{enumerate}
		\item\label{equivarianceoverA0trivial.1} There exists an $\Autz(X)$-invariant curve $D\subset X$, whose blowup followed by the contraction of the strict transform of $\pi^{-1}(\sigma)$ yields an $\Autz(X)$-equivariant square birational map $$X \dashrightarrow \PP(\O_{\AAA_0}\oplus \O_{\AAA_0}((b-1)\sigma)).$$
		\item\label{equivarianceoverA0trivial.2} $\Autz(X)$ is not relatively maximal and is conjugate to a proper connected algebraic subgroup of $\Autz(\PP^1 \times \AAA_0)$.
	\end{enumerate}
\end{lemma}

\begin{proof}
	(1): Let $S\subset X$ be the section of $\pi\colon X\to \AAA_0$ corresponding to the line subbundle $\O_{\AAA_0}\subset \O_{\AAA_0}\oplus \O_{\AAA_0}(b\sigma)$ and $D=\pi^{-1}(\sigma)\cap S$. We choose trivializations of $\tau\pi$ such that the transition maps are:
	\[
	\begin{array}{cccc}
		\phi \colon & (C\setminus \{q\}) \times \FF_b & \dashrightarrow & (C\setminus \{p\}) \times \FF_b \\
		& (x,[y_0:y_1 \ ; z_0:z_1]) & \longmapsto & \left(x,\left[y_0:y_1 \ ; z_0: \alpha(x)^{-1}z_0+ z_1\right]\right).
	\end{array}
	\]
	Under this choice of trivializations, we have $D=\{y_1=z_0=0\}$. 
	
	Let $f \in \mathrm{ker}(\pi_*)$. Via the trivializations of the $\FF_b$-bundle $\tau\pi\colon X\to C$, the automorphism $f$ induces the automorphisms
	\[
	\begin{array}{cccc}
		f_q \colon & (C\setminus \{q\})\times \FF_b & \longrightarrow & (C\setminus \{q\})\times \FF_b \\
		& (x,[y_0:y_1\ ; z_0:z_1]) & \longmapsto & \left(x,\left[y_0:\mu_q(x)y_1 + \left(\sum_{k=0}^{b} a_{q,k}(x)z_0^kz_1^{b-k}\right)y_0\ ; z_0:z_1\right]\right), \\
		f_p \colon & (C\setminus \{p\})\times \FF_b & \longrightarrow & (C\setminus \{p\})\times \FF_b \\
		& (x,[y_0:y_1\ ; z_0:z_1]) & \longmapsto & \left(x,\left[y_0:\mu_p(x)y_1 + \left(\sum_{k=0}^{b} a_{p,k}(x)z_0^kz_1^{b-k}\right)y_0\ ; z_0:z_1\right]\right) ,
	\end{array}
	\]
	where $\mu_q \in \O_C(C\setminus \{q\})^*$, $\mu_p \in \O_C(C\setminus \{p\})^*$, $a_{q,k} \in \O_C(C\setminus \{q\})$ and $a_{p,k} \in \O_C(C\setminus \{p\})$ for every $k\in \{0,\cdots,b\}$. The condition $\phi f_q = f_p \phi$ implies, by identifying the coefficients in front of $z_1^b$ and $z_0z_1^{b-1}$, the following equalities:
	\begin{align}
		a_{q,0}(x)  & =   a_{p,0}(x), \label{eq0}\\
		a_{q,1}(x)  & =  b\alpha^{-1}(x)a_{p,0}(x)+a_{p,1}(x).\label{eq1}
	\end{align}
	The equality (\ref{eq0}) implies that $a_{p,0} = a_{q,0} \in \kk$. The rational function $\alpha^{-1}$ has poles exactly at $p$ and $q$, while $a_{p,1}$ is a constant or a rational function having poles of order at least two exactly at $q$. Using (\ref{eq1}), we obtain $a_{p,0}=0$ and this implies that the curve $D$ is $\mathrm{ker}(\pi_*)$-invariant. Moreover, the morphism $\pi_*$ is surjective by Proposition \ref{imagepi*}, and since $\pi$ is decomposable, we also have $\mathbb{G}_m\subset \mathrm{ker}(\pi_*)$. This implies that $D$ is $\Autz(X)$-invariant. 
	
	The blowup of the curve $D$ followed by the contraction of the strict transform of $\pi^{-1}(\sigma)$ yields an $\Autz(X)$-equivariant birational map $\psi\colon X\dashrightarrow X_{b-1}$ locally given by
	\[
	\begin{array}{ccc}
		U\times \FF_b & \dashrightarrow & U\times \FF_{b-1} \\
		(x,[y_0:y_1 \ ; z_0:z_1]) & \longmapsto & (x,[z_0y_0:y_1 \ ; z_0:z_1]),
	\end{array}
	\]
	where $U=C\setminus \{q\}$ or $C\setminus \{p\}$. We get that the transition maps of the $\FF_{b-1}$-bundle $\tau\pi_{b-1}\colon X_{b-1}\to C$ are
	\[
	\begin{array}{cccc}
		\phi_{b-1}\colon & (C\setminus \{q\})\times \FF_{b-1} & \dashrightarrow & (C\setminus \{p\})\times \FF_{b-1} \\
		& (x,[y_0:y_1 \ ; z_0:z_1]) & \longmapsto & (x,[y_0:y_1 \ ; z_0:\alpha^{-1}(x)z_0+z_1]).
	\end{array}
	\]
	Then $X_{b-1}=\PP(\O_{\AAA_0} \oplus \O_{\AAA_0}((b-1)\sigma))$ by Lemmas \ref{fiberproduct} and \ref{A0_trivial_representative}.
	
	(2): We repeat the process to decrease $b$. Eventually, we obtain an $\Autz(X)$-equivariant birational map $X\dashrightarrow  \PP(\O_{\AAA_0}\oplus \O_{\AAA_0})$, and the latter is isomorphic to $\PP^1 \times \AAA_0$. By Proposition \ref{FF0viaCxPP1}, the pair $(\PP^1\times \AAA_0,p_2\colon \PP^1\times \AAA_0 \to \AAA_0)$ is superstiff, so $\Autz(X)$ is conjugate to a proper subgroup of $\Autz(\PP^1 \times \AAA_0)$; thus, $\Autz(X)$ is not relatively maximal.
\end{proof}

\begin{notation}
	For $b>0$, we denote by $X_{(\AAA_0,b,0)}$ the unique indecomposable $\PP^1$-bundle with invariants $(\AAA_0,b,0)$.
\end{notation}

\begin{lemma}\label{AAAnotmax}
	For every $b>0$, there exists an $\Autz(X_{(\AAA_0,b,0)})$-equivariant square birational map $X_{(\AAA_0,b,0)} \dashrightarrow \PP(\O_{\AAA_0}\oplus \O_{\AAA_0}((b+1)\sigma))$ and $\Autz(X_{(\AAA_0,b,0)})$ is not relatively maximal.
\end{lemma}

\begin{proof}
	We denote by $(\AAA_0)_b\subset X_{(\AAA_0,b,0)}$ the surface obtained by taking the union of the $(-b)$-sections along the fibers of $\tau\pi\colon X_{(\AAA_0,b,0)}\to C$. We choose the trivializations of $\tau\pi\colon X_{(\AAA_0,b,0)}\to C$ given in Lemma \ref{A0_trivial_representative}, such that the transition maps of $\tau \pi$ equal
	\[
	\begin{array}{cccc}
		\phi_{pq} \colon & (C\setminus \{q\}) \times \FF_b & \dashrightarrow & (C\setminus \{p\}) \times \FF_b \\
		& (x,[y_0:y_1 \ ; z_0:z_1]) & \longmapsto & \left(x,\left[y_0:y_1 + \alpha(x)^{-1}z_1^{b}y_0\ ; z_0: \alpha(x)^{-1}z_0+ z_1\right]\right),
	\end{array}
	\]
	and under this choice, we have $(\AAA_0)_b = \{y_0=0\}$.
	Since the unique minimal section $\sigma$ of $\tau\colon \AAA_0\to C$ is $\Autz(\AAA_0)$-invariant, it follows by Blanchard's lemma that the curve $\pi^{-1}(\sigma)\cap (\AAA_0)_b = \{y_0=z_0=0\}$ is $\Autz(X_{(\AAA_0,b,0)})$-invariant. Its blowup followed by the contraction of the strict transform of $\pi^{-1}(\sigma)$ yields an $\Autz(X_{(\AAA_0,b,0)})$-equivariant birational map $X_{(\AAA_0,b,0)} \dashrightarrow X_{b+1}$, locally given by 
	\[
	\begin{array}{ccc}
		U \times \FF_b & \dashrightarrow & U\times \FF_{b+1} \\
		(x,[y_0:y_1 \ ; z_0:z_1]) & \longmapsto & (x,[y_0:z_0y_1 \ ; z_0:z_1]),
	\end{array}
	\]
	where $U=C\setminus \{p\}$ or $C\setminus \{q\}$.
	We obtain that the transition maps of $\pi_{b+1}\colon X_{b+1}\to \AAA_0$ equal
	\[
	\begin{array}{cccc}
		& (C\setminus \{q\}) \times \FF_{b+1} & \dashrightarrow & (C\setminus \{p\}) \times \FF_{b+1} \\
		& (x,[y_0:y_1 \ ; z_0:z_1]) & \longmapsto & \left(x,\left[y_0:y_1 +  \alpha(x)^{-1}z_0z_1^{b}y_0\ ; z_0: \alpha(x)^{-1}z_0+ z_1\right]\right).
	\end{array}
	\]
	By Lemma \ref{FFb-bundle}, the $\PP^1$-bundle $\pi_{b+1}$ has invariants $(\AAA_0,b+1,0)$. 
	
	Next we show that $X_{b+1}$ is $\AAA_0$-isomorphic to the decomposable $\PP^1$-bundle $\PP(\O_{\AAA_0}\oplus \O_{\AAA_0}((b+1)\sigma))$. We use the same notations as in the proof of Lemma \ref{A0_trivial_representative}, with $U_j = C\setminus \{q\}$, $U_i = C\setminus \{p\}$, $c_{ij,k} = \alpha^{-1}$ if $k=1$ and else, $c_{ij,k}=0$. Using the equality (\ref{A0_trivial_=}) in the proof of Lemma \ref{A0_trivial_representative}, without changing the $\AAA_0$-isomorphism class of $X_{b+1}$, we can replace $p_{ij}$ by:
	\[
	 \sum_{k=0}^{b+1} c_{ij,k}(x) z_0^kz_1^{(b+1)-k} - \sum_{k=0}^{b+1} c_{j,k}(x) z_0^kz_1^{(b+1)-k} + \sum_{\substack{s,t\in \{0,\cdots, b+1\}, \\ s+t \leq b+1}} \binom{(b+1)-s}{t}c_{i,s}(x) \alpha^{-t}(x)z_0^{s+t}z_1^{(b+1)-(s+t)}.
	\]
	Applying the same iteration as in the proof of Lemma \ref{A0_trivial_representative}, we replace all the $c_{ij}$ by zero.
\end{proof}

\begin{proposition}\label{prop:(A0,b,0)}
	Let $b>0$ and $\pi\colon X\to \AAA_0$ be a $\PP^1$-bundle with invariants $(\AAA_0,b,0)$. Then $\Autz(X)$ is not relatively maximal.
\end{proposition}

\begin{proof}
	This is a direct consequence of the lemmas of this section.
\end{proof}

\subsection{$\PP^1$-bundles with invariants $(\AAA_0,b,D)$ with $b>0$ and $D\in \Pic^0(C)$ non-trivial}

~\bigskip

A crucial element in the proof of the following lemma is that $\mathrm{Ext}^1(\O_C,\O_C(D)) = 0$. This will guarantee us that the $\PP^1$-bundle $\pi_{|\pi^{-1}(\sigma)}\colon X_{|\pi^{-1}(\sigma)}\to \sigma$ admits two disjoint invariant sections, whose blowups determine the only two $\Autz(X)$-equivariant Sarkisov diagrams of type II starting from $\pi$. This also allows us to show that $\pi$ is the decomposable $\PP(\O_{\AAA_0} \oplus \O_{\AAA_0}(b\sigma+\tau^*(D)))$, whose automorphism group is conjugate to $\Autz(\PP(\O_{\AAA_0} \oplus \O_{\AAA_0}(\tau^*(D))))$.

\begin{lemma}\label{A0_nontrivial}
	Let $\pi\colon X\to \AAA_0$ be a $\PP^1$-bundle with invariants $(\AAA_0,b,D)$ where $b>0$ and $D\in \Pic^0(C)\setminus \{0\}$. We denote by $(\AAA_0)_b\subset X$ the surface spanned by the $(-b)$-sections along the fibers of $\tau\pi$. Then the following hold:
	\begin{enumerate}
		\item\label{A0_nontrivial.1} There exists an $\Autz(X)$-invariant curve $l_0^X= \pi^{-1}(\sigma)\cap (\AAA_0)_b$, whose blowup followed by the contraction of the strict transform of $\pi^{-1}(\sigma)$ induces an $\Autz(X)$-equivariant square birational map $\phi_{b,b+1}$ and the following commutative diagram:
		\[
		\begin{tikzcd}
			X \arrow[rr,dashed,"\phi_{b,b+1}"] \arrow[rd,"\pi" swap] && X_{b+1}\arrow[ld,"\pi_{b+1}"] \\
			& \AAA_0 & ,
		\end{tikzcd}
		\]
		where $\pi_{b+1}\colon X_{b+1}\to \AAA_0$ is a $\PP^1$-bundle with invariants $(\AAA_0,b+1,D)$.
		\item\label{A0_nontrivial.2} There exists an $\Autz(X)$-invariant curve $l_1^X\subset \pi^{-1}(\sigma)$ disjoint from $l_0^X$, whose blowup followed by the contraction of the strict transform of $\pi^{-1}(\sigma)$ induces an $\Autz(X)$-equivariant square birational map $\psi_{b,b-1}$ and the following commutative diagram:
		\[
		\begin{tikzcd}
			X_{b-1} \arrow[rd,"\pi_{b-1}" swap] && X\arrow[ld,"\pi"] \arrow[ll,dashed,"\psi_{b,b-1}" swap] \\
			& \AAA_0 & ,
		\end{tikzcd}
		\]
		where $\pi_{b-1}\colon X_{b-1}\to \AAA_0$ is a $\PP^1$-bundle with invariants $(\AAA_0,b-1,D)$ if $b\geq 2$; and $X_0$ is the $\PP^1$-bundle $\PP(\O_{\AAA_0} \oplus \O_{\AAA_0}(\tau^*(D)))$.
		\item\label{A0_nontrivial.3} The $\PP^1$-bundle $\pi\colon X \to \AAA_0$ is isomorphic to the decomposable $\PP^1$-bundle \[ X\cong \PP(\O_{\AAA_0} \oplus \O_{\AAA_0}(b\sigma +\tau^*(D)))\]
		and $ \Autz(X)$ is conjugate to $\Autz(\PP(\O_{\AAA_0} \oplus \O_{\AAA_0}(\tau^*(D))))$ via a square birational map.
	\end{enumerate}
	Moreover, the only $\Autz(X)$-equivariant Sarkisov diagrams of type II starting from $\pi\colon X\to \AAA_0$ are those given in \ref{A0_nontrivial.1} and \ref{A0_nontrivial.2}.
\end{lemma}

\begin{proof}	
		By Lemma \ref{FFb-bundle}, we can choose the trivializations of $\tau\pi$ such that the transition maps equal
		\[
		\begin{array}{cccc}
			\theta_{ij}\colon &U_j\times \FF_b & \dasharrow  &U_i\times \FF_b \\
			&(x,[y_0,y_1\ ;z_0,z_1]) & \longmapsto &(x,[y_0:\lambda_{ij}(x) y_1 + p_{ij}(x,z_0,z_1)y_0\ ; z_0:\alpha_{ij}(x)z_0+z_1]), 
		\end{array}
		\]
		where $\lambda_{ij}\in O_C(U_{ij})^*$ denotes the cocyles of the line bundle $\O_C(D)$, $\alpha_{ij}\in \O_C(U_{ij})$ and $p_{ij}\in \O_C(U_{ij})[z_0,z_1]_b$. Moreover, the lower triangular matrices
		$\begin{pmatrix}
			1 & 0 \\
			\alpha_{ij} & 1
		\end{pmatrix}$
		are the transition matrices of $\AAA_0$, by Lemma \ref{transitionAtiyah}. Under this choice of trivializations, $\sigma$ is the zero section $\{z_0=0\}\subset \AAA_0$.
		
		Restricting the transition maps $\theta_{ij}$ on the surface $\pi^{-1}(\sigma) = \{z_0=0\}$, we get that the transition maps of the $\PP^1$-bundle $\pi_{|\pi^{-1}(\sigma)} \colon \pi^{-1}(\sigma) \to \sigma$ equals
		\[
		\begin{array}{cccc}
			(\theta_{ij})_{|\pi^{-1}(\sigma)} \colon & U_j\times \PP^1 & \dashrightarrow & U_i\times \PP^1 \\
			& (x,[y_0:y_1]) & \longmapsto & (x,[y_0:\lambda_{ij}(x) y_1 + p_{ij}(x,0,1)y_0]).
		\end{array}
		\]
		Therefore, there exists a rank-$2$ vector bundle $\E_\sigma$ over $\sigma$ such that $\PP(\E_\sigma) \simeq \pi^{-1}(\sigma)$ and $\PP(\E_\sigma)$ has a section given by $\{y_0=0\}$ which corresponds to the line subbundle $\O_\sigma(D)\subset \E_\sigma$. The quotient $\E_\sigma/\O_\sigma(D)$ is trivial; hence, $\E_\sigma$ fits into a short exact sequence
		$$
		0 \to \O_\sigma(D) \to \E_\sigma \to \O_\sigma \to 0. 
		$$
		Since $D$ is non-trivial of degree zero, we have that $\mathrm{Ext}^1(\O_\sigma,\O_\sigma(D))=0$, i.e., $\E_\sigma \simeq \O_\sigma \oplus \O_\sigma(D)$ and $\pi^{-1}(\sigma) \simeq \PP(\O_\sigma \oplus \O_\sigma(D))$. The section $\sigma$ is isomorphic to the elliptic curve $C$, hence $\pi^{-1}(\sigma)$ admits two disjoint sections of self-intersection zero which are $\Autz(\pi^{-1}(\sigma))$-invariant by Lemma \ref{geometryofruledsurface} \ref{geometryofruledsurface.2}. These two disjoint sections $l_0^X=\{y_0=z_0=0\}$ and $l_1^X$ correspond respectively to the line subbundles $\O_\sigma(D)\subset \E_\sigma$ and $\O_\sigma\subset \E_\sigma$. Since $\sigma$ is $\Autz(\AAA_0)$-invariant, it follows by Blanchard's lemma that $\pi^{-1}(\sigma)$ is $\Autz(X)$-invariant. Hence $l_0^X$ and $l_1^X$ are $\Autz(X)$-invariant as well. 
		
		Without loss of generality, we can assume that the $l_1^X$ equals locally $\{(x,[1:v_i(x)\ ; 0:1])\} \subset U_i\times \FF_b$ for some $v_i \in \O_C(U_i)$. Applying the isomorphisms
		\[
		\begin{array}{cccc}
			& U_i\times \FF_b & \dashrightarrow & U_i\times \FF_b \\
			& (x,[y_0:y_1\ ; z_0:z_1]) & \longmapsto & (x,[y_0:y_1-v_i(x)z_1^by_0\ ; z_0:z_1]),
		\end{array}
		\]
		we don't change the invariants and we can now assume that $l_1^X = \{y_1=z_0=0\}$. In particular, we have that $z_0$ divides $p_{ij}$.
		
		(1): The blowup of $l_0^X$ followed by the contraction of the strict transform of $\pi^{-1}(\sigma)$ yields an $\Autz(X)$-equivariant birational map $\phi\colon X\dashrightarrow X_{b+1}$, where $\pi_{b+1}\colon X_{b+1} \to \AAA_0$ is a $\PP^1$-bundle such that $\tau\pi_{b+1}$ is an $\FF_{b+1}$-bundle, and $\phi$ is locally given by 
		\[
		\begin{array}{cccc}
			\phi_i \colon & U_i \times \FF_b & \dashrightarrow & U_i \times \FF_{b+1} \\
			& (x,[y_0:y_1 \ ; z_0:z_1 ]) & \longmapsto & (x,[y_0:z_0y_1 \ ; z_0:z_1 ]).
		\end{array}
		\]
		Computing $\phi_i \theta_{ij}\phi_j^{-1}$, we get the transition maps of the $\FF_{b+1}$-bundle $\tau\pi_{b+1}\colon X_{b+1}\to C$ and we obtain that the $\PP^1$-bundle $\pi_{b+1}$ has invariants $(\AAA_0,b+1,D)$.
		
		(2): Similarly, the blowup of $l^X_{1}$ followed by the contraction of the strict transform of $\pi^{-1}(\sigma)$ yields an $\Autz(X)$-equivariant birational map $\psi\colon X\dasharrow X_{b-1}$, where $\pi_{b-1}\colon X_{b-1}\to \AAA_0$ is a $\PP^1$-bundle such that $\tau\pi_{b-1}$ is an $\FF_{b-1}$-bundle, and $\psi$ is locally given by:
		\[
		\begin{array}{cccc}
			\psi_i\colon & U_i\times \FF_b & \dasharrow & U_i\times \FF_{b-1} \\
			&(x,[y_0:y_1\ ;z_0:z_1]) & \longmapsto & (x,[z_0 y_0:y_1\ ; z_0:z_1]).
		\end{array}
		\]
		Computing $\psi_i \theta_{ij}\psi_j^{-1}$, we get the transition maps of the $\FF_{b-1}$-bundle $\tau\pi_{b-1}\colon X_{b-1}\to C$ equal
		\[
		\begin{array}{cccc}
		 & U_j\times \FF_{b-1} & \dasharrow & U_i\times \FF_{b-1} \\
		&(x,[y_0:y_1\ ;z_0:z_1]) & \longmapsto & (x,[ y_0:\lambda_{ij}(x)y_1 + p_{ij}(x,z_0,z_1)z_0^{-1}y_0\ ; z_0:z_1]).
		\end{array}		
		\]
		 If $b\geq 2$, we obtain that the $\PP^1$-bundle $\pi_{b-1}$ has invariants $(\AAA_0,b-1,D)$. If $b=1$, we obtain that $X_0$ is isomorphic to the fiber product $S\times_C \AAA_0$, where $S$ is a geometrically ruled surface with transition maps
		 \[
		 \begin{array}{cccc}
		 	& U_j\times \PP^1 & \dasharrow & U_i\times \PP^1 \\
		 	&(x,[y_0:y_1]) & \longmapsto & (x,[ y_0:\lambda_{ij}(x)y_1 + p_{ij}(x,z_0,z_1)z_0^{-1}y_0]).
		 \end{array}		
		 \]
		 Therefore, $S$ is the projectivization of a rank-$2$ vector bundle, which yields an element of $\mathrm{Ext}^1(\O_C,\O_C(D))$. The latter is the zero vector space since $D$ is non-trivial of degree zero; thus, $S$ is $C$-isomorphic to the decomposable $\PP^1$-bundle $\PP(\O_C \oplus \O_C(D))$. Thus, $X$ is $\AAA_0$-isomorphic to $\PP(\O_{\AAA_0} \oplus \O_{\AAA_0}(\tau^*(D)))$.
		
		(3): Apply inductively \ref{A0_nontrivial.2}. For each $b>1$, the base locus of $\psi_{b,b-1}^{-1}$ is the $\Autz(X_{b-1})$-invariant curve $\pi_{b-1}^{-1}(\sigma) \cap (\AAA_0)_{b-1}$. For $b=0$, the base locus of $\psi_{1,0}^{-1}$ is the curve $l_0^{X_0} = \{y_0 = z_0 =0 \} \subset X_0$, which is one of the two disjoint sections of the $\PP^1$-bundle $(\pi_0)_{|\pi_0^{-1}(\sigma)}\colon \pi_0^{-1}(\sigma) \to \sigma$; hence, $l_0^{X_0}$ is $\Autz(X_0)$-invariant by Lemmas \ref{geometryofruledsurface} and \ref{autofiberproduct}. Therefore, $\Autz(X)$ is conjugate to $\Autz(X_0)$, where $X_0=\PP(\O_{\AAA_0} \oplus \O_{\AAA_0}(\tau^*(D)))$.
		
		Moreover, we can choose new trivializations of $\tau \pi_0\colon X_0\to C$ such that the transition maps are
		\[
		\begin{array}{cccc}
		 &U_j\times \FF_0 & \dasharrow  &U_i\times \FF_0 \\
			&(x,[y_0,y_1\ ;z_0,z_1]) & \longmapsto &(x,[y_0:\lambda_{ij}(x) y_1 \ ; z_0:\alpha_{ij}(x)z_0+z_1]).
		\end{array}
		\]
		The blowup of the curve $l_0^{X_0} = \{y_0 = z_0 =0 \} \subset X_0$ followed by the contraction of the strict transform of $\pi_0^{-1}(\sigma)$ is the inverse of $\psi_{1,0}\colon X_1 \dashrightarrow X_0$ obtained at the last step of the induction above. This gives new trivializations of $\tau\pi_1$ such that the transition maps equal 	
		\[
		\begin{array}{cccc}
			 &U_j\times \FF_1 & \dasharrow  &U_i\times \FF_1 \\
			&(x,[y_0,y_1\ ;z_0,z_1]) & \longmapsto &(x,[y_0:\lambda_{ij}(x) y_1\ ; z_0:\alpha_{ij}(x)z_0+z_1]),
		\end{array}
		\]
		and this shows that $\tau\pi_1$ is the decomposable $\PP^1$-bundle $\PP(\O_{\AAA_0} \oplus \O_{\AAA_0}(\sigma +\tau^*(D)))$. For each $k\in \{1,\ldots,b-1\}$, using the birational map $\phi_{k,k+1}\colon X_{k}\dashrightarrow X_{k+1}$, we get new transition maps for the $\FF_{k+1}$-bundle $\tau\pi_{k+1}\colon X_{k+1}\to C$, that only depend on $\lambda_{ij}\in \O_C(U_{ij})^*$ and $\alpha\in \O_C(U_{ij})$. Therefore, we obtain that $X=\PP(\O_{\AAA_0} \oplus \O_{\AAA_0}(b\sigma +\tau^*(D)))$ by Lemma \ref{FFb-bundle}.
		
		The morphism $\pi_*\colon \Autz(X)\to \Autz(\AAA_0)$ is also surjective by Proposition \ref{imagepi*} \ref{imagepi*.2}. Since $\pi$ is decomposable, the group $\mathbb{G}_m$ acts on the fibers of $\pi$ by fixing the two disjoint sections of $\pi$ corresponding to the line subbundles $\O_{\AAA_0}(b\sigma +\tau^*(D))$ and $\O_{\AAA_0}$. Therefore, $l_0^X$ and $l_1^X$ are the only $\Autz(X)$-orbits of dimension one and are the $\Autz(X)$-orbits of minimal dimension. This implies that the only $\Autz(X)$-equivariant Sarkisov diagrams of type II are those given in \ref{A0_nontrivial.1} and \ref{A0_nontrivial.2}.
\end{proof}

Combining the results of this section, we reduce to the classification to the study of $\FF_0$-bundles:

\begin{proposition}\label{equivariantto0.A0}
	Let $\pi\colon X\to \AAA_0$ be a $\PP^1$-bundle with invariants $(\AAA_0,b,D)$, where $b>0$, and such that $\Autz(X)$ is relatively maximal. Then $D$ is non-trivial of degree zero and there exists an $\Autz(X)$-equivariant square birational map $X\dashrightarrow \PP(\O_{\AAA_0} \oplus \O_{\AAA_0}(\tau^*(D)))$. 
\end{proposition}

\begin{proof}
	By Proposition \ref{degnonnuldec}, $\deg(D)=0$. If $D$ is trivial, then $\Autz(X)$ is not relatively maximal by Section \ref{section:(A_0,b,0)}, which is a contradiction. Then $D$ is non-trivial of degree zero and the statement follows from Lemma \ref{A0_nontrivial}.
\end{proof}

\subsection{The decomposable $\PP^1$-bundle $\PP(\O_{\AAA_0} \oplus \O_{\AAA_0}(\tau^*(D)))$}

\begin{proposition}\label{prop:decomposable_over_A0}
	Let $X= \PP(\O_{\AAA_0} \oplus \O_{\AAA_0}(\tau^*(D)))$ with structure morphism $\pi\colon X\to \AAA_0$ and where $D$ is non-trivial of degree zero. Then the following hold:
	\begin{enumerate}
		\item $\Autz(X)$ is relatively maximal, but the pair $(X,\pi)$ is not stiff.
		\item Let $\pi'\colon X'\to S$ be a $\PP^1$-bundle and $S$ be a geometrically ruled surface over $C$. Then $\Autz(X)$ and $\Autz(X')$ are conjugate via a square birational map if and only if there exists $b\in \mathbb{Z}$ such that $$X'\cong \PP(\O_{\AAA_0}\oplus \O_{\AAA_0}(b\sigma + \tau^*(D) )).$$
	\end{enumerate}
\end{proposition}

\begin{proof}
	The $\PP^1$-bundle $ \PP(\O_{\AAA_0} \oplus \O_{\AAA_0}(\tau^*(D)))$ is $\AAA_0$-isomorphic to the fiber product $\PP(\O_{C} \oplus \O_{C}(D))\times_C \AAA_0$. By Lemma \ref{SarkisovIIIandIV}, there exists a unique Sarkisov diagram of type IV that permutes the two fibrations of the fiber product. Moreover, an $\Autz(X)$-Sarkisov diagram of type I or II is uniquely determined by the datum of the blowup of an $\Autz(X)$-orbit of dimension one. By Lemmas \ref{geometryofruledsurface} and \ref{autofiberproduct}, it follows that the $\Autz(X)$-orbits of dimension one are precisely the intersection of $\pi^{-1}(\sigma)$ with the two disjoint sections of $\pi$ corresponding to the line subbundles $\O_{\AAA_0}\subset \O_{\AAA_0} \oplus \O_{\AAA_0}(\tau^*(D))$ and $\O_{\AAA_0}(\tau^*(D))\subset \O_{\AAA_0} \oplus \O_{\AAA_0}(\tau^*(D))$. 
	
	We can choose trivializations of $\tau\pi$ such that the transition maps equal:
	\[
	\begin{array}{ccc}
		U_j \times \FF_0 & \dashrightarrow & U_i \times \FF_0 \\
		(x,[y_0:y_1 \ ; z_0 : z_1]) & \longmapsto & (x,[y_0:\lambda_{ij}(x)y_1 \ ; z_0 : \alpha_{ij}(x)z_0 +  z_1]),
	\end{array}
	\]
	where $\alpha_{ij} \in \O_C(U_{ij})$ and $\lambda_{ij} \in \O_C(U_{ij})^*$ are the cocycles of $\O_C(D)$. Under this choice of trivializations, the two $\Autz(X)$-invariants curves are $l_{00} =\{y_0 = z_0 =0\}$ and $l_{10} =\{y_1 = z_0 =0\}$. The blowup of $l_{00}$ (resp. $l_{10}$) followed by the contraction of strict transform yields an $\Autz(X)$-equivariant birational map $X\dashrightarrow  \PP(\O_{\AAA_0} \oplus \O_{\AAA_0}(\sigma+ \tau^*(D)))$ (resp. $\PP(\O_{\AAA_0} \oplus \O_{\AAA_0}(-\sigma+ \tau^*(D)))$). Using the equivariant birational maps in Lemma \ref{A0_nontrivial}, we obtain an $\Autz(X)$-equivariant birational map $X\dashrightarrow \PP(\O_{\AAA_0} \oplus \O_{\AAA_0}(b\sigma+ \tau^*(D)))$ for every $b\in \mathbb{Z}$, which conjugates the respective automorphism groups. 
	
	When $b>0$, it also follows from Lemma \ref{A0_nontrivial} that the $\Autz(X)$-Sarkisov diagrams of type II starting from $\PP(\O_{\AAA_0} \oplus \O_{\AAA_0}(b\sigma+ \tau^*(D)))$ yield the Sarkisov links: $$\PP(\O_{\AAA_0} \oplus \O_{\AAA_0}(b\sigma+ \tau^*(D))) \dashrightarrow \PP(\O_{\AAA_0} \oplus \O_{\AAA_0}((b\pm 1)\sigma+ \tau^*(D))).$$
	Thus, there is no other $\PP^1$-bundle $\pi\colon X'\to S'$ than the ones stated, such that $\Autz(X)$ and $\Autz(X')$ are conjugate via a square birational map. 
\end{proof}

The fiber product $\PP(\O_{C} \oplus \O_{C}(D))\times_C \AAA_0$ is equipped with two projections: one onto $\AAA_0$ and another one onto $\PP(\O_{C} \oplus \O_{C}(D))$. The second projection is a $\PP^1$-bundle over $\PP(\O_{C} \oplus \O_{C}(D))$, whose automorphism group will be studied in the final section, see Lemma $\ref{A0xS}$. 

\subsection{The fiber product $\AAA_0\times_C \AAA_0$}

\begin{proposition}\label{A0A0notmax}
	For both projections onto $\AAA_0$, the automorphism group $\Autz(\AAA_0\times_C \AAA_0)$ is not relatively maximal.
\end{proposition}

\begin{proof}
	Without loss of generality, we consider the projection onto the second factor $\pi\colon \AAA_0\times_C \AAA_0 \to \AAA_0$. Let $\sigma$ be the unique minimal section of $\tau\colon \AAA_0\to C$. We also choose trivializations of $\tau$ such that the transition maps of $\tau\pi$ equal
	\[
	\begin{array}{ccc}
		U_j \times \FF_0  & \dashrightarrow & U_i \times \FF_0  \\
		(x,[y_0:y_1 \ ; z_0 : z_1 ]) & \longmapsto & (x,[y_0: \alpha(x)^{-1}y_0 + y_1\ ; z_0 : \alpha(x)^{-1}z_0+z_1 ]),
	\end{array}
	\]
	By Lemma \ref{autofiberproduct}, we have that $\Autz(\AAA_0\times_C \AAA_0) \simeq \Autz(\AAA_0) \times_{\Autz(C)} \Autz(\AAA_0)$, so the curve $\widetilde{C}=\{y_0=z_0=0\}$ is $\Autz(\AAA_0\times_C \AAA_0)$-invariant. The blowup of $\widetilde{C}$ followed by the contraction of the strict transform of $\pi^{-1}(\sigma)$ yields an $\Autz(\AAA_0\times_C \AAA_0)$-equivariant square birational map $\AAA_0\times_C \AAA_0\dashrightarrow X_1$ locally given by
	\[
	\begin{array}{ccc}
		U_k \times \FF_0 & \dashrightarrow & U_k\times \FF_1 \\
		(x,[y_0:y_1 \ ; z_0 : z_1 ]) & \longmapsto & (x,[y_0:z_0y_1 \ ; z_0 : z_1 ]),
	\end{array}
	\]
	where $U_k$ is any trivializing open subset of $\tau\pi$. Then $\pi_1\colon X_1\to \AAA_0$ is a $\PP^1$-bundle such that $\tau\pi_1\colon X_1\to C$ is the $\FF_1$-bundle with the following transition maps
	\[
	\begin{array}{cccc}
		& U_j \times \FF_1 & \dashrightarrow & U_i \times \FF_1 \\
		& (x,[y_0:y_1 \ ; z_0 : z_1 ]) & \longmapsto & (x,[y_0:y_1 + \alpha(x)^{-1}z_0y_0 \ ; z_0 : \alpha(x)^{-1}z_0+z_1 ]).
	\end{array}
	\]
	In particular, $\pi_1$ has invariants $(\AAA_0,1,0)$; thus, $\Autz(X_1)$ is not relatively maximal by Proposition \ref{prop:(A0,b,0)}.
\end{proof}

Gathering the results of this section, we now conclude with:

\begin{proof}[Proof of Proposition \ref{main_prop:A0}]
	Assume that $\tau\pi$ is an $\FF_b$-bundle with $b>0$ and $\Autz(X)$ is relatively maximal. Then by Propositions \ref{degnonnuldec} and \ref{prop:(A0,b,0)}, the $\PP^1$-bundle $\pi$ has invariants $(\AAA_0,b,D)$, where $D\in \Pic^0(C)$ is non-trivial. By Lemma \ref{A0_nontrivial}, we have
	\[
	X \cong \PP(\O_{\AAA_0}\oplus \O_{\AAA_0}(b\sigma + \tau^*(D))),
	\]
	and there exists an $\Autz(X)$-equivariant square birational map $X \dashrightarrow \PP(\O_{\AAA_0}\oplus \O_{\AAA_0}( \tau^*(D)))$. This reduces to the case of $\FF_0$-bundles arising from a $\PP^1$-bundle over $\AAA_0$. Those are studied through Propositions \ref{FF0viaCxPP1}, \ref{A_0A_1notmax}, \ref{prop:decomposable_over_A0}, and \ref{A0A0notmax}.
\end{proof}

\section{Automorphisms of $\PP^1$-bundles over $S=\PP(\O_C\oplus \L)$, with $\L\in \Pic^0(C)$ non-trivial}\label{SectionSL}

Let $C$ be an elliptic curve and $\L$ a non-trivial line bundle of degree zero. In this section, $S$ denotes the geometrically ruled surface $\PP(\O_C\oplus \L)$, with the structure morphism $\tau\colon S\to C$, which admits two disjoint minimal sections of self-intersection. Those are not linearly equivalent but are $\Autz(S)$-invariant. We fix $\sigma$ to be the section of $\tau$ corresponding to the line subbundle $\L\subset \O_C\oplus \L$ such that $\O_S(\sigma) \simeq \O_S(1)$. 

\bigskip

The strategy of this section is similar as for $\PP^1$-bundles over $\AAA_0$. We show that if $X$ is a $\PP^1$-bundle over $S$ such that $\Autz(X)$ is relatively maximal, then $\Autz(X)$ is conjugate to the automorphism group of an $\FF_0$-bundle. The main result of this section is the following:

\begin{proposition}\label{main_prop:S}
	Let $\pi\colon X \to S$ be a $\PP^1$-bundle such that $\tau\pi$ is an $\FF_b$-bundle, with $b\in \mathbb{Z}$, and let $D$ be a divisor such that $\O_C(D)\cong \L$. Then $\Autz(X)$ is relatively maximal if and only if one of the following case occurs:
	\begin{enumerate}
		\item Case $b=0$. Then $X$ is isomorphic to one of the following $\PP^1$-bundles:
		\[
		S \times \PP^1,~S \times_C \AAA_1,~S \times_C \AAA_0,~ S \times_C \PP(\O_C\oplus \O_C(E));
		\]
		provided that $D$ has infinite order in the third case, and that in the last case $E\in \Pic^0(C)$ is non-trivial and not a multiple of $D$ in the last case. Moreover, the pair $(X,\pi)$ is superstiff in the first two cases, and it is not stiff in the last two.
		\item\label{main_prop:S.2} Case $b\neq 0$. Then $X$ is isomorphic to one of the following $\PP^1$-bundles:
		\begin{enumerate}
			\item\label{main_prop:S.2.i} The indecomposable $\PP^1$-bundle $X\cong \AAA_{(S,b,nD)}$ for some $b>0$, $n\in \{0,\ldots,b\}$, and $D$ has infinite order (see Lemma \ref{lem:induction}). There exists an $\Autz(X)$-equivariant square birational map 
			\[
			X \dashrightarrow S\times_C\AAA_0 = \AAA_{(S,0,0)}.
			\]
			\item\label{main_prop:S.2.ii} The decomposable $\PP^1$-bundle 
			$X\cong \PP(\O_{S}\oplus \O_{S}(b\sigma +{\tau}^{*}(nD+E)))$,
			for some $n\in \mathbb{Z}$, provided that $E$ is not a multiple of $D$. There exists an $\Autz(X)$-equivariant square birational map 
			\[
			X \dashrightarrow \PP(\O_{S}\oplus \O_{S}({\tau}^{*}(mD+E)))
			\]
			for every $m\in \mathbb{Z}$.
		\end{enumerate}
	\end{enumerate}
	Let $\pi'\colon X'\to T$ be a $\PP^1$-bundle over a geometrically ruled surface $T$. Moreover, the following hold:
	\begin{itemize}
		\item There exists an $\Autz(S\times_C \AAA_0)$-equivariant square birational map 
		\[
		S\times_C \AAA_0 \dashrightarrow X'
		\]
		if and only if $T\cong S$ and $X' \cong \AAA_{(S,b,nD)}$ as in Case \ref{main_prop:S.2} \ref{main_prop:S.2.i}.
		\item There exists an $\Autz(S\times_C \PP(\O_C\oplus \O_C(E)))$-equivariant square birational map 
		\[
		S\times_C \PP(\O_C\oplus \O_C(E)) \dashrightarrow X'
		\]
		if and only if $T\cong S$ and $X' \cong \PP(\O_{S}\oplus \O_{S}(b\sigma \oplus {\tau}^{*}(nD+E)))$ for some $b,n\in \mathbb{Z}$. If $b\neq 0$, then $X'$ is a $\PP^1$-bundle as in Case \ref{main_prop:S.2} \ref{main_prop:S.2.ii}.
	\end{itemize}
\end{proposition}

\subsection{$\PP^1$-bundles with invariants $(\PP(\O_C\oplus \L),b,0)$ with $b>0$}

\begin{lemma}\label{L-torsion}
	Assume that $\L$ is a torsion line bundle and let $\pi\colon X\to S$ be a $\PP^1$-bundle with invariants $(S,b,0)$, where $b>0$. Then $\Autz(X)$ is not relatively maximal.
\end{lemma}

\begin{proof}
	By Lemma \ref{FFb-bundle}, we can choose the trivializations of $\tau\pi$ such that the transition maps equal
	\[
	\begin{array}{ccc}
		U_j\times \FF_b & \dashrightarrow & U_i\times \FF_b \\
		(x,[y_0:y_1\ ; z_0:z_1]) & \longmapsto & (x,[y_0:y_1+ p_{ij}(x,z_0,z_1)y_0\ ; z_0:a_{ij}(x)z_1]),
	\end{array}
	\]
	where $p_{ij}\in \O_C(U_{ij})[z_0,z_1]_b$ and $a_{ij} \in \O_C(U_{ij})^*$ are the cocycles of $\L$. Under this choice of trivializations, the section $\sigma$ is the zero section $\{z_0=0\}$.
	The curve $\{y_0 =z_0 = 0\}$ is $\Autz(X)$-invariant, as it is the intersection of $\pi^{-1}(\sigma)$ with the surface spanned by the $(-b)$-sections along the fibers of $\tau\pi$. Its blowup followed by the contraction of the strict transform of the surface $\pi^{-1}(\sigma)=\{z_0=0\}$ yields an $\Autz(X)$-equivariant square birational map $X\dasharrow X_{b+1}$, which is locally given by
	\[
	\begin{array}{ccc}
		U_i\times \FF_b & \dashrightarrow & U_i\times \FF_{b+1} \\
		(x,[y_0:y_1\ ;z_0:z_1]) & \longmapsto & (x,[y_0:z_0y_1\ ;z_0:z_1]).
	\end{array}
	\]
	Computing the transition maps of the $\FF_{b+1}$-bundle $\tau\pi_{b+1}\colon X_{b+1}\to C$, we get that the $\PP^1$-bundle $\pi_{b+1}\colon X_{b+1}\to S$ has invariants $(S,b+1,0)$. Repeating this process, we get an $\Autz(X)$-equivariant square birational map $\phi\colon X\dasharrow X_{B}$ for some $B>b$ which is a multiple of the order of $\L$. In particular, $\L^B \simeq \O_C$. The transition maps of the $\FF_{B}$-bundle $\tau\pi_B \colon X_B \to C$ equal
	\[
	\begin{array}{cccc}
		\psi_{ij}\colon & U_j\times \FF_{B}& \dashrightarrow & U_i\times \FF_{B} \\
		&(x,[y_0:y_1\ ;z_0:z_1]) & \longmapsto & (x,[y_0:y_1 + z_0^{B-b}p_{ij}(x,z_0,z_1)y_0\ ;z_0:a_{ij}(x)z_1]).
	\end{array}
	\]
	The base locus of $\phi^{-1}$ is the curve $l_{10}^{X_B} = \{y_1=z_0=0\}\subset X_{B}$. Since $\L^{-B} \simeq \O_C$, there exists $\xi \in \Gamma(C,\L^{-B})\setminus \{0\}$ vanishing nowhere, such that $\xi_{|U_i}  =a_{ij}^{-B} \xi_{|U_j}$. Then, define the following automorphisms: 
	\[
	\begin{array}{cccc}
		f_{i}\colon & U_i\times \FF_{B} & \dashrightarrow & U_i\times \FF_{B} \\
		& (x,[y_0:y_1\ ;z_0:z_1]) & \longmapsto & (x,[y_0:y_1+ \xi_{|U_i}(x) z_1^By_0\ ;z_0:z_1]).
	\end{array}
	\]
	Since $f_{i} \psi_{ij} = \psi_{ij} f_{j}$, the $(f_i)_i$ yields an element $f\in \Autz(X)$ which does not leave invariant the curve $l_{10}^{X_B}$. Thus, $\Autz(X)$ is not relatively maximal.
\end{proof}

Next, we show that if $\L$ has infinite order, then there exists an $\Autz(X)$-equivariant birational map to an $\FF_0$-bundle.

\begin{lemma}\label{Sinvarianttriviallem}
	Assume $\L$ has infinite order and let $\pi\colon X\to S$ be a $\PP^1$-bundle with invariants $(S,b,0)$, where $b>0$. We denote by $S_b\subset X$ the surface spanned by the $(-b)$-sections along the fibers of $\tau\pi$. Then the following hold:
	\begin{enumerate}
		\item \label{Sinvarianttriviallem.1} There exists an $\Autz(X)$-invariant curve $l^X_{10}\subset \pi^{-1}(\sigma)$, disjoint from $S_b$, such that its blowup followed by the contraction of the strict transform of $\pi^{-1}(\sigma)$ yields an $\Autz(X)$-equivariant birational map $\psi_{b,b-1}$ and the following commutative diagram
		\[
		\begin{tikzcd}
			X_{b-1} \arrow[rd, "\pi_{b-1}" swap] && \arrow[ll,dashed,"\psi_{b,b-1}" swap] \arrow[ld,"\pi"] X \\
			& S &,
		\end{tikzcd}
		\]		
	 	where $\pi_{b-1}\colon X_{b-1}\to S$ is a $\PP^1$-bundle such that $\tau\pi_{b-1}$ is an $\FF_{b-1}$-bundle. Moreover, if $b > 1$, then $\pi_{b-1}$ has invariants $(S,b-1, 0)$.
		\item\label{Sinvarianttriviallem.3} $\Autz(X)$ is conjugate to an algebraic subgroup of $\Autz(X_0)$, where $X_0 = \PP^1\times S$ or $\AAA_0\times_C S$.
	\end{enumerate}
\end{lemma}	

\begin{proof}
	We fix trivializations of $\tau$ such that $\sigma$ is the zero section $\{z_0=0\}$, and the transition maps of $(\tau\pi)$ equal
	\[
	\begin{array}{cccc}
		\phi_{ij}\colon & U_j\times \FF_b & \dashrightarrow & U_i\times \FF_b \\
		&(x,[y_0:y_1\ ;z_0:z_1]) & \longmapsto  & (x,[y_0:y_1 + p_{ij}(x,z_0,z_1)y_0\ ;z_0:a_{ij}(x)z_1]),
	\end{array}
	\]
	where $a_{ij}\in \O_C(U_{ij})^*$ are the cocycles of $\L$ and $p_{ij}(x,z_0,z_1)\in \O_C(U_{ij})[z_0,z_1]_b$.
	
	(1) By Lemma \ref{FFb-bundle}, there exists a rank $2$-vector bundle $\E$ such that $\PP(\E)\simeq X$ which fits into the short exact sequence 
	$$0\to \O_S(b\sigma) \to \E \to \O_S\to 0.$$ 
	Restricting to $\sigma$ yields a $\PP^1$-bundle $\pi_{\sigma}\colon \PP(\E_{\sigma})\to \sigma$, where $\E_{\sigma}$ is a rank-$2$ vector bundle which fits into the exact sequence
	$$0\to \O_{\sigma}(b \sigma_{|\sigma}) \to \E_{\sigma} \to \O_{\sigma}\to 0.$$ 
	Taking $z_0=0$ in the transition maps of $\tau\pi$, we get that the transition maps of the $\PP^1$-bundle $\pi_{\sigma}$ over $\sigma$ equal
	\[
	\begin{array}{ccc}
		U_j \times \PP^1 & \dasharrow &U_i\times \PP^1 \\
		(x,[y_0:y_1]) & \longmapsto & (x,[y_0:a_{ij}(x)^{-b} y_1+ p_{ij}(x,0,1)y_0]).
	\end{array}
	\]
	This shows that $\O_\sigma(b\sigma_{|\sigma}) \simeq \L^{-b}$. Since $\L$ has infinite order, it implies that $\L^{-b}$ is a non-trivial line bundle of degree zero and $\mathrm{Ext}^1(\O_{\sigma},\L^{-b})=0$; hence, $\E_{\sigma}$ is $\sigma$-isomorphic to $\O_{\sigma} \oplus \L^{-b}$. By Lemma \ref{geometryofruledsurface}, the geometrically ruled surface $\PP(\E_\sigma)$ admits exactly two disjoint sections of self-intersection zero. This yields two $\Autz(X)$-invariant curves: $l_{00}^X = \pi^{-1}(\sigma) \cap S_b$, and $l_{10}^X\subset \pi^{-1}(\sigma)$ disjoint from $S_b$.
	
	Next, we show that we can choose the trivializations of $\tau\pi$, such that $l_{10}^X=\{y_1=z_0=0\}$. For each $i$, there exist $u_i,v_i \in \O_\sigma(U_i)$ such that ${l_{10}^X}_{|U_i} = \{x,[u_i(x):v_i(x)\ ;0:1]\}$. Since $l_{10}^X$ is disjoint from the curve $l_{00}^X$, we can also assume that $u_i=1$. 
	For each $i$, we define the automorphism 
	\[
	\begin{array}{cccc}
		f_i\colon &  U_i\times \FF_b & \longrightarrow &U_i\times\FF_b\\
		& (x,[y_0:y_1\ ;z_0:z_1])& \longmapsto &(x,[y_0: y_1 - v_i(x)z_1^by_0\ ;\ z_0:z_1]),
	\end{array}
	\]
	which sends $l_{10}^X$ to $\{y_1=z_0=0\}$. Computing the composition $f_i \phi_{ij} f_j^{-1}$, we obtain the new transition maps of $\tau\pi$. Those replace $p_{ij}$ by new polynomials $\tilde{p}_{ij}\in \O_C(U_{ij})[z_0,z_1]_b$ such that $\tilde{p}_{ij}(x,0,1) = 0$, i.e., $z_0$ divides $\tilde{p}_{ij}$.
	The blowup of $l_{10}^X = \{y_1=z_0=0\}\subset X$ followed by the contraction of the strict transform of $\pi^{-1}(\sigma)=\{z_0=0\}$ yields an $\Autz(X)$-equivariant birational map $\psi_{b,b-1}\colon X\dasharrow X_{b-1}$, where $\tau\pi_{b-1}\colon X_{b-1}\to C$ is an $\FF_{b-1}$-bundle, and $\psi_{b,b-1}$ is locally given by 
	\[
	\begin{array}{ccc}
		U_i\times \FF_b & \dashrightarrow & U_i\times \FF_{b-1} \\
		(x,[y_0:y_1\ ;z_0:z_1]) &\longmapsto &(x,[z_0y_0:y_1\ ;z_0:z_1]).
	\end{array}
	\]
	Computing the transition maps of $\tau \pi_{b-1}$, we obtain:
		\[
	\begin{array}{cccc}
		 & U_j\times \FF_b & \dashrightarrow & U_i\times \FF_b \\
		&(x,[y_0:y_1\ ;z_0:z_1]) & \longmapsto  & (x,[y_0:y_1 + z_0^{-1}\tilde{p}_{ij}(x,z_0,z_1)y_0\ ;z_0:a_{ij}(x)z_1]),
	\end{array}
	\]
	which is well-defined as $z_0$ divides $\tilde{p}_{ij}$.
	 If $b>1$, we obtain that $\pi_{b-1}\colon X_{b-1}\to S$ has invariants $(S,b-1,0)$. 
	 
	 (2): Applying inductively \ref{Sinvarianttriviallem.1}, we get an $\Autz(X)$-equivariant birational map $X\dashrightarrow X_0$, where $X_0$ is the $\FF_0$-bundle over $C$ with the transition maps
	\[
	\begin{array}{ccc}
		U_j \times \FF_0 & \dashrightarrow & U_i\times \FF_0 \\
		(x,[y_0:y_1 \ ;z_0:z_1]) & \longmapsto & (x,[y_0:y_1 + c_{ij}(x)y_0 \ ;z_0:a_{ij}(x)z_1]),
	\end{array}
	\]
	and where $c_{ij} \in \O_C(U_{ij})$. This implies that $X_0 \simeq \widetilde{S} \times_C S$, where $\widetilde{S}$ is a $\PP^1$-bundle over $C$ with transition matrices 
	$$
	\begin{pmatrix}
		1 & 0 \\ c_{ij} & 1
	\end{pmatrix}.
	$$
	
	Therefore, the geometrically ruled surface $\widetilde{S}$ is the projectivization of a rank-$2$ vector bundle which yields an element of $\mathrm{Ext}^1(\O_C,\O_C)$. This latter $\kk$-vector space is a one-dimensional: if the extension splits then $\widetilde{S}\cong C\times \PP^1$; else, $\widetilde{S}\cong \AAA_0$. 
\end{proof}

\subsection{$\PP^1$-bundles with invariants $(\PP(\O_C\oplus \L),b,D)$ such that $b>0$ and $D\in \Pic^0(C)$ non-trivial }

~\bigskip

Recall that $S = \PP(\O_C\oplus \L)$ admits two disjoint sections of minimal self-intersection, corresponding to the line subbundles $\O_C$ and $\L$. We also choose $\sigma$ corresponding to $\L$ and such that $\O_S(1) \simeq \O_S(\sigma)$. 

\bigskip

The idea of the following lemma is similar to Lemmas $\ref{A0_nontrivial}$ and $\ref{Sinvarianttriviallem}$. Instead of blowing up curves contained in $\pi^{-1}(\sigma)$, we search for curves contained in $\pi^{-1}(\tilde{\sigma})$, where $\tilde{\sigma}$ is the section corresponding to the line subbundle $\O_C\subset \O_C \oplus \L$. Since $\sigma$ and $\tilde{\sigma}$ are disjoint, we have $\O_S(b\sigma)_{|\tilde{\sigma}} \cong \O_{\tilde{\sigma}}$, so the restriction of the short exact sequence
\[
0 \to \O_S(b\sigma+\tau^*(D)) \to \E \to \O_S \to 0
\]
yields an element of $\mathrm{Ext}^1(\O_{\tilde{\sigma}}, \O_{\tilde{\sigma}}(D))$, which is again trivial. Therefore, the surface $\pi^{-1}(\tilde{\sigma})$ contains two $\Autz(X)$-invariant curves.

\begin{lemma}\label{SLnontriv}
		Let $\pi\colon X\to S$ be a $\PP^1$-bundle with invariants $(S,b,D)$, where $b>0$ and $D$ is non-trivial of degree zero. The following hold:
		\begin{enumerate}
			\item\label{SLnontriv.1} There exists an $\Autz(X)$-invariant curve $l_X^{11}\subset \pi^{-1}(\tilde{\sigma})$, whose blowup followed by the contraction of the strict transform of $\pi^{-1}(\tilde{\sigma})$ yields an $\Autz(X)$-equivariant square birational map $\psi_{b,b-1}$ and the following commutative diagram:
			\[
			\begin{tikzcd}
				X_{b-1} \arrow[rd, "\pi_{b-1}" swap] && \arrow[ll,dashed,"\psi_{b,b-1}" swap] \arrow[ld,"\pi"] X \\
				& S &,
			\end{tikzcd}
			\]		
			where $\pi_{b-1}\colon X_{b-1} \to S$ is a $\PP^1$-bundle such that $\tau \pi_{b-1}$ is an $\FF_{b-1}$-bundle over $C$. If $b>1$, then $\pi_{b-1}$ has invariants $(S,b-1,D-\D)$, where $\D\in \Pic^0(C)$ is such that $\L \cong \O_C(\D)$.
			\item If $\Autz(X)$ is relatively maximal, then $\Autz(X)$ is conjugate to an algebraic subgroup of the automorphism group of an $\FF_0$-bundle.
		\end{enumerate} 	
\end{lemma}

\begin{proof}	
		(1): By Lemma \ref{FFb-bundle} \ref{canonicalextension}, there exists a rank $2$-vector bundle $\E$ over $S$ such that $\PP(\E)\simeq X$, which fits into the short exact sequence $$0\to \O_S(b\sigma+\tau^*(D)) \to \E \to \O_S\to 0.$$ 
		The restriction to $\tilde{\sigma}$ yields the short exact sequence $$0\to \O_{\tilde{\sigma}}(b \sigma_{| \tilde{\sigma}}+D) \to \E_{\tilde{\sigma}} \to \O_{\tilde{\sigma}}\to 0.$$ 
		Since the sections $\sigma$ et $\tilde{\sigma}$ are disjoint, the $0$-cycle $\sigma\cdot \tilde{\sigma}$ is trivial, so $\O_\sigma(b \sigma_{| \tilde{\sigma}})\simeq \O_{\tilde{\sigma}}$. Since $D$ is non-trivial of degree $0$, it follows that $\mathrm{Ext}^1(\O_{\tilde{\sigma}},\O_{\tilde{\sigma}}(D)) =0$, so the geometrically ruled surface $X_{\tilde{\sigma}}$ is decomposable and isomorphic to $\PP(\O_{\tilde{\sigma}} \oplus \O_{\tilde{\sigma}}(D))$; hence, it admits exactly two disjoint sections of self-intersection zero by Lemma \ref{geometryofruledsurface}. This gives two $\Autz(X)$-invariant curves contained in $\pi^{-1}(\tilde{\sigma})$: $l^{X}_{01}\subset S_b$, where $S_b$ is the surface spanned by the $(-b)$-sections along the fibers of $\tau\pi$, and $l^{X}_{11}$ is disjoint from $S_b$.
		
		As in the proof of Lemma \ref{Sinvarianttriviallem}, we can choose the trivializations of $\tau\pi$ such that the transition maps equal
		\[
		\begin{array}{cccc}
		\phi_{ij}	&	U_j \times \FF_b & \dasharrow &U_i\times \FF_b \\
			& (x,[y_0:y_1\ ;z_0:z_1]) & \longmapsto &(x,[y_0:\lambda_{ij}(x) y_1+ p_{ij}(x,z_0,z_1)y_0\ ; z_0: a_{ij}(x)z_1]),
		\end{array}
		\]
		where $a_{ij},\lambda_{ij}\in \O_{\tilde{\sigma}}(U_{ij})^*$ are the cocycles of the line bundles $\L$ and $\O_{C}(D)$, respectively; and $z_1$ divides $p_{ij}(x,z_0,z_1)\in \O_{\tilde{\sigma}}(U_{ij})[z_0,z_1]_b$. Under this choice of trivializations, we have $l_{01}^X = \{y_0 = z_1 = 0\}$ and $l_{11}^X = \{y_1 = z_1 = 0\}$
		
		The blowup of $l^X_{11}=\{y_1=z_1=0\}$ followed by the contraction of the strict transform of the surface $\pi^{-1}(\tilde{\sigma})=\{z_1=0\}$ yields an $\Autz(X)$-equivariant birational map $\psi_{b,b-1} \colon X \dashrightarrow X_{b-1}$, locally given by
		 \[
		 \begin{array}{cccc}
		 			\psi_i\colon & U_i \times \FF_b & \dashrightarrow &U_i\times \FF_{b-1} \\
		& (x,[y_0:y_1\ ;z_0:z_1])& \longmapsto &(x,[z_1y_0: y_1 \ ;\ z_0:z_1]).
		 \end{array}
		 \]
		 The projection onto $(x,[z_0:z_1])$ equips $X_{b-1}$ with a structure of $\PP^1$-bundle $\pi_{b-1}\colon X_{b-1}\to S$ such that $\tau \pi_{b-1}$ is an $\FF_{b-1}$-bundle over $C$. Computing $\psi_i \phi_{ij} \psi_j^{-1}$, we obtain the transition maps of $\tau\pi_{b-1}$:
	 	\[
	 	\begin{array}{cccc}
			&	U_j\times \FF_{b-1} & \dashrightarrow &U_i\times \FF_{b-1}\\
			& (x,[y_0:y_1\ ;z_0:z_1])& \longmapsto & (x,[y_0:{a_{ij}(x)}^{-1}\lambda_{ij}(x)y_1 +{a_{ij}(x)}^{-1}z_1^{-1} p_{ij}(x,z_0,z_1)y_0\ ;z_0:a_{ij}(x)z_1]).
	 	\end{array}
	 	\]
		This is well-defined as $z_1$ divides $p_{ij}$. The coefficients ${a_{ij}(x)}^{-1}\lambda_{ij}(x)$ are the cocyles of the line bundle $\L^{-1} \otimes \O_C(D)$. Therefore, if $b>1$, then $\pi_{b-1}$ has invariants $(S,b-1, D-\D)$. 
		
		(2): We apply \ref{SLnontriv.1} inductively. The induction terminates when one of the two following cases arise: either we obtain an $\Autz(X)$-equivariant birational map $X\dashrightarrow X_0$, where $X_0$ is an $\FF_0$-bundle, or $D-n\D$ is trivial for some $n<b$. In the second case, $\L$ has infinite order by Lemma \ref{L-torsion}, and we conclude using Lemma \ref{Sinvarianttriviallem}. 
\end{proof}

Combining the previous results, we obtain the analogue of Proposition $\ref{equivariantto0.A0}$:

\begin{proposition}\label{equivariantto0.SL}
	Let $\pi\colon X\to S$ be a $\PP^1$-bundle with invariants $(S,b,D)$, where $b>0$, and such that $\Autz(X)$ is relatively maximal. Then $\deg(D)=0$ and there exists an $\Autz(X)$-equivariant square birational map $X\dashrightarrow X'$, where $\pi'\colon X'\to S$ is a $\PP^1$-bundle such that $\tau\pi'$ is an $\FF_0$-bundle.
\end{proposition}

By Proposition $\ref{equivariantto0.SL}$, it remains to study the automorphism groups of $\FF_0$-bundles over $C$ that arise from a $\PP^1$-bundle structure over $\PP(\O_C\oplus \L)$. 

\subsection{The fiber product $\AAA_0\times_C \PP(\O_C\oplus \L)$}

~\medskip

Recall that we work over an elliptic curve $C$. Let $D$ a non-trivial divisor of degree zero such that $\L\cong \O_C(D)$ and 
\[
S=\PP(\O_C\oplus \O_C(D)).
\] 
We denote by 
$\tau\colon S\to C$ and $\tau_0\colon \AAA_0\to C$ the structure morphisms.

\begin{lemma}\label{lem:induction}
	Let $\AAA_0\times_C S$. For every integer $b\geq 0 $ and $n\in \{0,\cdots,b\}$, there exists a $\PP^1$-bundle  \[
	\pi_{(S,b,nD)}\colon \AAA_{(S,b,nD)}\to S
	\]
	and an $\Autz(\AAA_0\times_C S)$-equivariant square birational map $\AAA_0\times_C S\dashrightarrow \AAA_{(S,b,nD)}$. Moreover, the following hold:
	\begin{enumerate}
		\item For $b=n=0$, then $\AAA_{(s,0,0)}=\AAA_0 \times_C S$ and $\pi_{(s,0,0)}$ is the projection onto $S$.
		
		\item For $b \geq 0$ and $n\in \{0,\ldots ,b\}$, we construct $\AAA_{(S,b+1,nD)}$ and $\AAA_{(S,b+1,(n+1)D)}$ from $\AAA_{(S,b,nD)}$. The $\PP^1$-bundle $\AAA_{(S,b,nD)}$ has two $\Autz(\AAA_{(S,b,nD)})$-invariant curves $l_{00}$ and $l_{01}$ over the minimal sections of $S$. Let $m\in \{0,1\}$, then the blowup of $l_{0m}$ followed by the contraction of the strict transform of $\pi_{(S,b,nD)}^{-1}(\pi_{(S,b,nD)}(l_{0m}))$ yields an $\Autz(\AAA_{(S,b,nD)})$-equivariant square birational map 
		\[
		\phi_{b,n,m}\colon \AAA_{(S,b,nD)} \dashrightarrow \AAA_{(S,b+1,(n+m)D)}.
		\]
		Moreover, the $\PP^1$-bundle $\AAA_{(S,b+1,(n+m)D)}$ has invariants $(S,b+1,(n+m)D)$.
		\item The $\FF_b$-bundle $\tau \pi_{(S,b,nD)}$ has transition maps of the form
		\[
		\begin{array}{ccc}
			U_j \times \FF_b & \dashrightarrow & U_i\times \FF_b \\
			(x,[y_0:y_1 \ ;z_0:z_1]) & \longmapsto & (x,[y_0:a_{ij}(x)^ny_1 + a_{ij}(x)^n \alpha_{ij}(x) z_0^{b-n}z_1^ny_0 \ ;z_0:a_{ij}(x)z_1]).
		\end{array}
		\]
		where 
		\[
		\begin{pmatrix}
			1 & \alpha_{ij} \\ 0 & 1
		\end{pmatrix}
		\text{ and }		
		\begin{pmatrix}
			1 & 0 \\ 0 & a_{ij}
		\end{pmatrix}
		\]
		are respectively the transition matrices of $\AAA_0$ and $S$.
	\end{enumerate}
\end{lemma}
	
\begin{proof}
	 Let $p_1,p_2$ be the projections from $\AAA_0\times_C S$ onto the first and second factors, respectively. Let $\sigma_0$ be the unique minimal section of $\tau_0$, and let $\sigma$ and $\sigma_D$ be the two disjoint minimal sections of $\tau$ (see Lemma \ref{geometryofruledsurface}). By Lemma \ref{autofiberproduct}, the curves $l_{00}=p_1^{-1}(\sigma_0) \cap p_2^{-1}(\sigma_D) $ and $l_{01}=p_1^{-1}(\sigma_0) \cap p_2^{-1}(\sigma)$ are $\Autz(X)$-invariant.
	
	We can choose open subsets $U_i\subset C$ which trivialize simultaneously the $\PP^1$-bundles $\tau$ and $\tau_0$, and such that the transition maps of the $\FF_0$-bundle $\AAA_0 \times_C S$ over $C$ equal
	\[
	\begin{array}{ccc}
		U_j \times \FF_0 & \dashrightarrow & U_i\times \FF_0 \\
		(x,[y_0:y_1\ ;z_0:z_1]) & \longmapsto & (x,[y_0:y_1+\alpha_{ij}(x)y_0\ ;z_0:a_{ij}(x)z_1]),
	\end{array}
	\]
	where $a_{ij}\in \O_C(U_{ij})^*$ are the cocycles of the line bundle $\O_C(D)$, and $\alpha_{ij}\in \O_C(U_{ij})$ are given in Lemma $\ref{transitionAtiyah}$. 
	
	Under this choice of coordinates, $l_{00} = \{ y_0=z_0=0\}$ and $l_{01}=\{y_0=z_1=0\}$. For $n\in\{0,1\}$, the blowup $l_{0i}$ followed by the contraction of $p_2^{-1}(p_2(l_{0n}))$ yields an $\Autz(\AAA_0\times_C S)$-equivariant birational map $\phi_{0,0,n}\colon \AAA_0\times_C S\dashrightarrow \AAA_{(S,1,nD)}$, where $\AAA_{(S,1,nD)}$ is a $\PP^1$-bundle with invariants $(S,1,nD)$. This can be seen as follows.
	Locally, $\phi_{0,0,n}$ is given by
	\[
	\begin{array}{ccc}
		U \times \FF_0  & \dashrightarrow & U \times \FF_1 \\
		(x,[y_0:y_1\ ;z_0:z_1]) & \longmapsto & (x,[y_0:z_n y_1 \ ;z_0:z_1]).
	\end{array}
	\]
	Hence, it follows that the transition maps of $\AAA_{(S,1,0)}$ and $\AAA_{(S,1,D)}$ respectively equal
	\[
	\begin{array}{ccc}
		U \times \FF_0  & \dashrightarrow & U \times \FF_1 \\
		(x,[y_0:y_1\ ;z_0:z_1]) & \longmapsto & (x,[y_0: y_1 +\alpha_{ij}(x)y_0z_0 \ ;z_0:a_{ij}(x)z_1]).,\\
		
		(x,[y_0:y_1\ ;z_0:z_1]) & \longmapsto & (x,[y_0: a_{ij}(x) y_1 +a_{ij}(x) \alpha_{ij}(x)y_0z_1 \ ;z_0:a_{ij}(x)z_1]).
	\end{array}
	\]
	
	Assume from now on that $b>0$ and let $X=\AAA_{(S,b,nD)}$. The surface $\{y_0=0\}$ is spanned by the $(-b)$-sections along the fibers of the $\FF_b$-bundle $\tau \pi_{(S,b,nD)}$ over $C$; hence, it is $\Autz(X)$-invariant. In particular, the two curves $l_{00}^X = \{y_0=z_0=0\}$ and $l_{01}^X = \{y_0=z_1=0\}$ are $\Autz(X)$-invariant. Then the blowup of $l_{0m}$ followed by the contraction of the strict transform of $\pi_{(S,b,nD)}^{-1}(\pi_{(S,b,nD)}(l_{0m}))$ yields an $\Autz(X)$-equivariant square birational map $\phi_{b,n,m}\colon X\dashrightarrow \AAA_{(S,b+1,(n+m)D)}$, locally given by 
	\[
	\begin{array}{ccc}
		U_i \times \FF_b & \dashrightarrow & U_i\times \FF_{b+1} \\
		(x,[y_0:y_1\ ; z_0:z_1]) & \longmapsto & (x,[y_0:z_m y_1\ ; z_0:z_1]).
	\end{array}
	\]
	At each step above, using the birational map $\phi_{b,n,m}$ and the transition maps to $\AAA_0\times_C S$, we obtain the transition maps of the $\AAA_{(S,b,nD)}$.
\end{proof}

In the following lemma, we show that $\AAA_0\times_C S$ is equivariantly birational to all the other models $\AAA_{(S,b,nD)}$, provided that $D$ has infinite order.

\begin{lemma}\label{equivariantlinksfromAAA}
	 Assume that $D$ has infinite order and let $X=\AAA_{(S,b,nD)}$. Then the $\Autz(X)$-invariant curves are all of the form 
	 \[
	 l_{kl}^X = \{y_k = z_l = 0 \},
	 \]
	 where $k,l\in \{0,1\}$. The blowup of $l_{kl}^X $ followed by the contraction of the strict transform of $\pi_{(S,b,nD)}^{-1} (\pi_{(S,b,nD)}(l_{kl}^X))$ yields an $\Autz(X)$-equivariant square birational map $\phi_{kl}\colon X\dashrightarrow Y$. For each pair $(b,n)$, the following table lists the possible choices for $l_{kl}$ and the resulting varieties $Y$:
	
	\bigskip 
	
	\begin{center}
		\begin{tabular}{|c|c|c|}
			\hline 
			$(b,n)$ & $(k,l)$ & $ Y$ \\
			\hline
			$(0,0)$ & 
			\begin{tabular}{c}
				$(0,0)$  \\ \hline $(0,1)$
			\end{tabular} & 
			\begin{tabular}{c}
				$\AAA_{(S,1,0)}$  \\ \hline $\AAA_{(S,1,D)}$
			\end{tabular} \\
			
			\hline 
			$b>n=0$ & 
			\begin{tabular}{c}
				$(0,0)$  \\ \hline $(0,1)$ \\ \hline $(1,0)$
			\end{tabular} & 
			\begin{tabular}{c}
				$\AAA_{(S,b+1,0)}$  \\ \hline $\AAA_{(S,b+1,D)}$ \\ \hline $\AAA_{(S,b-1,0)}$
			\end{tabular} \\
			\hline
			
			\hline 
			$b=n>0$ & 
			\begin{tabular}{c}
				$(0,0)$  \\ \hline $(0,1)$ \\ \hline $(1,1)$
			\end{tabular} & 
			\begin{tabular}{c}
				$\AAA_{(S,b+1,bD)}$  \\ \hline $\AAA_{(S,b+1,(b+1)D)}$ \\ \hline $\AAA_{(S,b-1,(b-1)D)}$
			\end{tabular} \\
			\hline
			
			$b>n>0$ & 
			\begin{tabular}{c}
				$(0,0)$  \\ \hline $(0,1)$ \\ \hline $(1,0)$ \\ \hline $(1,1)$
			\end{tabular} & 
			\begin{tabular}{c}
				$\AAA_{(S,b+1,nD)}$  \\ \hline $\AAA_{(S,b+1,(n+1)D)}$ \\ \hline $\AAA_{(S,b-1,nD)}$ \\ \hline $\AAA_{(S,b-1,(n-1)D)}$
			\end{tabular} \\
			\hline
			
		\end{tabular}
	\end{center}

	\bigskip

	Moreover, every $\phi_{kl}$ is induced by a Sarkisov diagram of type II starting from $\pi_{(S,b,nD)}\colon X\to S$ and such that $\phi_{kl} \Autz(X) \phi_{kl}^{-1} = \Autz(Y)$. Finally, there are no other $\Autz(X)$-equivariant Sarkisov diagrams of type II apart from those listed above.
\end{lemma}

\begin{proof}
	By Lemma \ref{geometryofruledsurface}, $S$ admits two disjoint $\Autz(S)$-invariant sections $\sigma=\{z_0=0\}$ and $\sigma' =\{z_1=0\}$, which respectively correspond to the line subbundles $\O_C(D)\subset \O_C\oplus \O_C(D)$ and $\O_C\subset \O_C\oplus \O_C(D)$. Using Lemma \ref{autofiberproduct} and Proposition \ref{imagepi*}, we obtain that the morphism ${\pi_{(S,b,nD)}}_*\colon \Autz(X) \to \Autz(S)$ is surjective. Again by Lemma \ref{geometryofruledsurface}, the $\Autz(X)$-orbit of any point in $X\setminus \pi_{(S,b,nD)}^{-1}(\sigma\cup \sigma')$ has dimension at least two. Therefore, the one-dimensional $\Autz(X)$-orbits are contained in $\pi_{(S,b,nD)}^{-1}(\sigma)$ or $\pi_{(S,b,nD)}^{-1}(\sigma')$.
	
	If $b=n=0$, then $\Autz(X)\simeq \Autz(\AAA_0)\times_{\Autz(C)} \Autz(S)$ by Lemma \ref{autofiberproduct}. Let $\sigma_0$ be the unique minimal section of $\AAA_0$ and $\pi_0\colon \AAA_0\times_C S\to \AAA_0$, $[y_0:y_1 \ ;z_0 : z_1]\mapsto [y_0:y_1]$ be the projection on the first factor. Then by Lemma \ref{geometryofruledsurface}, there exist exactly two $\Autz(X)$-invariant curves $l_{00}^X=\pi_0^{-1}(\sigma_0) \cap \pi_{(S,0,0)}^{-1}(\sigma) = \{ y_0=z_0=0\}$ and $l_{01}^X=\pi_0^{-1}(\sigma_0) \cap \pi_{(S,0,0)}^{-1}(\sigma')=\{y_0=z_1=0\}$. Their blowups induce the square birational maps $\phi_{(S,0,0)}\colon X \dashrightarrow \AAA_{(S,1,0)}$ and $\phi_{(S,0,1)}\colon X \dashrightarrow \AAA_{(S,1,D)}$ in Lemma \ref{lem:induction}.
	
	Assume from now on that $b>0$. The cases $(k,l) = (0,0)$ and $(0,1)$ follow again from Lemma \ref{lem:induction}. It remains to study the cases where $k=1$. 
	
	If $n=0$, taking $z_0=0$ in the transition maps of $\tau\pi_{(S,b,0)}$, we get that the $\PP^1$-bundle $\pi_{(S,b,0)}^{-1}(\sigma)\to \sigma$ has transition maps
	\[
	\begin{array}{ccc}
		U_j \times \PP^1 & \dashrightarrow & U_i\times \PP^1 \\
		(x,[y_0:y_1 ]) & \longmapsto & (x,[y_0:a_{ij}(x)^{-b}y_1]);
	\end{array}
	\]
	hence, it is isomorphic to $\PP(\O_\sigma\oplus \O_\sigma(-bD))$. Since $D$ has infinite order, it follows that $-bD$ is not trivial, so $\PP(\O_\sigma\oplus \O_\sigma(-bD))$ admits two disjoint invariant sections, by Lemma \ref{geometryofruledsurface}. The latters give two disjoint $\Autz(X)$-invariant curves: one is the curve $l_{00}^X$ that we have already considered, and the other one is $l_{10}^X=\{y_1 = z_0 = 0\}$. Moreover, by Lemma \ref{trickBrion}, there is no other $\Autz(X)$-invariant curve in the geometrically ruled surface $\PP(\O_\sigma\oplus \O_\sigma(-bD))$. Taking $z_1=0$ in the transition maps of $\tau\pi_{(S,b,0)}$, we get that $\pi_{(S,b,0)}^{-1}(\sigma')$ is isomorphic to $\AAA_0$, so again by Lemma \ref{trickBrion}, there is no other $\Autz(X)$-invariant curve than $l_{01}^X$ in the surface $\pi_{(S,b,0)}^{-1}(\sigma')$. Thus, $l_{00}^X$, $l_{01}^X$ and $l_{10}^X$ are the only $\Autz(X)$-invariant curves.
	
	The blowup of $l_{10}^X$ followed by the contraction of the strict transform of $\pi_{(S,b,0)}^{-1}(\sigma)$ yields a $\Autz(X)$-equivariant birational map $\phi_{10}$, locally given by 
	\[
	\begin{array}{ccc}
		U_i \times \FF_b & \dashrightarrow & U_i\times \FF_{b-1} \\
		(x,[y_0:y_1\ ; z_0:z_1]) & \longmapsto & (x,[z_0y_0: y_1\ ; z_0:z_1]).
	\end{array}
	\]
	Computing the transition maps, we get that $Y=\AAA_{(S,b-1,0)}$, and this proves the the case $b>n=0$.
	
	The last two cases are proven similarly: the restriction above $\sigma$ and $\sigma'$ gives a $\PP^1$-bundle which is either isomorphic to $\AAA_0$ and it contains a unique $\Autz(X)$-invariant curve, or a decomposable geometrically ruled surface with exactly two disjoint $\Autz(X)$-invariant curves. Finally, using the table with $Y$ as the source variety, we obtain the inverse of $\phi_{kl}$, which is an $\Autz(Y)$-equivariant birational map $Y\dashrightarrow X$; thus, $\phi_{kl}\Autz(X)\phi_{kl}^{-1} = \Autz(Y)$.
\end{proof}

\begin{proposition}\label{prop:A_SbnD}
	The $\PP^1$-bundle $\pi_{(S,b,nD)}\colon \AAA_{(S,b,nD)}\to S$ is indecomposable. Then $\Autz(\AAA_{(S,b,nD)})$ is relatively maximal if and only if $D$ has infinite order, but the pair $(\AAA_{(S,b,nD)},\pi_{(S,b,nD)})$ is not stiff. 
\end{proposition}

\begin{proof}
	First, we consider the case $n=0$. Then the surface $\pi_{(S,b,0)}^{-1}(\sigma') = \{z_1=0\}$ is isomorphic to $\AAA_0$. Therefore, $\pi_{(S,b,0)}$ does not admit two disjoint sections, i.e., it is indecomposable. The same argument holds if $b=n>0$, and considering the surface $\pi_{(S,b,bD)}^{-1}(\sigma) = \{z_0=0\}$.
	
	Assume that $b>n>0$, and by contradiction that $\AAA_{(S,b,nD)}$ is decomposable. Then $\pi_{(S,b,nD)}$ has invariants $(S,b,nD)$ and there exist trivializations of $\tau\pi_{(S,b,nD)}$ such that the transition maps equal
	\[
	\begin{array}{cccc}
		\theta_{ij} \colon & U_j \times \FF_b & \dashrightarrow & U_i\times \FF_b \\
		&(x,[y_0:y_1 \ ;z_0:z_1]) & \longmapsto & (x,[y_0:a_{ij}(x)^ny_1 \ ;z_0:a_{ij}(x)z_1]).
	\end{array}
	\]
	The surface $\pi_{(S,b,nD)}^{-1}(\sigma') = \{z_1=0\}$ is now isomorphic to $\PP(\O_{\sigma'}\oplus \O_{\sigma'}(nD))$. The section corresponding to the line subbundle $\O_{\sigma'}$ yields an $\Autz(X)$-invariant curve $l_{11}^X$, which is the same as the one in the last line of the table of Lemma \ref{equivariantlinksfromAAA}. This gives a birational map $\phi_{11}\colon X\dashrightarrow \AAA_{(S,b-1,(n-1)D)}$, locally given by 
	\[
	\begin{array}{cccc}
		\phi_{11,i}\colon & U_i \times \FF_b & \dashrightarrow & U_i\times \FF_{b-1} \\
		& (x,[y_0:y_1\ ; z_0:z_1]) & \longmapsto & (x,[z_1y_0: y_1\ ; z_0:z_1]).
	\end{array}
	\]
	Computing $\phi_{11,i}\theta_{ij}\phi_{11,j}^{-1}$, we obtain new transition maps for $\AAA_{(S,b-1,(n-1)D)}$, which depend only on the cocyles $a_{ij}\in \O_C(U_{ij})^*$. This implies that $\AAA_{(S,b-1,(n-1)D)}$ is also decomposable by Lemma \ref{FFb-bundle}. Repeating this process finitely many times, we obtain by induction that $\AAA_{(S,b-n,0)}$ is decomposable, which is a contradiction.
	
	Assume that $D$ has infinite order. By Theorem \ref{Floris}, every $\Autz(\AAA_{(S,b,nD)})$-equivariant birational factors through a sequence of $\Autz(\AAA_{(S,b,nD)})$-equivariant Sarkisov links. The table given in Lemma \ref{equivariantlinksfromAAA} lists all such links of type II, each of which conjugates $\Autz(\AAA_{(S,b,nD)})$ with the automorphism group of the target. Therefore, $\Autz(\AAA_{(S,b,nD)})$ is relatively maximal and is conjugate to any other $\Autz(\AAA_{(S,b',n'D)})$, where $b'\geq 0$ and $0\leq n' \leq b'$. 
	
	Finally, if $D$ has finite order, we use the $\Autz(X)$-equivariant square birational map $\phi_{01}$ inductively to construct an $\Autz(X)$-equivariant square birational map $\AAA_{(S,b,nD)} \dashrightarrow \AAA_{(S,b',0)}$, where $b'>0$, and we conclude with Lemma \ref{L-torsion}.
\end{proof}

Combining the previous results on $\AAA_0\times_C S$, we obtain the main result of this section:

\begin{proposition}\label{A0xS}
	Let $X=\AAA_0\times_C S$, let $p_1$ and $p_2$ denote the projections onto $\AAA_0$ and onto $S$, respectively. Let $\pi'\colon X'\to S'$ be a $\PP^1$-bundle over a geometrically ruled surface $S'$. Then the following hold:
	\bigskip
	\begin{enumerate}
		\item With respect to $p_1$, $\Autz(X)$ is relatively maximal and $\Autz(X)$ is conjugate to $\Autz(X')$ via a square birational map if and only if $S'=\AAA_0$ and $X'\cong \PP(\O_{\AAA_0}\oplus \O_{\AAA_0}(b\sigma +\tau^*(D)))$, where $b\in \mathbb{Z}$ and $\sigma$ denotes the unique minimal section of $\tau\colon \AAA_0 \to C.$
		\bigskip 
		\item With respect to $p_2$, $\Autz(X)$ is relatively maximal if and only if $D$ has infinite order. In that case, $\Autz(X)$ is conjugate to $\Autz(X')$ via a square birational map if and only if $S' = S$ and $X'\cong \AAA_{(S,b,nD)}$, where $b\geq 0$ and $0\leq n\leq b$.
	\end{enumerate}
\end{proposition}

\begin{proof}
	This follows from Propositions \ref{prop:decomposable_over_A0} and \ref{prop:A_SbnD}.
\end{proof}

\subsection{The fiber product $\PP(\O_C\oplus \O_C(D))\times_C \PP(\O_C\oplus \O_C(E))$}

~\bigskip

Recall that $D,E\in \Pic^0(C)$ are non-trivial. In this section, we denote by $\tau_D\colon S_D \to C$ and $\tau_E \colon S_E\to C$ the $\PP^1$-bundles $\PP(\O_C\oplus \O_C(D))$ and $\PP(\O_C\oplus \O_C(E))$, respectively. We denote by $\sigma_D$ and $\sigma_E$ the sections of $\tau_D$ and of $\tau_E$ corresponding to the line subbundles $\O_C(D)\subset \O_C\oplus \O_C(D)$ and $\O_C(E) \subset \O_C\oplus \O_C(E)$, respectively.

\bigskip

The natural projections $\pi_D \colon S_D \times_C S_E\to S_D$ and $\pi_E \colon S_D \times_C S_E\to S_E$ are precisely the decomposable $\PP^1$-bundles $\PP(\O_{S_D} \oplus \O_{S_D}(\tau_D^*(E)))$ and $\PP(\O_{S_E} \oplus \O_{S_E}(\tau_E^*(D)))$. We first start with the following elementary observation:

\begin{lemma}\label{lem:D+nEtrivial}
		If $E$ is a multiple of $D$, then $\Autz(S_D \times_C S_E)$ is not relatively maximal with respect to the $\PP^1$-bundle structure $\pi_D$.
\end{lemma}

\begin{proof}
	Let $X=S_D \times_C S_E$ and assume first that $nD+E$ is trivial for some $n>0$. By Lemma \ref{FFb-bundle}, we can choose the trivializations of $\tau_D\pi_D$ such that the transition maps equal
	\[
	\begin{array}{ccc}
		U_j \times \FF_0 & \dashrightarrow & U_i \times \FF_0 \\
		(x,[y_0:y_1 \ ;z_0:z_1]) &\longmapsto & (x,[y_0:e_{ij}(x)y_1 \ ;z_0:d_{ij}(x)z_1]),
	\end{array}
	\]
	where $d_{ij},e_{ij}\in \O_C(U_{ij})^*$ are the cocycles of $\O_C(D)$ and $\O_C(E)$, respectively. Under this choice of trivializations, we have $\sigma =\{z_1=0\}\subset S_D$ and $\sigma_E =\{y_0=0\} \subset S_E$. By Lemmas \ref{geometryofruledsurface} and \ref{autofiberproduct}, the curve $l_{01}^X=\{y_0=z_1=0\}$ is $\Autz(X)$-invariant. Its blowup followed by the contraction of the strict transform of $\pi_D^{-1}(\sigma)$ yields an $\Autz(X)$-equivariant square birational map $X\dashrightarrow X_1$, locally given by
	\[
	\begin{array}{ccc}
		U_i \times \FF_0 & \dashrightarrow & U_i \times \FF_1 \\
		(x,[y_0:y_1 \ ;z_0:z_1]) &\longmapsto & (x,[y_0:z_1y_1 \ ;z_0:z_1]),
	\end{array}
	\]
	and where $\pi_1\colon X_1 \to S_D$ is a $\PP^1$-bundle. Computing the transition maps of the $\FF_1$-bundle $\tau_D\pi_1\colon X_1 \to C$, we get 
	\[
	\begin{array}{ccc}
		U_j \times \FF_1 & \dashrightarrow & U_i \times \FF_1 \\
		(x,[y_0:y_1 \ ;z_0:z_1]) &\longmapsto & (x,[y_0:d_{ij}(x)e_{ij}(x)y_1 \ ;z_0:d_{ij}(x)z_1]).
	\end{array}
	\]
	In particular, by Lemma \ref{FFb-bundle}, $\pi_1$ is the decomposable $\PP^1$-bundle $\PP(\O_{S_D}\oplus \O_{S_D}(\sigma_D +{\tau_D}^*(D+E)))$. 
	
	Let $b\in \mathbb{Z}$ and $X_b = \PP(\O_{S_D}\oplus \O_{S_D}(b\sigma_D +\tau_E^*(bD+E)))$. As above, for each $b>0$, there exists an $\Autz(X_b)$-invariant curve $l^{X_b}_{01}=\{y_0 = z_1 =0\}$, which induces an $\Autz(X_b)$-equivariant square birational map $\phi_b\colon X_b \dashrightarrow X_{b+1}$. By induction, for the integer $n>0$ such that $nD+E$ is trivial, we get an $\Autz(X)$-equivariant square birational map $X\dashrightarrow X_n$, where the $\FF_n$-bundle $\tau_D\pi_n\colon X_n\to C$ has transition maps
	\[
	\begin{array}{ccc}
		U_j \times \FF_n & \dashrightarrow & U_i \times \FF_n \\
		(x,[y_0:y_1 \ ;z_0:z_1]) &\longmapsto & (x,[y_0:d_{ij}(x)^ne_{ij}(x)y_1 \ ;z_0:d_{ij}(x)z_1]).
	\end{array}
	\]
	Since $D+nE$ is trivial, there exist $\mu_j\in \O_C(U_j)^*$ and $\mu_i\in \O_C(U_i)^*$ such that $d_{ij}^n e_{ij}= \mu_i^{-1} \mu_j$. Using Lemma \ref{Sisomorph}, we can replace $d_{ij}^ne_{ij}$ by $1$ in the transition maps above. 
	
	Now, to show that $\Autz(X)$ is not relatively maximal, it is sufficient to see that the curve $l^{X_n}_{11}=\{y_1=z_1=0\} \subset X_n$ is not $\Autz(X_n)$-invariant. Let $\mu \in \mathbb{G}_a$, for each $i$, we define the automorphism 
	\[
	\begin{array}{cccc}
		f_{\mu,i} \colon & U_i \times \FF_n & \dashrightarrow & U_i \times \FF_n \\
		& (x,[y_0:y_1 \ ;  z_0:z_1]) & \longmapsto & (x,[y_0:y_1 + \mu z_0^by_0\ ;  z_0:z_1]).
	\end{array}
	\]
	These automorphisms commute with the transition maps of $\tau_E\pi_n\colon X_n\to C$, so they give an automorphism $f_\mu \in \Autz(X_n)$ which does not preserve the curve $l^{X_n}_{11}$. 
	
	The proof for $n<0$ is similar. Instead at each step, we blowup the $\Autz(X_b)$-invariant curve $l^{X_{b}}_{11} = \{y_1 = z_1 =0 \}$ and contract the strict transform of $\pi_b^{-1}(\sigma)$. This yields an $\Autz(X_b)$-equivariant birational map $X_b \dashrightarrow X_{b-1}$. The end of the proof follows as in the case $n>0$.
\end{proof}

\begin{lemma}\label{invariantcurvesDDD}
	Let $b\in \mathbb{Z}$, let $E'\in \Pic^0(C)$ non-trivial which is not a multiple of $D$. Let 
	\[
	X=\PP(\O_{S_D}\oplus \O_{S_D}(b\sigma_D + {\tau_D}^*( E')))
	\]
	be the decomposable $\PP^1$-bundle with structure morphism $\pi\colon X\to S_D$. Then there exist exactly four $\Autz(X)$-invariant curves $l_{kl}^X$, where $k,l\in \{0,1\}$. For each $k,l$, the blowup of $l^X_{kl}$ followed by the contraction of the strict transform of ${\pi^{-1}}(\pi(l^X_{kl}))$ yields an $\Autz(X)$-equivariant square birational map $\phi_{kl}\colon X\dashrightarrow X_{kl}$, where $X_{kl}$ is given by the following table:
	\bigskip
	\begin{center}
		\begin{tabular}{|c|c|}
			\hline 
			$(k,l)$ & $X_{kl}$ \\
			\hline 
			$(0,0)$ &  $\PP(\O_{S_D}\oplus \O_{S_D}((b+1)\sigma_D + {\tau_D}^*( E')))$ \\
			\hline 
			$(0,1)$ &  $\PP(\O_{S_D}\oplus \O_{S_D}((b+1)\sigma_D + {\tau_D}^*( E'+D)))$ \\
			\hline 
			$(1,0)$ &  $\PP(\O_{S_D}\oplus \O_{S_D}((b-1)\sigma_D + {\tau_D}^*( E')))$ \\			
			\hline 
			$(1,1)$ &  $\PP(\O_{S_D}\oplus \O_{S_D}((b-1)\sigma_D + {\tau_D}^*( E'-D)))$ \\			
			\hline
		\end{tabular}
	\end{center}
	
	\bigskip
	
	Moreover, each $\phi_{kl}$ is induced by an $\Autz(X)$-equivariant Sarkisov diagram of type $II$ and we have 
	\[
	\phi_{kl} \Autz(X) \phi_{kl}^{-1} =\Autz(X_{kl}).
	\]
	Finally, there are no other $\Autz(X)$-equivariant Sarkisov diagrams of type II starting from $\pi$ apart from those listed above.
\end{lemma}

\begin{proof}
	Let $f\in \Autz(S_D)$. Since $E'$ has degree zero, it is invariant by translations; hence, $f^*\tau_D^*(E') \sim \tau_D^*(E')$ by Blanchard's lemma (\ref{blanchard}). By Lemma \ref{geometryofruledsurface}, the section $\sigma_D$ is also $f$-invariant. This implies that $f^*X$ is $S_D$-isomorphic to $X$, so by Proposition \ref{Autalgebraicgroup}, the morphism $\pi_*\colon \Autz(X)\to \Autz(S_D)$ is surjective. In particular, the $\Autz(X)$-orbits have dimension at least one and the one-dimensional $\Autz(X)$-orbits lie in the geometrically ruled surfaces $\pi^{-1}(\sigma_D)$ and $\pi^{-1}(\sigma)$, where $\sigma$ is the other minimal section of $S_D$ disjoint from $\sigma_D$.
	
	By Lemma \ref{FFb-bundle}, we can choose the trivializations of $\tau_D\pi$ such that the transition maps equal
	\[
	\begin{array}{cccc}
		\theta_{ij}\colon &U_j \times \FF_b & \dashrightarrow & U_i\times \FF_b \\
		&(x,[y_0:y_1\ ; z_0:z_1]) & \longmapsto & (x,[y_0:\lambda_{ij}(x)y_1\ ; z_0:d_{ij}(x)z_1]),
	\end{array}
	\]
	where $\lambda_{ij},d_{ij}\in \O_C(U_{ij})^*$ are respectively the cocycles of the line bundles $\O_C(E')$ and $\O_C(D)$. With this choice of trivializations, we have that $\sigma_D = \{z_0=0\} \subset S_D$ and $\sigma = \{z_1=0\} \subset S_D$. 
	
	Taking $z_0=0$ in the transition maps above, we obtain that $\pi^{-1}(\sigma_D)$ is isomorphic to the decomposable $\PP^1$-bundle $\PP(\O_{\sigma_D} \oplus \O_{\sigma_D} (E'-bD))$. Since $E'-bD$ is non-trivial of degree zero by assumption, the curves $l^X_{00}=\{y_0=z_0=0\}$ and $l^X_{10}=\{y_1=z_0=0\}$ are $\Autz(X)$-invariant. Since $\pi$ is decomposable, the connected group $\mathbb{G}_m$ is contained in the kernel of the morphism $\pi_*\colon \Autz(X)\to \Autz(S_D)$ and acts transitively on the complement of those two curves in $\pi^{-1}(\sigma_D)$. Hence, $l^X_{00}$ and $l^X_{10}$ are the only $\Autz(X)$-invariant curves lying in $\pi^{-1}(\sigma_D)$. Similarly, taking $z_1=0$ in the transition maps above, we get that $l^X_{01}=\{y_0=z_1=0\}$ and $l^X_{11}=\{y_1=z_1=0\}$ are precisely the $\Autz(X)$-invariant curves contained in the surface $\pi^{-1}(\sigma)$. This implies that the curves $l^X_{kl}=\{y_k=z_l=0\}$, with $k,l\in \{0,1\}$, are the only $\Autz(X)$-invariant curves.
	
	 The blowup of $l^{X}_{0l}$ followed by the contraction of the strict transform of $\{z_l =0\}$ is locally given by 
	 \[
	 \begin{array}{ccc}
		U_i \times \FF_b & \dashrightarrow & U_i \times \FF_{b+1} \\
		(x,[y_0:y_1\ ; z_0:z_1]) & \longmapsto & (x,[y_0:z_ly_1\ ; z_0:z_1]).
	 \end{array}
	 \]
	 Similarly, the blowup of $l^X_{1l}$ followed by the contraction of the strict transform of $\{z_l =0\}$ is locally given by 
	 	 \[
	 \begin{array}{ccc}
	 	U_i \times \FF_b & \dashrightarrow & U_i \times \FF_{b-1} \\
	 	(x,[y_0:y_1\ ; z_0:z_1]) & \longmapsto & (x,[z_ly_0:y_1\ ; z_0:z_1]).
	 \end{array}
	 \]
	 In both cases, $\phi_{kl}\colon X\dashrightarrow X_{kl}$ is $\Autz(X)$-equivariant. By computing the transition maps $\phi_{kl}\theta_{ij}\phi_{kl}^{-1}$, we get the different $X_{kl}$ listed in the table, according to the values of the pair $(k,l)$. 
	 
	 We denote by $\pi_{kl}\colon X_{kl} \to S_D$ the corresponding $\PP^1$-bundle structure.
	 The base locus of $\phi_{00}^{-1}$ is the curve $l^{X_{00}}_{10} = \{y_1 = z_0 =0\} \subset X_{00}$. Taking $z_0=0$ in the transition maps of $\tau_D\pi_{00}$, we obtain that the geometrically ruled surface $\pi_{00}^{-1}(\pi_{00}(l^{X_{00}}_{10} ))$ is isomorphic to $\PP(\O_{\sigma_D}\oplus \O_{\sigma_D}(E'-(b+1)D))$. Since $E'-(b+1)D$ is non-trivial of degree zero, it follows by Lemma \ref{geometryofruledsurface} that $l^{X_{00}}_{10}$ is $\Autz(X_{00})$-invariant. 
	A similar argument follows for $\pi_{01}$, $\pi_{10}$ and $\pi_{11}$. For each pair $(k,l)$, we obtain that $\phi_{kl} \Autz(X) \phi_{kl}^{-1} =\Autz(X_{kl})$. Finally, every $\Autz(X)$-equivariant Sarkisov diagram of type I and II are uniquely determined by the blowup of an $\Autz(X)$-orbit $l_{kl}^X$ and the $2$-ray game.
\end{proof}

Using Lemma $\ref{invariantcurvesDDD}$, we determine the conjugacy class of $\Autz(S_D \times_C S_E)$ in the following proposition:

\begin{proposition}\label{prop:SDxSE}
	Let $X=S_D \times_C S_{E} $ and $\pi_D$ be the projection onto $S_D$. Assume that $E$ is not a multiple of $D$. Then the following hold:
	\begin{enumerate}
		\item\label{prop:SDxSE.1} $\Autz(X)$ is relatively maximal with respect to $\pi_D$ but the pair $(X,\pi_D)$ is not stiff.
		\item\label{prop:SDxSE.2} Let $X' \to T$ be a $\PP^1$-bundle over a geometrically ruled surface $T$. With respect to $\pi_D$, there exists an $\Autz(X)$-equivariant square birational map $X\dashrightarrow X'$ if and only if $T\cong S_D$ and there exist $b,n\in \mathbb{Z}$ such that \[
		X'\cong \PP(\O_{S_D}\oplus \O_{S_D}(b\sigma_D +{\tau_D}^{*}(nD+E)))).
		\]
	\end{enumerate}
\end{proposition}

\begin{proof}
	Using Lemma \ref{invariantcurvesDDD} inductively, there exists an $\Autz(X)$-equivariant square birational map $\phi\colon X\dashrightarrow X'$ if and only if there exist $n,b\in \mathbb{Z}$ such that $X' \cong \PP(\O_{S_D}\oplus \O_{S_D}(b\sigma_D +{\tau_{D}}^*(nD+E)))$. In this case, $\phi \Autz(X) \phi^{-1} = \Autz(X')$.
	Therefore, $\Autz(X)$ is relatively maximal and $(X,\pi_D)$ is not stiff.
\end{proof}

\begin{proposition}\label{SDxSE}
	 Let $X=S_D \times_C S_{E} $. Then the following hold: 
	 \begin{enumerate}
		\item If there exists $(n,m)\in \mathbb{Z}^2$ coprime such that $n D + m E$ is trivial, then $\Autz(X)$ is not a maximal connected algebraic subgroup of $\Bir(X)$.
		\item Else, $\Autz(X)$ is conjugate to $\Autz(X')$, where $\pi\colon X'\to S$ is a $\PP^1$-bundle over a geometrically ruled surface $S$, if and only if there exist $b\in \mathbb{Z}$ and 
		$
		\begin{pmatrix}
			\alpha & \beta \\ \gamma & \delta
		\end{pmatrix} 
		\in 
		\mathrm{SL}_2(\mathbb{Z})
		$
		such that
		\[
		X'\cong 
		\PP(\O_{S} \oplus \O_{S}(b\sigma + \tau^*(\alpha D+\beta E))),
		\]
		where $\tau\colon S\to C$ is the decomposable $\PP^1$-bundle $\PP(\O_C \oplus \O_C(\gamma D + \delta E))$ and $\sigma$ is the section of $\tau$ corresponding to the line subbundle $\O_C(\gamma D + \delta E)\subset \O_C \oplus \O_C(\gamma D + \delta E)$.
	 \end{enumerate}
\end{proposition}
	
\begin{proof}
	First, we work with the $\PP^1$-bundle structure $\pi_D\colon X\to S_D$.
	Let $(b,n_1)\in \mathbb{Z}^2$ and let
	\[
	X'= \PP(\O_{S_D}\oplus \O_{S_D}(b\sigma_D +{\tau_D}^{*}(n_1D+E)))).
	\]
	If $E$ is a multiple of $D$, then $\Autz(X)$ is not relatively maximal by Lemma \ref{lem:D+nEtrivial}. Else, by Proposition $\ref{prop:SDxSE}$, the groups $\Autz(X)$ and $\Autz(X')$ are conjugate via a square birational map. Since this holds for arbitrary $b\in \mathbb{Z}$, we can in particular choose $b=0$, and we get that $\Autz(X)$ is conjugate to $\Autz(S_{D} \times_C S_{n_1D+E})$, for every $n_1 \in \mathbb{Z}$. Considering the case $b=0$ is the only way to construct new $\Autz(X)$-equivariant birational maps to $\PP^1$-bundles over geometrically ruled surfaces, as $\FF_0$-bundles are equipped with another $\PP^1$-bundle structure.
	
	We associate to the fiber product $S_D\times_C S_E$ the pair of divisors $(D,E)\in \Pic^0(C)^2$. Notice that at the first step above, the pair of divisors $(D,E)$ is replaced by $(D,n_1D+E)$, where $n_1\in \mathbb{Z}$ is any integer. The $\PP^1$-bundles obtained by a sequence of Sarkisov of links are completely determined by the pair of divisors $(D,E)$ and the integer $b\in \mathbb{Z}$.
	
	Next, we work with the other projection from $S_{D} \times_C S_{n_1D+E}$ onto $S_{n_1D+ E}$. If $D$ is a multiple of $n_1D+E$ for an integer $n_1 \in \mathbb{Z}$, then $\Autz(X)$ is not a maximal connected algebraic subgroup of $\Bir(X)$ by Lemma \ref{lem:D+nEtrivial}. Else, using Lemma \ref{lem:D+nEtrivial}, we can replace the pair $(D,n_1D+E)$ by $(D + n_2(D+n_1E),D+n_1E)$, where $n_2 \in \mathbb{Z}$ is any integer. Then we proceed by induction. At each step, we replace the pair of divisors $(D',E')$ by $(D'+n'E',E')$ or by $(D',n'D'+E')$. Equivalently, we can multiply the vector $(D',E')$ by the matrices:
	\[
	\begin{pmatrix}
		1 & n' \\ 0 & 1
	\end{pmatrix}, \text{ or}
	\begin{pmatrix}
		1 & 0 \\ n' & 1
	\end{pmatrix},
	\] 
	for any $n'\in \mathbb{Z}$. This concludes the proof, since those matrices generate the group $\mathrm{SL}_2(\mathbb{Z})$.
	\end{proof}
	
	We now prove the main proposition of this section:
	
	\begin{proof}[Proof of Proposition \ref{main_prop:S}]
		Assume that $\Autz(X)$ is relatively maximal. If $\tau \pi$ is an $\FF_b$-bundle, with $b>0$, then there exist a $\PP^1$-bundle $\pi'\colon X' \to S$ such that $\tau\pi'$ is an $\FF_0$-bundle and
		an $\Autz(X)$-equivariant square birational map $X\dashrightarrow X'$ by Proposition \ref{equivariantto0.SL}. This reduces to the case of $\FF_0$-bundles arising from a $\PP^1$-bundle over $S$.
		
		If $b=0$, the first two cases follow from Proposition \ref{FF0viaCxPP1} and Lemma \ref{lem:SLA1max}. The last two cases were studied through Propositions \ref{A0xS} and \ref{prop:SDxSE}: the pair $(X,\pi)$ is not stiff and the conjugacy class of $\Autz(X)$ is given as follows.
		
		Let $\pi'\colon X'\to T$ be a $\PP^1$-bundle over a geometrically ruled surface $T$. By Proposition \ref{A0xS}, there exists an $\Autz(S\times_C \AAA_0)$-equivariant square birational map $S\times_C \AAA_0\dashrightarrow X'$ if and only if $T\cong S$ and $X'\cong \AAA_{(S,b,nD)}$ as in case \ref{main_prop:S.2} \ref{main_prop:S.2.i}. By Proposition \ref{prop:SDxSE}, there exists an $\Autz(S\times_C \PP(\O_C\oplus \O_C(E)))$-equivariant square birational map $S\times_C \PP(\O_C\oplus \O_C(E))\dashrightarrow X'$ if and only if $T\cong S$ and $X'\cong \PP(\O_{S}\oplus \O_{S}(b\sigma \oplus {\tau}^{*}(nD+E)))$ for some $b,n\in \mathbb{Z}$.
	\end{proof}
	
\section{Proof of the main results}\label{section:proofs}

\begin{proof}[Proof of Theorems $\ref{thmA}$ and $\ref{thmB}$]
	Let $\tau\colon S\to C$ and $\pi\colon X\to S$ be $\PP^1$-bundles such that $\Autz(X)$ is relatively maximal. By Proposition \ref{removal jumping}, we can assume that $\tau\pi\colon X\to C$ is an $\FF_b$-bundle. Using Theorem \ref{dim2max} with Proposition \ref{basesurfacemaximal}, we also get that $S$ is one of the following geometrically ruled surfaces:
	$C\times \PP^1$, or $\AAA_0$, or $\AAA_1$, or $\PP(\O_C\oplus \O_C(D))$ for some non-trivial divisor $D$ of degree zero. The three latter cases occur only when $C$ is an elliptic curve. To conclude, apply Propositions \ref{main_prop:CtimesP^1}, \ref{main_prop:A_1}, \ref{main_prop:A0}, and \ref{main_prop:S}.
\end{proof}

\begin{proof}[Proof of Proposition \ref{propB}]
	If $g\geq 2$ and $X = C\times \PP^1 \times \PP^1$, then $\Autz(C)$ is trivial and $\Autz(X)\cong \PGL_2(\kk)^2$.
	
	Assume now that $g=1$ and we consider first the case where $X =S_1\times_C S_2$ is a fiber product of two geometrically ruled surfaces $\tau_1\colon S_1 \to C$ and $\tau_2\colon S_2\to C$. We denote by $\pi_1$ and $\pi_2$ the projections from $X$ onto the first and second factors. 
	By Proposition \ref{candidateFF0max}, $(\tau_1 \pi_1)_* = (\tau_2 \pi_2)_*$ and surjects onto $\Autz(C)$. By Lemma \ref{fiberproduct}, $\pi_1$ (resp. $\pi_2$) is decomposable if and only if $\tau_2$ (resp. $\tau_1$) is decomposable. Moreover, by Lemma \ref{autofiberproduct}, 
	\[
	\Autz(X) \simeq \Autz(S_1)\times_{\Autz(C)} \Autz(S_2),
	\]
	and $\ker({\pi_1}_*) = \ker({\tau_2}_*)$ and  $\ker({\pi_2}_*) = \ker({\tau_1}_*)$. In each case, we know $\ker({\tau_1}_*)$ and $\ker({\tau_2}_*)$ from Lemma \ref{geometryofruledsurface}, and this determines the dimension of $\Autz(X)$ and its orbits. This proves the statement for all cases of (I), except (vii) and (viii). The last two cases (vii) and (viii), where $X$ is not a fiber product of two geometrically ruled surfaces, follow from Propositions \ref{modulioverCtimesPP^1} and \ref{SarkisovoverA1}.
\end{proof}

\begin{proof}[Proof of Corollary \ref{coroC}]
	If $g\geq 2$, the pair $(C\times \PP^1\times \PP^1,\Autz(C\times \PP^1\times \PP^1))$ is superstiff; hence, $\Autz(C\times \PP^1\times \PP^1)$ is a maximal connected algebraic subgroup. 
	
	Assume now that $g=1$. Cases (a) and (b) follow from Propositions \ref{FF0viaCxPP1},  \ref{A_1A_1max}, and \ref{A_0A_1notmax}. The last case, statement (c), follows from Theorem \ref{thmB}. Indeed, there is no $\Autz(X)$-equivariant birational map from $X$ to an $\FF_1$-bundle over $C$ in cases \ref{thmA.vi}, \ref{thmA.ix} with $b\neq 1$, and \ref{thmA.x}. Thus, in all those cases, there is no $\Autz(X)$-equivariant birational map to a $\PP^2$-bundle.
\end{proof}

	\bibliographystyle{alpha} 
	\bibliography{bib} 	
	
	%%%%%%%%%%%%%%%%%%%%%%%%%%%%%%%%%%%%%%%%%%%%%%%%%%%%%%%%%%%%%%
	
\end{document}